\pdfoutput=1
\RequirePackage{ifpdf}
\ifpdf 
\documentclass[pdftex]{sigma}
\else
\documentclass{sigma}
\fi

\usepackage{mathrsfs}
\usepackage{enumitem}
\usepackage{dsfont}

\usepackage{braket}
\usepackage{tikz}
\usetikzlibrary{matrix}

\numberwithin{equation}{section}

\newtheorem{thm}{Theorem}[section]
\newtheorem{cor}[thm]{Corollary}
\newtheorem{lem}[thm]{Lemma}
\newtheorem{prop}[thm]{Proposition}

\newtheorem{conj}[thm]{Conjecture}

\theoremstyle{definition}
\newtheorem{defn}[thm]{Definition}
\newtheorem{rem}[thm]{Remark}

\newtheorem{Example}[thm]{Example}

\newcommand{\nc}{\newcommand}
\newcommand{\rnc}{\renewcommand}

\renewcommand{\Bbb}{\mathbb}

\newcommand{\id}{\normalfont{\text{id}}}
\nc{\Xp}[1]{X^+(#1)}
\nc{\Xm}[1]{X^-(#1)}
\nc{\on}{\operatorname}
\nc{\ch}{\mbox{ch}}
\nc{\Z}{{\bold Z}}
\nc{\J}{{\mathcal J}}
\nc{\C}{{\bold C}}
\nc{\Q}{{\bold Q}}
\nc{\oC}{{\widetilde{C}}}
\nc{\oc}{{\tilde{c}}}
\nc{\ocI}{ \overline{\mathcal I}}
\nc{\og}{{\tilde{\gamma}}}
\nc{\lC}{{\overline{C}}}
\nc{\lc}{{\overline{c}}}
\nc{\Rt}{{\tilde{R}}}

\nc{\tW}{{\mathsf{W}}}
\nc{\tG}{{\mathsf{G}}}
\nc{\tI}{{\textsf{I}}}

\nc{\tE}{{\textsf{E}}}
\nc{\tF}{{\textsf{F}}}
\nc{\tK}{{\textsf{K}}}

\nc{\tx}{{\textsf{x}}}
\nc{\tho}{{\textsf{h}}}
\nc{\tk}{{\textsf{k}}}
\nc{\tep}{{\bf{\mathcal E}}}

\nc{\te}{{\textsf{e}}}
\nc{\tf}{{\textsf{f}}}

\nc{\odel}{{\overline{\delta}}}

\nc{\N}{{\Bbb N}}
\nc\beq{\begin{equation}}
 \nc\enq{\end{equation}}
\nc\lan{\langle}
\nc\ran{\rangle}
\nc\bsl{\backslash}
\nc\mto{\mapsto}
\nc\lra{\leftrightarrow}
\nc\hra{\hookrightarrow}
\nc\sm{\smallmatrix}
\nc\esm{\endsmallmatrix}
\nc\sub{\subset}
\nc\ti{\tilde}
\nc\fra{\frac}
\nc\und{\underline}
\nc\ov{\overline}
\nc\ot{\otimes}

\nc\ochi{\overline{\chi}}
\nc\bbq{\bar{\bq}_l}
\nc\bcc{\thickfracwithdelims[]\thickness0}
\nc\ad{\operatorname{ad}}
\nc\Ad{\operatorname{Ad}}
\nc\Hom{\operatorname{Hom}}
\nc\End{\operatorname{End}}
\nc\Ind{\operatorname{Ind}}
\nc\Res{\operatorname{Res}}
\nc\Ker{\operatorname{Ker}}
\rnc\Im{\operatorname{Im}}
\nc\sgn{\operatorname{sgn}}
\nc\tr{\operatorname{tr}}
\nc\Tr{\operatorname{Tr}}
\nc\supp{\operatorname{supp}}
\nc\card{\operatorname{card}}
\nc\bst{{}^\bigstar }
\nc\he{\heartsuit}
\nc\clu{\clubsuit}
\nc\spa{\spadesuit}
\nc\di{\diamond}
\nc\cW{{\mathcal W}}
\nc\cG{{\mathcal G}}
\nc\cZ{{\mathcal Z}}
\nc\ocW{\overline{\mathcal W}}
\nc\ocZ{\overline{\mathcal Z}}
\nc\al{\alpha}
\nc\bet{\beta}
\nc\ga{\gamma}
\nc\de{\delta}
\nc\ep{\epsilon}
\nc\io{\iota}
\nc\om{\omega}
\nc\si{\sigma}
\rnc\th{\theta}
\nc\ka{\kappa}
\nc\la{\lambda}
\nc\ze{\zeta}

\nc\vp{\varpi}
\nc\vt{\vartheta}
\nc\vr{\varrho}

\nc\odelta{\overline{\delta}}
\nc\Ga{\Gamma}
\nc\De{\Delta}
\nc\Om{\Omega}
\nc\Si{\Sigma}
\nc\Th{\Theta}
\nc\La{\Lambda}

\nc\boa{\bold a}
\nc\bob{\bold b}
\nc\boc{\bold c}
\nc\bod{\bold d}
\nc\boe{\bold e}
\nc\bof{\bold f}
\nc\bog{\bold g}
\nc\boh{\bold h}
\nc\boi{\bold i}
\nc\boj{\bold j}
\nc\bok{\bold k}
\nc\bol{\bold l}
\nc\bom{\bold m}
\nc\bon{\bold n}
\nc\boo{\bold o}
\nc\bop{\bold p}
\nc\boq{\bold q}
\nc\bor{\bold r}
\nc\bos{\bold s}
\nc\bou{\bold u}
\nc\bov{\bold v}
\nc\bow{\bold w}
\nc\boz{\bold z}

\nc\ba{\bold A}
\nc\bb{\bold B}
\nc\bc{\bold C}
\nc\bd{\bold D}
\nc\be{\bold E}
\nc\bg{\bold G}
\nc\bh{\bold H}
\nc\bi{\bold I}
\nc\bj{\bold J}
\nc\bk{\bold K}
\nc\bl{\bold L}
\nc\bm{\bold M}
\nc\bn{\bold N}
\nc\bo{\bold O}
\nc\bp{\bold P}
\nc\bq{\bold Q}
\nc\br{\bold R}
\nc\bs{\bold S}
\newcommand{\bt}{{\bf {\mathsf{T}}}}
\nc\bu{\bold U}
\nc\bv{\bold V}
\nc\bw{\bold W}
\nc\bz{\bold Z}
\nc\bx{\bold X}

\nc\ca{\mathcal A}
\nc\cb{\mathcal B}
\nc\cc{\mathcal C}
\nc\cd{\mathcal D}
\nc\ce{\mathcal E}
\nc\cf{\mathcal F}
\nc\cg{\mathcal G}
\rnc\ch{\mathcal H}
\nc\ci{\mathcal I}
\nc\cj{\mathcal J}
\nc\ck{\mathcal K}
\nc\cl{\mathcal L}
\nc\cm{\mathcal M}
\nc\cn{\mathcal N}
\nc\co{\mathcal O}
\nc\cp{\mathcal P}
\nc\cq{\mathcal Q}
\nc\car{\mathcal R}
\nc\cs{\mathcal S}
\nc\ct{\mathcal T}
\nc\cu{\mathcal U}
\nc\cv{\mathcal V}
\nc\cz{\mathcal Z}
\nc\cx{\mathcal X}
\nc\cy{\mathcal Y}

\nc\e[1]{E_{#1}}
\nc\ei[1]{E_{\delta - \alpha_{#1}}}
\nc\esi[1]{E_{s \delta - \alpha_{#1}}}
\nc\eri[1]{E_{r \delta - \alpha_{#1}}}
\nc\ed[2][]{E_{#1 \delta,#2}}
\nc\ekd[1]{E_{k \delta,#1}}
\nc\emd[1]{E_{m \delta,#1}}
\nc\erd[1]{E_{r \delta,#1}}

\nc\ef[1]{F_{#1}}
\nc\efi[1]{F_{\delta - \alpha_{#1}}}
\nc\efsi[1]{F_{s \delta - \alpha_{#1}}}
\nc\efri[1]{F_{r \delta - \alpha_{#1}}}
\nc\efd[2][]{F_{#1 \delta,#2}}
\nc\efkd[1]{F_{k \delta,#1}}
\nc\efmd[1]{F_{m \delta,#1}}
\nc\efrd[1]{F_{r \delta,#1}}

\nc\fa{\frak a}
\nc\fb{\frak b}
\nc\fc{\frak c}
\nc\fd{\frak d}
\nc\fe{\frak e}
\nc\ff{\frak f}
\nc\fg{\frak g}
\nc\fh{\frak h}
\nc\fj{\frak j}
\nc\fk{\frak k}
\nc\fl{\frak l}
\nc\fm{\frak m}
\nc\fn{\frak n}
\nc\fo{\frak o}
\nc\fp{\frak p}
\nc\fq{\frak q}
\nc\fr{\frak r}
\nc\fs{\frak s}
\nc\ft{\frak t}
\nc\fv{\frak v}
\nc\fz{\frak z}
\nc\fx{\frak x}
\nc\fy{\frak y}

\nc\fA{\frak A}
\nc\fB{\frak B}
\nc\fC{\frak C}
\nc\fD{\frak D}
\nc\fE{\frak E}
\nc\fF{\frak F}
\nc\fG{\frak G}
\nc\fH{\frak H}
\nc\fJ{\frak J}
\nc\fK{\frak K}
\nc\fL{\frak L}
\nc\fM{\frak M}
\nc\fN{\frak N}
\nc\fO{\frak O}
\nc\fP{\frak P}
\nc\fQ{\frak Q}
\nc\fR{\frak R}
\nc\fS{\frak S}
\nc\fT{\frak T}
\nc\fU{\frak U}
\nc\fV{\frak V}
\nc\fZ{\frak Z}
\nc\fX{\frak X}
\nc\fY{\frak Y}
\nc\tfi{\ti{\Phi}}
\nc\bF{\bold F}
\rnc\bol{\bold 1}

\nc\ua{\bold U_\A}

\nc\qinti[1]{[#1]_i}
\nc\q[1]{[#1]_q}
\nc\xpm[2]{E_{#2 \delta \pm \alpha_#1}} 
\nc\xmp[2]{E_{#2 \delta \mp \alpha_#1}}
\nc\xp[2]{E_{#2 \delta + \alpha_{#1}}}
\nc\xm[2]{E_{#2 \delta - \alpha_{#1}}}
\nc\hik{\ed{k}{i}}
\nc\hjl{\ed{l}{j}}
\nc\qcoeff[3]{\left[ \begin{smallmatrix} {#1}& \\ {#2}& \end{smallmatrix}
 \negthickspace \right]_{#3}}
\nc\qi{q}
\nc\qj{q}

\nc\ufdm{{_\ca\bu}_{\rm fd}^{\le 0}}

\nc\isom{\cong}

\nc{\pone}{{\Bbb C}{\Bbb P}^1}
\nc{\pa}{\partial}

\nc{\F}{{\mathcal F}}
\nc{\Sym}{{\goth S}}
\nc{\A}{{\mathcal A}}
\nc{\arr}{\rightarrow}
\nc{\larr}{\longrightarrow}

\nc{\ri}{\rangle}
\nc{\lef}{\langle}
\nc{\W}{{\mathcal W}}
\nc{\uqatwoatone}{{U_{q,1}}(\su)}
\nc{\dij}{\delta_{ij}}
\nc{\divei}{E_{\alpha_i}^{(n)}}
\nc{\divfi}{F_{\alpha_i}^{(n)}}
\nc{\Lzero}{\Lambda_0}
\nc{\Lone}{\Lambda_1}
\nc{\ve}{\varepsilon}
\nc{\bepsilon}{\bar{\epsilon}}
\nc{\bak}{\bar{k}}
\nc{\phioneminusi}{\Phi^{(1-i,i)}}
\nc{\phioneminusistar}{\Phi^{* (1-i,i)}}
\nc{\phii}{\Phi^{(i,1-i)}}
\nc{\Li}{\Lambda_i}
\nc{\Loneminusi}{\Lambda_{1-i}}
\nc{\vtimesz}{v_\ve \otimes z^m}

\nc\ag{\widehat{\goth{g}}}
\nc\teb{\tilde E_\boc}
\nc\tebp{\tilde E_{\boc'}}

 \newcommand{\eeq}{\end{equation}}
\newcommand{\ben}{\begin{eqnarray}}
 \newcommand{\een}{\end{eqnarray}}

\nc\Uq{U_q \mathfrak{sl}_2}
\nc\Uqhat{U_q \widehat{\mathfrak{sl}}_2}
\nc\Loop{\mathcal{L} U_q \mathfrak{sl}_2}
\newcommand{\h}{\frac{1}{2}}
\newcommand{\tha}{\frac{3}{2}}

\newcommand{\lambdaQ}{\bigl(q-q^{-1}\bigr)}

\newcommand{\ds}{\mathbb}

\newcommand{\fu}{{\langle 12 \rangle} }

\newcommand{\CE}{{\mathcal E}}
\newcommand{\CF}{{\mathcal F}}

\begin{document}

\allowdisplaybreaks

\newcommand{\arXivNumber}{2301.00781}

\renewcommand{\PaperNumber}{026}

\FirstPageHeading

\ShortArticleName{Fused K-Operators and the $q$-Onsager Algebra}

\ArticleName{Fused K-Operators and the $\boldsymbol{q}$-Onsager Algebra}

\Author{Guillaume LEMARTHE, Pascal BASEILHAC and Azat M.~GAINUTDINOV}

\AuthorNameForHeading{G.~Lemarthe, P.~Baseilhac and A.M.~Gainutdinov}

\Address{Institut Denis-Poisson CNRS/UMR 7013, Universit\'e de Tours, Universit\'e d'Orl\'eans,\\
 Parc de Grammont, 37200 Tours, France}
\Email{\mail{guillaume.lemarthe@idpoisson.fr}, \mail{pascal.baseilhac@idpoisson.fr},\newline
\hspace*{12.5mm} \mail{azat.gainutdinov@cnrs.fr}}

\ArticleDates{Received December 10, 2024, in final form February 12, 2026; Published online March 20, 2026}

\Abstract{We study universal solutions to reflection equations with a spectral parameter, so-called K-operators, within a general framework of universal K-matrices -- an extended version of the approach introduced by Appel--Vlaar. Here, the input data is a quasi-triangular Hopf algebra $H$, its comodule algebra~$B$ and a pair of consistent twists. In our setting, the universal K-matrix is an element of $B\otimes H$ satisfying certain axioms, and we consider the case $H=\mathcal{L} U_q \mathfrak{sl}_2$, the quantum loop algebra for $\mathfrak{sl}_2$, and $B=\mathcal{A}_q$, the alternating central extension of the $q$-Onsager algebra. Considering tensor products of evaluation representations of $\mathcal{L} U_q \mathfrak{sl}_2$ in ``non-semisimple'' cases, the new set of axioms allows us to introduce and study fused K-operators of spin-$j$; in particular, to prove that for all $j\in\frac{1}{2}\mathbb{N}$ they satisfy the spectral-parameter dependent reflection equation. We provide their explicit expression in terms of elements of the algebra ${\mathcal A}_q$ for small values of spin-$j$. The precise relation between the fused K-operators of spin-$j$ and evaluations of a universal K-matrix for ${\mathcal A}_q$ is conjectured based on supporting evidence. We finally discuss implications of our results on the K-operators for quantum integrable systems.}


\Classification{81R50; 81R10; 81U15}

\section{Introduction}

\subsection{Background}
In the context of quantum integrable systems, the R- and K-matrices are the basic ingredients for the construction of the monodromy matrix (or its double row version) leading to the generating function for mutually commuting quantities, the so-called transfer matrix. Here, the formal variable of the generating function is called `spectral parameter', denoted by $u$. By definition, the R-matrix is a solution of the Yang--Baxter equation with this spectral parameter, whereas the K-matrix satisfies a reflection equation, also known under the name boundary Yang--Baxter equation.
For a triple of finite-dimensional vector spaces $V^{(j_k)}$, for $k=1,2,3$, the corresponding Yang--Baxter equation in $\End\bigl(V^{(j_1)}\otimes V^{(j_2)}\otimes V^{(j_3)}\bigr)$ takes the form~\cite{Ba72,Ya67}
\begin{gather}
 R^{(j_1,j_2)}_{12}\left({\frac{u_1}{u_2}}\right) R^{(j_1,j_3)}_{13}\left({\frac{u_1}{u_3}}\right)
 R^{(j_2,j_3)}_{23}\left({\frac{u_2}{u_3}}\right)\nonumber\\
\qquad=
 R^{(j_2,j_3)}_{23}\left({\frac{u_2}{u_3}}\right)R^{(j_1,j_3)}_{13}\left({\frac{u_1}{u_3}}\right)
 R^{(j_1,j_2)}_{12}\left({\frac{u_1}{u_2}}\right),\label{YBj}
\end{gather}
%
where $u_k$ are the spectral parameters, and \smash{$R^{(j_n,j_m)}_{nm}(u)$} are the R-matrices on corresponding products of spaces:
$R_{12}=R \otimes {\mathbb I}$, $R_{23}= {\mathbb I}\otimes R$, $R_{13}= (\mathcal P \otimes {\mathbb I} ) R_{23} ( {\mathcal P} \otimes {\mathbb I})$,
with $\mathcal P (a \otimes b) = b \otimes a$.
Given an R-matrix $R^{(j_1,j_2)}(u)$ satisfying~\eqref{YBj}, the reflection equation in $\End\bigl(V^{(j_1)}\otimes V^{(j_2)}\bigr)$ is given by~\cite{Cher84,Skly88}
\begin{gather}
R_{12}^{(j_1,j_2)}\left({\frac{u_1}{u_2}}\right) {K}^{(j_1)}_1(u_1) R_{21}^{(j_2,j_1)}(u_1 u_2) { K}_2^{(j_2)}(u_2)\nonumber\\
\qquad =
{K}_2^{(j_2)}(u_2) R_{12}^{(j_1,j_2)}(u_1u_2) {K}_1^{(j_1)}(u_1) R_{21}^{(j_2,j_1)}\left({\frac{u_1}{u_2}}\right),\label{REj}
\end{gather}
where we set $K_1= K\otimes {\mathbb I}$, $K_2= {\mathbb I}\otimes K$ and
\[
R_{21}^{(j_2,j_1)}(u) = \mathcal{P}^{(j_2,j_1)} R^{(j_2,j_1)}(u) \mathcal{P}^{(j_1,j_2)}.
\]

Along the years, several examples of solutions to~\eqref{YBj} and~\eqref{REj} have been obtained for the case of 2- or 3-dimensional spaces $V^{(j_k)}$, see, e.g., \cite{VG94,GZ94,inami96,FZ80,ZZ78}. Interestingly, these $4\times 4$ R-matrices can be interpreted as intertwining operators for underlying action of the quantum affine algebra $\Uqhat$ on tensor product of so-called evaluation representations of spin-$\h$, or briefly these are spin-$\h$ solutions (and similarly for the 3-dimensional or spin-1 solutions). The spin-$\h$ expressions of the K-matrix are interpreted analogously as `twisted' intertwiners \big(with respect to the spectral parameter `reflection' $u\to u^{-1}$\big) for the action of a coideal subalgebra of the quantum affine algebra.

For higher values of spins $j_k$, constructing R- and K-matrices by brute force is increasingly complicated. To circumvent this problem, the so-called fusion method have been proposed which is summarized as follows. Starting from R- and K-matrix solutions of spin-$\h$, the fused R- and K-matrices of spin-$j$ are obtained inductively by tensoring (or ``fusing") the fundamental representations of $\Uq$ and projecting onto the highest spin sub-representation.
This procedure for the R-matrix has been originally developed in~\cite{Ka79,KR87,KRS81}, and for the K-matrix in~\cite{inami96,MN91}. For a more recent approach, see \cite{Be2015,NP15,RSV16}.

It is well known \cite{Dr0} that the Yang--Baxter equation \eqref{YBj} can be derived from the more general setting of Yang--Baxter algebras and the universal Yang--Baxter equation, see~\cite{DF93,JLM,JLMBD}, by specialization to the finite-dimensional representations $V^{(j_k)}$ that depend on the spectral parameter $u$. Namely, R-matrix solutions of \eqref{YBj} are obtained by specializing (affine) L-operators of the form
\begin{equation}
\widehat{{\bf L}}^{(j)}(u) \in H \otimes \End\bigl(V^{(j)}\bigr) \label{eq:L-op}
\end{equation}
that satisfy an equation in $H \otimes \End\bigl(V^{(j_1)}\otimes V^{(j_2)}\bigr)$
\[
R^{(j_1,j_2)}(u/v)\widehat{{\bf L}}_1^{(j_1)}(u)\widehat{{\bf L}}_2^{(j_2)}(v) = \widehat{{\bf L}}_2^{(j_2)}(v)\widehat{{\bf L}}_1^{(j_1)}(u) R^{(j_1,j_2)}(u/v).\label{YBAj}
\]
Here $H$ is assumed to be any quantum affine algebra and $V^{(j)}$'s are also known as evaluation representations.
Strictly speaking, for $V^{(j)}$ to be finite-dimensional the spectral parameter $u$ should be a non-zero complex number. However, it's more convenient to consider $u$ as a formal parameter and to work with a formal infinite-dimensional version of the evaluation representations. Then the first component of $\widehat{{\bf L}}^{(j)}(u)$ belongs to $H\big[\big[u^{-1}\big]\big]$, the algebra of formal power series in $u^{-1}$ with coefficients in $H$. In the case of $H$ the quantum affine algebra \smash{$\Uqhat$} these affine L-operators were proposed in~\cite[Section~IV]{DF93} for $j=\h$.

Furthermore, the L-operators themselves can be obtained by specializing 2nd component of the universal R-matrix $\mathfrak{R} \in H\otimes H$ that is known to exist in an appropriate completion of the tensor product~\cite[Section~13]{Dr0}. For instance, for $H=\Uqhat$ the universal R-matrix is expressed as an infinite product of $q$-exponentials
involving the root vectors~\cite{Da98,LS,Tolstoy1991}.\footnote{Strictly speaking, the universal R-matrices introduced in these works are for quantized Kac--Moody algebras, extensions of $U_q\hat{\mathfrak{g}}$ by a derivation. However, in our work we will not make any distinction between quantized affine and Kac--Moody algebras, as all representations we consider are of zero central charge
which means that we actually work with their quotients called quantum loop algebras.} The evaluation of $\mathfrak{R}$ on finite-dimensional representations requires to treat $u$ as a formal parameter to avoid convergence problems. This was suggested in~\cite[Section~13]{Dr0} and later studied for $H=U_q\hat{\mathfrak{g}}$ in~\cite[Section~4]{FR92}, see~\cite[Section~1]{He17} for a review.
We give an explicit calculation of the affine L-operator \smash{$\widehat{{\bf L}}^{(\h)}(u)$} from the universal R-matrix in Appendix~\ref{sec:eval-R}.

In the framework of the universal R-matrix, let us point out that the fusion method finds a~natural interpretation: it follows from one of the axioms~\eqref{univR1}--\eqref{univR3} on the universal R-matrix, see Section~\ref{sec2.1}.

For the reflection equation, it is well known~\cite{Skly88} that~\eqref{REj} has a representation-theoretic interpretation within the more general framework of reflection algebras, which are generated by the modes of the matrix entries of so-called \textit{K-operators}
\begin{gather}
\mathcal{K}^{(j)}(u) \in B\bigl(\bigl(u^{-1}\bigr)\bigr) \otimes \End\bigl(V^{(j)}\bigr), \label{eq:K-op}
\end{gather}
satisfying for all $j$ the following reflection equations\footnote{In the case of P-symmetric R-matrices associated to $H=\Uqhat$, the reflection equation~\eqref{genREAj} for the K\nobreakdash-operators takes the more standard form~\eqref{evalpsi}.}
\begin{gather}
R_{12}^{(j_1,j_2)}\left({\frac{u_1}{u_2}}\right) \mathcal{K}^{(j_1)}_1(u_1) R_{21}^{(j_2,j_1)}(u_1u_2) \mathcal{K}_2^{(j_2)}(u_2)\nonumber \\
\qquad=
\mathcal{K}_2^{(j_2)}(u_2) R_{12}^{(j_1,j_2)}(u_1u_2) \mathcal{K}_1^{(j_1)}(u_1) R_{21}^{(j_2,j_1)}\left(\frac{u_1}{u_2}\right). \label{genREAj}
\end{gather}
Here $B$ is in general assumed to be a comodule algebra over $H$, a notion generalizing coideal subalgebras in~$H$.\footnote{There are also K-operators that might not fit our setting. For instance, the K-operators associated with a~$q$\nobreakdash-oscillator algebra or the Askey--Wilson algebra were constructed in \cite{Ba05,BK03} respectively, see also~\cite{BF11} for quotients of a few higher rank generalizations of $O_q$. However, we are not aware of any comodule algebra structure for these algebras.}
In other words, unlike K-matrices, which have scalar entries, K-operators are matrices with entries belonging to $B\bigl(\bigl(u^{-1}\bigr)\bigr)$, i.e., they are formal Laurent series in $u^{-1}$ with coefficients in the algebra $B$ whose defining relations are extracted via expanding the reflection equation~\eqref{genREAj} and equating coefficients of $u_1^n u_2^m$ for each $n$, $m$.

In this paper, we consider \smash{$H=\Uqhat$} with zero central charge, so-called quantum loop algebra, and the comodule algebra of primary interest for us is $B=\mathcal{A}_q$, the alternating central extension~\cite{BSh1,Ter21} of the $q$-Onsager algebra $O_q$. Our main motivation in this comodule algebra comes from applications to open spin-chains through the K-operators, as it will be discussed below and in more detail in~\cite{LBG}.
It is important to note that we do not have any explicit form of a K-operator for $O_q$, while it is known for the comodule algebra $B=\mathcal{A}_q$~\cite{BSh1} in the fundamental case $j=\h$. Moreover, with this K-operator \smash{$\mathcal{K}^{(\h)}(u)$} the reflection equation~\eqref{genREAj} for $j_1=j_2=\h$ is the sole \textit{defining} relation of $\mathcal{A}_q$. We thus have a very compact presentation of $\mathcal{A}_q$ which is analogous to the famous FRT presentation~\cite{DF93} of \smash{$H=\Uqhat$}.

The K-operators can also produce their smaller cousins~-- the K-matrices~-- by evaluating one-dimensional representations of $B$ on the 1st tensor component of $\mathcal{K}^{(j)}(u)$.
Indeed, we show in~\cite{LBG} that one-dimensional representations of $\mathcal{A}_q$ produce out of the spin-$\h$ K-operator \smash{$\mathcal{K}^{(\h)}(u)$} the most general K-matrix of the open XXZ spin-$\h$ chain~\cite{VG94,GZ94}. Furthermore, on tensor-product representations of $\mathcal{A}_q$, the K-operator \smash{$\mathcal{K}^{(j)}(u)$} images agree with `dressed' K-matrices or Sklyanin's operators, and the important problem of functional relations between transfer matrices of different spins, the so-called \textit{TT-relations}, can be investigated at the level of algebraic relations of ${\mathcal A}_q$, as we discuss in more detail in Section~\ref{intro:Kop-TT}. The main goal of our program is to construct such a general framework that allows the development of universal TT-relations holding in $\mathcal{A}_q$, which would then specialize to the known (conjectured) TT-relations for transfer matrices in the literature such as~\cite{CYSW14,FNR07}.
This paper provides the first step in this direction by developing a formalism of K-operators with the spectral parameter, as we discuss it now.

\subsection{Universal K-matrix and comodule algebras}
As it was with the L-operators, it is expected that the K-operators $\mathcal{K}^{(j)}(u)$ from~\eqref{eq:K-op} can arise from a universal K-matrix satisfying a universal reflection equation.
The concept of a universal K-matrix is not new, it has been studied originally in~\cite{CG92, DKM03,E08, KSS92,Tom}. More recently, significant progress has been made in the works of~\cite{AV20,BKo19,BW13, Ko20}.
Each of the definitions of universal K-matrices introduced in these recent references has
different axioms that were chosen to serve different purposes. We can highlight 3 directions corresponding to the algebra
in whose completion the universal K-matrix lies:
\begin{enumerate}\itemsep=0pt
\item[(1)] $H=U_q\mathfrak{g}$, in~\cite{BKo19};
\item[(2)] $B \otimes H$, where $B$ is a coideal subalgebra of $H=U_q\mathfrak{g}$, in~\cite{Ko20};
\item[(3)] $H=U_q\hat{\mathfrak{g}}$, with the data of a certain automorphism of $H$, called `twist', in~\cite{AV20}.
\end{enumerate}

Let us clarify that (1) and (3) imply a choice of $B$ as a coideal subalgebra in $H$ to form a quantum symmetric pair $(H, B)$. In all the 3 directions, the universal K-matrix satisfies intertwining relations with respect to $B$, and not $H$ as the R-matrix, and which might be twisted in the direction (3).

Let us start with direction (2). Kolb introduced a universal K-matrix $\mathscr{K}\in B \otimes H$ that satisfies certain axioms~\cite[Definition~2.7]{Ko20}
\begin{gather}
 \mathscr{K} \Delta_B (b) = \Delta_B (b) \mathscr{K} \quad \text{for all $b \in B$}, \qquad
 ( \Delta_B \otimes \id ) (\mathscr{K}) = \mathfrak{R}_{32} \mathscr{K}_{13} \mathfrak{R}_{23}, \nonumber\\
 ( \id \otimes \Delta) ( \mathscr{K}) = \mathfrak{R}_{32} \mathscr{K}_{13} \mathfrak{R}_{23} \mathscr{K}_{12}. \label{intro:ax-Kolb}
\end{gather}
It satisfies an algebraic (non-matrix) version of the reflection equation that belongs to $B \otimes H \otimes H$,
$
\mathfrak{R}_{32} \mathscr{K}_{13} \mathfrak{R}_{23} \mathscr{K}_{12} = \mathscr{K}_{12} \mathfrak{R}_{32} \mathscr{K}_{13} \mathfrak{R}_{23}
$.
The specialization of its second tensor component in $H$, to a finite-dimensional representation leads to K-operators that do not depend on a~spectral parameter, denoted as $\mathcal{K} \in B \otimes \End\bigl(V^{(j)}\bigr)$, satisfying a reflection equation of the form~${
R_{12} \mathcal{K}_1 R_{21} \mathcal{K}_2 = \mathcal{K}_2 R_{12} \mathcal{K}_1 R_{21}}$,
which belongs to~$B \otimes \End\big(V^{(j_1)}\big) \otimes \End\big(V^{(j_2)}\big)$.
Next, let us clarify that direction (1) is a special case of (2). Indeed, applying the counit to the first component of $\mathscr{K}$, we obtain an object $k \in H$ that was introduced by Balagovi\'c and Kolb in~\cite{BKo19}. Its specialization to a finite-dimensional representation leads to K-matrices (with scalar entries).

However, axiomatics of directions (1) and (2) is not suitable for the reflection equation with spectral parameter~\eqref{REj}. To change this, Appel and Vlaar proposed in~\cite{AV20} a twisted version of the approach (1) in~\cite{BKo19}, this is direction (3). It generalizes (1) in the sense that identifying the twist with a diagram automorphism (1) is recovered. Their definition of universal K-matrix is associated with a coideal subalgebra of $H=U_q\hat{\mathfrak{g}}$ and includes a pair of consistent twists.
It was shown in~\cite{AV20} that such a universal K-matrix exists for a large class of coideals known as quantum symmetric pairs. Furthermore, it was shown in~\cite{AV22} that these universal K-matrices are well defined on finite-dimensional evaluation representations of $H=U_q\hat{\mathfrak{g}}$ and provide K-matrix solutions of~\eqref{REj}.
Unfortunately, no explicit expression of these universal K-matrices is known even in simplest cases, and the existence result requires a specific form of twists.

Furthermore, whereas the axiomatics and results of~\cite{AV20,AV22} are well suited for coideal subalgebras like $O_q$, they cannot be used for the algebra $\mathcal{A}_q$. Indeed, $\mathcal{A}_q$ is an extension of $O_q$ by an infinite number of algebraically independent central elements and, to the best of our knowledge, it cannot be realized as a coideal subalgebra of $H=\Uqhat$
or any other quasi-triangular Hopf algebra.
We therefore cannot use the axiomatics and results of~\cite{AV20,AV22}, and it requires an adjustment in order to treat general comodule algebras like $\mathcal{A}_q$.

In this paper, to produce K-operator solutions of~\eqref{genREAj}, we combine the approaches (2) and~(3), and introduce a twisted version $\mathfrak{K}\in B \otimes H$ of the universal K-matrix from~\cite{Ko20} via a new set of axioms~\eqref{univK1}--\eqref{univK3} in Section~\ref{sec:universal-K} that generalize those in~\eqref{intro:ax-Kolb} and at the same time recover the approach (3) via the counit application. The notion of a twist $\psi$ as a certain automorphism of~$H$ is important here because, as it will be discussed in Section~\ref{sec4}, a natural choice of $\psi$ allows the derivation of the K-operator reflection equation~\eqref{genREAj} from the $\psi$-twisted universal reflection equation satisfied by $\mathfrak{K}$ and given in Proposition~\ref{propRE}. We also note that here $B$ is not assumed to be a coideal in $H$, in particular we do not assume $(H,B)$ to be a quantum symmetric pair, but $B$ is a right comodule algebra over~$H$.

We conjecture that such a universal K-matrix $\mathfrak{K}$ for our choice of $\psi=\eta$ defined in~\eqref{autrho} exists,\footnote{In the context of 1-component universal K-matrix $k \in H$, there are existence results for other choices of twists in~\cite{AV22}. We discuss this in Section~\ref{sec8}.} so that its 2nd component evaluation on $\Uqhat$ representations leads to known examples of K-operators with a spectral parameter (that can be further evaluated to get known examples of K-matrices~\cite{LBG}).
However, the only known example is for the fundamental or spin-$\h$ representation, the \smash{$\mathcal{K}^{(\h)}(u)$} for $B=\mathcal{A}_q$, i.e., the K-operator providing the reflection algebra presentation of~${\mathcal A}_q$ from Theorem~\ref{def:Aq0}. It is thus natural to investigate further the spectral parameter dependent K-operators for arbitrary spin representations.
The universal K-matrix axioms~\eqref{univK1}--\eqref{univK3} provide the main guiding directions for such investigation, more importantly in definition of the underlying fusion construction, as it was the case for the interpretation of the (affine) L-operators and R-matrices.

\subsection{Goal and main results}
The purpose of this paper is to construct K-operators with a spectral parameter for arbitrary spin representations and show how they relate to specializations of a universal K-matrix $\mathfrak{K} \in B \otimes H$.
For this, we develop \textit{universal fusion method} for the K-operators \eqref{eq:K-op} based on our new set of axioms~\eqref{univK1}--\eqref{univK3}, and apply them in the case of the comodule algebra $B={\mathcal A}_q$.
The main results are the following.
Using the axiom~\eqref{univK2} and the detailed analysis of tensor product representation of $\Uqhat$, we derive \textit{fused} K-operators $\mathcal{K}^{(j)}(u)$ for arbitrary values of spin\nobreakdash-$j$, starting from the spin\nobreakdash-$\h$ K-operator, see Definition~\ref{spinjfusedK} based on the result in Proposition~\ref{propfusK}. One of our main results (see Theorem~\ref{prop:fusedRE}) is that they satisfy the reflection equations~\eqref{genREAj} for all~$j_1$,~$j_2$ and with no assumption made on existence of $\mathfrak{K}$. For $j=1,\tha$, we give explicit expressions of the fused K-operators in terms of generating functions of $\mathcal{A}_q$.

Let us comment more on constructions behind the universal fusion method.
Drawing inspiration from the evaluation of one of the axioms of the universal K-matrix, and \textit{without making any assumptions about the existence of $\mathfrak{K}$}, we define \textit{fused spin-$j$} K-operators
\begin{equation*}
\mathcal{K}^{(j)}(u) \in \mathcal{A}_q\bigl(\bigl(u^{-1}\bigr)\bigr) \otimes \End\bigl(\mathbb{C}^{2j+1}\bigr)
\end{equation*}
 through the following recursion based on the fundamental spin-$\frac{1}{2}$ K-operator
\begin{equation}\label{intro:fusion}
\mathcal{K}^{(j)}(u) = \mathcal{F}^{(j)}_{\langle 12 \rangle} \mathcal{K}_1^{(\frac 12)}\bigl(u q^{-j+\frac 12}\bigr) R^{(\frac 12,j-\frac 12)}\bigl(u^2 q^{-j+1}\bigr) \mathcal{K}_2^{(j-\frac 12)}\bigl(u q^{\frac 12}\bigr) \mathcal{E}^{(j)}_{\langle 12 \rangle}.
\end{equation}
Here, we used $\Uqhat$-intertwiners for evaluation representations
\begin{equation*}
\mathcal{E}^{(j)}\colon\ \mathbb{C}_u^{2j+1} \rightarrow \mathbb{C}_{u_1}^{2}
\otimes \mathbb{C}_{u_2}^{2j}, \qquad u_1 = u q^{-j+\frac 12}, \qquad u_2=u q^{\frac 12} ,
\end{equation*}
and we denote their pseudo-inverse maps
\smash{$\mathcal{F}^{(j)} \colon \mathbb{C}^{2}_{u_1} \otimes \mathbb{C}_{u_2}^{2j} \rightarrow \mathbb{C}_u^{2j+1}$}.
The condition imposed on~$u_1$ and $u_2$ ensures that the tensor product of the evaluation representations of $\Uqhat$ admits a spin-$j$ sub-representation (which is not a direct summand). This reducibility condition is expressed in terms of ratio of the evaluation parameters~\cite[Section~4.9]{CP}.
Using this criterion, we explicitly construct $\mathcal{E}^{(j)}$ and $\mathcal{F}^{(j)}$ by determining their matrix expressions, as described in Section~\ref{intert-fus}.

Inspiration for the above universal fusion~\eqref{intro:fusion} was taken from fusion relations in Proposition~\ref{propfusK} satisfied by the spin-$j$ K-operator
\begin{equation}\label{intro-K-op}
{\bf K}^{(j)}(u) = \bigl(\id \otimes \pi_{u^{-1}}^j\bigr) (\mathfrak{K}) \in\ B\bigl(\bigl(u^{-1}\bigr)\bigr) \otimes \End\bigl(V^{(j)}\bigr).
\end{equation}
obtained from the 1-component evaluation of $\mathfrak{K}$.
Here, the universal K-matrix $\mathfrak{K}$ is considered for the pair of twists $(\psi, J) = (\eta, 1 \otimes 1)$ with $\eta$ defined in~\eqref{autrho}.
Moreover, a careful analysis requires treating $u$ as a formal variable,
and the representation $\pi^j_{u} \colon H \rightarrow \End\bigl(V^{(j)}\big[u^{\pm1}\big]\bigr)$ used in~\eqref{intro-K-op} is what we call \textit{formal evaluation} representation of \smash{$H=\Uqhat$} defined in Section~\ref{sec:formal-ev}. It is an infinite-dimensional analogue of the finite-dimensional evaluation representations of $H$. These representations of quantum affine algebras are not of highest weight type, and also known under the name `quantum loop modules' \cite{chari,kashi}.
In Section~\ref{sec:formal-ev}, we also analyze tensor products of the formal evaluation representations $\pi^j_{u}$ and obtain the intertwining operators $\mathcal{E}^{(j)}$ and their pseudo-inverses $\mathcal{F}^{(j)}$ in this framework with $u$ being a formal variable.

It is important to note that, in contrast to the situation with L-operators, there is no single example in the literature of an evaluation~\eqref{intro-K-op}, even for $j=\h$.
On the other side, the formula~\eqref{intro:fusion} gives explicit expressions for K-operators.
Furthermore, we propose a precise relation that links the K-operators ${\bf K}^{(j)}(u)$ in~\eqref{intro-K-op} and the fused spin-$j$ K-operators $\mathcal{K}^{(j)}(u)$ in~\eqref{intro:fusion} associated with $\mathcal{A}_q$. This is stated in Conjecture~\ref{conj1}: the two K-operators equal up to a central and invertible element in $\mathcal{A}_q\bigl(\bigl(u^{-1}\bigr)\bigr)$, which is defined by a functional relation and has a specific coaction.
This conjecture is supported by showing independently that the fused K-operators~\eqref{intro:fusion} satisfy a set of relations arising from the evaluation of the axioms of the universal K-matrix.

\subsection[K-operators for A\_q and universal TT-relations]{K-operators for $\boldsymbol{{\mathcal A}_q}$ and universal TT-relations}\label{intro:Kop-TT}

A construction of arbitrary spin K-operators for ${\mathcal A}_q$ is not only important for a better understanding of the universal K-matrix formalism for comodule algebras, but first of all because of applications in quantum integrable systems. Namely, as it is explained in~\cite{LBG}, a generating function
of a commutative subalgebra of ${\mathcal A}_q$ can be constructed using the spin-$j$ K-operator~${\mathcal{K}}^{(j)}(u)$, that can be viewed as a spin-$j$ \textit{universal} transfer matrix $\bt^{(j)}(u) \in \mathcal{A}_q\bigl(\bigl(u^{-1}\bigr)\bigr)$. A family of {\it universal} TT-relations between the universal transfer matrices can be calculated directly in the algebra ${\mathcal A}_q$
\begin{equation} \label{intro:TT}
\bt^{(j)}(u) = \bt^{(j-\frac 12)}\bigl(u q^{-\frac 12}\bigr) \bt^{(\frac 12)}\bigl(u q^{j-\frac 12}\bigr) + f^{(j)}(u) \bt^{(j-1)}\bigl(u q^{-1}\bigr),
\end{equation}
with $ \bt^{(0)}(u)\equiv 1$, \smash{$ \bt^{(-\h)}(u)\equiv 0$} and $f^{(j)}(u)$ is a generating function of central elements in $\mathcal{A}_q$ that is explicitly calculated~\cite{LBG}. Considering
tensor product (or spin-chain) representations of
the algebra $B={\mathcal A}_q$, one can show that the K-operators ${\mathcal{K}}^{(j)}(u)$ and the universal transfer matrices $\bt^{(j)}(u)$ map, respectively, to the `dressed' K-matrices and transfer matrices associated with various examples of integrable spin chains of XXZ type with generic integrable boundary conditions. Importantly, whereas $f^{(j)}(u)$ in~(\ref{intro:TT}) is universal, its images with respect to the various $q$-Onsager quotients associated with the spin chains are not:
\textit{the boundary conditions of the open integrable spin chains dictate what value the image of $f^{(j)}(u)$ takes.}
In other words, the algebra ${\mathcal A}_q$ in a certain sense governs all known integrable boundary conditions through its quotients and degenerate cases~\cite{BB16}. Thus, many properties of the integrable models can be studied even before specializations to the spin-chain representations of ${\mathcal A}_q$, and so the K-operators of ${\mathcal A}_q$ can be used in the \textit{representation-independent} analysis of the related integrable models.
The applications of our results to integrable models are discussed more in Section~\ref{sec8} and in our next paper~\cite{LBG}.

The text is organized as follows. In Section~\ref{sec2}, the formalism of universal R-matrix is reviewed following~\cite{Dr0}. Furthermore, we upgrade the universal K-matrix axioms of \cite{AV20,Ko20} in Definition~\ref{defunivK}. In our framework, the K-operators are obtained as evaluations of the universal K-matrix $\mathfrak{K} \in B \otimes H$ (if it exists) satisfying a
$\psi$-twisted universal reflection equation, see Proposition~\ref{propRE}.
In this section, we also review basic definitions of the quantum loop algebra~${H=\Loop}$, the $q$-Onsager algebra $O_q$ and its alternating central extension $B={\mathcal A}_q$.

In Sections~\ref{analysUq}--\ref{sec5},
we develop the universal fusion method. Namely, Section~\ref{analysUq} is devoted to a~detailed analysis of the tensor product of evaluation representations of $\Loop$ and their formal-parameter versions.
In Section~\ref{sec4}, we assume existence of a~universal K-matrix $\mathfrak{K}$ for a~comodule algebra $B$ and a~certain twist pair.
The L-operators and K-operators of spin-$j$ are defined as formal evaluations of $\mathfrak{R}$ and $\mathfrak{K}$, respectively; see Definitions~\ref{defL} and~\ref{defbK}.
The main results of this section are the so-called `fusion' relations satisfied by the L- and K-operators of spin-$j$, see Propositions~\ref{propfusR} and~\ref{propfusK}, respectively.
In Section~\ref{sec5}, we do not assume existence of a~universal K-matrix for ${\mathcal A}_q$, and instead
introduce \textit{fused K-operators} of spin-$j$ in Definition~\ref{spinjfusedK} by a recursion based on the fundamental K-operator giving the reflection algebra presentation of ${\mathcal A}_q$. Theorem~\ref{prop:fusedRE} shows that they satisfy the reflection equation~\eqref{genREAj} \textit{for all choices of~$j_1$,~$j_2$}. Explicit expressions of the fused K-operators for $j=1,\tha$ are derived in Section~\ref{sec5.3}.
We also give compact expressions for the fused R-matrices and the fused K-operators in~\eqref{dvpR},~\eqref{dvpK}, solely in terms of the fundamental R-matrix and K-operator.
In Section~\ref{sec6}, the precise relation between the spin-$j$ K-operators from Section~\ref{sec4} and the fused K-operators from Section~\ref{sec5} leads to Conjecture~\ref{conj1}, with supporting evidence discussed.
Finally, in the concluding Section~\ref{sec8}, we give a brief summary of our results and discuss a few perspectives of applications of the K-operator formalism to quantum integrable systems.

In Appendix~\ref{appC}, we adapt the universal R-matrix constructed in~\cite{Tolstoy1991} to our conventions, and compute the corresponding Ding--Frenkel type L-operators \smash{$\widehat{{\bf L}}^{(\h)}(u) = {\bf L}^\pm(u)$} and the spin-$\h$ L-operator \smash{${\bf L}^{(\h)}(u)$} by evaluation of the universal R-matrix.
In Appendix~\ref{apB}, we give ordering relations for the generating functions of ${\mathcal A}_q$.
Finally, in Appendix~\ref{apD}, we give a proof of one of our main results, Theorem~\ref{prop:fusedRE}, that the fused K-operators satisfy the reflection equations for any pair of spins.

{\bf Notations.}
We denote the set of natural numbers by ${\mathbb N} = \{0, 1, 2, \dots\}$ and the positive integers by $\mathbb{N}_+= \{1, 2, \dots \}$.

All algebras are considered over the field of complex numbers $\mathbb{C}$, if not stated otherwise. Though the results till Section~\ref{exUqAq} are valid also over general fields, and many results in Section~\ref{sec5} can be directly generalized to algebraically closed fields of zero characteristic, we fix for simplicity the ground field to be $\mathbb{C}$.
Let $q\in\mathbb{C}^*$, and we assume in this paper that $q$ is not a root of unity. The $q$-commutator is
\begin{equation}
\big[X,Y\big]_q=qXY-q^{-1}YX \label{def:qcom}
\end{equation}
and $\big[X,Y\big]=\big[X,Y\big]_1=XY-YX$.
We denote the $q$-numbers by $[n]_q= (q^n-q^{-n})/\bigl(q-q^{-1}\bigr)$.

We denote by $\mathbb{I}_{2j}$ the $2j\times2j$ identity matrix. We also use Pauli matrices
\begin{gather*}
\sigma_+=\begin{pmatrix}
 0 &1 \\
 0& 0
\end{pmatrix},\qquad \sigma_-=\begin{pmatrix}
 0 &0 \\
 1& 0
\end{pmatrix}, \\ \sigma_x=\begin{pmatrix}
 0 & 1 \\
 1 & 0
\end{pmatrix}, \qquad \sigma_y=\begin{pmatrix}
 0 & -i \\
 i & 0
\end{pmatrix}, \qquad \sigma_z=\begin{pmatrix}
 1 &0 \\
 0& -1
\end{pmatrix}.
\end{gather*}

All generating functions considered in this paper like quantum determinants $\gamma(u)$ or $\Gamma(u)$ are formal Laurent series in $u^{-1}$ with coefficients in the corresponding algebra like $\mathcal{A}_q$. In other words, they are of the form
${\sum_{n\in \mathbb{Z}}} f_n u^{-n}$,
where $f_n \in \mathcal{A}_q$ and all but finitely many $f_n$ for $n<0$ vanish. In this case, we use the notation $\mathcal{A}_q\bigl(\bigl(u^{-1}\bigr)\bigr)$, and the convention that every rational function of the form $1/p(u)$, where $p(u)$ is a Laurent polynomial, is expanded in $u^{-1}$. The space of formal power series with coefficients $f_n$ in $\mathcal{A}_q$, i.e., when all $f_n$ are zero for $n<0$, is denoted by $\mathcal{A}_q\big[\big[u^{-1}\big]\big]$.

The subscript $[j]$ attached to an algebra-valued matrix $T$ indicates on which
factor in the tensor product of modules the entries of $T$ act nontrivially, and we use the convention
\begin{gather}\label{eq:conv-square}
((T)_{[\mathsf 2]}(T')_{[\mathsf 1]}(T'')_{[\mathsf 2]})_{ij}
=\sum_{k,\ell=1}^2 (T')_{k\ell} \otimes (T)_{ik}(T'')_{\ell j}.
\end{gather}

We also give a table of the most used notations:
\begin{center}\renewcommand{\arraystretch}{1.2}
\begin{tabular}{lll}
\hline
$\Uqhat$ & Definition~\ref{affineuq} & quantum affine algebra for $\mathfrak{sl}_2$ \\
$\Loop$ & Definition~\ref{def:loop} & quantum loop algebra for $\mathfrak{sl}_2$ \\
$\mathcal{A}_q$ & Definition~\ref{thm:m1com} & alternating central extension \\
&& of the $q$-Onsager algebra \\
\hline\\
\hline
$\mathsf{ev}_u$ & equation~\eqref{evu}& (formal) evaluation map from $\Loop$ to $\Uq$ \\
$\pi^{j}_u$ &equation~\eqref{evalrep}& (formal) evaluation representations from $\Loop$ \\
&&to $\End\bigl(\mathbb{C}^{2j+1}\bigr)$ \\
$\mathcal{E}^{(j+\h)}$ &Lemma~\ref{lem-E}& $\Loop$-intertwiner for fusion $(j+\h) \rightarrow (\h,j) $\\
$\bar{\mathcal{E}}^{(j-\h)}$ &Lemma~\ref{lem-barE}& $\Loop$-intertwiner \\
&& for reduction $(j-\h)\rightarrow(\h,j)$\\
$\mathcal{F}^{(j+\h)}$ &equations~\eqref{exprF}--\eqref{Fp2}& pseudo-inverse of $\mathcal{E}^{(j+\h)}$ \\
$\bar{\mathcal{F}}^{(j-\h)}$ &equations~\eqref{exprbF}--\eqref{coefbar} & pseudo-inverse of $\bar{\mathcal{E}}^{(j-\h)}$ \\
\hline \\
\hline
$\mathcal{L}^{(j)}(u)$ &equation~\eqref{cal-fused-L-uq}& fused L-operator in $\Uq\big[u,u^{-1}\big] \otimes \End\bigl(\mathbb{C}^{2j+1}\bigr)$ \\
$R^{(j_1,j_2)}(u)$ &equation~\eqref{fusRstj1j2}& fused R-matrix in $\End\big(\mathbb{C}^{2j_1+1}\big) \otimes \End\big(\mathbb{C}^{2j_2+1}\big)$ \\
$\mathcal{K}^{(j)}(u)$ &equation~\eqref{fusedunormK}& fused K-operator for $\mathcal{A}_q$ \\
&&in $u^{4j^2-2j} \mathcal{A}_q\big[\big[u^{-1}\big]\big] \otimes \End\bigl( \mathbb{C}^{2j+1}\bigr)$
\\ \hline\\
\hline
$\mathfrak{R}$ &Definition~\ref{def-univR}& universal R-matrix in $H\otimes H$ \\
${\bf L}^{\pm}(u)$ &equation~\eqref{defLpm}& affine L-operators in $\Loop\big[\big[ u^{\mp1}\big]\big] \otimes \End\bigl(\mathbb{C}^2\bigr) $ \\
${\bf L}^{(j)}(u)$ &Definition~\ref{defL}& spin-$j$ L-operator in $\Uq\big[\big[u^{-1}\big]\big] \otimes \End\bigl(\mathbb{C}^{2j+1}\bigr)$\\
$\mathcal{R}^{(j_1,j_2)}(u)$ &equation~\eqref{evalR} & spin-$j$ R-matrix in $\End\big(\mathbb{C}^{2j_1+1}\big) \otimes \End\big(\mathbb{C}^{2j_2+1}\big)$ \\
$\mathfrak{K}$ &Definition~\ref{defunivK}& universal K-matrix in $B \otimes H$ \\
$(\psi,J)$ &Definition~\ref{def:twist-pair}& twist pair: algebra automorphism $\psi$, \\
&& and Drinfeld twist $J$ \\
${\bf K}^{(j)}(u)$ &Definition~\ref{defbK}& spin-$j$ K-operator for $\mathcal{A}_q$ \\
&& in $\mathcal{A}_q\bigl(\bigl(u^{-1}\bigr)\bigr)\otimes \End\bigl(\mathbb{C}^{2j+1}\bigr)$\\
\hline
\end{tabular}
\end{center}

\section{Universal R- and K-matrices} \label{sec2}
Firstly, we recall the definition of a quasi-triangular Hopf algebra $H$ with the associated universal R-matrix that satisfies the universal Yang--Baxter equation. Then, inspired by the works~\cite{Ko20} and \cite{AV20}, we define in Section~\ref{sec:universal-K} a universal K-matrix associated with $H$, a pair of its consistent twists $(\psi,J)$, and its comodule algebra $B$ -- this is an element in $B\otimes H$ that satisfies a universal reflection equation (also called $\psi$-twisted reflection equation), see Proposition~\ref{propRE}. We also show a precise relation of our notion of the universal K-matrix to those introduced in~\cite{Ko20} and~\cite{AV20}. In~Section~\ref{exUqAq}, we introduce the main example of $H$ and $B$ considered in this paper: the quantum loop algebra $H = \Loop$ and its comodule algebra $B= {\mathcal A}_q$, the alternating central extension of the $q$-Onsager algebra.

\subsection{Universal R-matrix} \label{sec2.1}
Let $H$ be a Hopf algebra with coproduct $\Delta\colon H\rightarrow H \otimes H$, the counit $\epsilon\colon H \to \mathbb{C}$ and the antipode $S \colon H \to H$, which are subject to consistency conditions, for which we refer to~\cite[Chapter~4]{CP95}. We denote the opposite coproduct $\Delta^{\rm op}= \mathfrak{p}\circ\Delta$, where $\mathfrak{p}$ is the permutation operator, i.e., $\mathfrak{p}(x\otimes y) = y \otimes x$, for $x, y\in H$.
For $\mathfrak{R} \in H \otimes H,$ we use the notation $\mathfrak{R}_{12} = \mathfrak{R} \otimes 1$, $\mathfrak{R}_{23} = 1 \otimes \mathfrak{R}$, $\mathfrak{R}_{13} = \mathfrak{p}_{23}(\mathfrak{R}_{12})$.
\begin{defn}[\cite{Dr0}] \label{def-univR} For a Hopf algebra $H$, an invertible element $\mathfrak{R} \in H \otimes H$ is called a~universal R-matrix if it satisfies
 \begin{gather}
 \label{univR1} \tag{R1}
 \mathfrak{R}\Delta(x)=\Delta^{\rm op}(x) \mathfrak{R} \qquad \forall x \in H, \\ \label{univR2} \tag{R2}
 ( \Delta \otimes \id) (\mathfrak{R})= \mathfrak{R}_{13}\mathfrak{R}_{23},\\ \label{univR3} \tag{R3}
 ( \id \otimes \Delta) (\mathfrak{R}) = \mathfrak{R}_{13}\mathfrak{R}_{12}.
 \end{gather}
 If such $\mathfrak{R}$ exists, then the pair $(H,\mathfrak{R})$ is called a quasi-triangular Hopf algebra.
\end{defn}

We note that the universal R-matrix necessarily satisfies
\begin{gather} \label{SR}
 (S \otimes \id) ( \mathfrak{R}) = \mathfrak{R}^{-1} = ( \id \otimes S ) (\mathfrak{R}), \\ \label{epsR}
 (\epsilon \otimes \id ) (\mathfrak{R}) = 1 = (\id \otimes \epsilon) (\mathfrak{R}).
\end{gather}
Using the relations~\eqref{univR1}--\eqref{univR3}, one can show that the universal R-matrix satisfies the universal Yang--Baxter equation
\begin{equation}\label{YB-noparam}
 \mathfrak{R}_{12}\mathfrak{R}_{13}\mathfrak{R}_{23}=\mathfrak{R}_{23}\mathfrak{R}_{13}\mathfrak{R}_{12}.
\end{equation}
It is well known~\cite{Dr0} that the universal R-matrix coming from a quasi-triangular Hopf algebra gives a way to generate R-matrices on tensor product of representations, via evaluations as we will see in Section \ref{sec4}.

\subsection{Twist pairs}
We begin with introducing $\psi$-twisting for Hopf algebras, as in~\cite[Section~2.4]{AV20}.

\begin{defn}\label{def:HR} Let $(H,\mathfrak{R})$ be a quasi-triangular Hopf algebra and $\psi\colon H\rightarrow H$ an algebra automorphism. The $\psi$-twisting of $(H,\mathfrak{R})$ is the quasi-triangular Hopf algebra $\big(H^\psi,\mathfrak{R}^{ \psi \psi}\big)$ obtained from $(H,\mathfrak{R})$ by pullback through $\psi$, i.e., $H^\psi$ is the Hopf algebra with same multiplication, new coproduct, counit and antipode
 \begin{equation} \label{Dpsi}
 \Delta^\psi= (\psi \otimes \psi) \circ \Delta \circ \psi^{-1}, \qquad \epsilon^\psi = \epsilon \circ \psi^{-1}, \qquad S^\psi = \psi \circ S \circ \psi^{-1},
 \end{equation}
 and the universal R-matrix is given by
 \begin{equation} \label{Rpsipsi}
 \mathfrak{R}^{\psi\psi}= (\psi\otimes\psi) (\mathfrak{R}).
 \end{equation}
\end{defn}

Let $H^{{\rm cop}}$ be the Hopf algebra with the coproduct $\Delta^{\rm op}$, the antipode $S^{-1}$ and its R-matrix is $\mathfrak{R}_{21}$, while the other structure maps are the same as for $H$. In what follows, we also use the $\psi$-twisting of $H^{{\rm cop}}$, denoted by $H^{{\rm cop},\psi}$, with $\Delta^{{\rm op},\psi}=\mathfrak{p}\circ\Delta^\psi$. We then immediately get the following lemma.

\begin{lem}\label{lempsi}
 The pair $\big(H^{{\rm cop},\psi},\mathfrak{R}_{21}^{\psi\psi}\big)$ is a quasi-triangular Hopf algebra.
\end{lem}

To introduce the concept of universal K-matrix, we recall Drinfeld twists.

\begin{defn}[\cite{Dr0,Dr89}] A Drinfeld twist of a Hopf algebra H is an invertible element $J \in H \otimes H$ satisfying the property
 \begin{equation} \label{epsJ}
 (\epsilon \otimes \id)(J)=1=(\id \otimes \epsilon)(J)
 \end{equation}
 and the cocycle identity
 \begin{equation}\label{cocycle}
 (J \otimes 1) (\Delta \otimes \id)(J) = (1\otimes J) (\id \otimes \Delta)(J).
 \end{equation}
\end{defn}
Given a Drinfeld twist $J$ one obtains a new quasi-triangular Hopf algebra $(H_J, \mathfrak{R}_J)$ with the coproduct and the antipode~\cite{Dr89b}, see also in our conventions~\cite[Theorem~2.3.4]{book-M},
\begin{equation}\label{DeltaTw}
 \Delta_J(x)=J \Delta(x) J^{-1}, \qquad S_J(x) = U x U^{-1}, \qquad \forall x \in H,
\end{equation}
where we set $U = \sum J_1 S(J_2)$ and with $J= \sum J_1 \otimes J_2$.
The universal R-matrix is given by
\begin{equation} \label{Rj}
 \mathfrak{R}_J = J_{21} \mathfrak{R} J^{-1}.
\end{equation}
We now recall~\cite[Definition~2.2]{AV20}.

\begin{defn} \label{def:twist-pair}
 Let $(H,\mathfrak{R})$ be a quasi-triangular Hopf algebra.
 A twist pair $(\psi,J)$ is the datum of an algebra automorphism $\psi \colon H\rightarrow H$ and a Drinfeld twist $J \in H \otimes H$ such that $\big(H^{{\rm cop},\psi}, \mathfrak{R}_{21}^{\psi \psi}\big)= ( H_J, \mathfrak{R}_J)$, i.e., such that $\epsilon^\psi=\epsilon$,
 \begin{gather} \label{Dopsi}
 \Delta^{{\rm op},\psi}(x) =J \Delta(x) J^{-1} \qquad \forall x\in H,\qquad
 \mathfrak{R}_{21}^{\psi \psi} = J_{21} \mathfrak{R} J^{-1}.
 \end{gather}
\end{defn}

\subsection{Universal K-matrix}\label{sec:universal-K}
Appel and Vlaar introduced the notion of a cylindrical bialgebra~\cite[Definition~2.3]{AV20} which is a quasi-triangular bialgebra with a twist pair and a universal solution of a twisted reflection equation in $H\otimes H$. This approach is convenient for scalar K-matrix solutions of the parameter-dependent reflection equation~\eqref{REj}. It is however not a convenient framework for K-operator solutions of more general reflection equation~\eqref{genREAj} associated to general $H$-comodule algebras~$B$.
To change this, we define a universal K-matrix via axioms~\eqref{univK1}--\eqref{univK3} and show that it satisfies a universal reflection equation in $B \otimes H \otimes H$ that specializes on representations to the reflection equation~\eqref{genREAj}, as will be further demonstrated in Section~\ref{sec4}. We begin with the following standard definition.
\begin{defn} \label{def:right-com}
 An algebra $B$ is called a right comodule algebra over a Hopf algebra $H$ if there exists an algebra map $\delta \colon B \rightarrow B \otimes H $, which we call right coaction, such that the coassociativity and counital conditions hold
 \begin{equation} \label{def-coideal}(\id \otimes \Delta) \circ \delta = (\delta \otimes \id) \circ \delta, \qquad (\id \otimes \epsilon) \circ \delta = \id.
 \end{equation}
\end{defn}
Let $(\psi,J)$ be a twist pair for a Hopf algebra $H$ and $B$ is a right comodule algebra over $H$.
We define a universal K-matrix $\mathfrak{K} \in B\otimes H$. Here we use the notation $\mathfrak{K}_{12} = \mathfrak{K} \otimes 1$, $\mathfrak{K}_{13} = \mathfrak{p}_{23}( \mathfrak{K}_{12})$.
\begin{defn}\label{defunivK}
 We say that $\mathfrak{K}\in B\otimes H$ is a universal K-matrix if the following relations hold for all $b \in B$:
 \begin{gather}\label{univK1}\tag{K1}
 \mathfrak{K} \delta(b) = \delta^\psi(b)\mathfrak{K}, \qquad \text{with } \delta^\psi= (\id \otimes \psi)\circ \delta,\\ \label{univK2}\tag{K2}
 (\delta \otimes \id) (\mathfrak{K}) = \bigl(\mathfrak{R}^\psi\bigr)_{32} \mathfrak{K}_{13} \mathfrak{R}_{23},\\ \label{univK3}\tag{K3}
 (\id \otimes \Delta) (\mathfrak{K}) = J_{23}^{-1} \mathfrak{K}_{13} \mathfrak{R}_{23}^\psi \mathfrak{K}_{12},
 \end{gather}
 where
 \begin{equation} \label{Rpsi}
 \mathfrak{R}^\psi = (\psi \otimes \id) (\mathfrak{R}).
 \end{equation}
\end{defn}
We note that $\delta^\psi$ defines a comodule structure on $B$ over $H^\psi$. Therefore, by~\eqref{univK1} $\mathfrak{K}$ intertwines two actions of $B$ on $B \otimes H$, given by $\delta$ and $\delta^\psi$ respectively. By analogy with~\eqref{univR1}, we call~\eqref{univK1} the twisted intertwining relation.
We now make several remarks concerning this definition.
\begin{rem}\label{rem:29} From the axioms~\eqref{univK2}--\eqref{univK3}, we get some relations on the level of the algebra.
 \begin{enumerate}[label=(\roman*),itemsep=0em]
 \item We provide a consistency check of the axioms~\eqref{univK2} and~\eqref{univK3}.
 From coassociativity property~\eqref{def-coideal}, we have
 $(\id \otimes \Delta \otimes \id ) \circ ( \delta \otimes \id ) ( \mathfrak {K}) = (\delta \otimes \id \otimes \id ) \circ (\delta \otimes \id ) (\mathfrak{K})$.
 This relation is checked using~\eqref{univR2} and~\eqref{univK2}. Indeed, the left-hand side equals $\bigl(\mathfrak{R}^\psi\bigr)_{43} \bigl(\mathfrak{R}^\psi\bigr)_{42} \mathfrak{K}_{14} \mathfrak{R}_{24} \mathfrak{R}_{34}$ where we used $ (\Delta \otimes \id) \bigl( \bigl(\mathfrak{R}^\psi\bigr)_{21}\bigr) = \bigl(\mathfrak{R}^\psi\bigr)_{32} \bigl(\mathfrak{R}^\psi\bigr)_{31} $, while the right-hand side gives the same expression using twice~\eqref{univK2}.
 The counital property in~\eqref{def-coideal} is checked using~\eqref{univK3} and~\eqref{epsR} as follows:
 \begin{equation}
 (\id \otimes \epsilon \otimes \id ) \circ ( \delta \otimes \id ) ( \mathfrak{K}) = (\id \otimes \epsilon \otimes \id ) ( \bigl(\mathfrak{R}^\psi\bigr)_{32} \mathfrak{K}_{13} \mathfrak{R}_{23} ) = \mathfrak{K}.
 \end{equation}
 \item Recall that
 $ (\id \otimes \epsilon ) \circ \Delta = \id$.
 Then, using~\eqref{univK3},~\eqref{epsR},~\eqref{epsJ}, and that $\epsilon$ is an algebra map
 \begin{align}
 \mathfrak{K}& = \lbrack (\id \otimes \id \otimes \epsilon)\circ (\id \otimes \Delta) \rbrack (\mathfrak{K}) = (\id \otimes \id \otimes \epsilon) \big(J_{23}^{-1} \mathfrak{K}_{13} \mathfrak{R}_{23}^\psi \mathfrak{K}_{12}\bigr)\nonumber\\
 &= (\lbrack (\id \otimes \epsilon) (\mathfrak K) \rbrack \otimes 1 ) \mathfrak{K}.\label{eq:espK}
 \end{align}
 Applying again $( \id \otimes \epsilon)$ on~\eqref{eq:espK}, we find that $(\id \otimes \epsilon)(\mathfrak{K})$ is an idempotent.
 If in addition, $\mathfrak{K}$ is invertible, it follows $(\id \otimes \epsilon)(\mathfrak{K}) = 1$
 which is the analogue of \eqref{epsR} for $\mathfrak{R}$.
 \end{enumerate}
\end{rem}

\begin{rem}\label{rem:Kolb}
 The previous definitions of universal K-matrices from the works~\cite{AV20,Ko20} can be obtained as special cases of Definition~\ref{defunivK}.
 \begin{enumerate}[label=(\roman*),itemsep=0em]
 \item
 It is easy to check that $\big(\id, \mathfrak{R}_{21}^{-1}\big)$ is a twist pair. In particular, one finds that~\eqref{univK1}--\eqref{univK3}, for $\psi = \id$ and $J= \mathfrak{R}_{21}^{-1}$, correspond to the axioms for the universal K-matrix defined in~\cite[Definition~2.7]{Ko20}. Moreover, assuming \eqref{univK1}--\eqref{univK3} hold for a twist pair $\big(\psi, \mathfrak{R}_{21}^{-1}\big)$, then the algebra automorphism ${\mathcal K}_{KY}:= \bigl(\id\otimes \psi^{-1}\bigr)\circ Ad(\mathfrak{K})$ of $H\otimes B$, with $Ad(\mathfrak{K}):= \mathfrak{K}(-)\mathfrak{K}^{-1}$, satisfies the axioms of~\cite[Definition~6.11]{KY20} with ${\mathcal R}:=Ad(\mathfrak{R})$.

 \item
 Assume that $B$ is a right coideal subalgebra of $H$, i.e., $\delta=\Delta|_B$, and define
 ${\bf \mathcal{K}}= (\epsilon \otimes \id )(\mathfrak{K})$.
 Applying the counit on the first tensor factor of the universal K-matrix in~\eqref{univK1}--\eqref{univK3} yields
 \begin{gather} \label{eps-K1}
 \mathcal{K} b = \psi(b) \mathcal{K} \qquad \forall b \in B, \\ \label{eps-K2}
 \mathfrak{K} = \bigl(\mathfrak{R}^\psi\bigr)_{21}\mathcal{K}_2 \mathfrak{R}, \\ \label{eps-K3}
 \Delta|_B(\mathcal{K}) = J^{-1} \mathcal{K}_2 \mathfrak{R}^\psi \mathcal{K}_1.
 \end{gather}
 We recall that in this setting of $B$ a coideal subalgebra of $H$, Appel and Vlaar introduced in~\cite{AV20} one-component universal K-matrices $k \in H$.
 The formulas~\eqref{eps-K1} and~\eqref{eps-K3}, with the identification $\mathcal{K} =k$, correspond respectively to the defining relations of the so-called cylindrically invariant subalgebra $B$ and the cylindrical pair $(H,\mathfrak{R})$, see~\cite[Definition~2.3]{AV20}.
 We thus get the one-component universal K-matrix of~\cite{AV20} from our 2-component~$\mathfrak{K}$.
 The opposite is also true: starting from a solution $\mathcal{K}$ of~\eqref{eps-K1} and~\eqref{eps-K3}, and assuming that~$\mathfrak{K}$ defined by~\eqref{eps-K2} lies in $B \otimes H$, one can check that $\mathfrak{K}$ satisfies the axioms~\eqref{univK1}--\eqref{univK3}.
 Indeed, checking~\eqref{univK1} and~\eqref{univK2} is straightforward, while for~\eqref{univK3}
 we use the equalities
 \begin{align*}
 (\id \otimes \Delta) \big\lbrack \bigl(\mathfrak{R}^\psi\bigr)_{21} \big\rbrack = J_{23}^{-1} \bigl(\mathfrak{R}^\psi\bigr)_{31} \bigl(\mathfrak{R}^\psi\bigr)_{21} J_{23}, \qquad
 \bigl(\mathfrak{R}^\psi\bigr)_{21} \mathfrak{R}^\psi_{23} \mathfrak{R}_{13} = \mathfrak{R}_{13} \mathfrak{R}^\psi_{23} \bigl(\mathfrak{R}^\psi\bigr)_{21}.
 \end{align*}
 The first one is obtained using
 $\Delta(\psi(x)) = J^{-1} \lbrack( \psi \otimes \psi) \circ \Delta^{\rm op}(x) \rbrack J$
 with $(\id \otimes \Delta^{\rm op}) (\mathfrak{R}_{21}) = \mathfrak{R}_{31} \mathfrak{R}_{21}$
 while we applied $(\id \otimes \psi\otimes \id)\circ\mathfrak{p}_{12}$ to the universal Yang--Baxter equation~\eqref{YB-noparam} to get the second equality.
 Then, we indeed get~\eqref{univK3}
 \begin{align*}
 (\id \otimes \Delta) ( \mathfrak{K}) &= J_{23}^{-1} \bigl(\mathfrak{R}^\psi\bigr)_{31} \mathcal{K}_3
 \bigl(\mathfrak{R}^\psi\bigr)_{21}
 \mathfrak{R}^\psi_{23}
 \mathfrak{R}_{13}
 \mathcal{K}_2 \mathfrak{R}_{12}
 \\
 &= J_{23}^{-1} \bigl(\mathfrak{R}^\psi\bigr)_{31} \mathcal{K}_3 \mathfrak{R}_{13} \mathfrak{R}^\psi_{23} \bigl(\mathfrak{R}^\psi\bigr)_{21} \mathcal{K}_2 \mathfrak{R}_{12}
 = J_{23}^{-1} \mathfrak{K}_{13} \mathfrak{R}^\psi_{23} \mathfrak{K}_{12}.
 \end{align*}
 \end{enumerate}
\end{rem}

We now derive a universal reflection equation based on the axioms~\eqref{univK1}--\eqref{univK3}.

\begin{prop} \label{propRE}
 Let $(\psi,J)$ be a twist pair. The universal K-matrix satisfies the $\psi$-twisted reflection equation
 \begin{equation} \label{psiRE}
 \mathfrak{K}_{12} \bigl( \mathfrak{R}^{\psi}\bigr)_{32} \mathfrak{K}_{13} \mathfrak{R}_{23} = \mathfrak{R}_{32}^{\psi \psi} \mathfrak{K}_{13} \mathfrak{R}_{23}^\psi \mathfrak{K}_{12}.
 \end{equation}
\end{prop}
\begin{proof}
 Multiply~\eqref{univK2} on the left by $\mathfrak{K}_{12}$ to get
 \begin{equation} \label{proP1}
 \mathfrak{K}_{12}[(\delta \otimes \id) (\mathfrak{K})] = \mathfrak{K}_{12} \bigl(\mathfrak{R}^\psi\bigr)_{32} \mathfrak{K}_{13} \mathfrak{R}_{23}.
 \end{equation}
 Then, using~\eqref{univK1}, the left-hand side of this equation equals
 \begin{equation} \label{proP2}
 [(\id \otimes \psi \otimes \id) \circ (\delta \otimes \id )(\mathfrak{K})] \mathfrak{K}_{12} = \mathfrak{R}_{32}^{\psi \psi} \mathfrak{K}_{13} \mathfrak{R}_{23}^\psi \mathfrak{K}_{12},
 \end{equation}
 where we used again~\eqref{univK2}.
 Equating~\eqref{proP1} with~\eqref{proP2}, the equation~\eqref{psiRE} follows.
\end{proof}

\subsection[Example of (H,B)]{Example of $\boldsymbol{(H,B)}$} \label{exUqAq}
We now introduce our main example of the pair $(H,B)$ where the Hopf algebra $H= \Loop$ is the quantum loop algebra and its right comodule algebra $B={\mathcal A}_q$ is a central extension of the $q$-Onsager algebra.
We first recall the definition of the quantum affine algebra \smash{$\Uqhat$}.
\begin{defn}\label{affineuq}
 Define the extended Cartan matrix $(a_{ij})_{i,j\in\{0,1\}}$ with $a_{ii} = 2$, $a_{ij} = -2$ for~${i \neq j}$.
 The quantum affine algebra \smash{$\Uqhat$} is a Hopf algebra generated by the elements $E_i$, $F_i$, \smash{$K_i^{\pm \h}$}, $ i \in \{0,1\}$, satisfying
 \begin{gather}
 K_i^\h E_j = q^{\frac{a_{ij}}{2}} E_j K_i^\h, \qquad
 K_i^\h F_j = q^{-\frac{a_{ij}}{2}} F_j K_i^\h, \qquad
 [E_i,F_j]= \delta_{i,j} \frac{K_i - K_i^{-1}}{q-q^{-1}}, \nonumber\\
 K_i^\h K_i^{-\h} = K_i^{-\h} K_i^\h = 1, \qquad K_0^\h K_1^\h =K_1^\h K_0^\h, \label{affine01}
 \end{gather}
 with the $q$-Serre relations, recall the definition of the $q$-commutator \eqref{def:qcom},
 \begin{equation}\label{affine03}
 [E_i,[E_i,[E_i,E_j]_q]_{q^{-1}}]=0, \qquad [F_i,[F_i,[F_i,F_j]_q]_{q^{-1}}]=0.
 \end{equation}
 Let $\Delta \colon \Uqhat \to \Uqhat \otimes \Uqhat$, $\epsilon \colon \Uqhat \to \mathbb{C}$ and $S \colon \Uqhat \to \Uqhat$ be respectively the coproduct, the counit and the antipode. They are given by
 \begin{gather}
 \Delta(E_i)= E_i \otimes K_i^{\h} + K_i^{-\h} \otimes E_i, \qquad \Delta(F_i) =F_i \otimes K_i^{\h} + K_i^{-\h} \otimes F_i, \nonumber\\ \Delta\bigl(K_i^{\pm \h}\bigr) = K_i^{ \pm\h} \otimes K_i^{\pm \h}, \label{affinecoprodefk}\\ \label{affine-counit}
 \epsilon(E_i)=\epsilon(F_i)=0, \qquad \epsilon\bigl(K_i^{\pm \h}\bigr)=1, \\
\label{affine-antipode}
 S(E_i)=-qE_i,\qquad S(F_i)=- q^{-1}F_i, \qquad S\bigl(K_i^{\pm\h}\bigr)= K^{\mp\h}_i.
 \end{gather}
\end{defn}
The element \smash{$K_0^\h K_1^\h$} is central.

\begin{defn}\label{def:loop}
 The quantum loop algebra $\Loop$ is the unital associative ${\mathbb C}$-algebra generated by the elements $E_i$, $F_i$, \smash{$K_i^{\pm \h} $}; $ i \in \{0,1\}$, which satisfy the defining relations of the \smash{$\Uqhat$} algebra~\eqref{affine01}--\eqref{affine03} with the extra relation
 \smash{$
 K_0^\h K_1^\h=1$}.
 The Hopf algebra structure defined by~\eqref{affinecoprodefk}--\eqref{affine-antipode} descends to $\Loop$.
\end{defn}

Note that a presentation of $\Loop$ of Faddeev--Reshetikhin--Taktadjan type is also known~\cite{FRT87}. It involves Ding--Frenkel L-operators satisfying the Yang--Baxter algebra \cite{DF93}, see details in Section~\ref{sec:comod}.

The explicit form of the universal R-matrix associated with the quantized Kac--Moody algebras was proposed in~\cite[Theorem~2]{Tolstoy1991} and in~\cite[equation~(58)]{Tolstoy1992},\footnote{As the universal R-matrix $\mathfrak{R}$ of $\Uqhat$ has the form of a product over an \textit{infinite} set of root vectors, and so strictly speaking $\mathfrak{R}$ is not an element of the tensor product $\Uqhat \otimes \Uqhat$ but belongs to its appropriate completion, see, e.g.,~\cite[Section~2.5]{AV22}. This fact however does not affect the use of the R- or K-matrix axioms.} in terms of root vectors for~$\Uqhat$, see also~\cite{Da98,LS}.
In our work, we consider the universal R-matrix $\mathfrak{R}$ associated with the Hopf algebra quotient $\Loop$ and review its explicit form in our conventions in Appendix~\ref{appC}. It is important to note that the coproduct convention we use is different to the one used in~\cite{Tolstoy1991}. The two coproducts and corresponding universal R-matrices are related via the automorphism~$\nu$
\begin{gather}\label{autnu}
 \nu(E_i) = E_i K^{\h}_i, \qquad \nu(F_i) = K^{-\h}_iF_i, \qquad \nu\bigl(K_i^{\pm \h}\bigr) = K_i^{\pm \h}.
\end{gather}
We then have
\begin{equation} \label{DeltaTK}
 \Delta^{{\rm KT}}= (\nu \otimes \nu) \circ \Delta \circ \nu^{-1},\qquad \bigl(\nu^{-1} \otimes \nu^{-1}\bigr) \bigl(\mathfrak{R}^{{\rm KT}}\bigr) = \mathfrak{R},
\end{equation}
where $\mathfrak{R}^{{\rm KT}}$ denotes the image of the universal R-matrix given in~\cite[equation~(58)]{Tolstoy1992} under the quotient map to $\Loop$, i.e., after setting the central charge to zero.

Let us introduce an automorphism $\eta$ of $\Loop$
\begin{gather}
 \eta(E_0)=F_1, \qquad \eta(E_1) = F_0, \qquad \eta\bigl(K_0^\h\bigr) = K_1^{-\h}, \nonumber\\
 \eta(F_1) = E_0, \qquad \eta(F_0)=E_1, \qquad \eta\bigl(K_1^\h\bigr)=K_0^{-\h}.\label{autrho}
\end{gather}
\begin{Example}\label{extwist}
 For $H=\Loop$, both the pairs
 $(\psi,J)=(\eta,1 \otimes 1)$ and $(\psi,J)=(\eta,\mathfrak{R}_{21}\mathfrak{R})$, with the automorphism $\eta$ from~\eqref{autrho}, are examples of twist pairs from Definition~\ref{def:twist-pair}. The first choice is the main example considered in this paper. To verify~\eqref{Dopsi}, using the definition of $\Delta$ given in~\eqref{affinecoprodefk}, we get indeed $
 \Delta^{{\rm op},\eta}=\Delta$. It remains to show that $\mathfrak{R}_{21}^{\eta \eta}= \mathfrak{R}$. From Lemma~\ref{lempsi}, the pair $\bigl(H^{{\rm cop},\eta},\mathfrak{R}_{21}^{\eta\eta}\bigr)$ satisfies
 \eqref{univR1}--\eqref{univR3} with the only substitution $\mathfrak{R} \rightarrow \mathfrak{R}_{21}^{\eta \eta}$.
 Recall that an invertible solution of~\eqref{univR1}--\eqref{univR3} of the form~\cite[equation~(42)]{Tolstoy1992} is unique~\cite[Theorem 7.1]{Tolstoy1992}, and is given by~\cite[equation~(58)]{Tolstoy1992}.
 We then note that $\mathfrak{R}^{\eta\eta}_{21}$ is of the same form~\cite[equation~(42)]{Tolstoy1992}. As $\mathfrak{R}^{\eta \eta}_{21}$ satisfies the same equations as $\mathfrak{R}$, as we have $\Delta^{{\rm op},\eta}=\Delta$, it follows from the uniqueness that they are equal and $(\eta, 1 \otimes 1)$ is thus a twist pair.
\end{Example}

\subsubsection{Evaluation map}
In further sections, we will use the so-called evaluation map to study the tensor product representation of the algebra $\Loop$.
We first recall that the quantum algebra $\Uq$ is a Hopf subalgebra of \smash{$\Uqhat$} generated by $E_1$, $F_1$ and $K_1$, and for brevity we denote them by $E$, $F$, and~$K$, respectively.
Now, consider the algebra map
$ \varphi \colon \Loop \rightarrow \Uq$
defined by the equations\looseness=-1
\begin{gather*}
\begin{split}
 & \varphi(E_0)= F, \qquad \varphi(F_0)=E, \qquad \varphi\bigl(K_0^\h\bigr)=K^{-\h}, \\
 & \varphi(E_1)=E, \qquad \varphi(F_1)=F, \qquad \varphi\bigl(K_1^\h\bigr) = K^\h.
 \end{split}
\end{gather*}
Let $\phi_u$ be the $\mathbb{Z}$-gradation automorphism of $\Loop$,
$
 \phi_u \colon \Loop \rightarrow \Loop$,
with $u \in \mathbb{C}^*$, defined~by
\begin{equation}\label{phiu}
 \phi_u(E_i) = u^{-1} E_i, \qquad \phi_u(F_i)=u F_i, \qquad \phi_u \bigl(K_i^\h\bigr) = K_i^\h, \qquad i=0,1.
\end{equation}
Then the evaluation map $\mathsf{ev}_u$\footnote{For a more general evaluation map, see, for instance,~\cite[equation~(4.32)]{Boos2012}. In this paper, we set $s_0=s_1=-1$.}
\begin{equation}\label{evu}
 \mathsf{ev}_u \colon\ \Loop \rightarrow \Uq
\end{equation}
is defined by the composition
\begin{equation} \label{eval}
 \mathsf{ev}_u = \varphi \circ \phi_u.
\end{equation}

\subsubsection[Comodule algebra A\_q]{Comodule algebra $\boldsymbol{\mathcal{A}_q}$} \label{sec:OqAq}
Before giving our example of the comodule algebra $B=\mathcal{A}_q$ over $\Loop$, which is a central extension of the $q$-Onsager algebra, we recall the definition of the latter.
The definition of the $q$-Onsager algebra $O_q$~\cite{Ba05,Ter03} is given in terms of two generators $W_0$, $W_1$ satisfying the so-called $q$-Dolan--Grady relations
\begin{gather*}
 \lbrack {\normalfont W}_0, \lbrack {\normalfont W}_0, \lbrack {\normalfont W}_0, {\normalfont W}_1\rbrack_q \rbrack_{q^{-1}} \rbrack =\rho \lbrack {\normalfont W}_0, {\normalfont W}_1 \rbrack,\qquad
 \lbrack {\normalfont W}_1, \lbrack {\normalfont W}_1, \lbrack {\normalfont W}_1, {\normalfont W}_0\rbrack_q \rbrack_{q^{-1}}\rbrack = \rho \lbrack {\normalfont W}_1, {\normalfont W}_0 \rbrack,
\end{gather*}
where we fix
\begin{equation}
 \rho=k_+k_-\bigl(q+q^{-1}\bigr)^2 \qquad \text{with} \ k_\pm\in \mathbb{C}^*.
 \label{rho}
\end{equation}

The algebra $B=\mathcal{A}_q$ has first appeared in~\cite{BK05,BSh1} in the form of a reflection algebra that we review in Section~\ref{sec5}. It was understood later on that $\mathcal{A}_q$ is isomorphic to the central extension of the $q$-Onsager algebra~\cite{BasBel,Ter21}, namely to $O_q \otimes \mathbb{C}[Z]$ where $\mathbb{C}[Z] = \mathbb{C}[z_1,z_2,\dots]$ is a polynomial algebra with infinitely many indeterminates $z_i$, for $i\in\mathbb{N}$. Here, we recall a so-called \textit{compact} presentation of this algebra.
\begin{defn}[see~\cite{Ter21b}]\label{thm:m1com}
 $\mathcal{A}_q$ is an associative algebra over ${\mathbb C}$ generated by ${\normalfont \tW}_0$, ${\normalfont \tW}_1$, $\lbrace {\normalfont \tG}_{k+1} \rbrace_{k \in \mathbb N}$ subject to the following defining relations
 \begin{gather*}
 \lbrack {\normalfont \tW}_0, \lbrack {\normalfont \tW}_0, \lbrack {\normalfont \tW}_0, {\normalfont \tW}_1\rbrack_q \rbrack_{q^{-1}} \rbrack =\rho \lbrack {\normalfont \tW}_0, {\normalfont \tW}_1 \rbrack,\qquad
 \lbrack {\normalfont \tW}_1, \lbrack {\normalfont \tW}_1, \lbrack {\normalfont \tW}_1, {\normalfont \tW}_0\rbrack_q \rbrack_{q^{-1}}\rbrack = \rho \lbrack {\normalfont \tW}_1, {\normalfont \tW}_0 \rbrack, \\
 \lbrack {\normalfont \tW}_1, {\normalfont \tG}_1 \rbrack =
 \lbrack {\normalfont \tW}_1, \lbrack {\normalfont \tW}_1, {\normalfont \tW}_0 \rbrack_q \rbrack,\qquad
 \lbrack {\normalfont \tG}_1, {\normalfont \tW}_0 \rbrack =
 \lbrack \lbrack {\normalfont \tW}_1, {\normalfont \tW}_0 \rbrack_q, {\normalfont \tW}_0 \rbrack,\\
 \lbrack {\normalfont \tW}_1, {\normalfont \tG}_{k+1} \rbrack =
 \rho^{-1}
 \lbrack {\normalfont \tW}_1, \lbrack {\normalfont \tW}_1, \lbrack {\normalfont \tW}_0,
 {\normalfont \tG}_{k}
 \rbrack_q
 \rbrack_q
 \rbrack, \qquad k\geq1,\\
 \lbrack {\normalfont \tG}_{k+1}, {\normalfont \tW}_0\rbrack =
 \rho^{-1}
 \lbrack\lbrack \lbrack {\normalfont \tG}_{k}, {\normalfont \tW}_1 \rbrack_q,
 {\normalfont \tW}_0 \rbrack_q,
 {\normalfont \tW}_0
 \rbrack, \qquad k\geq1,\\
 \lbrack {\normalfont \tG}_{k+1}, {\normalfont \tG}_{\ell+1} \rbrack=0, \qquad k,\ell \in \mathbb{N},
 \end{gather*}
 where $\rho$ is given by~\eqref{rho}.
\end{defn}

We discuss the coaction map $\delta\colon {\mathcal A}_q \rightarrow {\mathcal A}_q\otimes \Loop$ using the reflection algebra presentation of $\mathcal{A}_q$ in Section \ref{sec:subcoac}.
The corresponding coaction of the fundamental generators $\tW_0$, $\tW_1$
takes the form
\begin{align}
 \delta( {\tW}_0)&= 1 \otimes \bigl( k_+ q^{\h} E_1 K_1^{\h} + k_- q^{-\h} F_1 K_1^{\h} \bigr) + {\tW}_0 \otimes K_1,\label{coW0}\\
 \delta( {\tW}_1)&= 1 \otimes \bigl( k_+ q^{-\h} F_0 K_0^{\h} + k_- q^\h E_0 K_0^{\h} \bigr)+ {\tW}_1 \otimes K_0.\label{coW1}
\end{align}
In what follows, it is sufficient to consider the evaluation of the coaction map $\delta$, denoted by $\delta_w$
\begin{equation}
 \delta_w = (\id \otimes \normalfont{\mathsf{ev}}_w) \circ \delta \colon \ \mathcal{A}_q \rightarrow \mathcal{A}_q \otimes \Uq.\label{defevco}
\end{equation}

The proof of the following proposition is postponed to the end of Section \ref{sec:evalcoac}.
\begin{prop} \label{prop:coact}
 The evaluated coaction map $\delta_w\colon \mathcal{A}_q \rightarrow \mathcal{A}_q \otimes \Uq$ is such that
 \begin{gather*}
 \delta_w(\normalfont{\tW}_0)= 1 \otimes \bigl(k_+ q^{\h} w^{-1} EK^{\h} + k_- q^{-\h} w F K^{\h}\bigr) + \normalfont{\tW}_0 \otimes K,\\
 \delta_w(\normalfont{\tW}_1)=1 \otimes \bigl(k_+ q^{-\h} wE K^{-\h} + k_- q^{\h} w^{-1}F K^{-\h}\bigr) + \normalfont{\tW}_1 \otimes K^{-1}, \\
 \delta_w(\normalfont{\tG}_{k+1})=
 \frac{k_-}{k_+}\frac{\bigl(q-q^{-1}\bigr)^2}{q+q^{-1}}
 \left(\normalfont{\tG}_k - \frac{\bigl(q+q^{-1}\bigr)}{\rho}\big[\normalfont{\tW}_0,\big[\normalfont{\tW}_0,\normalfont{\tG}_k \big]_q\big] \right)\otimes F^2\nonumber \\
 \phantom{\delta_w(\normalfont{\tG}_{k+1})=}{}+\frac{\bigl(q-q^{-1}\bigr)}{k_+\bigl(q+q^{-1}\bigr)} \bigl( q^{-\h} w^{-1} \big[ \tW_0,\tG_{k+1}\big]_q \otimes F K^\h + q^\h w \big[ \tG_{k+1},\tW_1\big]_q \otimes F K^{-\h} \bigr)\nonumber \\
 \phantom{\delta_w(\normalfont{\tG}_{k+1})=}{}+\normalfont{\tG}_{k+1} \otimes 1 - \frac{\normalfont{\tG}_k}{q+q^{-1}} \otimes \bigl( w^2 K^{-1} + w^{-2}K\bigr),
 \end{gather*}
 with the initial condition
 %
 \smash{$\normalfont{\tG}_{k}\bigr\rvert_{k=0} =
 k_+ k_- \frac{(q+q^{-1})^2}{q-q^{-1}}$}.
\end{prop}
\section[Tensor product representations of LUqsl2]{Tensor product representations of $\boldsymbol{\Loop}$}\label{analysUq}

Our fusion procedure for K-operators requires first analysis of the tensor product of evaluation representations of $\Loop$. The reducibility criteria in terms of ratios of the evaluation parameters for these tensor products are known~\cite[Section~4.9]{CP}. In this section,
we study the sub-quotient structure of the tensor products $\mathbb{C}_{u_1}^{2}
\otimes \mathbb{C}_{u_2}^{2j+1}$ in more details. We first show that with conditions on the ratio of evaluation parameters as in Lemmas~\ref{lem-E} and~\ref{lem-barE}, this tensor product representation of $\Loop$ admits a sub-representation either of spin-$\bigl(j+\h\bigr)$ or of spin-$\bigl(j-\h\bigr)$. As a main new result of this section, we then provide explicit matrix form of the corresponding intertwining operators and their pseudo-inverses.
This is done in Sections~\ref{intert-fus} and~\ref{sec:reduction-map} for finite-dimensional evaluation representations, and then in Section~\ref{sec:formal-ev} for their infinite-dimensional analogues with formal evaluation parameters.

First, recall that finite-dimensional irreducible representations of $\Uq$ are labeled by a non-negative integer or half-integer $j$, of the dimension $2j+1$, and denoted simply by $\mathbb{C}^{2j+1}$. Its basis is given by $\ket{j, m}$ with $ m \in \{-j, -j+1, \dots, j-1, j\}$, and with the action
\begin{gather}
 E\ket{j,m} = A_{j,m} \ket{j,m+1}, \qquad F\ket{j,m} = B_{j,m}\ket{j,m-1},\nonumber \\ K^{\pm \h}\ket{j,m}= q^{\pm m} \ket{j,m},\label{evalEFK}
\end{gather}
with
\begin{equation} \label{AB}
 A_{j,m}=\sqrt{\left[j-m\right]_q \left[j+m+1\right]_q}, \qquad B_{j,m}=\sqrt{ \left[ j+m\right]_q \left[j-m+1\right]_q}.
\end{equation}
Let $\pi^{j}$ be the corresponding representation map of $\Uq$,
\smash{$
 \pi^{j} \colon \Uq \rightarrow \End\bigl(\mathbb{C}^{2j+1}\bigr)$}.
Now, given~${u\in \mathbb{C}^*}$, we then define the evaluation representations \smash{$\pi^{j}_u \colon \Loop \rightarrow \End\bigl(\mathbb{C}^{2j+1}\bigr)$}
by
\begin{gather}\label{evalrep}
 \pi^{j}_u= \pi^j \circ \mathsf{ev}_u,
\end{gather}
where $\mathsf{ev}_u$ is defined in~\eqref{eval}, and the corresponding $\Loop$-module by $\mathbb{C}^{2j+1}_{u}$.
We study now tensor products of these representations
\[ 
 \bigl(\pi^{\h}_{u_1}\otimes \pi^{j}_{u_2}\bigr) \circ\Delta \colon \ \Loop \rightarrow \End\bigl(\mathbb{C}_{u_1}^2\otimes \mathbb{C}_{u_2}^{2j+1}\bigr),
\]
and look at special points in the evaluation parameters space so that a proper sub-representation emerges.
The strategy is the following: we first construct basis vectors $\{w_k\}$, $\{v_\ell\}$ for the decomposition with respect to the subalgebra generated by $\{E_1,F_1,K_1\}$. Then, we study the action of $\{E_0,F_0,K_0\}$ on these basis vectors. We find that there are only two ratios of evaluation parameters up to a sign when we get a proper sub-representation, and we also construct explicitly the corresponding intertwining maps.

\subsection[Analysis of the tensor product representation of LUqsl2]{Analysis of the tensor product representation of $\boldsymbol{\Loop}$} \label{sub:tensprod}
Consider the subalgebra generated by $\{E_1$, $F_1$, $K_1\}$ and construct basis vectors $\{w_k\}$, $\{v_\ell\}$, where $k=0, 1, \dots, 2j+1$, $\ell=0, 1, \dots, 2j-1$ and $j \in \h \mathbb{N}_+$, corresponding to the tensor product decomposition $\mathbb{C}^2 \otimes \mathbb{C}^{2j+1} = \mathbb{C}^{2j+2} \oplus \mathbb{C}^{2j}$. We denote by $w_0$ and $v_0$ the highest weight vectors of the corresponding spins-$\bigl(j+\h\bigr)$ and $\bigl(j-\h\bigr)$. These are defined by the relations
\begin{gather} 
 \big[\bigl(\pi^{\h}_{u_1}\otimes \pi^{j}_{u_2}\bigr)\Delta(E_1)\big]w_0=0, \qquad \big[\bigl(\pi^{\h}_{u_1}\otimes \pi^{j}_{u_2}\bigr)\Delta(K_1)\big]w_0=q^{2j+1} w_0,\nonumber\\ \label{def:v0}
 \big[\bigl(\pi^{\h}_{u_1}\otimes \pi^{j}_{u_2}\bigr)\Delta(E_1)\big]v_0=0,\qquad \big[\bigl(\pi^{\h}_{u_1}\otimes \pi^{j}_{u_2}\bigr)\Delta(K_1)\big]v_0=q^{2j-1} v_0.
\end{gather}
Solutions to these equations are uniquely determined, up to a scalar, by
\begin{equation} \label{hwv-w0}
 w_0 = \ket{\uparrow} \otimes \ket{j,j}, \qquad v_0= \ket{\uparrow} \otimes \ket{j,j-1 } - \frac{u_1}{u_2} q^{-j-\h}A_{j,j-1} \ket{\downarrow} \otimes \ket{j,j},
\end{equation}
with $\ket{\uparrow}=\ket{\h,\h}$, $\ket{\downarrow}=\ket{\h,-\h}$, and where $A_{j,m}$ is given in~\eqref{AB}.
The other basis vectors are constructed via the action of $F_1$
\begin{equation} \label{wks}
 w_k = \big[ \bigl(\pi_{u_1}^{\h} \otimes \pi_{u_2}^{j}\bigr) \Delta(F_1) \big]^k w_0, \qquad v_\ell =\big[ \bigl( \pi_{u_1}^{\h} \otimes \pi_{u_2}^{j}\bigr) \Delta(F_1) \big]^\ell v_0,
\end{equation}
where $0\leq k \leq 2j+1$ and $0\leq \ell \leq 2j-1$, and we set these vectors to zero if their index is outside of the indicated range.

Now, we study the action of the generators $E_0$ and $F_0$ on these basis vectors, and we begin with the action on $w_k$'s
\begin{gather}
 \big[\bigl(\pi^{\h}_{u_1}\otimes \pi^{j}_{u_2}\bigr)\Delta(E_0)\big] w_k = e_{1,k}(u_1,u_2) w_{k+1} + e_{2,k}(u_1,u_2) v_{k}, \nonumber\\
 \big[\bigl(\pi^{\h}_{u_1}\otimes \pi^{j}_{u_2}\bigr)\Delta(F_0)\big]w_k=f_{1,k}(u_1,u_2) w_{k-1}+f_{2,k}(u_1,u_2)v_{k-2},\label{act-gen}
 \end{gather}
where the coefficients are
\begin{gather}\label{eq:coef1}
 e_{1,k}(u_1,u_2)= \frac{u_1^{-2} + \lbrack 2j\rbrack_q u_2^{-2}}{ \lbrack 2j+1\rbrack_q},\qquad e_{2,k}(u_1,u_2) = \frac{q^{2j+\h}\bigl(u_1^2-u_2^2 q^{-2j-1}\bigr) \sqrt{[2j]_q}}{u_1^2 u_2 [ 2j+1 ]_q }, \\
 f_{1,k}(u_1,u_2) = \bigl( u_1^2 + u_2^2 \lbrack 2j \rbrack_q \bigr) \frac{ \lbrack 2j-k+2 \rbrack_q \lbrack k \rbrack_q} { \lbrack 2j+1 \rbrack_q }, \\
 f_{2,k}(u_1,u_2) = \frac{\bigl(u_1^2 -u_2^2 q^{-2j-1}\bigr) \sqrt{[2j]_q} [k]_q[k-1]_q }{ u_2^{-1} q^{-2j-\h} [2j+1]_q}.
 \label{eq:coef2}
\end{gather}

We next look at the action on $v_{\ell}$'s
\begin{gather}
 \big[\bigl(\pi^{\h}_{u_1}\otimes \pi^{j}_{u_2}\bigr)\Delta(E_0)\big] v_\ell = \bar{e}_{1,\ell}(u_1,u_2) v_{\ell+1} + \bar{e}_{2,\ell}(u_1,u_2) w_{\ell+2},\nonumber \\
 \big[\bigl(\pi^{\h}_{u_1}\otimes \pi^{j}_{u_2}\bigr) \Delta(F_0)\big]v_\ell=\bar{f}_{1,\ell}(u_1,u_2) v_{\ell-1}+\bar{f}_{2,\ell}(u_1,u_2)w_{\ell},\label{act-gen-2}
\end{gather}
where the coefficients are
\begin{gather*}
 \bar{e}_{1,\ell}(u_1,u_2) = \frac{u_2^{-2} \lbrack 2j+2 \rbrack_q - u_1^{-2} }{\lbrack 2j+1 \rbrack_q},
 \qquad \bar{e}_{2,\ell}(u_1,u_2) = \frac{ u_2^2 q^{2j+1}-u_1^2}{q^{2j+\h} u_1^2 u_2^3 \sqrt{[2j]_q} [2j+1]_q}, \\
 \bar{f}_{1,\ell}(u_1,u_2) = \bigl(u_1^2 - u_2^2 \lbrack 2j+2 \rbrack_q\bigr)\frac{\lbrack \ell-2j \rbrack_q \lbrack \ell \rbrack_q}{\lbrack 2j+1 \rbrack_q}, \\
 \bar{f}_{2,\ell}(u_1,u_2) = \frac{\bigl(u_1^2 - u_2^2 q^{2j+1}\bigr) \lbrack \ell -2j \rbrack_q \lbrack 2j+1-\ell \rbrack_q }{u_2 q^{2j+\frac 12} \sqrt{\lbrack 2j\rbrack_q} \lbrack 2j+1 \rbrack_q}.
\end{gather*}

We first note that in~\eqref{act-gen} and~\eqref{act-gen-2} the coefficients with indices $2,k$ and $2,\ell$, respectively, correspond to contribution in the action that mixes the two spin components. In order to find proper sub-representations, we thus need to analyze the roots of these functions $e_{2,k}$, $f_{2,k}$, $\bar{e}_{2,\ell}$ and $\bar{f}_{2,\ell}$.
First, for generic values of $u_1/u_2$ these functions do not vanish which corresponds to the well known fact that this tensor product is irreducible.

From the expression of the coefficients~\eqref{eq:coef1} and \eqref{eq:coef2}, one finds that
\[
e_{2,k}(u_1,u_2)=f_{2,k}(u_1,u_2) = 0\qquad \text{iff} \quad u_1/u_2=\pm q^{-j-\h}
\] and
\[
\bar{e}_{2,\ell}(u_1,u_2)=\bar{f}_{2,\ell}(u_1,u_2)=0\qquad \text{iff} \quad u_1/u_2=\pm q^{j+\h}.
\]
Note that we do not have simultaneously these four coefficients equal to zero so we do not have a~direct sum decomposition. Instead, by fixing the ratio $u_1/u_2$ to the special values we have a~sub-representation as depicted below in Figure~\ref{fig2} by the dotted rectangles.

\begin{figure}[!ht]
 \centering
 \begin{minipage}{.5\textwidth}
 \centering
 \begin{tikzpicture}[scale=0.8, every node/.style={scale=0.8}]
 \draw[rounded corners,dashed] (-0.5, 1.3) rectangle (9.5, -1.3) {};
 \node[draw,circle,inner sep=4.5pt](0) at (1.5*0,0){\tiny{$w_0$}};
 \node[draw,circle,inner sep=4.5pt](1) at (1.5*1,0){\tiny{$w_1$}};
 \node[draw,circle,inner sep=4.5pt](2) at (1.5*2,0){\tiny{$w_2$}};
 \node[](3) at (4.5,0){};
 \node[draw,circle,inner sep=3pt](3) at (1.5*3,0){\dots};
 \node[draw,circle,inner sep=0.5pt](4) at (1.5*4,0){\tiny{$w_{2j-1}$}};
 \node[draw,circle,inner sep=3.5pt](5) at (1.5*5,0){\tiny{$w_{2j}$}};
 \node[draw,circle,inner sep=0.5pt](6) at (1.5*6,0){\tiny{$w_{2j+1}$}};
 \node[draw,circle,inner sep=3.5pt](7) at (1.5*1,-2.5){\tiny{$v_0$}};
 \node[draw,circle,inner sep=3.5pt](8) at (1.5*2,-2.5){\tiny{$v_1$}};
 \node[draw,circle,inner sep=3pt](9) at (1.5*3,-2.5){\dots};
 \node[draw,circle,inner sep=0.5pt](10) at (1.5*4,-2.5){\tiny{$v_{2j-2}$}};
 \node[draw,circle,inner sep=0.5pt](11) at (1.5*5,-2.5){\tiny{$v_{2j-1}$}};

 \draw[->,line width=0.05mm,color=blue,>=stealth] (1.805,-3.2) to [bend right=5] (2);
 \draw[->,line width=0.05mm,color=blue,>=stealth] (3.30,-3.2) to [bend right=5] (3);
 \draw[->,line width=0.05mm,color=blue,>=stealth] (3.30+1.5,-3.2) to [bend right=5] (4);
 \draw[->,line width=0.05mm,color=blue,>=stealth] (3.3+3,-3.2) to [bend right=5] (5);

 \draw[->,line width=0.05mm,color=red,>=stealth] (2.7,-1.82) to [bend left=10](1);
 \draw[->,line width=0.05mm,color=red,>=stealth] (4.2,-1.78) to [bend left=10](2);
 \draw[->,line width=0.05mm,color=red,>=stealth] (5.7,-1.76) to [bend left=10](3);
 \draw[->,line width=0.05mm,color=red,>=stealth] (7.2,-1.76) to [bend left=10](4);
 \draw[->,line width=0.05mm,color=red,>=stealth] (7) to [bend left=15](0);
 \node[color=red] at (0.9,-1.6){\tiny{$F_0$}};

 \foreach \i [count=\j] in {0,...,5} {
 \draw[<-,>=stealth] (\i) to[bend left] node[midway,above,inner sep=1pt] {\tiny{$E_1$}} (\j);
 \draw[->,>=stealth] (\i) to[bend right] node[midway,below,inner sep=1pt] {\tiny{$F_1$}} (\j);
 \draw[<-,>=stealth,line width=0.05mm,color=blue] (\j) to[bend left=85] node[midway,below,inner sep=2pt] {\tiny{$E_0$}} (\i);
 \draw[->,>=stealth,line width=0.05mm,color=red] (\j) to[bend right=85] node[midway,above,inner sep=2pt] {\tiny{$F_0$}} (\i);
 }

 \foreach \i [count=\j from 8] in {7,...,10} {
 \draw[<-,>=stealth] (\i) to[bend left] node[midway,above,inner sep=1pt] {\tiny{$E_1$}} (\j);
 \draw[->,>=stealth] (\i) to[bend right] node[midway,below,inner sep=1pt] {\tiny{$F_1$}} (\j);
 \draw[<-,>=stealth,line width=0.05mm,color=blue] (\j) to[bend left=85] node[midway,below,inner sep=2pt] {\tiny{$E_0$}} (\i);
 \draw[->,>=stealth,line width=0.05mm,color=red] (\j) to[bend right=85] node[midway,above,inner sep=2pt] {\tiny{$F_0$}} (\i);
 }
 \end{tikzpicture}
$u_1/u_2=\pm q^{-j-\h}$
 \end{minipage}%
 \begin{minipage}{.5\textwidth}
 \centering
 \begin{tikzpicture}[scale=0.8, every node/.style={scale=0.8}][h!]
 \draw[rounded corners,dashed] (1, -1.3) rectangle (8, -3.7) {};
 \node[draw,circle,inner sep=4.5pt](0) at (1.5*0,0){\tiny{$w_0$}};
 \node[draw,circle,inner sep=4.5pt](1) at (1.5*1,0){\tiny{$w_1$}};
 \node[draw,circle,inner sep=4.5pt](2) at (1.5*2,0){\tiny{$w_2$}};
 \node[](3) at (4.5,0){};
 \node[draw,circle,inner sep=3pt](3) at (1.5*3,0){\dots};
 \node[draw,circle,inner sep=0.5pt](4) at (1.5*4,0){\tiny{$w_{2j-1}$}};
 \node[draw,circle,inner sep=3.5pt](5) at (1.5*5,0){\tiny{$w_{2j}$}};
 \node[draw,circle,inner sep=0.5pt](6) at (1.5*6,0){\tiny{$w_{2j+1}$}};
 \node[draw,circle,inner sep=3.5pt](7) at (1.5*1,-2.5){\tiny{$v_0$}};
 \node[draw,circle,inner sep=3.5pt](8) at (1.5*2,-2.5){\tiny{$v_1$}};
 \node[draw,circle,inner sep=3pt](9) at (1.5*3,-2.5){\dots};
 \node[draw,circle,inner sep=0.5pt](10) at (1.5*4,-2.5){\tiny{$v_{2j-2}$}};
 \node[draw,circle,inner sep=0.5pt](11) at (1.5*5,-2.5){\tiny{$v_{2j-1}$}};

 \draw[->,>=stealth,line width=0.05mm,color=blue] (0.3,-0.73) to [bend right] (7);
 \draw[->,>=stealth,line width=0.05mm,color=blue] (1.8,-0.73) to [bend right] (8);
 \draw[->,>=stealth,line width=0.05mm,color=blue] (3.3,-0.73) to [bend right] (9);
 \draw[->,>=stealth,line width=0.05mm,color=blue] (4.8,-0.73) to [bend right] (10);
 \draw[->,>=stealth,line width=0.05mm,color=blue] (6.3,-0.73) to [bend right] (11);

 \draw[->,>=stealth,line width=0.05mm,color=red] (8.705,0.73) to [bend right=4] (11);
 \draw[->,>=stealth,line width=0.05mm,color=red] (7.205,0.73) to [bend right=4] (10);
 \draw[->,>=stealth,line width=0.05mm,color=red] (8.705-3,0.73) to [bend right=4] (9);
 \draw[->,>=stealth,line width=0.05mm,color=red] (8.705-4.5,0.73) to [bend right=4] (8);
 \draw[->,>=stealth,line width=0.05mm,color=red] (8.705-6,0.73) to [bend right=4] (7);

 \foreach \i [count=\j] in {0,...,5} {
 \draw[<-,>=stealth] (\i) to[bend left] node[midway,above,inner sep=1pt] {\tiny{$E_1$}} (\j);
 \draw[->,>=stealth] (\i) to[bend right] node[midway,below,inner sep=1pt] {\tiny{$F_1$}} (\j);
 \draw[<-,>=stealth,line width=0.05mm,color=blue] (\j) to[bend left=85] node[midway,below,inner sep=2pt] {\tiny{$E_0$}} (\i);
 \draw[->,>=stealth,line width=0.05mm,color=red] (\j) to[bend right=85] node[midway,above,inner sep=2pt] {\tiny{$F_0$}} (\i);
 }

 \foreach \i [count=\j from 8] in {7,...,10} {
 \draw[<-] (\i) to[bend left] node[midway,above,inner sep=1pt] {\tiny{$E_1$}} (\j);
 \draw[->,>=stealth] (\i) to[bend right] node[midway,below,inner sep=1pt] {\tiny{$F_1$}} (\j);
 \draw[<-,>=stealth,line width=0.05mm,color=blue] (\j) to[bend left=85] node[midway,below,inner sep=2pt] {\tiny{$E_0$}} (\i);
 \draw[->,>=stealth,line width=0.05mm,color=red] (\j) to[bend right=85] node[midway,above,inner sep=2pt] {\tiny{$F_0$}} (\i);
 }
 \end{tikzpicture}
 $u_1/u_2= \pm q^{j+\h}$
 \end{minipage}
 \caption{Action of $E_i$, $F_i$ on $w_k$, $v_\ell$ for fixed values of $u_1/u_2$. The red (resp.\ blue) arrows correspond to the action of $F_0$ (resp.\ $E_0$), and the branching points of the arrows correspond to the linear combinations from~\eqref{act-gen} and~\eqref{act-gen-2}.}
 \label{fig2}
\end{figure}

We thus find that fixing $u_1/u_2= q^{\mp(j+\h)}$ gives a spin-$(j\pm\h)$ sub-representation of $\Loop$.

\subsection[Intertwining map E(j+1/2) and its pseudo-inverse F(j+1/2)]{Intertwining map $\boldsymbol{\mathcal{E}^{(j+\h)}}$ and its pseudo-inverse $\boldsymbol{\mathcal{F}^{(j+\h)}}$} \label{intert-fus}
We now study the spin-$(j+\h)$ sub-representation when \smash{$u_1/u_2= q^{-j-\h}$}, and construct the corresponding intertwining operator explicitly.
Consider two linear operators \smash{$\mathcal{E}^{(j+\h)}$} and \smash{$\mathcal{F}^{(j+\h)}$} for any $j \in \frac{1}{2} \mathbb{N}$
\begin{gather} \label{intertE}
 \mathcal{E}^{(j+\h)}\colon\ \mathbb{C}^{2j+2}_u \rightarrow \mathbb{C}_{u_1}^{2}
 \otimes \mathbb{C}_{u_2}^{2j+1},\\ \label{intertF}
 \mathcal{F}^{(j+\h)} \colon\ \mathbb{C}_{u_1}^{2} \otimes \mathbb{C}_{u_2}^{2j+1} \rightarrow \mathbb{C}^{2j+2}_u,
\end{gather}
with the matrix form
\begin{equation}\label{exprE}
 \mathcal{E}^{(j+\h)}= \sum_{a=1}^{4j+2}\sum_{b=1}^{2j+2} \mathcal{E}_{a,b}^{(j+\h)} E_{a,b}^{(2j,j)}, \qquad \mathcal{F}^{(j+\h)}\mathcal{E}^{(j+\h)} = {\mathbb I}_{2j+2},
\end{equation}
that depends on scalars \smash{$\mathcal{E}^{(j+\h)}_{a,b}$}, and here \smash{$E_{a,b}^{(j_1,j_2)}$} denotes the matrix of dimension $(2j_1+2) \times (2j_2+2)$ with 1 at position $(a,b)$ and $0$ otherwise. Here, we choose the bases of the source $\{ \ket{j+\h,m} \}$ with $m=j+\h, j-\h, \dots, -j+\h, -j-\h$ and the target
\[
\{\ket{\uparrow} \otimes \ket{ j,j}, \dots, \ket{\uparrow} \otimes \ket{j,-j}, \ket{\downarrow} \otimes \ket{j,j}, \dots, \ket{\downarrow} \otimes \ket{j,-j} \}.
\]
We calculate the coefficients \smash{$\mathcal{E}^{(j+\h)}_{a,b}$} provided \smash{$\mathcal{E}^{(j+\h)}$} is a $\Loop$-intertwiner for the condition \smash{$u_1/u_2=q^{-j-\h}$} and \smash{$u_2=uq^\h$}, that was found in the previous subsection for $j\in \h \mathbb{N}_+$.
First of all for $j=0$, we have for any $u$ that $\pi_u^{0} = \epsilon$, where the counit is defined in~\eqref{affine-counit}, i.e., the trivial representation of $\Loop$. Then identifying $\mathbb{C}^2\otimes \mathbb{C}$ with $\mathbb{C}^2$, it follows
from~\eqref{intertE},~\eqref{intertF} and~\eqref{exprE}
that \smash{$\mathcal{E}^{(\h)}=\mathcal{F}^{(\h)}={\mathbb I}_2$}.

\begin{lem}\label{lem-E}
 Let \smash{$u_1/u_2= q^{-j-\h}$} and \smash{$u_2= u q^{\h}$}, then the map \smash{$\mathcal{E}^{(j+\h)}$} in~\eqref{exprE} is a $\Loop$-intertwiner
 \begin{equation} \label{lem-intE}
 \mathcal{E}^{(j+\h)} \bigl(\pi_u^{j+\h}\bigr) (x) = \bigl(\pi_{u q^{-j}}^{\h} \otimes \pi_{u q^{\h}}^j\bigr) (\Delta(x)) \mathcal{E}^{(j+\h)} \qquad \forall x \in \Loop,
 \end{equation}
 if and only if its entries are given for any $j \in \frac{1}{2} \mathbb{N}_+$ by
 \begin{gather}
 \mathcal{E}^{(j+\h)}_{1,1}=1, \qquad \mathcal{E}^{(j+\h)}_{1+n,1+n}= \prod_{p=0}^{n-1} \frac{ B_{j,j-p}}{B_{j+\h,j+\h-p}},\qquad \mathcal{E}^{(j+\h)}_{2j+1+m,1+m}=[m]_q \frac{\mathcal{E}^{(j+\h)}_{m,m}}{B_{j+\h,j+\tha-m}},\!\!\!\label{E-proj}
 \end{gather}
 where $n=1, 2, \dots, 2j$, $m =1, 2, \dots, 2j+1$ and $B_{j,m}$ is given in~\eqref{AB}, and all the other entries are zero.
\end{lem}

\begin{proof}
 Recall that for $u_1=u_2 q^{-j-\h}$ we have a spin-$(j+\h)$ sub-representation as depicted in left-top of Figure~\ref{fig2}, and therefore we also have a corresponding intertwining operator that we are going to construct.
 In this case, the sub-representation has the basis $\{w_k \mid 0 \leq k \leq 2j +1\}$ given in~\eqref{wks}.
 Let us introduce a basis $\{w'_k \mid 0 \leq k \leq 2j+1\}$ in the source of the map \smash{$\mathcal{E}^{(j+\h)}$}: we define
 \[
 w_0'=\ket{j+\tfrac{1}{2},j+\tfrac{1}{2}} \qquad \text{and} \qquad w_k'= \big[\pi_u^{j+\h} (F_1)\big]^k w_0'.
 \] The intertwining property reads
 \begin{equation} \label{commutE}
 \mathcal{E}^{(j+\h)} \big[\pi_u^{j+\h} (x) \big] w_k' = \big[ \bigl(\pi_{u_1}^{\h} \otimes \pi_{u_2}^{j} \bigr) ( \Delta(x)) \big] \mathcal{E}^{(j+\h)} (w_k') \qquad \forall x \in \Loop.
 \end{equation}
 This equation for $x=K_0, K_1$ and $k=0$ gives \smash{$\mathcal{E}^{(j+\h)}(w_0') = w_0$}. Then for $k=0$ and $x= (F_1)^n$, with $n=1, 2, \dots, 2j+1$, one shows
 \begin{equation} \label{intE-map}
 \mathcal{E}^{(j+\h)} ( w_n' ) = w_n \qquad \text{for all} \ n.
 \end{equation}
 Using this, we then check directly that~\eqref{commutE} indeed holds now for all $k$ and $x=F_1$ and $x=E_1$. It remains to check that~\eqref{commutE} is satisfied for the action of $E_0$ and $F_0$. By straightforward calculations using~\eqref{phiu},~\eqref{evalrep} and~\eqref{affinecoprodefk}, one gets
 \begin{gather} \label{u-p1}
 \big[\pi^{j+\h}_u (E_0)\big] w_k'=u^{-2}w_{k+1}', \qquad \big[\pi^{j+\h}_u (F_0)\big] w_k' =u^{2}w_{k-1}', \\ \label{u-p2}
 \big[\bigl(\pi^{\h}_{u_1} \otimes \pi^{j}_{u_2}\bigr)\Delta(E_0)\big] w_k =q u_2^{-2} w_{k+1}, \qquad\big[\bigl(\pi^{\h}_{u_1} \otimes \pi^{j}_{u_2}\bigr)\Delta(F_0)\big] w_k=q^{-1} u_2^{2} w_{k-1}.
 \end{gather}
 Using~\eqref{u-p1} and~\eqref{u-p2}, one finds that~\eqref{commutE} holds for \smash{$u_2= u q^{\h}$}.
 Finally, we get the matrix elements of \smash{$\mathcal{E}^{(j+\h)}$}. The basis vectors read explicitly
 \begin{align*}
 w_k&=u^k\left ( \prod_{p=0}^{k-1} B_{j,j-p} \ket{\uparrow} \otimes \ket{j,j-k} + [k]_q \prod_{p=0}^{k-2} B_{j,j-p} \ket{\downarrow} \otimes \ket{j,j-k+1} \right ), \\
 w_k'&=u^k \prod_{p=0}^{k-1} B_{j+\h,j+\h-p} \ket{j+\tfrac{1}{2},j+\tfrac{1}{2}-k},
 \end{align*}
 and we set \smash{${\prod_{p=0}^n} B_{j,j-p} = 1$} if $n$ is negative.
 Then solving~\eqref{intE-map} for \smash{$\mathcal{E}^{(j+\h)}$} in~\eqref{exprE}, one gets~\eqref{E-proj} (recall~\eqref{exprE} that the basis is $\ket{j+\h,m}$ with $m=j+\h, j-\h, \dots, -j+\h, -j-\h$, and not $w'_k$).
\end{proof}

We now give an expression of \smash{$\mathcal{F}^{(j+\h)}$} which is a pseudo-inverse of \smash{$\mathcal{E}^{(j+\h)}$}. It takes the form
\begin{equation}\label{exprF}
 \mathcal{F}^{(j+\h)}= \sum_{a=1}^{2j+2}\sum_{b=1}^{4j+2} \mathcal{F}_{a,b}^{(j+\h)} E_{a,b}^{(j,2j)},
\end{equation}
where \smash{$\mathcal{F}^{(j+\h)}_{a,b}$} are scalars. The solution of \smash{$\mathcal{F}^{(j+\h)} \mathcal{E}^{(j+\h)} = {\mathbb I}_{2j+2}$} is not unique. For instance, we fix the entries of \smash{$\mathcal{F}^{(j+\h)}$} for $n=2, 3, \dots, 2j+1$ as follows:
\begin{gather} \label{Fp1}
 \mathcal{F}_{1,1}^{(j+\h)}=1,\qquad \mathcal{F}_{n,n+2j}^{(j+\h)}=\frac{ \mathcal{E}_{n+2j,n}^{(j+\h)}}{\bigl(\mathcal{E}_{n,n}^{(j+\h)}\bigr)^2+\bigl(\mathcal{E}_{n+2j,n}^{(j+\h)}\bigr)^2},\\
 \mathcal{F}_{2j+2,4j+2}^{(j+\h)} = \bigl(\mathcal{E}_{4j+2,2j+2}^{(j+\h)}\bigr)^{-1}, \qquad \mathcal{F}_{n,n}^{(j+\h)}=\frac{1-\mathcal{F}_{n,n+2j}^{(j+\h)}\mathcal{E}_{n+2j,n}^{(j+\h)}}{\mathcal{E}^{(j+\h)}_{n,n}} , \label{Fp2}
\end{gather}
and all other entries are zero. This choice is important because it allows the factorization of the R-matrix as in Lemma~\ref{lemEHF} below.
We finally note that any pseudo-inverse of $\mathcal{E}^{(j+\h)}$, in particular the one given above, is not a $\Loop$-intertwiner because the sub-representation involved is not a direct summand, recall the structure in Figure~\ref{fig2}.

\subsection[The maps bE(j+1/2) and bF(j+1/2)]{The maps $\boldsymbol{\bar{\mathcal{E}}^{(j-\h)}}$ and $\boldsymbol{\bar{\mathcal{F}}^{(j-\h)}}$}\label{sec:reduction-map}

We now study the spin-$(j-\h)$ sub-representation when \smash{$u_1/u_2= q^{j+\h}$}.
Introduce the two maps~\smash{$\bar{\mathcal{E}}^{(j-\h)}$} and \smash{$\bar{\mathcal{F}}^{(j-\h)}$} for any $j \in \frac{1}{2} \mathbb{N}_+$
\[
 \bar{\mathcal{E}}^{(j-\h)}\colon\ \mathbb{C}^{2j}_u \rightarrow \mathbb{C}_{u_1}^{2}
 \otimes \mathbb{C}_{u_2}^{2j+1},\qquad
 \bar{\mathcal{F}}^{(j-\h)}\colon\ \mathbb{C}_{u_1}^{2} \otimes \mathbb{C}_{u_2}^{2j+1} \rightarrow \mathbb{C}^{2j}_u,
\]
given by
\begin{equation}\label{exprbarE}
 \bar{\mathcal{E}}^{(j-\h)}= \sum_{a=1}^{4j+2}\sum_{b=1}^{2j} \bar{\mathcal{E}}_{a,b}^{(j-\h)} E^{(2j,j-1)}_{a,b}, \qquad \bar{\mathcal{F}}^{(j-\h)} \bar{\mathcal{E}}^{(j-\h)} = {\mathbb I}_{2j},
\end{equation}
where \smash{$\bar{\mathcal{E}}^{(j-\h)}_{a,b}$} are certain scalars. The bases of the source and the target of \smash{$\bar{\mathcal{E}}^{(j-\h)}$} are respectively $\{ \ket{j-\h,m} \}$ with $m=j-\h, j-\tha, \dots, -j+\tha, -j+\h$ and $\{\ket{\uparrow} \otimes \ket{ j,j}, \dots,\ket{\uparrow} \otimes \ket{j,-j}, \ket{\downarrow} \otimes \ket{j,j}, \dots, \ket{\downarrow} \otimes \ket{j,-j} \}$. \allowbreak
\begin{lem}\label{lem-barE}
 Let \smash{$u_1/u_2= q^{j+\h}$} and \smash{$u_2= u q^{\h}$}. Then, the map \smash{$\bar{\mathcal{E}}^{(j-\h)}$} in~\eqref{exprbarE} is a $\Loop$-intertwiner
 \begin{equation} \label{lem-intEbar}
 \bar{\mathcal{E}}^{(j-\h)} \bigl(\pi_u^{j-\h}\bigr) (x) = \bigl(\pi_{u q^{j+1}}^{\h} \otimes \pi_{u q^{\h}}^j\bigr) (\Delta(x)) \bar{\mathcal{E}}^{(j-\h)} \qquad \forall x \in \Loop,
 \end{equation}
 if and only if its entries are given for any $j \in \frac{1}{2} \mathbb{N}_+$ by
 \begin{gather}
 \bar{\mathcal{E}}_{2,1}^{(j-\h)}=1, \qquad
 \bar{\mathcal{E}}^{(j-\h)}_{2+n,1+n}= \prod_{p=0}^{n-1} \frac{ B_{j,j-p-1}}{B_{j-\h,j-\h-p}}, \nonumber\\ \bar{\mathcal{E}}^{(j-\h)}_{2j+2+m,1+m}=\frac{[m-2j]_q}{B_{j,j-m}}
 \bar{\mathcal{E}}_{2+m,1+m}^{(j-\h)},\label{barE-proj}
 \end{gather}
 where $n=1, 2, \dots, 2j-1$, $m=0, 1, \dots, 2j-1$ and $B_{j,m}$ is given in~\eqref{AB}, and all the other entries are zero.
\end{lem}

\begin{proof}
 Recall that for \smash{$u_1=u_2 q^{j+\h}$} we have a spin-$(j-\h)$ sub-representation as depicted in right-bottom of Figure~\ref{fig2}. We also have a corresponding intertwining operator \smash{$\bar{\mathcal{E}}^{(j-\h)}$} that we now construct.
 In this case, the sub-representation has the basis \smash{$\{v_\ell \mid 0 \leq \ell \leq 2j-1 \}$} given in~\eqref{wks}. Let us introduce a basis $\{v'_k \mid 0 \leq k \leq 2j-1\}$ in the source of the map \smash{$\bar{\mathcal{E}}^{(j-\h)}$}: we~define
 \[
 {v_0'=\ket{j-\tfrac{1}{2},j-\tfrac{1}{2}}} \qquad \text{and} \qquad v_k'= \big[\pi_u^{j-\h} (F_1)\big]^k v_0'.
 \]
 Analogously to the proof of Lemma~\ref{lem-E}, one shows \smash{$\bar{\mathcal{E}}^{(j-\h)} (v_\ell') = v_\ell$}.
 Then, one finds the constraint on the ratio $u/u_2$ and the coefficients~\eqref{barE-proj} are obtained by solving the latter equation for \smash{$\bar{\mathcal{E}}^{(j-\h)}$} in~\eqref{exprbarE}.
\end{proof}

We now give expression of \smash{$\bar{\mathcal{F}}^{(j-\h)}$} which is a pseudo-inverse of \smash{$\bar{\mathcal{E}}^{(j-\h)}$}. It takes the form
\begin{equation}\label{exprbF}
 \bar{\mathcal{F}}^{(j-\h)}= \sum_{a=1}^{2j} \sum_{b=1}^{4j+2} \bar{\mathcal{F}}_{a,b}^{(j-\h)} E^{(j-1,2j)}_{a,b},
\end{equation}
where \smash{$\bar{\mathcal{F}}^{(j-\h)}_{a,b}$} are scalars.
The solution of \smash{$\bar{\mathcal{F}}^{(j-\h)} \bar{\mathcal{E}}^{(j-\h)} = {\mathbb I}_{2j}$} is not unique. Similarly to the fusion case above, we fix the entries of \smash{$\bar{\mathcal{F}}^{(j-\h)}$} for $n=1, 2, \dots, 2j$ this way
\begin{equation}\label{coefbar}
 \bar{\mathcal{F}}_{n,n+2j+1}^{(j-\h)} = \frac {\bar{\mathcal{E}}_{n+2j+1,n}^{(j-\h)}}{\bigl(\bar{\mathcal{E}}_{n+1,n}^{(j-\h)}\bigr)^2 + \bigl(\bar{\mathcal{E}}_{n+2j+1,n}^{(j-\h)}\bigr)^2 }, \qquad
 \bar{\mathcal{F}}_{n,n+1}^{(j-\h)} = \frac { 1 - \bar{\mathcal{F}}_{n,n+2j+1}^{(j-\h)} \bar{\mathcal{E}}_{n+2j+1,n}^{(j-\h)} } { \bar{\mathcal{E}}_{n+1,n}^{(j-\h)}},
\end{equation}
and all other entries are zero. Similarly to the previous case, \smash{$\bar{\mathcal{F}}^{(j-\h)}$} is not an intertwiner.

In what follows, we will need action on tensor product using opposite coproduct. For this new action,\footnote{Recall that for a bialgebra $H$, we can define another bialgebra $H^{{\rm cop}}$ with the coproduct $\Delta^{\rm op}$. Therefore, $\Delta^{\rm op}$~also defines an action of the algebra $H$ on the tensor product of $H$-modules.} the corresponding intertwining operators appear at different evaluation parameters.
\begin{rem} \label{rem:oppcop}
 For the action of $\Loop$ on the tensor product given by the opposite coproduct~${\Delta^{\rm op}=\mathfrak{p}\circ\Delta}$ with $\Delta$ in~\eqref{affinecoprodefk}, we repeat the sub-representation analysis from Section~\ref{sub:tensprod} using the corresponding basis $\{\tilde{w}_k \mid 0 \leq k \leq 2j +1\}$ and $\{\tilde{v}_\ell \mid 0 \leq \ell \leq 2j-1\}$ defined by
 \[
 \tilde{w}_k = \big[ \bigl(\pi_{u_1}^{\h} \otimes \pi_{u_2}^{j}\bigr) \Delta^{\rm op}(F_1) \big]^k \tilde{w}_0, \qquad \tilde{v}_\ell =\big[ \bigl( \pi_{u_1}^{\h} \otimes \pi_{u_2}^{j}\bigr) \Delta^{\rm op}(F_1) \big]^\ell \tilde{v}_0,
 \]
 with $\tilde{w}_0=w_0$ in~\eqref{hwv-w0}, and $\tilde{v}_0$ is the solution of~\eqref{def:v0} with the substitution $\Delta \rightarrow \Delta^{\rm op}$
 \begin{equation} \nonumber
 \tilde{v}_0= \ket{\uparrow} \otimes \ket{j,j-1 } - \frac{u_1}{u_2} q^{j+\h}A_{j,j-1} \ket{\downarrow} \otimes \ket{j,j}.
 \end{equation}
 Then, we find that the conditions on the evaluations parameter are different. Indeed, we have for the spin-$(j+\h)$ sub-representation
 \begin{equation}\label{specop}
 u_1/u_2= \pm q^{j+\h}, \qquad u = u_2 q^{\h},
 \end{equation}
 and for the spin-$(j-\h)$ sub-representation
 \begin{equation} \label{specop-bar}
 u_1/u_2=\pm q^{-j-\h}, \qquad u = u_2 q^{\h}.
 \end{equation}
 Nevertheless, it leads to the intertwining operator \smash{$\mathcal{E}^{(j+\h)}$} \big(resp.\ \smash{$\bar{\mathcal{E}}^{(j-\h)}$}\big) with the same matrix elements as in~\eqref{E-proj} (resp.\ in~\eqref{barE-proj}). Indeed, the matrix elements are invariant under the replacement of $q$ by $q^{-1}$.
 Then, the corresponding intertwining properties for the conditions~\eqref{specop},~\eqref{specop-bar} and the choice of the positive sign read respectively for all $x \in \Loop$
 \begin{gather} \label{intert-op}
 \mathcal{E}^{(j+\h)} \bigl( \pi_{u}^{j+\h}\bigr) (x) = \bigl( \pi_{u q^j}^\h \otimes \pi_{u q^{-\h}}^j \bigr) ( \Delta^{\rm op} (x)) \mathcal{E}^{(j+\h)}, \\
 \label{intopEbar}
 \bar{\mathcal{E}}^{(j-\h)} \bigl(\pi_u^{j-\h}\bigr)(x) = \bigl(\pi_{u q^{-j-1}}^\h \otimes \pi^j_{u q^{-\h}}\bigr) (\Delta^{\rm op}(x)) \bar{\mathcal{E}}^{(j-\h)}.
 \end{gather}
\end{rem}

\subsection{Additional properties}\label{sub:addprop}
We conclude this section with a few observations on relations between the intertwining operator~\smash{$\mathcal{E}^{(j+\h)}$}, its pseudo-inverse~$\mathcal{F}^{(j+\h)}$ and the R-matrix.
In the literature the expression of the R-matrix \smash{$R^{(\h,j)}(u) \in \End\bigl(\ds{C}^2 \otimes \ds{C}^{2j+1}\bigr)$} is known~\cite{DN02,KR83}. It reads
\begin{align}
 R^{(\h,j)}(u) ={}& \prod_{k=0}^{2j-2} c\bigl(u q^{j-\h-k}\bigr)\bigl( \bigl(q-q^{-1}\bigr)\bigl( \sigma_+ \otimes \pi^{j}(F) + \sigma_- \otimes \pi^{j}(E) \bigr) \nonumber \\
 &+ uq^{\frac{1}{2}( {\mathbb I_{4j+2}}+\sigma_z\otimes \pi^j(H))} - u^{-1}q^{-\frac{1}{2}( {\mathbb I}_{4j+2} +\sigma_z\otimes \pi^j(H))}\bigr),\label{R-Rqg}
 \end{align}
where $\pi^j(E)$, $\pi^j(F)$ are given in~\eqref{evalEFK}, $
\bigl( \pi^j( H) \bigr)_{mn} = 2(j+1-n) \delta_{m,n}$, with $m$, $n=1, 2, \dots, 2j+1$, and where we use the scalar function
\begin{equation} \label{eq:cu}
 c(u) = u-u^{-1}.
\end{equation}
Note that this R-matrix satisfies the unitarity property
\begin{equation} \label{unitRhj}
 R^{(\h,j)}(u) R^{(\h,j)}\bigl(u^{-1}\bigr) = \left ( \prod_{k=0}^{2j-1} c\big(u q^{j+\h-k}\big) c\big(u^{-1} q^{j+\h-k}\big) \right) \mathbb{I}_{4j+2}.
\end{equation}

Let $\mathcal{H}^{(j+\h)}$ and $\bar{\mathcal{H}}^{(j-\h)}$ be invertible diagonal matrices given by
\begin{align}
 \mathcal{H}^{(j+\h)}&=\operatorname{Diag}\bigl(\mathcal{H}_{1}^{(j+\h)},\mathcal{H}_{2}^{(j+\h)},\dots, \mathcal{H}_{2j+2}^{(j+\h)}\bigr), \nonumber\\
 \bar{\mathcal{H}}^{(j-\h)}&=\operatorname{Diag}\bigl(\bar{\mathcal{H}}_{1}^{(j-\h)},\bar{\mathcal{H}}_{2}^{(j-\h)},\dots, \bar{\mathcal{H}}_{2j}^{(j-\h)}\bigr),\label{defH}
\end{align}
where \smash{$\mathcal{H}^{(j+\h)}_{m}$} and \smash{$\bar{\mathcal{H}}^{(j-\h)}_{n}$} are scalars.

Inspired by~\cite{Be2015}, the R-matrix~\eqref{R-Rqg} admits two special points for which its rank drops below its maximum. Then at these points, the R-matrix decomposes in terms of the intertwining operator $\mathcal{E}^{(j+\h)}$ and the operator $\mathcal{F}^{(j+\h)}$, defined above, as follows.

\begin{lem} \label{lemEHF}
 The R-matrix~\eqref{R-Rqg} at the point $u=q^{j+\h}$ has rank $2j+2$ and decomposes as
 \begin{equation} \label{decompR}
 R^{(\h,j)}\bigl(q^{j+\h}\bigr)= \mathcal{E}^{(j+\h)}\mathcal{H}^{(j+\h)}\mathcal{F}^{(j+\h)},
 \end{equation}
 where \smash{$\mathcal{E}^{(j+\h)}$} is fixed by Lemma~{\rm\ref{lem-E}}, \smash{$\mathcal{F}^{(j+\h)}$} is given in~\eqref{exprF} with~\eqref{Fp1},~\eqref{Fp2} and the entries of \smash{$\mathcal{H}^{(j+\h)}$} are
 \begin{gather*}
 \mathcal{H}_{1}^{(j+\h)}=\mathcal{H}_{2j+2}^{(j+\h)}=\left ( \prod_{k=0}^{2j-2} c\bigl(q^{2j-k}\bigr) \right) \bigl(q^{2j+1}-q^{-2j-1}\bigr), \\
 \mathcal{H}_{n}^{(j+\h)} = \left (\prod_{k=0}^{2j-2} c\bigl(q^{2j-k}\bigr) \right) \frac{\bigl(q-q^{-1}\bigr) B_{j,-j-1+n}} {\mathcal{E}_{n,n}^{(j+\h)} \mathcal{F}_{n,n+2j}^{(j+\h)}},
 \end{gather*}
 for $n=2, 3, \dots, 2j+1$.
\end{lem}

\begin{proof}
 Recall that $\mathcal{E}^{(j+\h)}$ is given in Lemma~\ref{lem-E} and its pseudo-inverse $\mathcal{F}^{(j+\h)}$ is not unique. However, imposing~\eqref{decompR}, with $\mathcal{H}^{(j+\h)}$ defined in~\eqref{defH}, fixes both $\mathcal{F}^{(j+\h)}$ and $\mathcal{H}^{(j+\h)}$ uniquely as we now show.
 Indeed, solving \smash{$ \mathcal{F}^{(j+\h)}\mathcal{E}^{(j+\h)} = {\mathbb I}_{2j+2}$} imposes~\eqref{Fp2} and \smash{$\mathcal{F}_{1,1}^{(j+\h)}=1$}. There are still $2j$ unfixed coefficients \smash{$\mathcal{F}^{(j+\h)}_{n,n+2j}$}, with $n=2, 3, \dots, 2j+1$. They are fixed as in~\eqref{Fp1}, as well as the entries of \smash{$\mathcal{H}^{(j+\h)}$}, by solving~\eqref{decompR}.
\end{proof}

Then, with the decomposition~\eqref{decompR} and using the pseudo-inverse property $\mathcal{F}^{(j+\h)}\mathcal{E}^{(j+\h)}= \mathbb{I}_{2j+2}$, we have the following.
\begin{cor} \label{corEHF}
 The following relations hold:
 \begin{gather}
 \mathcal{E}^{(j+\h)}\mathcal{H}^{(j+\h)}=R^{(\h,j)}\bigl(q^{j+\h}\bigr)\mathcal{E}^{(j+\h)},\label{EHF1}\\
 \mathcal{H}^{(j+\h)}\mathcal{F}^{(j+\h)}=\mathcal{F}^{(j+\h)} R^{(\h,j)}\bigl(q^{j+\h}\bigr), \label{EHF2}\\
 \label{EHF3}
 R^{(\h,j)}\bigl(q^{j+\h}\bigr)= \mathcal{E}^{(j+\h)} \mathcal{F}^{(j+\h)}R^{(\h,j)}\bigl(q^{j+\h}\bigr).
 \end{gather}
\end{cor}
We note that Lemma~\ref{lemEHF} and Corollary~\ref{corEHF} will be used many times in Sections~\ref{sec5} and~\ref{sec6}, in particular to prove the reflection equation in Theorem~\ref{prop:fusedRE}.

Similarly to Lemma~\ref{lemEHF}, for the second special point we have the following.

\begin{lem} \label{lemEHF-2}
 The R-matrix~\eqref{R-Rqg} at the point $u=q^{-j-\h}$ has rank $2j$ and is decomposed as
 \[
 R^{(\h,j)}\bigl(q^{-j-\h}\bigr)= \bar{\mathcal{E}}^{(j-\h)} \bar{\mathcal{H}}^{(j-\h)} \bar{\mathcal{F}}^{(j-\h)},
 \]
 where \smash{$\bar{\mathcal{E}}^{(j-\h)}$} is fixed by Lemma~{\rm\ref{lem-barE}}, \smash{$\bar{\mathcal{F}}^{(j-\h)}$} is given in~\eqref{exprbF} with~\eqref{coefbar} and
 the entries of~\smash{$\bar{\mathcal{H}}^{(j-\h)}$} are
 \[
 \bar{\mathcal{H}}_{n}^{(j-\h)}= \left (\prod_{k=0}^{2j-2} c(q^{-k-1}) \right) \frac {\bigl(q-q^{-1}\bigr) B_{j,-j+n} }{ \bar{\mathcal{E}}^{(j-\h)}_{n+2j+1,n} \bar{\mathcal{F}}_{n,n+1}^{(j-\h)}},\qquad n=1,2,\dots, 2j.
 \]
\end{lem}

\subsection{Formal evaluation representations}\label{sec:formal-ev}

In the next sections, when we discuss L- and K-operators, and in calculations of 1-component evaluation of the universal R-matrix, we actually work with formal parameter~$u$, not a complex number, because the evaluated expressions would not be well defined if $u$ were a number as in the evaluation representations considered above. Actually, in the evaluation of the universal R-matrix, such considerations were already accounted for in~\cite[Section~13]{Dr0} and~\cite[Section~4]{FR92}, see~\cite[Section~1]{He17} for a review. For an evaluation of the universal K-matrix, similar precautions should be taken. One way to proceed is to introduce a formal-parameter analogue of the finite-dimensional evaluation representations from~\eqref{evalrep} and to establish the intertwining properties analogous to Lemmas~\ref{lem-E} and~\ref{lem-barE} on this formal level. In the literature, these representations of quantum affine algebras are also known~\cite{chari} as \textit{quantum loop modules}.

From now on, we assume that $u^{\pm1}$ is formal, and keeping for brevity the same notations, we~define the formal evaluation representations $\pi^{j}_u$ as
\begin{equation}\label{evalrep-form}
 \pi^{j}_u= \pi^j \circ \mathsf{ev}_u \colon\ \Loop \rightarrow \End\bigl(\mathbb{C}^{2j+1}\bigr)\big[u^{\pm1}\big],
\end{equation}
where the formal evaluation map
\begin{equation}\label{formal-evu}
 \mathsf{ev}_u \colon \ \Loop \rightarrow \Uq \big[u^{\pm1}\big],
\end{equation}
is defined by same formulas as in~\eqref{evu} but with $u^{\pm1}$ formal variables. We thus have explicitly
\begin{gather}
 \pi^j_u(E_0) = u^{-1} \pi^j(F), \qquad \pi^j_u(F_0)=u \pi^j(E), \qquad \pi^j_u \bigl(K_0^\h\bigr) = \pi^j\bigl(K^{-\h}\bigr),\nonumber\\
 \pi^j_u(E_1)= u^{-1} \pi^j(E), \qquad \pi^j_u(F_1) = u \pi^j(F), \qquad \pi^j_u\bigl(K_1^\h\bigr)=\pi^j\bigl(K^\h\bigr).\label{evalrep-form-explicit}
\end{gather}
In the case $u$ is a complex number, this action is also known as the principal gradation~\cite{gradation-paper}. In~our formal $u$ case, we will also call it the action of \textit{principal gradation}.

The map~\eqref{evalrep-form} and formulas~\eqref{evalrep-form-explicit} define the action of $\Loop$ on the space $\mathbb{C}^{2j+1}\big[u^{\pm1}\big]$ of Laurent polynomials in $u$ with coefficients being vectors in $\mathbb{C}^{2j+1}$. For example, the action of~$E_0$ is given by
$
\pi^j_u (E_0) (x u^n) = \pi^j(F)(x) u^{n-1}$, $ x\in \mathbb{C}^{2j+1}$, $ n\in \mathbb{Z}$,
and extended linearly to all elements in $\mathbb{C}^{2j+1}\big[u^{\pm1}\big]$, and similarly for the other generators.
Here, we use the same notation for the corresponding representation \smash{$\pi^j_u \colon \Loop \to \End\bigl(\mathbb{C}^{2j+1}\big[u^{\pm1}\big]\bigr)$}, and call it \textit{formal evaluation representation}.
We denote the corresponding infinite-dimensional module by
$\mathbb{C}^{2j+1}_u$ and sometimes call it also \textit{quantum loop module of spin-$j$}. In what follows, we will often use the notation $\mathbb{C}^{2j+1}_{z u}$, for $z\in \mathbb{C}^*$, which means that the $\Loop$ action is on the same underlying vector space $\mathbb{C}^{2j+1}\big[u^{\pm1}\big]$ but one needs to replace $u$ in the action~\eqref{evalrep-form-explicit} by $z u$.

Let $\{x_k\}_{0\leq k\leq 2j}$ be a basis in the spin-$j$ module $\mathbb{C}^{2j+1}$ over $\Uq$ such that $x_0$ is the highest weight vector, or $\ket{j,j}$ in the notations of~\eqref{evalEFK}, and $x_k := \pi^j(F)^k x_0$. Then the vectors
\begin{equation}\label{basis-loop-module}
 x_{k,n} := x_k u^n,\qquad 0\leq k\leq 2j, \quad n\in \mathbb{Z},
\end{equation}
form a basis in $\mathbb{C}^{2j+1}_u$ with the action of $\Loop$ up to scalars described as follows: $E_0$ changes the index $(k,n)$ by $(+1,-1)$, while $F_0$ changes it in the opposite way $(k,n)\to (k-1,n+1)$, and similarly $E_1$ changes the index $(k,n)$ by $(-1,-1)$ while $F_1$ changes it by $(+1,+1)$, and finally~$K_0$ and $K_1$ don't change the index $(k,n)$, they act as $K^{-1}$ and $K$, respectively, on $x_k$, i.e., by the corresponding $\Uq$ weight.

For the $\Loop$ action in the basis~\eqref{basis-loop-module}, let us consider an example of the $j=1$ case in more details. The action of $\Loop$ is better to present diagrammatically as in Figure~\ref{fig:loop-module}, where each node corresponds to some $x_{k,n}$ while edges are the actions of $E_i$ and $F_i$ described up to scalars above.
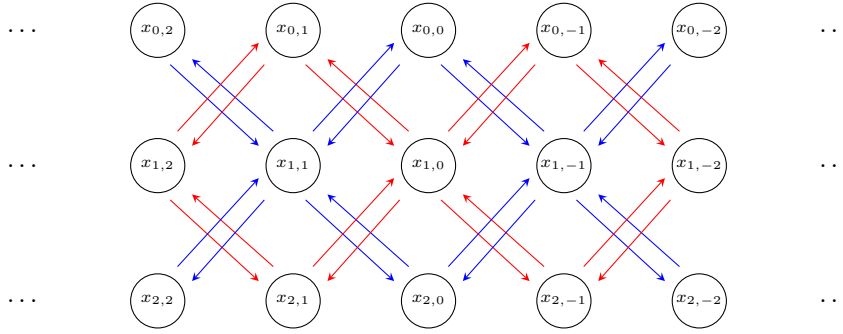
\begin{figure}[!ht]
 \centering
 \begin{tikzpicture}[scale=0.9, every node/.style={scale=0.9}]

 \def\m{1}
 \def\n{2}
 \pgfmathsetmacro{\col}{2*\n}
 \pgfmathsetmacro{\row}{2*\m}
 \def\size{8}
 \def\sizeee{10}
 \def\sep{2}

 \foreach \x in {0,...,\col} {
 \foreach \y in {0,...,\row} {
 \pgfmathtruncatemacro{\cX}{\row-\y}
 \pgfmathtruncatemacro{\cY}{\n-\x}
 \pgfmathtruncatemacro{\labelnode}{\x+100*\y}

 \node [draw,circle,inner sep=1pt,minimum size=\size mm] (\labelnode) at (\x*\sep,\y*\sep){
 \text{\tiny $x_{\cX,\cY}$ }};
 \node [minimum size=\sizeee mm] (\labelnode) at (\x*\sep,\y*\sep){};
 }
 }
 %


 \draw [->,>=stealth,red] (1.120) -- (100.320);
 \draw [<-,>=stealth,red] (1.150) -- (100.290);

 \draw [->,>=stealth,blue] (2.120) -- (101.320);
 \draw [<-,>=stealth,blue] (2.150) -- (101.290);

 \draw [->,>=stealth,red] (3.120) -- (102.320);
 \draw [<-,>=stealth,red] (3.150) -- (102.290);

 \draw [->,>=stealth,blue] (4.120) -- (103.320);
 \draw [<-,>=stealth,blue] (4.150) -- (103.290);

 \draw[->,>=stealth,blue] (0.60) -- (101.200);
 \draw[<-,>=stealth,blue] (0.30) -- (101.230);

 \draw [->,>=stealth,red] (1.60) -- (102.200);
 \draw [<-,>=stealth,red] (1.30) -- (102.230);

 \draw[->,>=stealth,blue] (2.60) -- (103.200);
 \draw[<-,>=stealth,blue] (2.30) -- (103.230);

 \draw [->,>=stealth,red] (3.60) -- (104.200);
 \draw [<-,>=stealth,red] (3.30) -- (104.230);


 \draw [->,>=stealth,blue] (101.120) -- (200.320);
 \draw [<-,>=stealth,blue] (101.150) -- (200.290);

 \draw [->,>=stealth,red] (102.120) -- (201.320);
 \draw [<-,>=stealth,red] (102.150) -- (201.290);

 \draw [->,>=stealth,blue] (103.120) -- (202.320);
 \draw [<-,>=stealth,blue] (103.150) -- (202.290);

 \draw [->,>=stealth,red] (104.120) -- (203.320);
 \draw [<-,>=stealth,red] (104.150) -- (203.290);

 \draw[->,>=stealth,red] (100.60) -- (201.200);
 \draw[<-,>=stealth,red] (100.30) -- (201.230);

 \draw [->,>=stealth,blue] (101.60) -- (202.200);
 \draw [<-,>=stealth,blue] (101.30) -- (202.230);

 \draw[->,>=stealth,red] (102.60) -- (203.200);
 \draw[<-,>=stealth,red] (102.30) -- (203.230);

 \draw [->,>=stealth,blue] (103.60) -- (204.200);
 \draw [<-,>=stealth,blue] (103.30) -- (204.230);

 \node at (-\sep,\sep) {$\dots$};
 \node at (-\sep,0) {$\dots$};
 \node at (-\sep,2*\sep) {$\dots$};

 \node at (5*\sep,\sep) {$\dots$};
 \node at (5*\sep,0) {$\dots$};
 \node at (5*\sep,2*\sep) {$\dots$};

 \end{tikzpicture}
 \caption{Quantum loop module $\mathbb{C}^{3}_u$ of spin-1 where the red and blue diagrams correspond to irreducible components of the $\Loop$ action.} \label{fig:loop-module}
\end{figure}

As we can see from this example, the formal-evaluation representation $\mathbb{C}^{3}_u$ is fully reducible due to the presence of two different sub-diagrams, one is blue and the other is red. Each of these sub-diagrams describes an irreducible action. It is not hard to see that this action arises from the other version of formal-evaluation representation in the homogeneous gradation.
Recall~\cite{gradation-paper} that the homogeneous gradation action is as in~\eqref{evalrep-form-explicit} for $E_0$ and $F_0$, while $E_1$ and~$F_1$ act just as $E$ and~$F$, respectively. Let us denote the corresponding $\Loop$-module on the vector space~$\mathbb{C}^{2j+1}\big[u^{\pm1}\big]$ by \smash{$\widetilde{\mathbb{C}}^{2j+1}_u$} and call it \textit{quantum loop module of spin-$j$ in the homogeneous gradation}. It is clear from the action in the basis of \smash{$\widetilde{\mathbb{C}}^{2j+1}_u$} analogous to $x_k u^n$ that this module is irreducible, while in our example $\mathbb{C}^{3}_u$ is decomposed as $\mathbb{C}^{3}_u\cong \widetilde{\mathbb{C}}^{3}_u \oplus \widetilde{\mathbb{C}}^{3}_u$.

We have a similar result for arbitrary spin $j$: indeed, the diagram for $j=\h$ is obtained by simply deleting the last row of nodes labeled by $x_{2,n}$ and all arrows attached to them, so we get a disjoint union of two zig-zag's, while for $j>1$ one just repeats the obvious pattern of arrows which gives again two independent sub-diagrams of different color. Let us reformulate this observation as the following lemma.

\begin{lem}\label{lem:form-ev-decomp}
 $\mathbb{C}^{2j+1}_u$ is a direct sum of two copies of \smash{$\widetilde{\mathbb{C}}^{2j+1}_u$}.
\end{lem}

This effect of decomposition of the action in one gradation via actions in the other gradation is in contrast to the finite-dimensional evaluation representations where the choice of gradation is not really important.

We now want to study tensor products of these infinite-dimensional representations
\begin{gather} \label{formal-tens-rep}
 \bigl(\pi^{\h}_{u_1}\otimes \pi^{j}_{u_2}\bigr) \circ\Delta \colon \ \Loop \rightarrow \End\bigl(\mathbb{C}_{u_1}^2\otimes \mathbb{C}_{u_2}^{2j+1}\bigr),
\end{gather}
where we used that the homomorphism \smash{$\bigl(\pi^{\h}_{u_1}\otimes \pi^{j}_{u_2}\bigr) \circ\Delta$} sends $\Loop$ to $\End\bigl(\smash{\mathbb{C}_{u_1}^2}\bigr)\otimes\allowbreak \End\bigl(\smash{\mathbb{C}_{u_2}^{2j+1}}\bigr)$ which is a subalgebra in $\End\bigl(\smash{\mathbb{C}_{u_1}^2}\otimes \smash{\mathbb{C}_{u_2}^{2j+1}}\bigr)$, and this algebra embedding is implicitly used in~\eqref{formal-tens-rep}.

We then note that as a vector space $\mathbb{C}_{u_1}^2\otimes \mathbb{C}_{u_2}^{2j+1}$ is isomorphic to $\mathbb{C}^2\otimes \mathbb{C}^{2j+1}\big[u_1^{\pm1},u_2^{\pm1}\big]$ -- the space of Laurent polynomials in two variables $u_1$ and $u_2$ with coefficients in the $2(2j+1)$-dimensional vector space -- or equivalently, to the product of $\mathbb{C}^2\otimes \mathbb{C}^{2j+1}$ and $\mathbb{C}\big[u_1^{\pm1},u_2^{\pm1}\big]$. Below, we will use these identifications implicitly.
Then, the action~\eqref{formal-tens-rep} can be written as
\begin{equation}\label{eq:loop-act-formal-tensor}
 a\bigl(\boldsymbol{v} f(u_1,u_2)\bigr) = a(\boldsymbol{v}) f(u_1,u_2) = \big[\bigl(\pi^{\h}_{u_1}\otimes \pi^{j}_{u_2}\bigr) \circ\Delta(a)\big] (\boldsymbol{v}) f(u_1,u_2)
\end{equation}
for any $a\in\Loop$, $\boldsymbol{v}\in \mathbb{C}^2\otimes \mathbb{C}^{2j+1}$ and $f(u_1,u_2)\in \mathbb{C}\big[u_1^{\pm1},u_2^{\pm1}\big]$. In what follows, we will shortly write the action as on the left-hand side instead of lengthy but more precise expression on the right-hand side.
The difference from the finite-dimensional case where $u_i$'s were generic complex numbers is that these infinite-dimensional tensor products are not irreducible anymore. Indeed, for every
non-zero complex number~$z$ the Laurent polynomial $u_1-z^{-1} u_2$
generates a~proper ideal in the algebra of Laurent polynomials because its inverse lives in a bigger ring of rational functions in $u_1$ and $u_2$.
Then, consider for a fixed $z\in \mathbb{C}^*$ the linear span of vectors of the following form:
\begin{equation}\label{eq:Iz}
 I_{z} :=\big\langle \boldsymbol{v} \bigl(u_1-z^{-1} u_2\bigr)f(u_1,u_2) \mid \boldsymbol{v}\in \mathbb{C}^2\otimes \mathbb{C}^{2j+1} , \, f(u_1,u_2)\in \mathbb{C}\big[u_1^{\pm1},u_2^{\pm1}\big]\big\rangle,
\end{equation}
which can be thought of as `ideal' generated by the Laurent polynomials $u_1-z^{-1} u_2$.
Analyzing the action of $F_i$ and $E_i$ on these vectors,
recall~\eqref{eq:loop-act-formal-tensor}, it is clear that $I_z$ is a proper $\Loop$-submodule of \smash{$\mathbb{C}_{u_1}^2\otimes \mathbb{C}_{u_2}^{2j+1}$}.

We can now define a so-called partial specialization ``identifying $u_2$ with $z u_1$" of the module~$\mathbb{C}_{u_1}^2\otimes \mathbb{C}_{u_2}^{2j+1}$ as the quotient by the submodule $I_z$ from~\eqref{eq:Iz}. We will use the following short-hand notations for this quotient:
\begin{equation}\label{eq:Iz-quotient}
 \mathbb{C}_{u_1}^2\otimes \mathbb{C}_{z u_1}^{2j+1}
 := \bigl(\mathbb{C}_{u_1}^2\otimes \mathbb{C}_{u_2}^{2j+1}\bigr)/ I_z,
\end{equation}
when the same formal parameter appears in both tensor factors.
And let us denote the canonical surjection
\begin{equation}\label{eq:proj-pz}
 p^j_z\colon\ \mathbb{C}_{u_1}^2\otimes \mathbb{C}_{u_2}^{2j+1} \twoheadrightarrow \mathbb{C}_{u_1}^2\otimes \mathbb{C}_{z u_1}^{2j+1},
\end{equation}
which is a $\Loop$-intertwining operator. We also note that the quotient module $\mathbb{C}_{u_1}^2\otimes \mathbb{C}_{z u_1}^{2j+1}$ can be identified with \smash{$\mathbb{C}^2\otimes \mathbb{C}^{2j+1}[u_1^{\pm1}]$} as a vector space. Indeed, first note that the images~${p^j_z(x\otimes y u_1^n u_2^m)}$, for any $x\in\mathbb{C}^2$ and $y\in \mathbb{C}^{2j+1}$ are all proportional to $(x\otimes y) u_1^{n+m}$, therefore every element in the quotient space \smash{$\mathbb{C}_{u_1}^2\otimes \mathbb{C}_{z u_1}^{2j+1}$} is a (finite) linear combination of the monomials~$(x\otimes y) u_1^n$, for $n\in \mathbb{Z}$, which makes it identical to $\mathbb{C}^2\otimes \mathbb{C}^{2j+1}[u_1^{\pm1}]$. With this identification, the $\Loop$ representation corresponding to the quotient module~\eqref{eq:Iz-quotient} will be denoted by
\beq\label{eq:pi-prod-form}
\pi_{u_1}^{\h} \otimes \pi_{z u_1}^j \colon \ \Loop \to \End\bigl(\mathbb{C}^2\otimes \mathbb{C}^{2j+1}\big[u^{\pm1}\big]\bigr).
\eeq
Also, we will often use the notation \smash{$\pi_{u_1}^{\h} \otimes \pi_{z u_1}^j (\Delta(x))$}, for $x\in \Loop$, to stress that the action~in~\eqref{eq:pi-prod-form} uses the coproduct by construction. In other words, \smash{$\pi_{u_1}^{\h} \otimes \pi_{z u_1}^j (\Delta(x))$} stands for~the action of $\Delta(x)$ on $\mathbb{C}_{u_1}^2\otimes \mathbb{C}_{u_2}^{2j+1}$ reduced to the quotient space by the surjective map $p^j_z$ from~\eqref{eq:proj-pz}. We note that the intertwining property of \smash{$p^j_z$} guarantees that this reduction is well defined, because the kernel of \smash{$p^j_z$} is invariant under the action of $\Delta(x)$ for all $x\in \Loop$.

In the quotient modules~\eqref{eq:Iz-quotient}, we can now identify for appropriate values of $z$ infinite-dimensional submodules isomorphic to certain quantum loop modules. First, denote by $W^j$ the subspace in $\mathbb{C}_{u_1}^2\otimes \mathbb{C}_{u_2}^{2j+1}$ spanned by
vectors $w_{k,n,m}$ obtained via the iterated action of $F_1$ on~${w_0 u_1^{n}u_2^{m}}$ where $w_0$ is the product~\eqref{hwv-w0} of highest-weight vectors in $\Uq$-modules of spin-$\h$ and~$j$
\[
 w_{k,n,m} := F_1^k (w_0) u_1^{n}u_2^{m}, \qquad 0\leq k \leq 2j+1,\quad n,m\in \mathbb{Z}.
\]
Note that as in the finite-dimensional situation if $k> 2j+1$ then the corresponding vectors~$w_{k,n,m}$ are just zero. We will also denote $w_k:=w_{k,0,0}$.

\begin{prop}\label{prop:Wj-image}
 The image of \smash{$W^j = \langle w_{k,n,m} \rangle_{0\leq k \leq 2j+1, n,m\in \mathbb{Z}} \subset \mathbb{C}_{u_1}^2\otimes \mathbb{C}_{u_2}^{2j+1}$} under the projection $p^j_{z}$ from~\eqref{eq:proj-pz} at \smash{$z=q^{j+\h}$} is a $\Loop$-submodule isomorphic to the quantum loop module \smash{$\mathbb{C}^{2j+2}_{u_1 q^{j}}$} of spin-$(j+\h)$.
\end{prop}
\begin{proof}
 First, the action of $E_1$ preserves $W^j$ because using the commutator relation between~$E_1$ and $F_1$ and the fact that all vectors $w_{0,n,m}$ are annihilated by $E_1$ we get
\smash{$
 E_1(w_{k,n,m}) \sim w_{k-1,n,m}$},
 and so the image of $W^j$ under $p^j_z$ is equally preserved by the action of $E_1$ because $p^j_z$ commutes with all $E_i$ and $F_i$.
 Let us now analyze the action of $E_0$ and $F_0$ on the vectors $w_{k,n,m}$. Observing that $w_{k,n,m} = w_k u_1^n u_2^m$, we get for the action the expressions similar to~\eqref{act-gen}
 \begin{gather*}
 E_0 (w_{k,n,m}) = E_0 (w_{k})u_1^n u_2^m = e_{1,k}(u_1,u_2) w_{k+1,n,m} + e_{2,k}(u_1,u_2) v_{k}u_1^n u_2^m, \\
 F_0(w_{k,n,m})= F_0 (w_{k})u_1^n u_2^m = f_{1,k}(u_1,u_2) w_{k-1,n,m}+f_{2,k}(u_1,u_2)v_{k-2}u_1^n u_2^m,
 \end{gather*}
 where $v_k$ is defined as in~\eqref{wks} and the coefficients are as in~\eqref{eq:coef1}--\eqref{eq:coef2}.
 Note that the images~$p^j_z(w_{k,n,m})$ are all proportional to $p^j_z(w_k) u_1^{n+m}$, so the vectors $w_{k,n}:=p^j_z(w_k) u_1^n$, for $n\in \mathbb{Z}$ and~$0{\leq k\leq 2j+1}$, form a basis in the image of $W^j$ under the projection $p^j_z$.
 Then, the action on these basis elements $w_{k,n}$ at $z=q^{j+\h}$ simplifies to
 \begin{equation}\label{act-gen-formal-simple}
 E_0 (w_{k,n}) = q^{-2j} w_{k+1,n-2}, \qquad
 F_0(w_{k,n})= q^{2j} [ 2j+2-k ]_q w_{k-1,n+2}.
 \end{equation}
 We thus see that $\Loop$ acts on $p^j_z\bigl(W^j\bigr)$. We now show that $p^j_z\bigl(W^j\bigr)$ is isomorphic to a quantum loop module.
 First note from~\eqref{act-gen-formal-simple} that the action of $E_0$ and $F_0$ changes the index $n$ by~$\pm2$ while the action of $E_1$ and $F_1$ keeps it unchanged. Therefore $p^j_z\bigl(W^j\bigr)$ is decomposed onto a~direct sum of two $\Loop$-modules: one generated by $w_{0,0}$ and the other by $w_{0,1}$.
 We identify~each of these modules with the quantum loop module \smash{$\widetilde{\mathbb{C}}^{2j+2}_{u_1 q^{j}}$} in the homogeneous gradation, using essentially the same analysis as in the proof of Lemma~\ref{lem-E}. Then using Lemma~\ref{lem:form-ev-decomp} we conclude that the image $p^j_z\bigl(W^j\bigr)$ is isomorphic to the quantum loop module ${\mathbb{C}}^{2j+2}_{u_1 q^{j}}$ which proves the proposition.
\end{proof}

We can now construct an intertwining operator corresponding to the quantum loop submodule described in Proposition~\ref{prop:Wj-image}
\begin{equation}\label{eq:intertE-formal}
 \mathcal{E}^{(j+\h)}\colon\ \mathbb{C}^{2j+2}_{uq^j} \rightarrow
 \mathbb{C}^{2}_{u}\otimes \mathbb{C}^{2j+1}_{u q^{j+\h}},
\end{equation}
where for simplicity we use the same notation as in the finite-dimensional setting of Section~\ref{intert-fus}.
First, we identify the quotient space \smash{$\mathbb{C}^{2}_{u}\otimes \mathbb{C}^{2j+1}_{\vphantom{A^a}\smash{u q^{j+\h}}}$} with \smash{$\mathbb{C}^2\otimes \mathbb{C}^{2j+1}\big[u^{\pm1}\big]$} as discussed below~\eqref{eq:proj-pz}. Using this identification, we define the action of $\mathcal{E}^{(j+\h)}$ on Laurent polynomials from $\mathbb{C}^{2j+2}_{uq^j} = \mathbb{C}^{2j+2}\big[u^{\pm1}\big]$
simply on their coefficients by the same matrix as in~\eqref{exprE}, with the matrix entries from Lemma~\ref{lem-E} given in the basis fixed below~\eqref{exprE}. This action on coefficients is well defined because the matrix entries do not depend on $u$. We then extend this action linearly to Laurent polynomials in $u$.
With the discussion below~\eqref{eq:pi-prod-form},
it is straightforward to check the intertwining property~\eqref{lem-intE} of such linear map $\mathcal{E}^{(j+\h)}$ along the same lines as in the proof of Lemma~\ref{lem-E}.
Similarly, its pseudo inverse map $\mathcal{F}^{(j+\h)}$ is first defined on coefficients of Laurent polynomials as in~\eqref{exprF} with the same matrix entries~\eqref{Fp1}--\eqref{Fp2}, and then extended by linearity. Note that $\mathcal{F}^{(j+\h)}$ is not a $\Loop$-intertwiner.

Finally, we have a construction of intertwining map \smash{$\bar{\mathcal{E}}^{(j-\h)}$} analogous to Section~\ref{sec:reduction-map}. This map corresponds to the spin-$(j-\h)$ quantum loop submodule in the quotient module \smash{$\mathbb{C}^{2}_{u}\otimes \mathbb{C}^{2j+1}_{\vphantom{A^a}\smash{u q^{-j-\h}}}$}.
To see this, one can follow the proof of Proposition~\ref{prop:Wj-image} but using the vector $v_0$ from~\eqref{hwv-w0} instead~of~$w_0$. We thus get a $\Loop$-intertwiner
\begin{equation}\label{eq:intertEbar-formal}
 \bar{\mathcal{E}}^{(j-\h)}\colon\ \mathbb{C}^{2j}_{uq^{-j-1}} \rightarrow
 \mathbb{C}^{2}_{u}\otimes \mathbb{C}^{2j+1}_{\vphantom{A^a}\smash{u q^{-j-\h}}}
\end{equation}
defined on coefficients of Laurent polynomials by the matrix~\eqref{exprbarE}, with the matrix entries from Lemma~\ref{lem-barE}. This map satisfies the same relation~\eqref{lem-intEbar}, now in the context of discussion below~\eqref{eq:pi-prod-form}.
Similarly, its pseudo inverse map \smash{$\bar{\mathcal{F}}^{(j-\h)}$} is defined on coefficients of Laurent polynomials via~\eqref{exprbF} and~\eqref{coefbar}.

\section[Spin-j L- and K-operators]{Spin-$\boldsymbol{ j}$ L- and K-operators} \label{sec4}

In Sections~\ref{sec:spin-j-L} and~\ref{sec:spin-j-K}, we define spin-$j$ L- and K-operators, respectively, as evaluations of universal R- and K-matrices for $H=\Loop$ and $B$ a comodule algebra for a certain twist pair. Using the intertwining operators constructed above,
we obtain main results of this section: explicit matrix form of the so-called \textit{fusion} and \textit{reduction} equations satisfied by the spin-$j$ L-operators, in Propositions~\ref{propfusR} and~\ref{barpropfusR}, respectively, and similarly for the K-operators in Propositions~\ref{propfusK} and~\ref{propfusedbarK}. We also establish the so-called P-symmetry and unitarity properties of the fused L-operators. In the final Section~\ref{sec:comod}, the comodule algebra structure on $B$ is characterized within the framework of the spin-$j$ K-operators and the Ding--Frenkel L-operators.

\subsection[Spin-j L-operators]{Spin-$\boldsymbol{j}$ L-operators}\label{sec:spin-j-L}
\subsubsection{Evaluation of the universal R-matrix}
As we noted before,
the universal R-matrix $\mathfrak{R}$ for $H=\Loop$ is an element of a certain completion of $\Loop \otimes \Loop$ (see details, e.g., in~\cite[Section~2.5]{AV22}), expressed as a product over an infinite set of root vectors~\cite[Theorem~1]{Tolstoy1991}, see also~\cite{Da98,Tolstoy1992,LS}. We give the universal R-matrix expression in our conventions in Appendix~\ref{appC}. In this section, we consider 1- and 2-component evaluation of $\mathfrak{R}$ on the formal evaluation representations
$\pi^j_u$ studied in the previous Section~\ref{sec:formal-ev}, recall~\eqref{evalrep-form} and~\eqref{evalrep-form-explicit}.

\begin{defn} \label{defL}
 For $j \in \h \mathbb{N}$, we define
 \begin{equation} \label{evalL}
 {\bf L}^{(j)}(u_1/u_2)= \bigl( \mathsf{ev}_{u_1}\otimes \pi_{u_2}^{j}\bigr)( \mathfrak{R}) \in \Uq[[u_2/u_1]] \otimes \End\bigl(\mathbb{C}^{2j+1}\bigr).
 \end{equation}
 We call ${\bf L}^{(j)}(u)$ the spin-$j$ L-operator. In particular, ${\bf L}^{(0)}(u)=1$ by~\eqref{epsR}.
\end{defn}

A few comments for the above definition are necessary:
\begin{itemize}\itemsep=0pt
 \item Strictly speaking, the right-hand side of~\eqref{evalL} does not make sense literally because $\mathfrak{R}$ does not belong to $\Loop\otimes \Loop$ but to its completion which is much bigger and \textit{a priori} is not of a tensor product form, so we cannot apply a map of this form. What makes sense is to use a bi-gradation on the positive part $\Loop^+$ generated by $E_i$: the degree of $E_0$ is $(1,0)$, while $E_1$ has degree $(0,1)$. In other words, replacing each $E_i$ by $z_i E_i$ in the $q$-exponent expression~\eqref{app:univRform} of $\mathfrak{R}$ produces $\mathfrak{R}(z_0,z_1)$
 such that $\smash{q^{-\h h_1 \otimes h_1}\mathfrak{R}(z_0,z_1)}\in \Loop^+\otimes \Loop^-[[z_0,z_1]]$. Now, we can apply a map of tensor product form coefficient-wise, and set $z_0=z_1=1$ in the end of the calculation. This is what we will always use implicitly in the exposition below.

 \item
 First, we evaluate 2nd tensor component of $\mathfrak{R}$ at $j=\h$ treating $u_2$ as a formal variable. This is done in Appendix~\ref{app-subsec:Lp} using the exponential form of $\mathfrak{R}$ reviewed in Appendix~\ref{app:C2}. This calculation essentially amounts to evaluating infinite series of the root vectors with~\eqref{evalrep-form-explicit}, that is why one needs to consider $u_2$ as a formal variable and not just a complex number. The result of such evaluation is the L-operator of type \smash{$\widehat{\bf L}^{(\h)}(u_2)$} from~\eqref{eq:L-op} for $H=\Loop$, and it takes the form of the Ding--Frenkel L-operator~\cite{DF93}. It is a $2\times 2$ matrix with entries from $\Loop[[u_2]]$.

 \item
 In the second step, it is straightforward to evaluate the remaining 1st component of $\mathfrak{R}$ with~\eqref{formal-evu}, see our calculation in Appendix~\ref{appC4}, and with the final result that depends on the ratio $u=u_1/u_2$ only
 \begin{gather}
{\bf L}^{(\h)}(u)= \mu(u)
 {\bf {\mathcal L}}^{(\h)}(u)\qquad \text{with}\nonumber\\
{\bf {\mathcal L}}^{(\h)}(u)=
 \begin{pmatrix}
 u q^{\h} K^{\h}-u^{-1}q^{-\h} K^{-\h} & \bigl(q-q^{-1}\bigr) F \\
 \bigl(q-q^{-1}\bigr) E & u q^{\h} K^{-\h}-u^{-1}q^{-\h} K^{\h}
 \end{pmatrix} , \label{Laxh}
 \end{gather}
 where the `normalization' $\mu(u)$ is the following formal power series in $u^{-1}$:
 \begin{equation} \label{expMU}
 \mu(u) = u^{-1} q^{-\h} e^{\Lambda(u^{-2}q^{-1} )},
 \end{equation}
 and $\Lambda(u)$ is a formal power series in $u$ with coefficients
 from the center $Z(\Uq)$, see its explicit expression in~\eqref{app:Lambda}.
 Formally, the result of evaluation in~\eqref{evalL} belongs to~$\Uq\big[\big[u_1^{-1},u_2\big]\big] \otimes \End\bigl(\mathbb{C}^2\bigr) $, however we see that it actually belongs to a smaller space which is $\Uq[[ u_2/u_1]] \otimes \End\bigl(\mathbb{C}^2\bigr) $.
 Furthermore, we show by the fusion construction discussed below that all the L-operators ${\bf L}^{(j)}(u_1/u_2)$ are indeed formal power series belonging to~$\Uq\big[\big[u^{-1}\big]\big] \otimes \End\bigl(\mathbb{C}^{2j+1}\bigr)$ with $u=u_1/u_2$, as indicated in Definition~\ref{defL}.
\end{itemize}

\begin{rem} \label{remLaxh}
 The spin-$\h$ L-operator ${\bf L}^{(\h)}(u)$
 was initially calculated
 in~\cite{Boos2010}, with different conventions on the coproduct, see~\cite[equation~(4.62)]{Boos2012}.\footnote{With the identification $\zeta \rightarrow u$, $s \rightarrow -2$, in particular \smash{$e^{\Lambda(q^{-1}\zeta^s )} \rightarrow \mu(u) u q^{\h}$}, and $s_0 \rightarrow -1$, $s_1 \rightarrow -1$.}
 The two L-operators are related as follows. Recall that the auto\-mor\-phism \eqref{DeltaTK} relates the universal R-matrix corresponding to our coproduct with the universal R-matrix of~\cite{Tolstoy1991}. Then~\eqref{evalL} for $j=\h$ reads
 \begin{equation} \label{step1-L}
 {\bf L}^{(\h)}(u_1/u_2) = \bigl(\id \otimes \pi^\h\bigr)\circ
 (\mathsf{ev}_{u_1} \otimes \mathsf{ev}_{u_2}) \circ \bigl( \nu^{-1} \otimes \nu^{-1}\bigr)\bigl(\mathfrak{R}^{\rm KT}\bigr),
 \end{equation}
 where $\nu$ is given in~\eqref{autnu}.
 Now, consider the automorphism $\tau \colon \Uq \rightarrow \Uq$ defined by
 \[ 
 \tau(E) = q^{-\h} E K^{-\h}, \qquad \tau(F) = q^{-\h} F K^{\h}, \qquad \tau \bigl( K^{\pm\h}\bigr) = K^{\pm \h}.
 \]
 Noting that it satisfies the property \smash{$\mathsf{ev}_{u} \circ \nu^{-1}(x) = \tau \circ \mathsf{ev}_{u q^{-\h}}(x)$}, for all $x\in \Loop$, and $\pi^\h\circ \tau = \pi^\h$,
 then~\eqref{step1-L} becomes
 \[
 {\bf L}^{(\h)}(u_1/u_2) =
 \bigl(\tau \otimes \pi^\h\bigr) \circ
 (\mathsf{ev}_{u_1 q^{-\h}} \otimes \mathsf{ev}_{u_2 q^{-\h}})\bigl(\mathfrak{R}^{\rm KT}\bigr) = (\tau \otimes \id) \bigl({\bf L}^{(KT)}(u_1/u_2)\bigr),
 \]
 where ${\bf L}^{(KT)}(u)$ is the L-operator from~\cite[equation~(4.62)]{Boos2012}.
\end{rem}

Evaluating the first component of ${\bf L}^{(j_2)}(u)$ on the spin-$j_1$ representation of $\Uq$, we get the R-matrix for any spin $j_1$, $j_2$,
\begin{align} \label{evalR}
 \mathcal{R}^{(j_1,j_2)}(u_1/u_2)&=\bigl(\pi^{j_1}_{u_1}\otimes \pi^{j_2}_{u_2}\bigr) (\mathfrak{R}) = \bigl(\pi^{j_1} \otimes \id \bigr) ({\bf L}^{(j_2)}(u_1/u_2) ),
\end{align}
which
is a matrix from $\End\bigl(\mathbb{C}^{2j_1+1}\otimes \mathbb{C}^{2j_2+1}\bigr)$ with entries in formal power series $\mathbb{C}\big[\big[u^{-1}\big]\big]$.

\begin{Example}
 For $j_1=j_2=\h$ in~\eqref{evalR}, the corresponding R-matrix is given by
 \begin{gather}
 \mathcal{R}^{(\h,\h)}(u) = \pi^{\h}(\mu(u)) R^{(\h,\h)}(u) \qquad \text{with} \nonumber\\
 R^{(\h,\h)}(u)= \begin{pmatrix}
 c(uq)& 0 & 0 &0 \\
 0 & c(u)& c(q)&0 \\
 0 & c(q)&c(u) & 0 \\
 0 & 0 & 0 &c(uq)
 \end{pmatrix},\label{evalRh}
 \end{gather}
 where $c(u)$ is given in~\eqref{eq:cu} and with
 \begin{equation} \label{pimu}
 \pi^{j}(\mu(u)) = u^{-1}q^{-\h} \exp \left ( \sum_{k=1}^\infty \frac{q^{k(2j+1)} + q^{-k(2j+1)}} {1+q^{2k}} \frac{u^{-2k}}{k} \right ) ,
 \end{equation}
 where we used the evaluation of the coefficients $C_k$ of $\Lambda(u)$ given in~\eqref{app:Lambda}, see~\cite[equation~(4.59)]{Boos2012}.
 Note that \smash{$R^{(\h,\h)}(u)$} coincides with the expression in~\eqref{R-Rqg} for $j=\h$.
\end{Example}

We recall that the L-operators satisfy the RLL relations.
Indeed, applying $\bigl(\mathsf{ev}_{u_1} \otimes \pi^{j_2}_{u_2} \otimes \pi^{j_3}_{u_3}\bigr)$ to~\eqref{YB-noparam}, one finds\footnote{The RLL equation belongs to the triple product \smash{$\Uq \otimes \End-(\mathbb{C}^{2j_1+1}\bigr) \otimes \End-(\mathbb{C}^{2j_2+1}\bigr)$}, and thus the L-operator should be written as \smash{${\bf L}^{(j)}_{0i}(u)$} but here we omit the label $0$ corresponding to $\Uq$.\label{L-not}}
\begin{equation} \label{univRLLv2}
 {\bf L}_1^{(j_2)}(u_1/u_2) {\bf L}_2^{(j_3)}(u_1/u_3) \mathcal{R}_{12}^{(j_2,j_3)}(u_2/u_3) = \mathcal{R}_{12}^{(j_2,j_3)}(u_2/u_3) {\bf L}_2^{(j_3)}(u_1/u_3){\bf L}_1^{(j_2)}(u_1/u_2).
\end{equation}

Recall also that the R-matrix satisfies the Yang--Baxter equation.
It is found by applying $(\pi^{j_1} \otimes \id \otimes \id)$ to the above equation and setting $u_3=1$,
\begin{equation}\label{YBj1j2}
 \mathcal{R}_{12}^{(j_1,j_2)}(u_1/u_2) \mathcal{R}_{13}^{(j_1,j_3)}(u_1) \mathcal{R}_{23}^{(j_2,j_3)}(u_2)=\mathcal{R}_{23}^{(j_2,j_3)}(u_2)\mathcal{R}_{13}^{(j_1,j_3)}(u_1)\mathcal{R}_{12}^{(j_1,j_2)}(u_1/u_2).
\end{equation}

We recall that the so-called quantum determinant is given by~\cite{Skly88}
 \begin{equation}\label{eq:qdet}
 {\gamma}(u)= \normalfont{\text{tr}}_{12} \bigl( \mathcal{P}_{12}^- {\mathcal L}_1^{(\h)}(u) {\mathcal L}_2^{(\h)}(u q) \bigr) = u^2q^2+u^{-2}q^{-2} - C,
 \end{equation}
where $\normalfont{\text{tr}}_{12}$ stands for the trace over $V_1\otimes V_2$, and ${\mathcal P}^-_{12}=(1-{\mathcal P})/2$ with the permutation matrix~\smash{${\mathcal P=R^{(\frac{1}{2},\frac{1}{2})}(1)/\bigl(q-q^{-1}\bigr)}$}, and the Casimir element $C$ is defined by %
\begin{gather} \label{Cas-Uqsl2}
 C= \bigl(q-q^{-1}\bigr)^2 FE + qK+q^{-1}K^{-1}.
\end{gather}

\begin{rem} \label{rem:invLh}
 The inverse of \smash{$\mathcal{L}^{(\h)}(u)$} is given by
 \begin{equation}\label{eq:Linv}
 \big\lbrack \mathcal{L}^{(\h)}(u) \big\rbrack^{-1} = - \frac{1}{\gamma\bigl(u q^{-1}\bigr)} \mathcal{L}^{(\h)}\bigl(u^{-1}\bigr),
 \end{equation}
 with $\gamma(u)$ from~\eqref{eq:qdet}. Note that $u^{-2}\gamma(u)$ is invertible in $Z(\Uq)\big[\big[u^{-1}\big]\big]$ due to~\cite[Lemma~4.1]{Ter21d}, and so $\gamma(u)$ is invertible too.
\end{rem}

\subsubsection[L-operators and fusion (h,j)rightarrow (j+h)]{L-operators and fusion $\boldsymbol{(\h,j)\rightarrow (j+\h)}$}
We study a so-called fusion relation for L-operators that expresses ${\bf L}^{(j+\h)}(u)$ in terms of product of ${\bf L}^{(j)}(u_1)$ and \smash{${\bf L}^{(\h)}(u_2)$}. For this, we evaluate 2nd and 3rd tensor components of the universal R-matrix equation~\eqref{univR3} on quantum loop modules with formal evaluation parameters fixed as~\smash{$u_2= q^{j+\h}u_1$}. This fixing requires taking the quotient~\eqref{eq:Iz-quotient} to relate the two formal variables $u_1$ and $u_2$ on tensor product~\eqref{formal-tens-rep} of two quantum loop modules \smash{$\mathbb{C}_{u_1}^2$} and~\smash{$\mathbb{C}_{u_2}^{2j+1}$}, respectively.
Using Proposition~\ref{prop:Wj-image}, on this quotient module
we get a $\Loop$-submodule isomorphic to the quantum loop module \smash{$\mathbb{C}^{2j+2}_{u_1 q^{j}}$} of spin-$\bigl(j+\h\bigr)$, with the corresponding intertwining operator \smash{${\mathcal{E}}^{(j+\h)}$} described below~\eqref{eq:intertE-formal}. As it was explained there, this intertwiner
is defined by Lemma~\ref{lem-E} and its pseudo-inverse \smash{$\mathcal{F}^{(j+\h)}$} is fixed by~\eqref{exprF} with~\eqref{Fp1},~\eqref{Fp2}.
Using the associated identity~\smash{$\mathcal{F}^{(j+\h)} \mathcal{E}^{(j+\h)} = \id_{\mathbb{C}^{2j+2}[u^{\pm1}]}$}, fusion relations satisfied by L-operators and R-matrices are exhibited in the following two propositions.

\begin{prop}\label{propfusR} For $j \in \h \mathbb{N}$, the spin-$j$ L-operators \eqref{evalL} satisfy the following relations:
 \begin{equation} \label{fused-L-uq}
 {\bf L}^{(j+\h)}(u)= \mathcal{F}^{(j+\h)}_{\langle 12 \rangle} {\bf L}_{2}^{(j)}\bigl(u q^{-\h}\bigr){\bf L}_{1}^{(\h)}\bigl(u q^{j}\bigr) \mathcal{E}^{(j+\h)}_{\langle 12 \rangle} \in \Uq\big[\big[u^{-1}\big]\big] \otimes \End\bigl(\mathbb{C}^{2j+2}\bigr),
 \end{equation}
 where we use the notation $\fu$ to indicate which spaces are fused, i.e., where the intertwin\-er~\smash{$\mathcal{E}^{(j+\h)}$} acts, and here we use its finite-dimensional matrix form given in~\eqref{exprE} with~\eqref{E-proj}.
\end{prop}

\begin{proof}
 By definition of the L-operator we have
 \smash{${\bf L}^{(j+\h)}(w/u)=\bigl( \mathsf{ev}_w \otimes \pi_u^{j+\h}\bigr) (\mathfrak{R})$}.
 Using the pseudo-inverse property \smash{$ \mathcal{F}^{(j+\h)} \mathcal{E}^{(j+\h)} = \id$}, we get
 \begin{align*}
 {\bf L}^{(j+\h)}(w/u)&= \bigl(1 \otimes \mathcal{F}^{(j+\h)}\mathcal{E}^{(j+\h)}\bigr) \big[\bigl( \mathsf{ev}_w \otimes \pi_u^{j+\h}\bigr) (\mathfrak{R}) \big] \\
 &= \bigl(1 \otimes \mathcal{F}^{(j+\h)} \bigr) \big[\bigl( \mathsf{ev}_{w} \otimes \pi_{uq^{-j}}^{\h} \otimes \pi_{uq^{\h}}^{j}\bigr) (\id \otimes \Delta)( \mathfrak{R}) \big ](1 \otimes\mathcal{E}^{(j+\h)})\\
 &= \bigl( 1 \otimes \mathcal{F}^{(j+\h)} \bigr) \big[\bigl( \mathsf{ev}_{w}\otimes \pi_{u q^{-j}}^{\h} \otimes \pi_{u q^{\h}}^{j} \bigr)(\mathfrak{R}_{13}\mathfrak{R}_{12}) \big]\bigl(1 \otimes \mathcal{E}^{(j+\h)} \bigr) \\
 &=\bigl( 1 \otimes \mathcal{F}^{(j+\h)} \bigr) {\bf L}_2^{(j)}\bigl(q^{-\h} w/u \bigr) {\bf L}_1^{(\h)}\bigl(q^{j} w/u \bigr) \bigl(1 \otimes \mathcal{E}^{(j+\h)}\bigr),
 \end{align*}
 where in the second line we use the convention discussed below~\eqref{eq:pi-prod-form}.
 Here, the second equality is obtained using the intertwining property~\eqref{lem-intE} which is applicable in the formal $u$ setting too, following the discussion below~\eqref{eq:intertE-formal}. Applied to elements of 2nd tensor component of $\mathfrak{R}$ this intertwining property takes the following form:
 \[
 \bigl(1 \otimes \mathcal{E}^{(j+\h)}\bigr) \big[( \id \otimes \pi_{u}^{j+\h} )(\mathfrak{R}) \big]= \big[\bigl(\id \otimes \pi_{u q^{-j}}^{\h} \otimes \pi_{uq^{\h}}^{j} \bigr) \circ ( \id \otimes \Delta ) ( \mathfrak{R})\big] \bigl(1 \otimes \mathcal{E}^{(j+\h)} \bigr),
 \]
 where the maps of the tensor product form are applied to $\mathfrak{R}(z_0,z_1)$ coefficient-wise with respect to the bi-grading we always use to treat properly the universal R-matrix expansion, recall the first comment below Definition~\ref{defL}.
 Then, the third equality in the above calculation is due to~\eqref{univR3}, and the last one is by definition of the L-operator.
\end{proof}

\begin{rem}\label{rem:L-RSV}
 Proposition~\ref{propfusR} is a generalization of~\cite[Proposition~3.3]{RSV16}, where a similar fusion formula is given for `L-operators' which are in our terminology $R$-matrices on the tensor product of the spin-$j$ and Verma modules. The L-operators of~\cite{RSV16} are recovered by evaluating entries of the matrices~\eqref{fused-L-uq} on the Verma modules of $\Uq$. The other novelty with respect to~\cite{RSV16} is that we give and use explicit expressions for the intertwiners $\mathcal{E}^{(j)}$ and their pseudo-inverses~$\mathcal{F}^{(j)}$.
\end{rem}

\begin{prop} The R-matrices \eqref{evalR}
 satisfy for all $j_1, j_2 \in \h \mathbb{N}$ the following relations in~$\mathbb{C}\big[\big[u^{-1}\big]\big] \otimes \End\bigl(\mathbb{C}^{2j_1+1}\otimes \mathbb{C}^{2j_2+1}\bigr)$:
 \begin{equation}
 \label{v2Rj1j2}
 \mathcal{R}^{(j_1,j_2)}(u)=\mathcal{F}^{(j_1)}_{\fu} \mathcal{R}_{13}^{(\h,j_2)}\bigl(u q^{-j_1+\h}\bigr) \mathcal{R}_{23}^{(j_1-\h,j_2)}\bigl(u q^{\h}\bigr) \mathcal{E}^{(j_1)}_{\fu},
 \end{equation}
 where
 \begin{equation} \label{fused-R-uq}
 \mathcal{R}^{(\h,j+\h)}(u)=\mathcal{F}^{(j+\h)}_{\langle 23 \rangle} \mathcal{R}_{13}^{(\h,j)}\bigl(u q^{-\h}\bigr) \mathcal{R}_{12}^{(\h,\h)}\bigl(u q^{j}\bigr) \mathcal{E}^{(j+\h)}_{\langle 23 \rangle}.
 \end{equation}
\end{prop}
\begin{proof}
 First we show~\eqref{fused-R-uq}. By definition, we have \smash{$\mathcal{R}^{(\h,j_2)}(u) = \bigl(\pi^{\h} \otimes \id\bigr)\bigl( {\bf L}^{(j_2)}(u)\bigr)$}, therefore the application of \smash{$\bigl(\pi^\h \otimes \id\bigr)$} on~\eqref{fused-L-uq} yields\footnote{The shifting of the labels $\{ 0, 1,2\}$ to $\{ 1, 2, 3\}$ is due to the convention that the first tensor component of~${\bf L}^{(j)}(u)$ is labeled by $0$.}~\eqref{fused-R-uq}.
 Equation~\eqref{v2Rj1j2} is shown similarly to Proposition~\ref{propfusR} using now~\eqref{univR2} and~\eqref{lem-intE}. It gives
 \begin{align*}
 \mathcal{R}^{(j_1,j_2)}(v/w) &= \bigl(\mathcal{F}^{(j_1)} \mathcal{E}^{(j_1)} \otimes \id\bigr) \circ \bigl(\pi^{j_1}_v \otimes \pi^{j_2}_{w}\bigr) (\mathfrak{R}) \\
 &= \mathcal{F}^{(j_1)}_\fu \bigl [ \bigl(\pi_{v q^{-j_1+\h}}^\h \otimes \pi_{v q^\h}^{j_1-\h} \otimes \pi_{w}^{j_2} \bigr) ( \mathfrak{R}_{13} \mathfrak{R}_{23}) \bigr ] \mathcal{E}^{(j_1)}_\fu,
 \end{align*}
 and using the evaluation of the universal R-matrix from~\eqref{evalR}, the result follows after substituting~${v/w \rightarrow u}$.
\end{proof}

\begin{rem} \label{uniFR}
 Recall that $\mathcal{F}^{(j+\h)}$ is not uniquely determined, see Section~\ref{intert-fus}. From the construction in the proof of Proposition~\ref{propfusR}, it is clear that taking different expressions for $\mathcal{F}^{(j+\h)}$ yields the same L-operators and R-matrices.
\end{rem}
\subsubsection[L-operators and reduction (h,j)rightarrow (j-h)]{L-operators and reduction $\boldsymbol{(\h,j)\rightarrow (j-\h)}$}
We now consider quantum loop submodules of spin-$(j-\h)$ and the corresponding $\Loop$-intertwiner \smash{$\bar{\mathcal{E}}^{(j-\h)}$} in~\eqref{eq:intertEbar-formal}.
With the intertwining property~\eqref{lem-intEbar},
we can write
\[
\bigl( 1 \otimes \bar{\mathcal{E}}^{(j-\h)} \bigr) \big\lbrack \bigl(\id \otimes \pi_u^{j-\h}\bigr) ( \mathfrak{R}) \big\rbrack= \big\lbrack \bigl(\id \otimes \pi_{u q^{j+1}}^{\h} \otimes \pi_{u q^{\h}}^j\bigr) \circ ( \id \otimes \Delta)(\mathfrak{R}) \big\rbrack \bigl( 1 \otimes \bar{\mathcal{E}}^{(j-\h)}\bigr),
\]
with the convention discussed below~\eqref{eq:pi-prod-form}. Then,
the proof of the following proposition essentially repeats the one of Proposition~\ref{propfusR}.

\begin{prop}\label{barpropfusR} The L-operators \eqref{evalL} satisfy for $j\in\h \mathbb{N}_+$ the following matrix relations with coefficient in $\Uq\big[\big[u^{-1}\big]\big]$
 \begin{equation} \label{fused-bar-L-uq}
 {\bf L}^{(j-\h)}(u)= \bar{\mathcal{F}}^{(j-\h)}_{\langle 12 \rangle} {\bf L}_{2}^{(j)}\bigl(u q^{-\h}\bigr){\bf L}_{1}^{(\h)}\bigl(u q^{-j-1}\bigr) \bar{\mathcal{E}}^{(j-\h)}_\fu.
 \end{equation}
\end{prop}

\subsubsection[P-symmetry of spin-j L-operators]{P-symmetry of spin-$\boldsymbol{ j}$ L-operators}
We now give our first application of the fusion formula in Proposition~\ref{propfusR} to show the so-called P-symmetry property of the spin-$j$ L-operators, in~\eqref{evR21} below. As a consequence, we show that the P-symmetry
\begin{equation}
 \label{Psym}
 \mathcal{R}_{21}^{(j_2,j_1)}(u) = \mathcal{R}^{(j_1,j_2)}(u),
\end{equation}
with
\begin{equation} \label{calRj2j1} \mathcal{R}_{21}^{(j_2,j_1)}(u) = \mathcal{P}^{(j_2,j_1)} \mathcal{R}^{(j_2,j_1)}(u) \mathcal{P}^{(j_1,j_2)}, \end{equation}
holds for any $j_1$, $j_2$, for the case of $H=\Loop$. Note that a proof of this equation for R-matrices can be found in~\cite[Lemma~2.1]{RSV16}. Here, we give a different proof by showing first a more general relation
\begin{equation} \label{evR21-R}
 \bigl(\mathsf{ev}_{u_1^{-1}} \otimes \pi^j_{u_2^{-1}}\bigr) (\mathfrak{R}) = \bigl(\mathsf{ev}_{u_1} \otimes \pi^j_{u_2}\bigr) (\mathfrak{R}_{21}),
\end{equation}
which can be interpreted as the P-symmetry on the level of L-operators. Indeed, identifying the left-hand side of~\eqref{evR21-R} with the spin-$j$ L-operator defined by~\eqref{evalL}, the equation~\eqref{evR21-R} reads
\begin{align} \label{evR21}
 {\bf L} ^{(j)}(u_2/u_1) &= \bigl( \normalfont{\mathsf{ev}}_{u_1}\otimes \pi_{u_2}^{j}\bigr) (\mathfrak{R}_{21}).
\end{align}
The proof of this equation is done by induction on $j$.
It is straightforward to check that~\eqref{evR21} holds for $j=\h$, by a calculation similar to the evaluation of $\mathfrak{R}$ in Appendix~\ref{appC}.
Now assume~\eqref{evR21} holds for a fixed value of $j$, we show it holds for $(j+\h)$. It is done by identifying the right-hand side of~\eqref{evR21} with~\eqref{fused-L-uq}. Indeed, by an analysis similar to the proof of Proposition~\ref{propfusR}, using the pseudo-inverse property \smash{$ \mathcal{F}^{(j+\h)} \mathcal{E}^{(j+\h)} = \id$},~\eqref{intert-op} and
$
(\id \otimes \Delta^{\rm op}) (\mathfrak{R}_{21}) = \mathfrak{R}_{31} \mathfrak{R}_{21}
$, one gets
\begin{align} \nonumber
 \bigl(\mathsf{ev}_{u_1} \otimes \pi_{u_2}^{j+\h}\bigr)(\mathfrak{R}_{21}) &= \bigl( 1 \otimes \mathcal{F}^{(j+\h)} \bigr) \big[ \bigl( \mathsf{ev}_{u_1} \otimes \pi_{u_2q^{j}}^\h \otimes \pi^j_{u_2 q^{-\h}} \bigr) ( \mathfrak{R}_{31} \mathfrak{R}_{21}) \big] \bigl( 1 \otimes \mathcal{E}^{(j+\h)} \bigr) \\
 &= \mathcal{F}_{\langle 12 \rangle}^{(j+\h)} {\bf L}_2^{(\h)}\bigl(u_2 q^{-\h}/u_1\bigr) {\bf L}_1^{(j)}\bigl(u_2 q^{j}/u_1 \bigr) \mathcal{E}_{\langle 12 \rangle}^{(j+\h)},\label{compL21}
\end{align}
where we used the assumption~\eqref{evR21} for a fixed $j$ to get the last line. Then, comparing~\eqref{compL21} with~\eqref{fused-L-uq}, one finds that indeed \smash{$\bigl(\mathsf{ev}_{u_1} \otimes \pi_{u_2}^{j+\h}\bigr)(\mathfrak{R}_{21}) = {\bf L}^{(j+\h)}(u_2/u_1)$}.

Now, by specializing the first tensor component of~\eqref{evR21} on the spin-$j_1$ representation and for $j=j_2$, one obtains
\begin{equation} \label{eq:R21}
 \mathcal{R}^{(j_1,j_2)}(u_2/u_1) = \bigl(\pi_{u_1}^{j_1} \otimes \pi_{u_2}^{j_2}\bigr) (\mathfrak{R}_{21}).
\end{equation}
Finally, the immediate corollary of the latter relation is the P-symmetry~\eqref{Psym}.
Indeed, the R-matrix defined in~\eqref{calRj2j1} can be interpreted as the evaluation of the flipped universal R-matrix
\begin{gather} \label{eval-Psym}
 \mathcal{R}_{21}^{(j_2,j_1)}(u_2/u_1) = \bigl( \pi^{j_1}_{u_1} \otimes \pi^{j_2}_{u_2}\bigr) ( \mathfrak{R}_{21} ),
\end{gather}
and thus~\eqref{Psym} holds.

\subsubsection{Fused L-operators and fused R-matrices}\label{subsec:norm}
For further technical needs, let us introduce
a higher spin generalization of $2\times2$ matrices~$\mathcal{L}^{(\h)}(u)$ from \eqref{Laxh}.
For any $j \in \frac{1}{2} \mathbb{N}_+$, we define \textit{the fused} L-operators $\mathcal{L}^{(j)}(u) \in \Uq\big[u,u^{-1}\big] \otimes \End\bigl( \mathbb{C}^{2j+1}\bigr) $ as
\begin{equation} \label{cal-fused-L-uq}
 {\mathcal L}^{(j+\h)}(u) := \mathcal{F}^{(j+\h)}_{\langle 12 \rangle} {\mathcal L}_{2}^{(j)}\bigl(u q^{-\h}\bigr){\mathcal L}_{1}^{(\h)}\bigl(u q^{j}\bigr) \mathcal{E}^{(j+\h)}_{\langle 12 \rangle}.
\end{equation}
Although not needed here, it can be proven directly by induction that $\mathcal{L}^{(j)}(u)$'s satisfy the Yang--Baxter equation (\ref{univRLLv2}), where ${\bf L}^{(j)}(u)$ are replaced by ${\mathcal L}^{(j)}(u)$.
We now give the relations between the spin-$j$ L-operators~\eqref{evalL}, obtained by evaluation of the universal R-matrix, and the fused L-operators~\eqref{cal-fused-L-uq}.
\begin{lem} \label{lem-muj} The spin-$j$ L-operators and the fused L-operators are related as follows:
 \begin{equation} \label{evalLj}
 {\bf L}^{(j)}(u) = \mu^{(j)}(u) {\mathcal L}^{(j)}(u),
 \end{equation}
 where
 \begin{equation}
 \mu^{(j)}(u)= \prod_{k=0}^{2j-1} \mu\bigl(u q^{j-\h-k}\bigr) \label{exp-mujh}
 \end{equation}
 is central in $\Uq\big[\big[u^{-1}\big]\big]$.
\end{lem}

\begin{proof}
 The relation~\eqref{evalLj} is shown by induction on $j$ using \eqref{fused-L-uq} and \eqref{cal-fused-L-uq} and the fact that~$\mu(u)$ is central in $\Uq$. The first step of the induction is given by~\eqref{Laxh}.
\end{proof}

As pointed out in \cite{Boos2012}, using the explicit expression \eqref{app:Lambda}
one can show that the ``normalising'' generating function $\mu(u)$ in \eqref{Laxh}, given by \eqref{expMU}, satisfies the following functional relation:
\begin{gather}\label{eq:mu}
 \mu(u) \mu(u q) \gamma(u) = 1,
\end{gather}
with $\gamma(u)$ introduced in~\eqref{eq:qdet}.
Actually, this functional relation can be independently derived from the fusion and reduction relations obtained in Propositions~\ref{propfusR} and~\ref{barpropfusR} without the knowledge of \eqref{expMU}, as we now show.
\begin{prop} \label{lem-mu} Assume the spin-$\h$ L-operator is of the form \eqref{Laxh} for some formal Laurent series $\mu(u)$ from $Z(\Uq)\bigl(\bigl(u^{-1}\bigr)\bigr)$. Then, the functional relation \eqref{eq:mu} holds.
\end{prop}
\begin{proof}
 Comparing \smash{${\bf L}^{(\h)}(u) = \mu(u) \mathcal{L}^{(\h)}(u)$} with~\eqref{fused-bar-L-uq} for $j=1$, it follows
 \begin{equation}\label{gen-mumu}
 \mathcal{L}^{(\h)}(u) = \mu\bigl(u q^{-1}\bigr) \mu\bigl(uq^{-2}\bigr) \bar{\mathcal{F}}^{(\h)}_\fu \mathcal{L}_2^{(1)}\bigl(u q^{-\h}\bigr) \mathcal{L}^{(\h)}_1\bigl(uq^{-2}\bigr) \bar{\mathcal{E}}_\fu^{(\h)},
 \end{equation}
 where we used $\mu^{(1)}(u)$ in~\eqref{exp-mujh}.
 For the right-hand side of~\eqref{gen-mumu}, after a direct calculation we find that
 \[
\bar{\mathcal{F}}^{(\h)}_\fu \mathcal{L}_2^{(1)}\bigl(u q^{-\h}\bigr) \mathcal{L}^{(\h)}_1\bigl(uq^{-2}\bigr) \bar{\mathcal{E}}_\fu^{(\h)} = \gamma(u q^{-2}) \mathcal{L}^{(\h)}(u).
 \]
 Then, since \smash{$\mathcal{L}^{(\h)}(u)$} is invertible by~\eqref{eq:Linv}, the relation~\eqref{eq:mu} follows.
\end{proof}

We obtain the following corollary of the functional relation~\eqref{eq:mu}.
\begin{cor}\qquad
 \begin{itemize}\itemsep=0pt
 \item [$(i)$] The inverse of ${\bf L}^{(\h)}(u)$ is given by
 \[
 \big\lbrack {\bf L}^{(\h)}(u) \big\rbrack^{-1} = - \mu\bigl(uq^{-1}\bigr) \mathcal{L}^{(\h)}\bigl(u^{-1}\bigr).
 \]

 \item [$(ii)$]
 The quantum determinant of the L-operator ${\bf L}^{(\h)}(u)$ is such that
 \begin{equation}
 \normalfont{\text{tr}}_{12} \bigl( \mathcal{P}_{12}^- {\bf L}_1^{(\h)}(u) {\bf L}_2^{(\h)}(u q) \bigr) = 1,\label{qdetLb}
 \end{equation}
 where $\mathcal{P}^-$ is given below~\eqref{eq:qdet}.
 \end{itemize}
\end{cor}
\begin{proof}
 Recall first from~\eqref{expMU} and Remark~\ref{rem:invLh} that $\mu(u)$ and $\gamma(u)$ are invertible. Then, to prove the first statement we use
 the first equality in~\eqref{Laxh}, \eqref{eq:Linv} and~\eqref{eq:mu}. Equation~\eqref{qdetLb} follows directly from the first equality in~\eqref{Laxh},~\eqref{eq:qdet} and the functional relation~\eqref{eq:mu}.
\end{proof}

\begin{rem}
 Assuming that $\mu(u)$ in~\eqref{Laxh} is an invertible Laurent series in $u$ with coefficients in a commutative subalgebra of $\Uq$, its explicit form~\eqref{expMU} can be derived, up to a sign, from the quadratic functional equation~\eqref{eq:mu}. Indeed, we can write $\mu(u)=\mu_0 u^{-p}g(u)$ for some scalar~$\mu_0$ and positive integer $p$, and $g(u)$ is an invertible power series with coefficients in a~commutative subalgebra of~$\Uq$ with constant term~${g_0=1}$. Using \eqref{eq:qdet}, the relation \eqref{eq:mu} fixes $p=1$ and~\smash{$\mu_0=\pm q^{-\h}$}. Therefore, \smash{$\mu(u)=\pm u^{-1} q^{-\h}g(u)$} where by \eqref{eq:mu} $g(u)$ satisfies
 \begin{gather}
 g(u)g(uq)\xi(u) =1 \qquad \mbox{with}\quad \xi(u)=1-Cu^{-2}q^{-2} + u^{-4}q^{-4}.\label{eq:ggu}
 \end{gather}
 Equating the left-hand side of the first equation in (\ref{eq:ggu}) with the one for $u$ replaced by $uq^{-1}$ is equivalent to $g(uq)\xi(u)=g\bigl(uq^{-1}\bigr)\xi\bigl(uq^{-1}\bigr)$. Inserting $g(u)=\sum_{k=0}^{\infty}g_ku^{-k}$ in the last equation, we get the following three-term recurrence relation
 \[
 g_k=-\frac{1}{(q^{k}-q^{-k})}\bigl(-\bigl(q^{k}-q^{-k+2}\bigr)Cq^{-2} g_{k-2} + \bigl(q^{k}-q^{-k+4}\bigr)q^{-4} g_{k-4}\bigr),
 \]
 with initial conditions $g_0=1$ and $g_{k}=0$ for all $k<0$. It gives $g_1=0$, which implies $g_{2\ell+1}=0$ for all $\ell\geq0$. Thus, $g(u)\in \Uq\big[\big[u^{-2}\big]\big]$ and by assumption $[g_{2k},g_{2\ell}]=0$ for all $\ell,k>0$. As every formal power series can be written in an exponential form using the Bell polynomials,\footnote{Note that \smash{$g_{2k}= \frac{B_k(q^{-1}\Lambda_1,\dots ,k!q^{-k}\Lambda_k)}{k!}$}, where $B_k(x_1,\dots ,x_k)$ are the complete Bell polynomials.} let us define $g(u)=\exp\bigl(\Lambda\bigl(u^{-2}q^{-1}\bigr)\bigr)$ with $\Lambda(u)=\sum_{k=1}\Lambda_k u^{k}$. Now, taking the logarithm of the relation \eqref{eq:ggu} with shifted argument $u \rightarrow uq^{-1}$, we get the relation
 \[ 
 \Lambda\bigl(u^{-2} q\bigr) + \Lambda\bigl( u^{-2} q^{-1}\bigr) =- \log\bigl(1-Cu^{-2} +u^{-4}\bigr).
 \]
 Introduce the central elements $C_k\in Z(\Uq)$ defined by \eqref{genCK} in order to rewrite the right-hand side as the power series
 \smash{$\sum_{k=1}^{\infty} C_k \frac{u^{-2k}}{k}$}. Comparing the coefficients of the power series on both sides of the relation above determines uniquely \smash{$\Lambda_k= \frac{C_k }{(q^k+q^{-k})k}$}, which concludes the proof.
\end{rem}

We now introduce fused R-matrices by analogy with~\eqref{v2Rj1j2} and~\eqref{fused-R-uq}
\begin{align} \label{fusRstj1j2}
 R^{(j_1,j_2)}(u)&=\mathcal{F}^{(j_1)}_{\fu} R_{13}^{(\h,j_2)}\bigl(u q^{-j_1+\h}\bigr) R_{23}^{(j_1-\h,j_2)}\bigl(u q^{\h}\bigr) \mathcal{E}^{(j_1)}_{\fu},
\end{align}
for $j_1 \geq 1$ and where
\begin{align} \label{fusRstraight}
 R^{(\h,j+\h)}(u)&=\mathcal{F}^{(j+\h)}_{\langle 23 \rangle} R_{13}^{(\h,j)}\bigl(u q^{-\h}\bigr) R_{12}^{(\h,\h)}\bigl(u q^{j}\bigr) \mathcal{E}^{(j+\h)}_{\langle 23 \rangle},
\end{align}
with \smash{$R^{(\h,\h)}(u)$} given in the right part of~\eqref{evalRh}, and show that~\eqref{fusRstraight} indeed agrees with the explicit expression for \smash{$R^{(\h,j)}(u)$} given in~\eqref{R-Rqg}.

\begin{lem} \label{lem-Rhj}
 The R-matrices~\eqref{evalR} and the fused R-matrices~\eqref{fusRstj1j2} are related by
 \begin{equation}\label{evalRj1j2}
 \mathcal{R}^{(j_1,j_2)}(u) = f^{(j_1,j_2)}(u) R^{(j_1,j_2)}(u),
 \end{equation}
 where
 \begin{equation} \label{eq:fj1j2}
 f^{(j_1,j_2)}(u) = u^{-4j_1j_2} q^{-2j_1j_2} \exp \left (\sum_{k=1}^\infty \frac {q^{2k} + q^{-2k}}{1+q^{2k}} \lbrack 2j_1 \rbrack_{q^k} \lbrack 2j_2 \rbrack_{q^k} \frac{u^{-2k}}{k}\right),
 \end{equation}
 and~\eqref{fusRstraight} agrees with~\eqref{R-Rqg}.
\end{lem}

\begin{proof}
 Firstly, we prove~\eqref{evalRj1j2} for $j_1=\h$. We show by induction on $j_2$ that
 \begin{equation}\label{evalRhj}
 \mathcal{R}^{(\h,j_2)}(u) = \pi^{\h}( \mu^{(j_2)}(u)) R^{(\h,j_2)}(u),
 \end{equation}
 and we identify \smash{$\pi^\h(\mu^{(j_2)})(u)$} with \smash{$f^{(\h,j_2)}(u)$}. For $j_2=\h$, it is given in~\eqref{evalRh}. Now, assuming~\eqref{evalRhj} holds for a fixed value of $j_2$, we show it holds for $j_2+\h$. Inserting~\eqref{fused-R-uq} and~\eqref{fusRstraight} in~\eqref{evalRhj} for $j_2\rightarrow j_2+\h$ and using~\eqref{exp-mujh}, one finds that the equality indeed holds. Then, using~\eqref{exp-mujh} and~\eqref{pimu} we find that{\samepage
 \begin{equation} \label{step:RF}
 \pi^\h(\mu^{(j_2)})(u) = u^{-2j_2} q^{-j_2} \exp \left (\sum_{k=1}^\infty \frac {q^{2k} + q^{-2k}}{1+q^{2k}} \lbrack 2j_2 \rbrack_{q^k} \frac{u^{-2k}}{k}\right),
 \end{equation}
 and it coincides with~\eqref{eq:fj1j2} for $j_1=\h$.}

 Secondly, we show that~\eqref{fusRstraight} agrees with~\eqref{R-Rqg}. On one hand, using~\eqref{Laxh} it is straightforward to find that (recall the notation in~\eqref{calRj2j1})
 \begin{equation}
 \mathcal{R}_{21}^{(j,\h)}(u) = \mathcal{P}^{(j,\h)} \big\lbrack \bigl( \pi^j \otimes \id \bigr)\bigl( {\bf L}^{(\h)}(u)\bigr) \big\rbrack \mathcal{P}^{(\h,j)}
 \end{equation}
 is proportional to~\eqref{R-Rqg}. On the other hand, from the P-symmetry in~\eqref{Psym}, one has $\smash{\mathcal{R}_{21}^{(j,\h)}(u)} = \smash{\mathcal{R}^{(\h,j)}(u)}$. Therefore, \smash{$\mathcal{R}^{(\h,j)}(u)$} is equally proportional to~\eqref{R-Rqg}, as well as the expression in~\eqref{fusRstraight}, recall~\eqref{evalRhj}. Finally, comparing the matrix entry $(1,1)$ of~\eqref{R-Rqg} and~\eqref{fusRstraight}, one finds that they are equal.

 Thirdly, assuming that $\mathcal{R}^{(j_1,j_2)}(u)$ is proportional to $R^{(j_1,j_2)}(u)$ as in~\eqref{evalRj1j2}, we show that $f^{(j_1,j_2)}(u)$ takes the form~\eqref{eq:fj1j2}.
 Replacing the R-matrices and the fused R-matrices in~\eqref{evalRj1j2} by~\eqref{v2Rj1j2},~\eqref{fusRstj1j2}, and using~\eqref{evalRhj} one gets
 \begin{equation} \label{exp-f}
 f^{(j_1,j_2)}(u) = \pi^{\h}\bigl( \mu^{(j_2)} \bigl(u q^{-j_1+\h}\bigr)\bigr) f^{(j_1-\h,j_2)}(u q^\h )
 = \prod_{k=0}^{2j_1-1} \bigl[\pi^\h \bigl( \mu^{(j_2)}\bigl(u q^{j_1-\h -k}\bigr)\bigr) \bigr],
 \end{equation}
 where we set $f^{(0,j_2)}(u) =1$. Finally, using~\eqref{step:RF} in the latter relation, one finds that $f^{(j_1,j_2)}(u)$ is indeed given by~\eqref{eq:fj1j2} and so the claim follows.
\end{proof}

As a consequence of the above Lemma~\ref{lem-Rhj}, we see that the P-symmetry~\eqref{Psym}
follows from equation~\eqref{Psym} with the fact that \smash{$f^{(j_1,j_2)}(u) = f^{(j_2,j_1)}(u)$} and that $\mu(u)$ is invertible.

\subsubsection{Unitarity properties of fused L-operators}
Later in the text, we will need various relations satisfied by the L-operators and R-matrices. They are obtained from the action of the quantum loop algebra (on tensor products) defined by the opposite coproduct, recall Remark~\ref{rem:oppcop}.
\begin{lem} The following relations hold:
 \begin{gather} \label{v2fused-Lop-uq}
 {\bf L}^{(j+\h)}(u)= \mathcal{F}^{(j+\h)}_{\langle 12 \rangle} {\bf L}_{1}^{(\h)}\bigl(u q^{-j}\bigr){\bf L}_{2}^{(j)}\bigl(u q^{\h}\bigr) \mathcal{E}^{(j+\h)}_{\langle 12 \rangle},\\ \label{v2fused-R-uq}
 \mathcal{R}^{(j_1,j+\h)}(u)= \mathcal{F}^{(j+\h)}_{\langle 23 \rangle } \mathcal{R}_{12}^{(j_1,\h)}(uq^{-j}) \mathcal{R}_{13}^{(j_1,j)}
 \bigl(u q^{\h}\bigr)\mathcal{E}^{(j+\h)}_{\langle 23 \rangle }, \\ \label{v2fused-R-uq12}
 \mathcal{R}^{(j+\h,j_2)}(u)= \mathcal{F}^{(j+\h)}_\fu \mathcal{R}_{23}^{(j,j_2)}\bigl(u q^{-\h}\bigr) \mathcal{R}_{13}^{(\h,j_2)} \bigl(u q^j\bigr)\mathcal{E}^{(j+\h)}_\fu.
 \end{gather}
\end{lem}
\begin{proof}
 Recall the discussion in Remark~\ref{rem:oppcop} about the action on tensor product of evaluation representations given by the opposite coproduct. The same comments also apply to the tensor product of quantum loop modules studied in Section~\ref{sec:formal-ev}. In particular, the intertwining property~\eqref{intert-op} holds also at the level of formal $u$.
 Now, combining~\eqref{intert-op} with $(\id \otimes \Delta^{\rm op}) (\mathfrak{R})= \mathfrak{R}_{12} \mathfrak{R}_{13}$, which is obtained by applying $\mathfrak{p}_{23}$ to~\eqref{univR3}, then~\eqref{v2fused-Lop-uq} is proven similarly to Proposition~\ref{propfusR}. Equation~\eqref{v2fused-R-uq} follows from~\eqref{v2fused-Lop-uq} by specialization of the 1st tensor component to the spin-{$j_1$} representation. The last relation~\eqref{v2fused-R-uq12} is proven similarly to~\eqref{v2Rj1j2} but using~\eqref{intert-op} and $(\Delta^{\rm op} \otimes \id ) (\mathfrak{R})= \mathfrak{R}_{23} \mathfrak{R}_{13}$, coming from the application of $\mathfrak{p}_{12}$ to~\eqref{univR2}, instead of~\eqref{lem-intE} and~\eqref{univR2}.
\end{proof}

Let us note that, using Remark~\ref{rem:invLh}, the L-operator~\eqref{Laxh} satisfies the unitarity property
\begin{gather} \label{unitLaxh}
 \mathcal{L}^{(\h)}(u) \mathcal{L}^{(\h)}\bigl(u^{-1}\bigr) = - \gamma\bigl(u q^{-1}\bigr) \mathbb{I}_{2}.
\end{gather}

\begin{prop} We have the following unitarity properties of the L-operators and R-ma\-tri\-ces:\samepage
 \begin{gather}\label{inverse-Lax}
 \mathcal{L}^{(j)}\bigl(u^{-1}\bigr) \mathcal{L}^{(j)}(u) = \mathcal{L}^{(j)}(u) \mathcal{L}^{(j)}\bigl(u^{-1}\bigr)= \left( \prod_{k=0}^{2j-1} -\gamma\bigl(u q^{-j-\h+k}\bigr) \right)\mathbb{I}_{2j+1},
\\
 R^{(j_1,j_2)}(u) R^{(j_1,j_2)}\bigl(u^{-1}\bigr) = R^{(j_1,j_2)}\bigl(u^{-1}\bigr) R^{(j_1,j_2)}(u) \propto \mathbb{I}_{(2j_1+1)(2j_2+1)}, \label{inverse-R}
 \end{gather}
 for any $j,j_1,j_2\in \h\mathbb{N}_+$ and where the proportionality coefficient is a Laurent polynomial in $u$.
\end{prop}
\begin{proof}
 First, we need another form of the RLL equation~\eqref{univRLLv2}
 \begin{equation} \label{univRLL}
 \mathcal{R}_{12}^{(j_1,j_2)}(u_1/u_2) {\bf L}_1^{(j_1)}(u_1) {\bf L}_2^{(j_2)}(u_2)={\bf L}_2^{(j_2)}(u_2){\bf L}_1^{(j_1)}(u_1)\mathcal{R}_{12}^{(j_1,j_2)}(u_1/u_2),
 \end{equation}
 which is found by applying $\bigl(\mathsf{ev}_{u} \otimes \pi^{j_1}_{v} \otimes \pi^{j_2}_{w}\bigr) \circ \mathfrak{p}_{23}$ to the universal Yang--Baxter equation~\eqref{YB-noparam}, using~\eqref{eval-Psym} and substituting $v \rightarrow u/u_1$, $w \rightarrow u/u_2$, where we also used the P-symmetry property~\eqref{eq:R21}.
 Recall from Lemma~\ref{lem-muj} that ${\bf L}^{(j)}(u)$ is proportional to $\mathcal{L}^{(j)}(u)$, and so we also have
 \begin{gather} \label{RLL}
 \mathcal{R}_{12}^{(j_1,j_2)}(u_1/u_2) \mathcal{L}_1^{(j_1)}(u_1) \mathcal{L}_2^{(j_2)}(u_2)=\mathcal{L}_2^{(j_2)}(u_2)\mathcal{L}_1^{(j_1)}(u_1)\mathcal{R}_{12}^{(j_1,j_2)}(u_1/u_2), \\
 \mathcal{L}^{(j+\h)}(u)= \mathcal{F}^{(j+\h)}_{\langle 12 \rangle} \mathcal{L}_{1}^{(\h)}\bigl(u q^{-j}\bigr)\mathcal{L}_{2}^{(j)}\bigl(u q^{\h}\bigr) \mathcal{E}^{(j+\h)}_{\langle 12 \rangle}, \label{v2fused-L-uq}
 \end{gather}
 where we used~\eqref{v2fused-Lop-uq} for 2nd equality and the fact that all $\mu^{(j)}(u)$ are invertible.
 Recall that the fused L-operators are defined in~\eqref{cal-fused-L-uq}. This expression matches the fusion relation~\eqref{v2fused-L-uq} because, due to Lemma~\ref{lem-muj}, both follow from the fusion relations satisfied by the spin-$j$ L-operators in~\eqref{fused-L-uq} and~\eqref{v2fused-Lop-uq}. Indeed, comparing the two fusion relations written in terms of fused L-operators $\mathcal{L}^{(j)}(u)$, the coefficients \smash{$\mu^{(j)} \bigl(u q^{-\h}\bigr) \mu \bigl(u q^j\bigr)$} and \smash{$\mu^{(j)} \bigl(u q^\h\bigr) \mu \bigl(u q^{-j}\bigr)$} on both sides cancel each other, as they are both equal to \smash{$\mu^{(j+\h)} (u)$}, see~\eqref{exp-mujh}. Now, using~\eqref{cal-fused-L-uq} and~\eqref{v2fused-L-uq} for $j=\h$, we get
 \begin{align}
 \mathcal{L}^{(1)}(u)\mathcal{L}^{(1)}\bigl(u^{-1}\bigr)={}& \mathcal{F}^{(1)}_\fu \mathcal{L}_1^{(\h)}\bigl(u q^{-\h}\bigr)\mathcal{L}_2^{(\h)}\bigl(u q^{\h}\bigr) \mathcal{E}^{(1)}_\fu \mathcal{F}^{(1)}_\fu \mathcal{L}_2^{(\h)}\nonumber\\
&\times\bigl(u^{-1} q^{-\h}\bigr)\mathcal{L}_1^{(\h)}\bigl(u^{-1} q^{\h}\bigr)\mathcal{E}^{(1)}_\fu. \label{invL1p1}
 \end{align}
 Then, we need to use~\eqref{unitLaxh} to show that~\eqref{invL1p1} is proportional to the identity matrix. However, there is an unwanted product $\mathcal{E}^{(1)} \mathcal{F}^{(1)}$ which is removed as follows. We recall first Corollary~\ref{corEHF}, then insert \smash{$\mathcal{H}^{(1)}_\fu \bigl[ \mathcal{H}^{(1)}_\fu\bigr]^{-1} = \mathbb{I}_3$} in the right-hand side of~\eqref{invL1p1} and use~\eqref{EHF1}, which gives
 \begin{align}
\mathcal{L}^{(1)}(u)\mathcal{L}^{(1)}\bigl(u^{-1}\bigr)={}& \mathcal{F}^{(1)}_\fu \mathcal{L}_1^{(\h)}\bigl(u q^{-\h}\bigr)\mathcal{L}_2^{(\h)}\bigl(u q^{\h}\bigr) \mathcal{E}^{(1)}_\fu \mathcal{F}^{(1)}_\fu \mathcal{L}_2^{(\h)}\nonumber \\
&\times\left(\frac{1}{u} q^{-\h}\right)\mathcal{L}_1^{(\h)}\left(\frac{1}{u} q^{\h}\right) R^{(\h,\h)}(q) \mathcal{E}^{(1)}_\fu \bigl[ \mathcal{H}^{(1)}_\fu\bigr]^{-1}\nonumber \\
 ={}&\mathcal{F}^{(1)}_\fu \mathcal{L}_1^{(\h)}\bigl(u q^{-\h}\bigr)\mathcal{L}_2^{(\h)}\bigl(u q^{\h}\bigr) \mathcal{E}^{(1)}_\fu \mathcal{F}^{(1)}_\fu R^{(\h,\h)}(q) \mathcal{L}_1^{(\h)}\left(\frac{1}{u} q^{\h}\right)\nonumber \\
&\times \mathcal{L}_2^{(\h)}\bigl(\frac{1}{u}q^{-\h}\bigr) \mathcal{E}^{(1)}_\fu \bigl[ \mathcal{H}^{(1)}_\fu\bigr]^{-1} \nonumber\\
 ={}&\mathcal{F}^{(1)}_\fu \mathcal{L}_1^{(\h)}\bigl(u q^{-\h}\bigr)\mathcal{L}_2^{(\h)}\bigl(u q^{\h}\bigr)\mathcal{L}_2^{(\h)}\bigl(u^{-1} q^{-\h}\bigr)\mathcal{L}_1^{(\h)}\bigl(u^{-1} q^{\h}\bigr)\mathcal{E}^{(1)}_\fu, \label{L1EF}
 \end{align}
 where we used the RLL equation given in~\eqref{RLL} to get the second line, and the property~\eqref{EHF3}, the RLL equation and~\eqref{EHF1} to get the third line. Finally, using~\eqref{unitLaxh}, we get
 \begin{gather*}
 \mathcal{L}^{(1)}\bigl(u^{-1}\bigr) \mathcal{L}^{(1)}(u) = \mathcal{L}^{(1)}(u) \mathcal{L}^{(1)}\bigl(u^{-1}\bigr)= \gamma\bigl(u q^{-\h}\bigr) \gamma\bigl(u q^{-\frac 32}\bigr) \mathbb{I}_3.
 \end{gather*}
 More generally, by induction we get the unitarity property for any spin-$j$ as in~\eqref{inverse-Lax}.
 Because the evaluation representation of the L-operators gives R-matrices, we also have~\eqref{inverse-R}.
\end{proof}

\subsection[Spin-j K-operators]{Spin-$\boldsymbol{j}$ K-operators}\label{sec:spin-j-K}
Here, we introduce analogs of spin-$j$ L-operators, that we call spin-$j$ K-operators, and show that they satisfy the reflection equation~\eqref{genREAj} together with fusion and reduction relations.

As always, we consider the case $H=\Loop$ as introduced in Section~\ref{exUqAq}, however,
without specifying its comodule algebra $B$.
Assume that a universal K-matrix $\mathfrak{K}$ exists for a~choice of~$B$ and the twist pair $(\psi,J)=(\eta,1\otimes1)$
where $\eta$ is defined in~\eqref{autrho}.
Similarly to the situation with the universal R-matrix $\mathfrak{R}$ for $\Loop$, by this assumption we actually mean that there exists a ``root vector" basis in $B$ such that $\mathfrak{K}$ can be formally written as a possibly infinite sum of products of the root vectors, i.e., each term in this sum is an element in $B \otimes \Loop$. In~particular, a basis of root vectors for $B=O_q$ was constructed in~\cite{BK17} and for $B=\mathcal{A}_q$ in~\cite{Ter21c}.
With this assumption, we
can consider evaluation of $\mathfrak{K}$ on the second tensor component.

\begin{defn}\label{defbK} For $j \in \frac{1}{2} \mathbb{N}$, introduce
 \begin{equation} \label{evalK}
 {\bf K}^{(j)}(u)= \bigl(\id \otimes \pi_{u^{-1}}^{j}\bigr)(\mathfrak{K}) \in B \bigl(\bigl(u^{-1}\bigr)\bigr)\otimes \End\bigl(\mathbb{C}^{2j+1}\bigr).
 \end{equation}
 We call ${\bf K}^{(j)}(u)$ the spin-$j$ K-operator.
\end{defn}
Similarly to the case of L-operators, recall Definition~\ref{defL} and comments below it, we consider here $u$ as a formal variable. In contrast to the L-operators, we do not have examples of explicit evaluations of $\mathfrak{K}$, and that is why in the above definition we assumed that the entries of K-operators are in the ring of formal Laurent series $B \bigl(\bigl(u^{-1}\bigr)\bigr)$ instead of power series $B \big[\big[u^{-1}\big]\big]$. However, in the case of $B={\mathcal A}_q$ we will see by Conjecture~\ref{conj1} in Section~\ref{sec6} that \smash{${\bf K}^{(\h)}(u)$} is~in $B\big[\big[u^{-1}\big]\big] \otimes \End\bigl(\mathbb{C}^{2}\bigr)$, compare also with the similar Appel--Vlaar's result~\cite{AV22} on 1-leg universal K-matrices for $B=O_q$.

By Proposition \ref{propRE}, the universal K-matrix $\mathfrak{K}$ satisfies the $\psi$-twisted reflection~\eqref{psiRE}.
\!We~now show that the evaluation of this equation leads to the reflection equation~\eqref{genREAj}.
To~do~so, we need evaluation of the $\psi$-twisted universal R-matrices~\eqref{Rpsi}. Firstly, note that~${\psi=\eta}$ from~\eqref{autrho} is such that,
recall the definition in~\eqref{eval},
\begin{equation} \label{actpsi}
 \mathsf{ev}_u \circ \eta = \mathsf{ev}_{u^{-1}}.
\end{equation}
Then the evaluations of the $\psi$-twisted universal R-matrices read
\begin{gather}
 \bigl(\pi_{u_1}^{j_1} \otimes \pi_{u_2}^{j_2}\bigr) ( \mathfrak{R}^{\eta} ) = \mathcal{R}^{(j_1,j_2)}(1/(u_1 u_2)),\label{evR1}
 \\ \label{evR2}
 \bigl( \pi_{u_1}^{j_1} \otimes \pi_{u_2}^{j_2} \bigr) ( (\mathfrak{R}^\eta)_{21} ) = \mathcal{R}_{21}^{(j_2,j_1)}(1/(u_1 u_2)), \\
 \bigl( \pi_{u_1}^{j_1} \otimes \pi_{u_2}^{j_2} \bigr) ( (\mathfrak{R}_{21})^\eta ) = \mathcal{R}_{21}^{(j_2,j_1)}(u_1u_2), \\ \label{evR4}
 \bigl( \pi_{u_1}^{j_1} \otimes \pi_{u_2}^{j_2} \bigr)( \mathfrak{R}^{\eta \eta}_{21} ) = \mathcal{R}_{21}^{(j_2,j_1)}(u_1 /u_2), \\ \label{evR5}
 \bigl( \pi_{u_1}^{j_1} \otimes \pi_{u_2}^{j_2} \bigr)( \mathfrak{R}^{\eta \eta}) = \mathcal{R}^{(j_1,j_2)}(u_2/u_1),
\end{gather}
where \smash{$\mathcal{R}_{21}^{(j_2,j_1)}(u)$} is defined in~\eqref{calRj2j1}.
Applying $\mathfrak{p}_{23}$ to~\eqref{psiRE} leads to the universal reflection equation~$
\mathfrak{K}_{13} (\mathfrak{R}^\eta)_{23} \mathfrak{K}_{12} \mathfrak{R}_{32} = \mathfrak{R}_{23}^{\eta \eta} \mathfrak{K}_{12}( \mathfrak{R}^\eta)_{32} \mathfrak{K}_{13}
$. Finally, applying \smash{$\bigl(\id \otimes \pi_{u_1^{-1}}^{j_1} \otimes \pi_{u_2^{-1}}^{j_2} \bigr)$} to the latter equation using~\eqref{evR1}--\eqref{evR5}, it follows that the K-operator~\eqref{evalK} satisfies the reflection equation for any values of $j_1$ and $j_2$\footnote{As for the L-operator \smash{${\bf L}^{(j)}(u)$},
 the K-operator should be written as \smash{${\bf K}^{(j)}_{0i}(u)$} but here we use standard notation~\smash{${\bf K}_i^{(j)}(u)$} where we omit the label $0$ corresponding to $B$.}
\begin{gather}
 \mathcal{R}_{12}^{(j_1,j_2)}(u_1/u_2) {\bf K}_1^{(j_1)}(u_1) \mathcal{R}_{21}^{(j_2,j_1)}(u_1 u_2) {\bf K}_2^{(j_2)}(u_2) \nonumber\\
 \qquad= {\bf K}_2^{(j_2)}(u_2)\mathcal{R}_{12}^{(j_1,j_2)}(u_1 u_2){\bf K}_1^{(j_1)}(u_1)\mathcal{R}_{21}^{(j_2,j_1)}(u_1/u_2), \label{genevalpsi}
\end{gather}

Due to the P-symmetry~\eqref{Psym}, the relations~\eqref{evR2}--\eqref{evR4} become
\begin{gather*}
 \bigl( \pi_{u_1}^{j_1} \otimes \pi_{u_2}^{j_2} \bigr) ( (\mathfrak{R}^\eta)_{21} ) = \mathcal{R}^{(j_1,j_2)}(1/(u_1 u_2)), \qquad
 \bigl( \pi_{u_1}^{j_1} \otimes \pi_{u_2}^{j_2} \bigr) ( (\mathfrak{R}_{21})^\eta ) = \mathcal{R}^{(j_1,j_2)}(u_1u_2), \\
 \bigl( \pi_{u_1}^{j_1} \otimes \pi_{u_2}^{j_2} \bigr)( \mathfrak{R}^{\eta \eta}_{21} ) = \mathcal{R}^{(j_1,j_2)}(u_1 /u_2),
\end{gather*}
and the reflection equation~\eqref{genevalpsi} becomes the standard reflection equation
\begin{gather}
 \mathcal{R}^{(j_1,j_2)}(u_1/u_2) {\bf K}_1^{(j_1)}(u_1) \mathcal{R}^{(j_1,j_2)}(u_1 u_2) {\bf K}_2^{(j_2)}(u_2)\nonumber \\
 \qquad = {\bf K}_2^{(j_2)}(u_2)\mathcal{R}^{(j_1,j_2)}(u_1 u_2){\bf K}_1^{(j_1)}(u_1)\mathcal{R}^{(j_1,j_2)}(u_1/u_2), \label{evalpsi}
\end{gather}
where we can also replace $\mathcal{R}^{(j_1,j_2)}(u)$ by $R^{(j_1,j_2)}(u)$ due to Lemma~\ref{lem-Rhj}.

\subsubsection[K-operators and fusion (h,j)rightarrow (j+h)]{K-operators and fusion $\boldsymbol{(\h,j)\rightarrow (j+\h)}$}
We follow here the same approach used for L-operators based on sub-representations in the tensor product of formal evaluation representations of $\Loop$, now considering~\eqref{univK3} instead of~\eqref{univR3}.
Recall the intertwining operator \smash{${\mathcal{E}}^{(j+\h)}$} is fixed by Lemma~\eqref{lem-E} and its pseudo-inverse~\smash{$\mathcal{F}^{(j+\h)}$} is given in~\eqref{exprF} with~\eqref{Fp1} and~\eqref{Fp2}.
We now obtain our first main result.
\begin{prop}\label{propfusK}
 The K-operators \eqref{evalK} satisfy for $j \in \h \mathbb{N}$
 \begin{gather} \label{fusedevalK}
 {\bf K}^{(j+\h)}(u) = \mathcal{F}^{(j+\h)}_\fu {\bf K}_2^{(j)}\bigl(u q^{-\h}\bigr) \mathcal{R}^{(\h,j)}\bigl(u^2 q^{j-\h}\bigr) {\bf K}_1^{(\h)}\bigl(u q^{j}\bigr) \mathcal{E}^{(j+\h)}_\fu.
 \end{gather}
\end{prop}

\begin{proof}
 By definition of the K-operator, we have
 \begin{equation*}
 {\bf K}^{(j+\h)}(u)=\bigl(\id \otimes \pi_{u^{-1}}^{j+\h}\bigr) ( \mathfrak{K}).
 \end{equation*}
 Using the pseudo-inverse property $ \mathcal{F}^{(j+\h)} \mathcal{E}^{(j+\h)} = \mathbb{I}_{2j+2}$, we get
 \begin{align*}
 \bigl(\id \otimes \pi_{u^{-1}}^{j+\h}\bigr) (\mathfrak{K}) &=(1 \otimes \mathcal{F}^{(j+\h)} \mathcal{E}^{(j+\h)} ) \bigl(\id \otimes \pi_{u^{-1}}^{j+\h}\bigr)( \mathfrak{K})\\
 &= \bigl(1 \otimes \mathcal{F}^{(j+\h)} \bigr) \big[\bigl(\id \otimes \pi_{u^{-1} q^{-j}}^{\h}\otimes \pi_{u^{-1} q^{\h}}^{j}\bigr) \circ ( \id \otimes \Delta) (\mathfrak{K}) \big] \bigl( 1 \otimes \mathcal{E}^{(j+\h)}\bigr) \\
 &= \bigl(1 \otimes \mathcal{F}^{(j+\h)} \bigr) \big[ \bigl(\id \otimes \pi_{u^{-1} q^{-j}}^{\h} \otimes \pi_{u^{-1} q^{\h}}^{j}\bigr) ( \mathfrak{K}_{13} \mathfrak{R}_{23}^\eta \mathfrak{K}_{12})\big]\bigl(1 \otimes \mathcal{E}^{(j+\h)}\bigr).
 \end{align*}
 The second equality is obtained using the intertwining property~\eqref{lem-intE}
 \[
 \bigl(1 \otimes \mathcal{E}^{(j+\h)} \bigr)\bigl(\id \otimes \pi_{u^{-1}}^{j+\h}\bigr) (\mathfrak{K}) = \big[\bigl(\id \otimes \pi_{u^{-1}q^{-j}}^{\h}\otimes \pi_{u^{-1} q^{\h}}^{j}\bigr) \circ ( \id \otimes \Delta) (\mathfrak{K}) \big] \bigl( 1 \otimes \mathcal{E}^{(j+\h)}\bigr).
 \]
 Then, the third equality is due to~\eqref{univK3} and finally, from the definition of the K-operator~\eqref{evalK} and the evaluation of the twisted universal R-matrix~\eqref{evR1}, the claim follows.
\end{proof}

We see from the above Proposition~\ref{propfusK} that ${\bf K}^{(j)}(u)$ is a formal power series in $u^{-1}$, i.e., it is in~{$B\big[\big[u^{-1}\big]\big]\otimes \End\bigl(\mathbb{C}^{2j+1}\bigr)$}, whenever \smash{${\bf K}^{(\h)}(u)$} is so.

\begin{rem} Similarly to Remark~\ref{uniFR}, it is clear from the proof of Proposition~\ref{propfusK} that the K-operator does not depend on the choice of $\mathcal{F}^{(j)}$, as long as it satisfies $\mathcal{F}^{(j)}\mathcal{E}^{(j)} = \mathbb{I}_{2j+1}$.
\end{rem}
\subsubsection[K-operators and reduction (h,j)rightarrow (j-h)]{K-operators and reduction $\boldsymbol{(\h,j)\rightarrow (j-\h)}$}
The proof of the following proposition is done similarly to the proof of the reduction relation~\eqref{fused-bar-L-uq} for the L-operators, thus we skip it.
Recall the intertwining operator $\bar{\mathcal{E}}^{(j-\h)}$ is fixed by Lemma~\ref{lem-barE} and its pseudo-inverse \smash{$\bar{\mathcal{F}}^{(j-\h)}$} is given in~\eqref{exprbF} with~\eqref{coefbar}.
\begin{prop} \label{propfusedbarK}
 The K-operators \eqref{evalK} satisfy for $j \in \h \mathbb{N}_+$
 \[
 {{\bf K}}^{(j-\h)}(u) = \bar{\mathcal{F}}_\fu^{(j-\h)} {\bf K}_2^{(j)}\bigl(uq^{-\h}\bigr) \mathcal{R}^{(\h,j)}\bigl(u^2 q^{-j-\tha}\bigr) {\bf K}_1^{(\h)}\bigl(u q^{-j-1}\bigr) \bar{\mathcal{E}}_\fu^{(j-\h)}.
 \]
\end{prop}
Recall that we assumed that the universal K-matrix exists for a given choice of $B$ and the twist pair $(\eta, 1\otimes 1)$. Therefore, the K-operator for a given spin is unique,
that is, similarly to the case of the L-operator, we obtain the same operator ${\bf K}^{(j)}(u)$ either using the fusion for~${(\h,j-\h) \rightarrow j}$ or using the reduction $(\h, j+\h) \rightarrow j$.
\begin{rem} \label{Dop-K} Consider the opposite coproduct $\Delta^{\rm op}=\mathfrak{p}\circ\Delta$ with the definition~\eqref{affinecoprodefk}. It follows from~\eqref{univK3}
 \begin{equation} \label{eq:opdK}
 (\id \otimes \Delta^{\rm op} )(\mathfrak{K}) = \mathfrak{K}_{12} \bigl(\mathfrak{R}^\psi\bigr)_{32} \mathfrak{K}_{13}.
 \end{equation}
 Recall that we obtained for $\Delta^{\rm op}$ an intertwining relation in~\eqref{intert-op} where \smash{$\mathcal{E}^{(j+\h)}$} is fixed as in~\eqref{E-proj}, see Remark~\ref{rem:oppcop}. Thus, we also take \smash{$\mathcal{F}^{(j+\h)}$} as defined in~\eqref{Fp1},~\eqref{Fp2}. Then, using~\eqref{intert-op} and~\eqref{eq:opdK}, we obtain a new fusion relation for any $j \in\h \mathbb{N}$
 \begin{equation}\label{fusedevalKv2}
 {\bf K}^{(j+\h)}(u) = \mathcal{F}^{(j+\h)}_\fu {\bf K}_1^{(\h)}\bigl(u q^{-j}\bigr) \mathcal{R}^{(\h,j)}\bigl(u^2 q^{-j+\h}\bigr) {\bf K}_2^{(j)}\bigl(u q^{\h}\bigr) \mathcal{E}^{(j+\h)}_\fu.
 \end{equation}
 Similarly, we also have the intertwining relation for $\Delta^{\rm op}$ given in~\eqref{specop-bar}
 with \smash{$\bar{\mathcal{E}}^{(j-\h)}$} fixed as in~\eqref{barE-proj}, see Remark~\ref{rem:oppcop}, and we take \smash{$\bar{\mathcal{F}}^{(j-\h)}$} as defined in~\eqref{exprbF} with~\eqref{coefbar}. Then, using~\eqref{eq:opdK} and~\eqref{intopEbar}, we obtain a new reduction relation for any $j\in \h \mathbb{N}_+$
 \begin{equation}\label{fusedevalbarKv2}
 {\bf K}^{(j-\h)}(u)= \bar{\mathcal{F}}_\fu^{(j-\h)} {\bf K}_1^{(\h)}\bigl(u q^{j+1}\bigr) \mathcal{R}^{(\h,j)}\bigl(u^2 q^{j+\tha}\bigr){\bf K}_2^{(j)}\bigl(u q^{\h}\bigr) \bar{\mathcal{E}}_\fu^{(j-\h)}.
 \end{equation}
\end{rem}

\subsection{Comodule algebra structure using K-operators} \label{sec:comod}
Given the Hopf algebra $H= \Loop$, it is known that the coproduct, antipode and counit can be formulated solely in terms of so-called Ding--Frenkel L-operators~\cite{DF93}
\begin{gather} \label{defLpm}
 \mathbf{L}^+(u)= \bigl( \id \otimes \pi_{u^{-1}}^{\h}\bigr)(\mathfrak{R}), \qquad \mathbf{L}^-(u)= \big\lbrack \bigl(\id \otimes \pi_{u^{-1}}^{\h}\bigr) (\mathfrak{R}_{21}) \big\rbrack^{-1}.
\end{gather}
They are computed in Appendix~\ref{appC}, see~\eqref{app:expLP} and~\eqref{app:expLM}.
Then, one finds that ${\bf L}^\pm(u)$ are formal power series in $u^{\mp1}$, i.e., ${\bf L}^{\pm}(u)$ are in $\Loop\big[\big[ u^{\mp1}\big]\big] \otimes \End\bigl(\mathbb{C}^2\bigr) $.
The modes of the entries of ${\bf L}^\pm(u)$ generate $\Loop$ and they satisfy the Yang--Baxter algebra relations
\begin{align} \label{RLpm1}
 R^{(\h,\h)}(u/v) \mathbf{L}^\pm_1(u) \mathbf{L}^\pm_2(v) &= \mathbf{L}^\pm_2(v) \mathbf{L}^\pm_1(u) R^{(\h,\h)}(u/v), \\ \label{RLpm2}
 R^{(\h,\h)}(u/v) \mathbf{L}^\pm_1(u) \mathbf{L}^\mp_2(v) &= \mathbf{L}^\mp_2(v) \mathbf{L}^\pm_1(u) R^{(\h,\h)}(u/v),
\end{align}
with \smash{$R^{(\h,\h)}(u)$} defined in~\eqref{evalRh}.
These four relations are obtained by applying \smash{$\bigl(\id \otimes \pi_{u^{-1}}^\h \otimes \pi_{v^{-1}}^\h\bigr)$} to the following four equations
\begin{gather*}
\mathfrak{p}_{23}\eqref{YB-noparam}, \qquad \mathfrak{R}_{31}^{-1} \mathfrak{R}_{21}^{-1} \lbrack \mathfrak{p}_{13}\eqref{YB-noparam}\rbrack \mathfrak{R}_{21}^{-1} \mathfrak{R}_{31}^{-1}, \\
 \mathfrak{R}_{31}^{-1} \lbrack \mathfrak{p}_{23} \circ \mathfrak{p}_{12}\eqref{YB-noparam}\rbrack \mathfrak{R}_{31}^{-1}, \qquad
 \mathfrak{R}_{23}^{-1} \mathfrak{R}_{21}^{-1} \lbrack \mathfrak{p}_{12}\eqref{YB-noparam}\rbrack\mathfrak{R}_{21}^{-1}\mathfrak{R}_{23}^{-1},
\end{gather*}
and using~\eqref{defLpm}, \eqref{evalR} with~\eqref{evalRh} and~\eqref{Psym}, while the unitarity property~\eqref{inverse-R} is also used for the last one.
The coproduct, antipode and counit of $\Loop$ are respectively given by, recall~\eqref{eq:conv-square}
\begin{gather} \label{coprodLpm}
 (\Delta \otimes \id)\bigl( {\bf L}^{\pm}(u)\bigr) = \bigl( {\bf L}^{\pm}(u) \bigr)_{[1]} \bigl( {\bf L}^{\pm}(u) \bigr)_{[2]}, \\ \label{antipLpm}
 (S \otimes \id)\bigl( {\bf L}^{\pm}(u)\bigr) = \bigl( {\bf L}^{\pm}(u)\bigr)^{-1}, \\ \label{epsLpm}
 (\epsilon \otimes \id)\bigl( {\bf L }^\pm (u)\bigr) = 1.
\end{gather}
These relations are easily understood using \eqref{univR2} and~\eqref{SR}--\eqref{epsR}: the relation \eqref{coprodLpm} is obtained by applying \smash{$\bigl(\id \otimes \id \otimes \pi_{u^{-1}}^\h\bigr)$} on~\eqref{univR2}, and similarly for ${\bf L}^{-}(u)$ using $\bigl(\smash{\pi_{u^{-1}}^\h \otimes \id \otimes \id}\bigr) \circ\allowbreak ( \id \otimes \Delta)\bigl(\mathfrak{R}^{-1}\bigr)$.
The relations \eqref{antipLpm}, \eqref{epsLpm} follow immediately from~\eqref{SR},~\eqref{epsR}.

Consider the subalgebra in $B$ generated by the matrix entries of the K-operator \smash{${\bf K}^{(\h)}(u)$}. They satisfy the reflection equation~\eqref{evalpsi} for $j_1=j_2=\h$. Similarly to the coproduct of~$\Loop$ discussed above, the coaction for this subalgebra can be expressed in terms of L- and K-operators.
\begin{prop} \label{prop:coac} The coaction map $\delta\colon B \rightarrow B \otimes \Loop$ restricted to the subalgebra generated by the matrix entries of \smash{${\bf K}^{(\h)}(u)$} is such that
 \begin{equation} \label{coacDK}
 (\delta \otimes \id)\bigl( {\bf K}^{(\h)}(u)\bigr)= \bigl( \lbrack {\bf L}^{-}\bigl(u^{-1}\bigr) \rbrack^{-1} \bigr)_{[2]} \bigl( {\bf K}^{(\h)}(u) \bigr)_{[1]} \bigl( {\bf L}^+(u) \bigr)_{[2]}.
 \end{equation}
\end{prop}
\begin{proof} From the fundamental axiom~\eqref{univK2}, the left-hand side of~\eqref{coacDK} can be written as
 \begin{equation*}
 \bigl(\id \otimes \id \otimes \pi_{u^{-1}}^{\h}\bigr) \circ ( \delta \otimes \id) ( \mathfrak{K}) = \bigl(\id \otimes \id \otimes \pi_{u^{-1}}^{\h}\bigr) \circ ( ( \mathfrak{R}^\eta )_{32} \mathfrak{K}_{13} \mathfrak{R}_{23}),
 \end{equation*}
 where $( \mathfrak{R}^\eta )_{32} = (\id \otimes \id \otimes \eta)(\mathfrak{R}_{32})$. Then, using $\pi_{u^{-1}}^j \circ \eta = \pi_u^j$ together with the definition of the K-operator and ${\bf L}^\pm(u)$ given respectively in~\eqref{evalK},~\eqref{defLpm}, the claim follows.
\end{proof}

In the case when $B$ is generated by the matrix entries of the K-operator \smash{${\bf K}^{(\h)}(u)$}, equation~\eqref{coacDK} expresses the coaction map for $B$ solely in terms of L- and K-operators. This is the case when~${B= \mathcal{A}_q}$ and it will be discussed in Section \ref{sec6}.

Let us now introduce a formal generalization of the evaluated coaction~\eqref{defevco}
\begin{equation}\label{formal_evco}
 \delta_w = (\id \otimes \normalfont{\mathsf{ev}}_w) \circ \delta\colon \ B \rightarrow B \otimes \Uq\big[w^{\pm1}\big]
\end{equation}
with a formal variable $w$. We will nevertheless call $\delta_w$ \textit{the evaluated coaction} for brevity.
We now consider the evaluated coaction applied to the matrix elements of the spin-$j$ K-operator, which is obtained by taking the image of~\eqref{univK2} under the formal evaluation map $\bigl( \id \otimes \textsf{ev}_w \otimes \pi_{u^{-1}}^{j} \bigr)$ using~\eqref{evalL},~\eqref{evR21},~\eqref{actpsi}, and is given by
\begin{equation}
 (\delta_{w} \otimes \id) \bigl( {\bf K}^{(j)}(u)\bigr) = \bigl( {\bf L}^{(j)}(u/w) \bigr)_{[2]} \bigl( {\bf K}^{(j)} (u) \bigr)_{[1]} \bigl( {\bf L}^{(j)}(u w) \bigr)_{[2]}. \label{delwK}
\end{equation}
On the other hand, the action of $\bigl(\id \otimes \pi^j_{v^{-1}}\bigr)$ on~\eqref{univK1} gives
\begin{equation} \label{evalbfK1}
 {\bf K}^{(j)}(v) \bigl(\id \otimes \pi^{j} \bigr)\big[\delta_{v^{-1}}(b)\big] = \bigl(\id \otimes \pi^{j} \bigr)[\delta_{v}(b)]{\bf K}^{(j)}(v).
\end{equation}
We call this equation \textit{the twisted intertwining relation} for ${\bf K}^{(j)}(u)$.

\section[Fused K-operators for A\_q]{Fused K-operators for $\boldsymbol{{\mathcal A}_q}$}\label{sec5}
In this section, we consider the comodule algebra $B= \mathcal{A}_q$ and related `fused' K-operators. Contrary to the previous section, here we do not assume the existence of a universal K-matrix. Instead, we introduce the fundamental K-operator for $B= \mathcal{A}_q$ and recall the corresponding Faddeev--Reshetikhin--Taktadjan type presentation following \cite{BasBel,BSh1}. Then, in Section~\ref{sec:fusKcal}, fused K-operators $\mathcal{K}^{(j)}(u)$ built from the fundamental K-operator by analogy with \eqref{fusedevalKv2} are shown to satisfy the reflection equation~\eqref{evalpsi} for all $j \in \frac {1}{2} \mathbb{N}_+$ where ${\bf K}^{(j)}(u)$ is replaced by $\mathcal{K}^{(j)}(u)$. This is the main result in this section. We also establish the unitarity and invertibility properties of~$\mathcal{K}^{(j)}(u)$ in Section~\ref{sec:unit-inv}, and examples of the fused K-operators are derived explicitly for small values of $j$ in Section~\ref{sec5.3}. In preparation to the discussion in the next section, in Section~\ref{sec:evalcoac} we calculate the evaluated coaction for $\mathcal{A}_q$ and also establish in Section~\ref{sec:twisted-rel-fusedK} the twisted intertwining relations for the fused K-operators which are similar to~\eqref{evalbfK1}.
\subsection[The fundamental K-operator for A\_q]{The fundamental K-operator for $\boldsymbol{\mathcal{A}_q}$} \label{sec:fusedAq}
An alternative presentation for $\mathcal{A}_q$ besides Definition~\ref{thm:m1com} is known, which takes the form
of a~reflection algebra~\cite{BSh1} that is recalled below.
Note that part of the material in this subsection is taken from \cite{BasBel,BSh1,Ter21}.
Let \smash{$R^{(\frac{1}{2},\frac{1}{2})}(u)$} be the symmetric $R$-matrix~\eqref{evalRh}
which satisfies the quantum Yang--Baxter equation~\eqref{YBj1j2} with the substitution \smash{$\mathcal{R}^{(j_k,j_\ell)}(u) \rightarrow R^{(\h,\h)}(u)$}.
We now introduce the K-operator that provides the \textit{reflection algebra} presentation of ${\mathcal A}_q$, with the parametrization from~\eqref{rho}.
\begin{thm}[\cite{BSh1}]\label{def:Aq0} $\mathcal{A}_q$ admits a presentation in the form of a reflection algebra. Introduce the generating functions
 \begin{alignat}{3}
 & {\cW}_+(u)=\sum_{k\in {\mathbb N}}{\normalfont \tW}_{-k}U^{-k-1}, \qquad&& {\cW}_-(u)=\sum_{k\in {\mathbb N}}{\normalfont \tW}_{k+1}U^{-k-1},&\label{c1}\\
& {\cG}_+(u)=\sum_{k\in {\mathbb N}}{\normalfont \tG}_{k+1}U^{-k-1}, \qquad&& {\cG}_-(u)=\sum_{k\in {\mathbb N}}{\normalfont \tilde{{\tG}}_{k+1}}U^{-k-1},&\label{c2}
 \end{alignat}
 where the shorthand notation $U=\bigl(qu^2+q^{-1}u^{-2}\bigr)/\bigl(q+q^{-1}\bigr)$ is used. The defining relations are given by
 \begin{gather}
 R^{(\frac{1}{2},\frac{1}{2})}(u/v) {\mathcal K}_1^{(\frac{1}{2})}(u) R^{(\frac{1}{2},\frac{1}{2})}(uv) {\mathcal K}_2^{(\frac{1}{2})}(v) = {\mathcal K}_2^{(\frac{1}{2})}(v) R^{(\frac{1}{2},\frac{1}{2})}(uv) {\mathcal K}_1^{(\frac{1}{2})}(u) R^{(\frac{1}{2},\frac{1}{2})}(u/v)
 \label{RE}
 \end{gather}
 with the R-matrix from~{\rm\eqref{evalRh}} and
 \begin{gather}
 {\mathcal K}^{(\frac{1}{2})}(u)=\!\begin{pmatrix}
 uq \cW_+(u)-u^{-1}q^{-1}\cW_-(u) &\dfrac{1}{k_-(q+q^{-1})}\cG_+(u)+\dfrac{k_+(q+q^{-1})}{q-q^{-1}} \\
 \dfrac{1}{k_+(q+q^{-1})}\cG_-(u) +\dfrac{k_-(q+q^{-1})}{q-q^{-1}}& uq \cW_-(u) -u^{-1}q^{-1}\cW_+(u)
 \end{pmatrix}\!,\label{K-Aq}\!\!\!
 \end{gather}
 and with the identification~\eqref{rho}.
\end{thm}

We call \smash{$\mathcal{K}^{(\h)}(u)$} \textit{the fundamental K-operator} for $\mathcal{A}_q$.
We note that $U^{-1}$ can be written as a~power series in $u^{-2}$ with the following expansion:
\[
 U^{-1} = \bigl(1+q^{-2}\bigr) u^{-2} \sum_{\ell =0}^{\infty} \bigl(-u^{-4} q^{-2}\bigr)^{\ell} .
\]
Thus, the generating functions $\mathcal{W}_\pm(u)$, $\mathcal{G}_\pm(u)$ start with $u^{-2}$. Consequently, the leading term of the diagonal entries of the K-operator in~\eqref{K-Aq} is at $u^{-1}$, while the leading term of the off-diagonal entries is at $u^0$. Therefore, \smash{$\mathcal{K}^{(\h)}(u)$} is in $ \mathcal{A}_q\big[\big[u^{-1}\big]\big]\otimes \End\bigl(\mathbb{C}^2\bigr)$.

When we write ${\cW}_\pm\bigl(u^{-1}\bigr)$ and ${\cG}_\pm\bigl(u^{-1}\bigr)$ it means replacing $u$ by $u^{-1}$ in $U$ and developing in power series in $u^{-2}$, this way we get{\samepage
\begin{equation}\label{eq:u-inv-convention}
 {\cW}_\pm\bigl(u^{-1}\bigr)= {\cW}_\pm\bigl(u q^{-1}\bigr)\qquad \text{and} \qquad
 {\cG}_\pm\bigl(u^{-1}\bigr)= {\cG}_\pm\bigl(u q^{-1}\bigr).
\end{equation}
In particular, this shows that \smash{$\mathcal{K}^{(\h)}\bigl(u^{-1}\bigr)$} is equally in $ \mathcal{A}_q\big[\big[u^{-1}\big]\big]\otimes \End\bigl(\mathbb{C}^2\bigr)$.}

Explicitly, in terms of the generating functions \eqref{c1}, \eqref{c2} the defining relations~\eqref{RE} read%
\begin{gather}
 [ \cW_\pm(u),\cW_\pm(v) ]=0, \qquad
 [ \cW_+(u),\cW_-(v) ] + [ \cW_-(u),\cW_+(v) ] =0,\label{eq2-Aq}
\\
 [ \cG_\epsilon(u),\cW_\pm(v) ] + [ \cW_\pm(u),\cG_\epsilon(v) ] =0, \qquad \epsilon=\pm,\label{eq3-Aq} \\
 [ \cG_\pm(u),\cG_\pm(v) ] =0,\qquad
[ \cG_+(u),\cG_-(v) ] + [ \cG_-(u), \cG_+(v) ]=0, \label{eq5-Aq}
\\
(U-V) [ \cW_\pm(u),\cW_\mp(v) ]= \frac{\bigl(q-q^{-1}\bigr)}{\rho\bigl(q+q^{-1}\bigr)} ( \cG_\pm(u) \cG_\mp(v)-\cG_\pm(v) \cG_\mp(u) )\nonumber\\
\phantom{(U-V) [ \cW_\pm(u),\cW_\mp(v) ]= }{} + \frac{1}{\bigl(q+q^{-1}\bigr)} ( \cG_\pm(u) -\cG_\mp(u)+\cG_\mp(v)-\cG_\pm(v) ), \\
 \cW_\pm(u)\cW_\pm(v)-\cW_\mp(u)\cW_\mp(v)+\frac{1}{\rho\bigl(q^2-q^{-2}\bigr)} [ \cG_\pm(u),\cG_\mp(v)] \nonumber\\
 \qquad+ \frac{1- UV}{U-V} ( \cW_\pm(u)\cW_\mp(v)-\cW_\pm(v)\cW_\mp(u)) =0, \\
U [ \cG_\mp(v), \cW_\pm(u) ]_q -V [ \cG_\mp(u),\cW_\pm(v) ]_q -\bigl(q-q^{-1}\bigr) ( \cW_\mp(u)\cG_\mp(v) - \cW_\mp(v) \cG_\mp(u) )\nonumber\\
 \qquad+ \rho ( U \cW_\pm(u) - V \cW_\pm(v) - \cW_\mp(u)+\cW_\mp(v) )=0,\label{eq8-Aq} \\
 U [ \cW_\mp(u),\cG_\mp(v) ]_q - V [ \cW_\mp(v),\cG_\mp(u) ]_q -\bigl(q-q^{-1}\bigr) (\cW_\pm(u) \cG_\mp(v) - \cW_\pm(v) \cG_\mp(u) ) \nonumber\\
 \qquad+ \rho ( U \cW_\mp(u)-V\cW_\mp(v) - \cW_\pm(u) +\cW_\pm(v) ) =0.\label{eq9-Aq}
\end{gather}
We note that the corresponding relations for the modes ${\normalfont \tW}_0$, ${\normalfont \tW}_1$, $\lbrace {\normalfont \tG}_{k+1} \rbrace_{k \in \mathbb N}$ give the compact presentation from Definition~\ref{thm:m1com}. The modes $\{ {\normalfont \tW}_{-k}$, ${\normalfont \tW}_{k+1}$, ${\normalfont \tG}_{k+1}$, $\tilde{{\normalfont \tG}}_{k+1} \}_{k \in \mathbb{N}}$ give actually
PBW type generators, called \textit{alternating}.
The expressions of the alternating generators $\tW_{-k}$ and~$\tW_{k+1}$ in terms of the generators of the compact presentation (called {\it essential} generators in~\cite{Ter21b}) are obtained recursively from two of the defining relations of ${\mathcal A}_q$ \cite[Definition~3.1, equations~(3.2) and~(3.3)]{BSh1}
\begin{gather}
{\tW}_{-k-1} = {\tW}_{k+1} + \frac{1}{\rho}\big[{\tW}_0,{\tG}_{k+1}\big]_q, \label{eq:Wkm-G}\\
{\tW}_{k+2} = {\tW}_{-k} + \frac{1}{\rho}
\big[{{\tG}}_{k+1},{\tW}_{1}\big]_q.\label{eq:Wkp-G}
\end{gather}

We introduce the following automorphism of $\mathcal{A}_q$
\begin{equation}
 \sigma\colon \ \cW_\pm(u) \to \cW_\mp(u), \qquad \cG_\pm(u) \to \cG_\mp(u), \qquad k_\pm \to k_\mp.\label{sigma}
\end{equation}
It is obtained from the reflection equation~\eqref{RE}. Indeed, note that the R-matrix is such that \smash{${\mathcal R}^{(\frac{1}{2},\frac{1}{2})}(u)=M{\mathcal R}^{(\frac{1}{2},\frac{1}{2})}(u)M$}, with $M=\sigma_x \otimes \sigma_x$. Consider the conjugation of the K-operator by~$\sigma_x$. Its entries read:
\[
(\sigma_x{\mathcal K}^{(\h)}(u)\sigma_x)_{i,j} = \bigl({\mathcal K}^{(\h)}(u)\bigr)_{3-i,3-j}
\qquad \text{for $1\leq i,j\leq 2$}.
\]
 Then, multiplying~\eqref{RE} on both sides by $M\otimes M$, the automorphism $\sigma$ follows.

In the following, we need the so-called quantum determinant associated with the K-operator. It is a generating function for central elements of $\mathcal{A}_q$, given by~\cite{Skly88}
\begin{gather}
 \Gamma(u)=\normalfont{\text{tr}}_{12}\bigl({\mathcal P}^{-}_{12} {\mathcal K}_1^{(\h)}(u) R^{(\frac{1}{2},\frac{1}{2})}\bigl(qu^2\bigr) {\mathcal K}_2^{(\h)}(u q) \bigr), \label{gammaform}
\end{gather}
where $\mathcal{P}^-$ is defined below~\eqref{eq:qdet}.
\begin{prop}[\cite{BasBel,Ter21}] The quantum determinant of the fundamental K-operator
 \begin{equation}
 \Gamma(u)=
 \frac{\bigl(u^2q^2-u^{-2}q^{-2}\bigr)}{2\bigl(q-q^{-1}\bigr)}\left( \Delta^{(\frac{1}{2})}(u) - \frac{2\rho}{q-q^{-1}}\right),\label{gamma}
 \end{equation}
 with
 \begin{align} \nonumber
 \Delta^{(\frac{1}{2})}(u)={}& -\bigl(q-q^{-1}\bigr)\bigl(q^2+q^{-2}\bigr)\bigl(\cW_+(u)\cW_+(uq) + \cW_-(u)\cW_-(uq)\bigr) \\
 &+\bigl(q-q^{-1}\bigr)(u^2q^2+u^{-2}q^{-2})\bigl(\cW_+(u)\cW_-(uq) + \cW_-(u)\cW_+(uq)\bigr)\nonumber\\
 &- \frac{\bigl(q-q^{-1}\bigr)}{\rho} \bigl(\cG_+(u)\cG_-(uq) + \cG_-(u)\cG_+(uq)\bigr)\nonumber\\
 & - \cG_+(u) - \cG_+(uq) - \cG_-(u) - \cG_-(uq), \label{deltau}
 \end{align}
 is such that \smash{$\big[\Gamma(u),{\mathcal K}_{mn}^{(\frac{1}{2})}(v)\big]=0$}. It generates the center of ${\mathcal A}_q$, denoted by $Z({\mathcal A}_q)$.
\end{prop}

We can expand $\Delta^{(\frac{1}{2})}(u)$ as a formal power series in $u^{-2}$
\begin{equation}\label{eq:Delta}
 \Delta^{(\h)}(u) = \sum_{k=0}^\infty u^{-2k-2} c_{k+1} \Delta_{k+1},
\end{equation}
where
\begin{equation}\label{eq:ck}
 c_{k} = - q^{-2k}\bigl(q+q^{-1}\bigr)^{k} \bigl(q^{k} + q^{-k}\bigr).
\end{equation}
An explicit formula expressing the central elements $\Delta_{k+1}$ in terms of the alternating generators we introduced above is given in \cite[Lemma~2.1]{BasBel}. We give a few of them, as they are used below.

\begin{Example}\label{ex:delta}
 First few modes from~\eqref{eq:Delta}, as elements of $Z({\mathcal A}_q)$, are
 \begin{align}
\Delta_1={}& \tG_{1} + {\normalfont \tilde{{\tG}}_{1}} -\bigl(q-q^ {-1}\bigr)\bigl(\tW_0\tW_1+\tW_1\tW_0\bigr),\label{delta1}\\
 \Delta_2={}& \tG_{2} +{\normalfont \tilde{{\tG}}_{2}} - \frac{\bigl(q^2-q^{-2}\bigr)}{\bigl(q^2+q^{-2}\bigr)}\bigl( q^{-1}\tW_0\tW_2 + q\tW_2\tW_0 + q^{-1}\tW_1\tW_{-1} + q\tW_{-1}\tW_{1} \bigr)\nonumber\\
 & + \frac{\bigl(q-q^{-1}\bigr)}{\bigl(q^2+q^{-2}\bigr)} \left(\bigl(q^2+q^{-2}\bigr)\bigl(\tW_0^2+\tW_1^2\bigr)+\frac{\tilde{\tG}_1\tG_1 + \tG_1\tilde{\tG}_1}{\rho}\right), \label{delta2}
 \end{align}
 where $\tW_{-1}$ and $\tW_2$ are determined by~\eqref{eq:Wkm-G} and~\eqref{eq:Wkp-G} at $k=0$, respectively.
\end{Example}

\begin{lem} \label{lem-invertKh}
 \begin{equation} \label{invertKh}
 \mathcal{K}^{(\h)}\bigl(u^{-1}\bigr)\mathcal{K}^{(\h)}(u) =
 \mathcal{K}^{(\h)}(u) \mathcal{K}^{(\h)}\bigl(u^{-1}\bigr) = \frac{\Gamma\bigl(u q^{-1}\bigr)}{c\bigl(u^{-2}\bigr)} \mathbb{I}_2,
 \end{equation}
 where $c(u)$ and $\Gamma(u)$ are respectively given in~\eqref{eq:cu} and~\eqref{gamma}, and where the right-hand side we consider as an element in $Z(\mathcal{A}_q)\big[\big[u^{-1}\big]\big]$.
\end{lem}

\begin{proof}
 Using our developing in series convention in~\eqref{eq:u-inv-convention} and the ordering relations of $\mathcal{A}_q$ given in Appendix~\ref{apB}, it is straightforward to calculate~\eqref{invertKh}.
\end{proof}

We will call the above property~\eqref{invertKh} the unitarity property of the fundamental K-operator by analogy with the L-operator.

Let us now show invertibility property of the fundamental K-operator $\mathcal{K}^{(\h)}(u)$ from~\eqref{K-Aq}.
We first show that $\Gamma(u)$ is invertible in the ring $\mathcal{A}_q\bigl(\bigl(u^{-1}\bigr)\bigr)$. We notice that the constant term of~\smash{$\Delta^{(\h)}(u)$} from~\eqref{deltau} is zero, which implies using~\eqref{gamma} that the constant term of~\smash{$u^{-2}\Gamma(u)$} is non-zero. Therefore, the latter formal power series is invertible in $Z(\mathcal{A}_q)\big[\big[u^{-1}\big]\big]$ by \cite[Lemma~4.1]{Ter21d}. We thus obtain that $\Gamma(u)$ is invertible in the ring $\mathcal{A}_q\bigl(\bigl(u^{-1}\bigr)\bigr)$. Explicit expressions of the modes of $\Gamma(u)^{-1}$ can be deduced from \cite[Lemma~4.1]{Ter21d} shifting the mode index by $-2$.
Since $\Gamma(u)$ is invertible, it follows by Lemma~\ref{lem-invertKh} that \smash{$\mathcal{K}^{(\h)}(u)$} is invertible too and its inverse is given by
\begin{equation}\label{rem-invKh}
 \big [ \mathcal{K}^{(\h)}(u) \big ]^{-1} = \frac{c\bigl(u^{-2}\bigr)}{\Gamma\bigl(u q^{-1}\bigr)} \mathcal{K}^{(\h)}\bigl(u^{-1}\bigr).
\end{equation}

\begin{rem}
 A central element of $\mathcal{A}_q$ denoted $\mathcal{Z}(t)$ has been proposed in~\cite[Definition~8.4]{Ter21}. It~is easily checked that adapting its expression to our conventions, one has
\[
 \frac{\Gamma(u q^{-\h})}{c(u^2q) } \rightarrow \mathcal{Z}(t),
\]
 with the identification
 \begin{gather*}
 \rho \rightarrow -\bigl(q^2-q^{-2}\bigr)^2, \qquad q^{-1} u^{-2} \rightarrow t, \qquad \mathcal{W}_\mp\bigl(u q^{\h}\bigr) \rightarrow S \mathcal{W}^\pm(S), \\ \mathcal{W}_\mp\bigl(u q^{-\h}\bigr) \rightarrow T \mathcal{W}^\pm(T),\qquad
 \mathcal{G}_+\bigl(u q^{\h}\bigr) + \rho/\bigl(q-q^{-1}\bigr) \rightarrow \mathcal{G}(S), \\ \mathcal{G}_-\bigl(u q^{-\h}\bigr) + \rho/\bigl(q-q^{-1}\bigr) \rightarrow \tilde{\mathcal{G}}(T).
 \end{gather*}
\end{rem}

\subsection[Fused K-operators for Aq]{Fused K-operators for $\boldsymbol{\mathcal{A}_q}$} \label{sec:fusKcal}
Recall \smash{$R^{(\h,\h)}(u)$} and the fundamental K-operator \smash{${\mathcal{K}}^{(\h)}(u)$}, given respectively in~\eqref{evalRh} and~\eqref{K-Aq}, satisfy the reflection equation~\eqref{RE}. By analogy with~\eqref{fusedevalKv2}, we now introduce fused K-operators~\smash{$\mathcal{K}^{(j)}(u) \in \mathcal{A}_q\bigl(\bigl(u^{-1}\bigr)\bigr) \otimes \End\bigl( \mathbb{C}^{2j+1}\bigr)$}.
\begin{defn} \label{spinjfusedK}
 For $j \in \frac{1}{2} \mathbb{N}_+$, the fused K-operators for $\mathcal{A}_q$ are
 \begin{equation}
 {\mathcal K}^{(j+\h)}(u) = \mathcal{F}^{(j+\h)}_\fu {\mathcal K}_1^{(\h)}\bigl(u q^{-j}\bigr) R^{(\h,j)}\bigl(u^2 q^{-j+\h}\bigr) {\mathcal K}_2^{(j)}\bigl(u q^{\h}\bigr) \mathcal{E}^{(j+\h)}_\fu, \label{fusedunormK}
 \end{equation}
 with \smash{${\mathcal K}^{(\h)}(u)$} defined in~\eqref{K-Aq}. We note that \smash{$\mathcal{K}^{(j)}(u) \in u^{4j^2-2j} \mathcal{A}_q\big[\big[u^{-1}\big]\big] \otimes \End\bigl( \mathbb{C}^{2j+1}\bigr)$}.
\end{defn}

The following theorem is our second main result.
\begin{thm}\label{prop:fusedRE}
 The fused K-operators \smash{$\mathcal{K}^{(j)}(u)$} satisfy the reflection equation for any $j_1, j_2 \in \h \mathbb{N}_+$
 \begin{gather}
 R^{(j_1,j_2)}(u_1/u_2) {\mathcal K}^{(j_1)}_1(u_1)R^{(j_1,j_2)}(u_1u_2) {\mathcal K}_2^{(j_2)}(u_2)\nonumber \\
 \qquad={\mathcal K}_2^{(j_2)}(u_2) R^{(j_1,j_2)}(u_1u_2) {\mathcal K}_1^{(j_1)}(u_1) R^{(j_1,j_2)}(u_1/u_2).\label{REj1j2}
 \end{gather}
\end{thm}

The proof of this theorem is very technical and was delegated to Appendix~\ref{apD}. We proceed by induction on $j_1$, $j_2$, using
several expressions for the fused R-matrices $R^{(j_1,j_2)}(u)$ defined in~\eqref{fusRstj1j2}.
Moreover, the decomposition of \smash{$R^{(\frac 12,j)}\big(q^{j+\frac 12}\big)$}, expressed in terms of \smash{$\mathcal{E}^{(j+\h)}$}, \smash{$\mathcal{F}^{(j+\h)}$}, and \smash{$\mathcal{H}^{(j+\h)}$} as given in Lemma~\ref{lemEHF}, is crucial for the proof of this theorem, along with the relations from Corollary~\ref{corEHF}.

\begin{rem}\label{rem:K-RSV}
 Definition~\ref{spinjfusedK} and Theorem~\ref{prop:fusedRE} provide a generalization of~\cite[Proposition~4.3]{RSV16}, where a similar fusion formula for diagonal K-matrices on the spin-$j$ representations was given. Our fused K-operators provide K-matrices on any representation of $\mathcal A_q$, and in particular general (both diagonal and non-diagonal) K-matrices on the spin-$j$ representations are recovered from~${\mathcal K}^{(j)}(u)$ evaluated on one-dimensional representations of $\mathcal A_q$, see~\cite[Section~2.3]{LBG} for more details.
\end{rem}

\subsection{Unitarity and invertibility properties} \label{sec:unit-inv}
We now discuss the unitarity and invertibility properties of the fused K-operators $\mathcal{K}^{(j)}(u)$ defined by~\eqref{fusedunormK}. Recall that \smash{$\mathcal{K}^{(\h)}(u)$} satisfies the unitarity property and is invertible, see Lemma~\ref{lem-invertKh} and expression in~\eqref{rem-invKh}, respectively. We generalize these properties for any spin-$j$.
\begin{prop}
 Let
 \begin{equation} \label{invertKj}
 \widehat{\mathcal{K}}^{(j+\h)}(u) = \mathcal{F}^{(j+\h)}_\fu \widehat{\mathcal{K}}^{(j)}_2\bigl(u q^{-\h}\bigr) R^{(\h,j)}\bigl(u^2 q^{j-\h}\bigr) \widehat{\mathcal{K}}_1^{(\h)}\bigl(u q^j\bigr) \mathcal{E}^{(j+\h)}_\fu,
 \end{equation}
 for $j \in \h \mathbb{N}_+$ and with \smash{$\widehat{\mathcal{K}}^{(\h)}(u) \equiv \mathcal{K}^{(\h)}(u)$}. Then
 \begin{align}
 \mathcal{K}^{(j)}(u) \widehat{\mathcal{K}}^{(j)}\bigl(u^{-1}\bigr) ={}& \left ( \prod_{k=0}^{2j-1} \frac{\Gamma\bigl(u q^{-j-\h+k}\bigr)}{ c\bigl(u^{-2} q^{2j-1-2k}\bigr)} \right )\nonumber\\
 &\times\left ( \prod_{k=0}^{2j-2} \prod_{\ell=0}^{2j-k-2} c\bigl(u^2 q^{2j-1-2k-\ell}\bigr) c\bigl(u^{-2} q^{1-k+\ell}\bigr) \right ) \mathbb{I}_{2j+1}\nonumber \\
={}& \widehat{\mathcal{K}}^{(j)}\bigl(u^{-1}\bigr) \mathcal{K}^{(j)}(u), \label{eq:invertKj}
 \end{align}
 where $\mathcal{K}^{(j)}(u)$, $\Gamma(u)$ and $c(u)$ are respectively given in~\eqref{fusedunormK},~\eqref{gamma},~\eqref{eq:cu}.
\end{prop}
\begin{proof}
 Recall that~\eqref{eq:invertKj} for $j=\h$ was proven in Lemma~\ref{lem-invertKh}. First, we show that~\eqref{eq:invertKj} holds for $j=1$. Using~\eqref{fusedunormK} and~\eqref{invertKj}, we find that the product \smash{$\mathcal{K}^{(1)}(u) \widehat{\mathcal{K}}^{(1)}\bigl(u^{-1}\bigr)$} equals
 \begin{gather*}
 \mathcal{F}^{(1)}_\fu \mathcal{K}^{(\h)}_1\bigl(u q^{-\h}\bigr) R^{(\h,\h)}(u^2) \mathcal{K}^{(\h)}_2\bigl(u q^\h\bigr) \mathcal{E}^{(1)}_\fu \mathcal{F}^{(1)}_\fu \widehat{\mathcal{K}}_2^{(\h)}\bigl(u^{-1} q^{-\h}\bigr) R^{(\h,\h)}\bigl(u^{-2}\bigr) \widehat{\mathcal{K}}_1^{(\h)}\\
\qquad\quad{} \times\bigl(u^{-1} q^\h\bigr) \mathcal{E}^{(1)}_\fu \\
\qquad{}= \mathcal{F}^{(1)}_\fu \mathcal{K}^{(\h)}_1\bigl(u q^{-\h}\bigr) R^{(\h,\h)}(u^2) \mathcal{K}^{(\h)}_2\bigl(u q^\h\bigr) \widehat{\mathcal{K}}_2^{(\h)}\bigl(u^{-1} q^{-\h}\bigr) R^{(\h,\h)}\bigl(u^{-2}\bigr)\\
\qquad\quad{} \times\widehat{\mathcal{K}}_1^{(\h)}\bigl(u^{-1} q^\h\bigr) \mathcal{E}^{(1)}_\fu,
 \end{gather*}
 where we removed the product $\mathcal{E}^{(1)} \mathcal{F}^{(1)}$ on the second line, similarly to the derivation of~\eqref{L1EF}. Then, using~\eqref{unitRhj} and~\eqref{invertKh} we have
 \smash{$ \mathcal{K}^{(1)}(u) \widehat{\mathcal{K}}^{(1)}\bigl(u^{-1}\bigr) = -\Gamma\bigl(u q^{-\tha}\bigr) \Gamma\bigl(u q^{-\h}\bigr) \mathbb{I}_3$}.
 Similarly, we get
\smash{$
 \widehat{\mathcal{K}}^{(1)}\bigl(u^{-1}\bigr) \mathcal{K}^{(1)}(u) = -\Gamma\bigl(u q^{-\tha}\bigr) \Gamma\bigl(u q^{-\h}\bigr) \mathbb{I}_3$}.
 More generally, by induction we get both equalities in~\eqref{eq:invertKj} for all spin $j$.
\end{proof}

\begin{rem}\label{rem-invertKj} The spin-$j$ fused K-operator $\mathcal{K}^{(j)}(u)$ is invertible and its inverse lies in \linebreak $\mathcal{A}_q\big[\big[u^{-1}\big]\big] \otimes \End\bigl(\mathbb{C}^{2j+1}\bigr)$ and is given by
 \begin{align*}
 \big [ \mathcal{K}^{(j)}(u) \big ]^{-1} ={}& \left [ \prod_{k=0}^{2j-1} \frac{\Gamma\bigl(u q^{-j-\h+k}\bigr)}{ c\bigl(u^{-2} q^{2j-1-2k}\bigr)} \right ]^{-1} \left [ \prod_{k=0}^{2j-2} \prod_{\ell=0}^{2j-k-2} c\bigl(u^2 q^{2j-1-2k-\ell}\bigr) c\bigl(u^{-2} q^{1-k+\ell}\bigr) \right ]^{-1} \\
&\times\widehat{\mathcal{K}}^{(j)}\bigl(u^{-1}\bigr).
 \end{align*}
\end{rem}
\begin{rem}
 By direct calculations we have checked for $j=1,\tha,2$ that $\widehat{\mathcal{K}}^{(j)}(u)$ is equal to $\mathcal{K}^{(j)}(u)$ defined in~\eqref{fusedunormK} and we expect this equality holds for any $j$. Note that $\mathcal{K}^{(j)}(u)$ and $\widehat{\mathcal{K}}^{(j)}(u)$, are direct analogs of the spin-$j$ K-operators ${\bf K}^{(j)}(u)$ defined in~\eqref{fusedevalKv2} and~\eqref{fusedevalK}, respectively.
\end{rem}

\subsection{Examples of fused K-operators} \label{sec5.3}
In this subsection, we give examples of spin-$1$ and spin-$\tha$ fused K-operators ${\mathcal K}^{(j)}(u)$ for ${\mathcal A}_q$, defined by \eqref{fusedunormK}. Recall the function $c(u)$ given in~\eqref{eq:cu}.

\subsubsection[Spin-1 fused K-operator]{Spin-$\boldsymbol{ 1}$ fused K-operator}
The expressions of $\mathcal{E}^{(j+\h)}$, $\mathcal{F}^{(j+\h)}$ in~\eqref{intertE},~\eqref{intertF}, for $j=\h$ read
\begin{equation*}
 \mathcal{E}^{(1)}=
 \begin{pmatrix}
 1& 0& 0\\
 0& \frac{1}{\sqrt{ [2]_q}}& 0\vspace{1mm}\\
 0 & \frac{1}{\sqrt{ [2]_q}}& 0\vspace{1mm}\\
 0& 0& 1
 \end{pmatrix}, \qquad \mathcal{F}^{(1)}=
 \begin{pmatrix}
 1 & 0 & 0 & 0\\
 0& \frac{\sqrt{ [2]_q}}{2} & \frac{\sqrt{[2]_q}}{2}& 0\vspace{1mm}\\
 0& 0 & 0 &1
 \end{pmatrix}.
\end{equation*}
From \eqref{fusRstraight}, the fused R-matrix reads
\[
 R^{(\h,1)}(u)=\mathcal{F}^{(1)}_{\langle 23 \rangle} R_{13}^{(\h,\h)}\bigl(u q^{-\h}\bigr) R_{12}^{(\h,\h)}\bigl(u q^{\h}\bigr) \mathcal{E}^{(1)}_{\langle 23 \rangle}
\]
and is given explicitly by
\begin{equation} \label{mat:Rh1}
 R^{\left (\h,1 \right)}(u)=c\bigl(u q^\h\bigr)
 \begin{pmatrix}
 c\bigl(u q^{\frac{3}{2}}\bigr)& 0& 0& 0 & 0& 0\\
 0 & c\bigl(u q^{\frac{1}{2}}\bigr) & 0 & c(q)\sqrt{[2]_q}& 0& 0\\
 0 & 0 & c\bigl(u q^{-\frac{1}{2}}\bigr) & 0 & c(q)\sqrt{[2]_q}&0 \\
 0 &c(q)\sqrt{[2]_q} & 0 & c\bigl(u q^{-\frac{1}{2}}\bigr) & 0 &0 \\
 0 & 0 & c(q)\sqrt{[2]_q} & 0 & c\bigl(u q^{\frac{1}{2}}\bigr) &0 \\
 0 & 0 & 0 & 0 & 0 &c\bigl(u q^{\frac{3}{2}}\bigr)
 \end{pmatrix}.
\end{equation}
From~\eqref{fusedunormK}, the fused K-operator is given by
\[
 {\mathcal K}^{(1)}(u) = \mathcal{F}^{(1)}_\fu {\mathcal K}_1^{(\h)}\bigl(u q^{-\h}\bigr) R^{(\h,\h)}(u^2) {\mathcal K}_2^{(\h)}\bigl(u q^{\h}\bigr) \mathcal{E}^{(1)}_\fu.
\]
Using the above expressions, one finds that the entries ${\mathcal K}_{mn}^{(1)}(u)$ are explicitly given by
\begin{gather}
 \mathcal{K}^{(1)}_{11}(u)= \bigl( c(q)^{-1} +\rho^{-1} \mathcal{G}_+\bigl(u q^{-\h}\bigr) \bigr) \bigl( \rho + c(q) \mathcal{G}_-\bigl(u q^\h\bigr) \bigr)\nonumber \\
 \phantom{\mathcal{K}^{(1)}_{11}(u)= }{}+c\bigl(u^2q\bigr) \bigl( u q^\h \mathcal{W}_+\bigl(u q^{-\h}\bigr) - u^{-1} q^{-\h} \mathcal{W}_-\bigl(u q^{-\h}\bigr) \bigr)\nonumber\\
\phantom{\mathcal{K}^{(1)}_{11}(u)=+ }{} \times\bigl( u q^{\frac 32} \mathcal{W}_+\bigl(u q^\h\bigr) - u^{-1} q^{-\frac 32} \mathcal{W}_-\bigl(u q^\h\bigr) \bigr), \nonumber\\
 \mathcal{K}^{(1)}_{12}(u)= \frac{\bigl(q+q^{-1}\bigr)^{-\frac 32}}{k_-} \bigl( c(u^2) \bigl( \rho c(q)^{-1} + \mathcal{G}_+\bigl(u q^{-\h}\bigr) \bigr) \bigl(u q^{\frac 32} \mathcal{W}_+\bigl(uq^\h\bigr) - u^{-1} q^{-\frac 32} \mathcal{W}_-\bigl(u q^\h\bigr) \bigr)\nonumber \\
 \phantom{\mathcal{K}^{(1)}_{12}(u)= }{} \bigl( \rho + c(q) \mathcal{G}_+\bigl(u q^{-\h}\bigr) \bigr) \bigl( u q^{\frac 32} \mathcal{W}_-\bigl(u q^\h\bigr) - u^{-1} q^{-\frac 32} \mathcal{W}_-\bigl(u q^\h\bigr) \bigr) \nonumber\\
 \phantom{\mathcal{K}^{(1)}_{12}(u)= }{} + c(u^2 q) \bigl( u q^\h \mathcal{W}_+\bigl(u q^{-\h}\bigr) - u^{-1} q^{-\h} \mathcal{W}_-\bigl(u q^{-\h}\bigr) \bigr) \bigl( \rho c(q)^{-1} + \mathcal{G}_+\bigl(u q^\h\bigr) \bigr) \bigr), \nonumber\\
 \mathcal{K}^{(1)}_{13}(u)= \frac{c(u^2)}{k_-^2c(q^2)^2} \bigl( \rho + c(q) \mathcal{G}_+\bigl(uq^{-\h}\bigr) \bigr) \bigl( \rho + c(q) \mathcal{G}_+\bigl(u q^\h\bigr) \bigr), \nonumber\\
 \mathcal{K}^{(1)}_{21}(u)= \frac{c(q)^{-1}}{2k_+\sqrt{q+q^{-1}}} \bigl( c\bigl(u^2q\bigr) \bigl( \rho + c(q) \mathcal{G}_-\bigl(u q^{-\h}\bigr) \bigr) \bigl( u q^{\frac 32} \mathcal{W}_+(u q^\h ) - u^{-1} q^{-\frac 32} \mathcal{W}_-\bigl(u q^\h\bigr) \bigr) \nonumber\\
 \phantom{\mathcal{K}^{(1)}_{21}(u)= }{} +\bigl( q^{-\h} ( u^{-3} + u(-2+q^2)) \mathcal{W}_-\bigl(u q^{-\h}\bigr) + q^{\h} (u^3 + u^{-1}(-2+q^{-2})) \mathcal{W}_+\bigl(u q^{-\h}\bigr) \bigr) \bigr) \nonumber\\
 \phantom{\mathcal{K}^{(1)}_{21}(u)=+ }{} \times \bigl( \rho + c(q) \mathcal{G}_-\bigl(u q^\h\bigr) \bigr),\nonumber
\\
 \mathcal{K}^{(1)}_{22}(u)= \frac{c\bigl(u^2q\bigr)}{2c(q)^2 \rho} \bigl( \bigl(\rho + c(q) \mathcal{G}_+\bigl(u q^{-\h}\bigr) \bigr) \bigl( \rho
 + c(q) \mathcal{G}_-\bigl(u q^\h\bigr) \bigr)\nonumber
 \\
\phantom{\mathcal{K}^{(1)}_{22}(u)= }{} + \bigl( \rho + c(q) \mathcal{G}_-\bigl(u q^{-\h}\bigr) \bigr) \bigl( \rho + c(q) \mathcal{G}_+\bigl(u q^\h\bigr) \bigr) \bigr)\nonumber \\
\phantom{\mathcal{K}^{(1)}_{22}(u)= }{}+ \frac{1}{2} \bigl( q^{-\h} (u^{-3} +u (-2+q^2)) \mathcal{W}_+\bigl(u q^{-\h}\bigr) + q^{\h} (u^{-1} (-2+q^2) +u^3) \mathcal{W}_-\bigl(u q^{-\h}\bigr) \bigr)\nonumber \\
\phantom{\mathcal{K}^{(1)}_{22}(u)= +}{}\times \bigl( u q^{\frac 32} \mathcal{W}_+\bigl(u q^\h\bigr) - u^{-1} q^{-\frac 32} \mathcal{W}_-\bigl(u q^\h\bigr) \bigr) \nonumber\\
\phantom{\mathcal{K}^{(1)}_{22}(u)= }{}+ \frac{1}{2} \bigl( q^{-\h} (u^{-3} +u (-2+q^2)) \mathcal{W}_-\bigl(u q^{-\h}\bigr) + q^{\h} (u^{-1} (-2+q^2) +u^3) \mathcal{W}_+\bigl(u q^{-\h}\bigr) \bigr)\nonumber
 \\
\phantom{\mathcal{K}^{(1)}_{22}(u)=+ }{}\times \bigl( u q^{\frac 32} \mathcal{W}_-\bigl(u q^\h\bigr) - u^{-1} q^{-\frac 32} \mathcal{W}_+\bigl(u q^\h\bigr) \bigr), \nonumber\\
 {\mathcal K}^{(1)}_{23}(u)=\sigma({\mathcal K}^{(1)}_{21}(u))
 , \qquad {\mathcal K}^{(1)}_{31}(u)=\sigma({\mathcal K}^{(1)}_{13}(u))\nonumber
 , \\
 {\mathcal K}^{(1)}_{32}(u)=
 \sigma({\mathcal K}^{(1)}_{12}(u))
 , \qquad {\mathcal K}^{(1)}_{33}(u)=\sigma({\mathcal K}^{(1)}_{11}(u)),\label{expK-spin1}
\end{gather}
where $\sigma$ is defined in~\eqref{sigma}. The last two lines describe the exchange of the entries of the fused K-operator due to the automorphism $\sigma$ and can be seen graphically
\begin{equation} \label{symK1}
 \begin{tikzpicture}[baseline=(current bounding box.center)]
 \draw [latex-latex] (-1.5,0.7) to (1.55,-0.75);
 \draw[latex-latex] (-1.5,0) -- (1.5,0);
 \draw[latex-latex] (0,0.7) -- (0,-0.7);
 \draw[latex-latex] (-1.5,-0.7) -- (1.5,0.7);
 \draw (-3.5,0) node[]{$\mathcal{K}^{(1)}(u)=$};
 \draw (3,0) node[]{.};
 \draw (-1.5,0.7) node[]{$\mathcal{K}_{11}^{(1)}(u)$};
 \draw (0,0.7) node[]{$\mathcal{K}_{12}^{(1)}(u)$};
 \draw (1.5,0.7) node[]{$\mathcal{K}_{13}^{(1)}(u)$};
 \draw (-1.5,0) node[]{$\mathcal{K}_{21}^{(1)}(u)$};
 \draw (0,0) node[]{$\mathcal{K}_{22}^{(1)}(u)$};
 \draw (1.5,0) node[]{$\mathcal{K}_{23}^{(1)}(u)$};
 \draw (-1.5,-0.7) node[]{$\mathcal{K}_{31}^{(1)}(u)$};
 \draw (0,-0.7) node[]{$\mathcal{K}_{32}^{(1)}(u)$};
 \draw (1.5,-0.7) node[]{$\mathcal{K}_{33}^{(1)}(u)$};
 \draw (-2.5,0) node[scale=2]{\Bigg{(}};
 \draw (2.5,0) node[scale=2]{\Bigg{)}};
 \end{tikzpicture}
\end{equation}
As shown in Lemma \ref{hjh}, the fused K-operator (\ref{evalKj}) for $j=1$ satisfies the reflection equation
\[
 R^{ (\h,1 )}(u/v) \mathcal{K}_1^{(\h)}(u) R^{(\h,1)}(u v) \mathcal{K}_2^{(1)}(v)=\mathcal{K}_2^{(1)}(v)R^{ (\h,1 )}(u v) \mathcal{K}_1^{(\h)}(u)R^{ (\h,1 )}(u/v).
\]
Note that the latter equation can be independently checked using the ordering relations given in Lemma \ref{rel-PBW-Aq}.

\subsubsection[Spin-tha fused K-operator]{Spin-$\boldsymbol{ \tha}$ fused K-operator}
The elements $\mathcal{E}^{(j+\h)}, \mathcal{F}^{(j+\h)}$ in~\eqref{intertE},~\eqref{intertF} for $j=1$ read
\begin{equation*}
 \mathcal{E}^{(\tha)} =
 \begin{pmatrix}
 1 & 0 & 0 &0 \\
 0& \sqrt{\frac{[2]_q}{[3]_q}} & 0 &0 \\
 0 & 0 & \frac{1}{ \sqrt{[3]_q}} &0 \\
 0 & \frac{1}{\sqrt{[3]_q}} & 0 &0 \\
 0 & 0 & \sqrt{\frac{[2]_q}{[3]_q}}&0 \\
 0 & 0 & 0 &1
 \end{pmatrix},\qquad
 \mathcal{F}^{(\tha)}=
 \begin{pmatrix}
 1 & 0 & 0 & 0 & 0 &0 \\
 0 & \frac{ \sqrt{[2]_q[3]_q}}{1+[2]_q} & 0 & \frac{\sqrt{[3]_q}}{1+[2]_q} & 0 &0 \\
 0&0 &\frac{\sqrt{[3]_q}}{1+[2]_q} &0 &
 \frac{ \sqrt{[2]_q[3]_q}}{1+[2]_q} &0 \\
 0 & 0 &0 &0 & 0 & 1
 \end{pmatrix}.
\end{equation*}
The fused R-matrix from~\eqref{fusRstraight} reads
\begin{align*}
 R^{(\h,\tha)}(u) & =\mathcal{F}^{(\tha)}_{\langle 23 \rangle} R_{13}^{\left (\h,1 \right)}\bigl(u q^{-\h}\bigr) R_{12}^{(\h,\h)}(u q) \mathcal{E}^{(\tha)}_{\langle 23 \rangle},
\end{align*}
given explicitly by
\[
 R^{(\h,\tha)}(u)=c(u)c(uq) c(q)
 \begin{pmatrix}
 \frac{c(uq^2)}{c(q)}& 0 &0 & 0 &0 &0 &0 &0 \\
 0 & \frac{c(uq)}{c(q)}&0 & 0 &\sqrt{[3]_q} &0 &0 &0 \\
 0 & 0 &\frac{c(u)}{c(q)}& 0 &0 &[2]_q &0 &0 \\
 0 & 0 &0 & \frac{c(u q^{-1})}{c(q)}&0 &0 & \sqrt{[3]_q} &0 \\
 0 & \sqrt{[3]_q} &0 & 0 &\frac{c(u q^{-1})}{c(q)}&0 &0 &0 \\
 0 & 0 & [2]_q & 0 &0 &\frac{c(u)}{c(q)} &0 &0 \\
 0 & 0 &0 & \sqrt{[3]_q} &0 &0 &\frac{c(uq)}{c(q)} &0 \\
 0 & 0 &0 & 0 &0 &0 &0 &\frac{c(uq^2)}{c(q)}
 \end{pmatrix}.
\]
From~\eqref{fusedunormK}, the fused K-operator reads
\[
 {\mathcal K}^{(\tha)}(u) = \mathcal{F}^{(\tha)}_\fu {\mathcal K}_1^{(\h)}\bigl(u q^{-1}\bigr) R^{\left (\h,1 \right)}(u^2 q^{-\h}) {\mathcal K}_2^{(1)}\bigl(u q^{\h}\bigr) \mathcal{E}^{(\tha)}_\fu.
\]
For instance, the first entry reads
\begin{align*}
 \mathcal{K}^{(\tha)}_{11}(u)={}& \frac{c(u^2)\bigl(q+q^{-1}\bigr)}{2 \rho c(q)} \bigl( \rho + c(q) \mathcal{G}_+\bigl(uq^{-1}\bigr)\bigr) \bigl( c(u^2q^2) \bigl(\rho +c(q) \mathcal{G}_-(u) \bigr) \\
 &\times\bigl(u q^2 \mathcal{W}_+(uq) -u^{-1} q^{-2} \mathcal{W}_-(u q) \bigr) + u^{-1} \bigl( \bigl(u^2 \bigl(q^2-2\bigr) + u^{-2}q^{-2}\bigr) \mathcal{W}_-(u)\\
 &\phantom{\times}{}+ \bigl(u^4 q^2 +q^{-2}-2\bigr) \mathcal{W}_+(u) \bigr) \bigl( \rho + c(q) \mathcal{G}_-(u q) \bigr) \bigr) \\
 & \phantom{\times}{}+ c\bigl(u^2\bigr)c\bigl(u^2q\bigr) \bigl( u\mathcal{W}_+\bigl(u q^{-1}\bigr) -u^{-1} \mathcal{W}_-\bigl(u q^{-1}\bigr) \bigr)\\
 &\times \bigl( \bigl( \rho^{-1} \mathcal{G}_+(u) +c(q)^{-1} \bigr) \bigl( \rho + c(q) \mathcal{G}_-(uq) \bigr) \\
 &\phantom{\times}{} + c\bigl(u^2q^2\bigr) \bigl(u q \mathcal{W}_+(u) -u^{-1} q^{-1} \mathcal{W}_-(u) \bigr) \bigl( u q^2 \mathcal{W}_+(uq) - u^{-1} q^{-2} \mathcal{W}_-(u q) \bigr) \bigr).
\end{align*}
The other explicit expressions of the entries in terms of the generating functions for ${\mathcal A}_q$ are not reported here for simplicity. Under the action of $\sigma$ from~\eqref{sigma}, the entries exchange according~to%
\begin{equation}\label{symK3h}
 \begin{tikzpicture}[baseline=(current bounding box.center)]
 \draw[latex-latex] (-1.5,0.7) to (3,-1.4);
 \draw[latex-latex] (-1.5,0) -- (3,-0.7);
 \draw[latex-latex] (-1.5,-0.7) -- (3,0);
 \draw[latex-latex] (-1.5,-1.4) --(3,0.7);
 \draw[latex-latex] (0,0.7) --(1.5,-1.4);
 \draw[latex-latex] (1.5,0.7) --(0,-1.4);
 \draw[latex-latex] (0,-0.7) --(1.5,0);
 \draw (-3.5,0) node[]{$\mathcal{K}^{(\tha)}(u)=$};
 \draw (3,0) node[]{,};
 \draw (-1.5,0.7) node[]{$\mathcal{K}_{11}^{(\tha)}(u)$};
 \draw (0,0.7) node[]{$\mathcal{K}_{12}^{(\tha)}(u)$};
 \draw (1.5,0.7) node[]{$\mathcal{K}_{13}^{(\tha)}(u)$};
 \draw (3,0.7) node[]{$\mathcal{K}_{14}^{(\tha)}(u)$};
 \draw (-1.5,0) node[]{$\mathcal{K}_{21}^{(\tha)}(u)$};
 \draw (0,0) node[]{$\mathcal{K}_{22}^{(\tha)}(u)$};
 \draw (1.5,0) node[]{$\mathcal{K}_{23}^{(\tha)}(u)$};
 \draw (3,0) node[]{$\mathcal{K}_{24}^{(\tha)}(u)$};
 \draw (-1.5,-0.7) node[]{$\mathcal{K}_{31}^{(\tha)}(u)$};
 \draw (0,-0.7) node[]{$\mathcal{K}_{32}^{(\tha)}(u)$};
 \draw (1.5,-0.7) node[]{$\mathcal{K}_{33}^{(\tha)}(u)$};
 \draw (3,-0.7) node[]{$\mathcal{K}_{34}^{(\tha)}(u)$};
 \draw (-1.5,-1.4) node[]{$\mathcal{K}_{41}^{(\tha)}(u)$};
 \draw (0,-1.4) node[]{$\mathcal{K}_{42}^{(\tha)}(u)$};
 \draw (1.5,-1.4) node[]{$\mathcal{K}_{43}^{(\tha)}(u)$};
 \draw (3,-1.4) node[]{$\mathcal{K}_{44}^{(\tha)}(u)$};
 \draw (-2.5,-0.35) node[scale=2.5]{\Bigg{(}};
 \draw (4,-0.35) node[scale=2.5]{\Bigg{)}};
 \end{tikzpicture}.
\end{equation}
By Theorem~\ref{prop:fusedRE}, the fused K-operator satisfies the reflection equation
\begin{equation*}
 R^{(\h,\tha)}(u/v) \mathcal{K}_1^{(\h)}(u) R^{(\h,\tha)}(u v) \mathcal{K}_2^{(\tha)}(v)=\mathcal{K}_2^{(\tha)}(v)R^{(\h,\tha)}(u v) \mathcal{K}_1^{(\h)}(u)R^{(\h,\tha)}(u/v).
\end{equation*}

\subsubsection[Spin-j fused K-operator]{Spin-$\boldsymbol{ j}$ fused K-operator}
Specializing the formula~\eqref{fusedunormK}, one gets the fused K-operator ${\mathcal K}^{(j)}(u)$ for any value of $j$ starting from~\eqref{K-Aq}. By analogy with the previous cases, note that one has the invariance of the R-matrix~\eqref{R-Rqg}
\begin{equation} \nonumber
 R^{(\h,j)}(u)= M^{(j)}R^{(\h,j)}(u)M^{(j)} \qquad \mbox{with} \quad M^{(j)}= \sigma_x \otimes \sum_{n=1}^{2j+1} E_{n,2j+2-n}^{(j,j)}.
\end{equation}
So, due to the automorphism $\sigma$ in~\eqref{sigma}, the entries of the K-operator of spin-$j$ exchange~as
\[
\mathcal{K}^{(j)}_{m,n}(u) = \sigma\bigl(\mathcal{K}^{(j)}_{2j+2-m,2j+2-n}(u)\bigr) \qquad \text{with}\quad 1 \leq m,n \leq 2j+1.
\]
 This is analogous to the property in~\eqref{symK1},~\eqref{symK3h}.

From the fusion formulas~\eqref{fusRstraight} and~\eqref{fusedunormK}, it is clear that the fused R-matrices and K-operators can be expressed only in terms of the fundamental K-operator and R-matrix, and the maps $\mathcal{E}^{(j)}$ and $\mathcal{F}^{(j)}$. They are given by
\begin{align}\begin{aligned}[b]
 R^{(\h,j)}(u) ={}& \left (\prod_{m=0}^{2j-2} \mathbb{I}_{2^{2m+1}} \otimes \mathcal{F}^{(j-\frac{m}{2})} \right ) \left ( \prod_{k=0}^{2j-1} R_{1 2j+1-k}^{(\h,\h)}(u q^{-j+\h+k}) \right ) \\
&\times
 \left (\prod_{m=0}^{2j-2} \mathbb{I}_{2^{2j-1-m}} \otimes \mathcal{E}^{(1+\frac{m}{2})} \right ),
 \end{aligned} \label{dvpR}
\end{align}
and
 \begin{align}
 \mathcal{K}^{(j)}(u) ={}& \left (\prod_{m=0}^{2j-2} \mathbb{I}_{2^m} \otimes \mathcal{F}^{(j-\frac{m}{2})} \right ) \prod_{k=1}^{2j} \left \{ \mathcal{K}_k^{(\h)}(u q^{k-j-\h}) \left [ \prod_{\ell=0}^{2j-k-1} R_{k 2j-\ell}^{(\h,\h)}(u^2q^{-2j+2k+\ell}) \right ] \right \}\nonumber \\
 & \times \left (\prod_{m=0}^{2j-2} \mathbb{I}_{2^{2j-2-m}} \otimes \mathcal{E}^{(1+\frac{m}{2})} \right ),\label{dvpK}
 \end{align}
where the product stands for the usual matrix product and the products are ordered from left to right in an increasing way in the indices.
The proof of~\eqref{dvpR} is straightforward by induction on $j$ using~\eqref{fusRstraight}, whereas the proof of~\eqref{dvpK} is more tedious.
We proceed by induction checking~\eqref{fusedunormK} using~\eqref{dvpR} and~\eqref{dvpK}. Then, one obtains a formula similar to~\eqref{dvpK} for~${j \rightarrow j + \h}$ but with unwanted products
of $\mathcal{E}^{(j)}\mathcal{F}^{(j)}$. They can be removed using the same trick as in equation~\eqref{L1EF}.
Firstly, multiply \eqref{fusedunormK} from the right by \smash{$\mathcal{H}^{(j+\h)}\big\lbrack\mathcal{H}^{(j+\h)}\big\rbrack^{-1}=\mathbb{I}_{2j+2}$} and use~\eqref{EHF1} to move \smash{$\mathcal{H}^{(j+\h)}$} to the left. Then, using successively the Yang--Baxter equation and the reflection equation, the products \smash{$\mathcal{E}^{(j)}\mathcal{F}^{(j)}$} are removed using the property~\eqref{EHF3}.

Note that in the literature, another fusion procedure was developed for the R-matrix, see \mbox{\cite{Ka79,KRS81}}, and for the K-matrix in~\cite{MN91}. In this case, the analogue of the formulas~\eqref{dvpK},~\eqref{dvpR} can be found for instance in~\cite[equations~(2.1) and~(2.7)]{FNR07}.

\subsection{Evaluated coaction of fused K-operators} \label{sec:evalcoac}

The fused K-operators $\mathcal{K}^{(j)}(u)$ are expected to have a simple relation with the spin-$j$ K-operators~${\bf K}^{(j)}(u)$ as will be discussed in Section~\ref{sec6}, similarly to the relations between~${\bf L}^{(j)}(u)$ and $\mathcal{L}^{(j)}(u)$. Therefore, the evaluated coaction $\delta_w$ defined in~\eqref{formal_evco} and applied to the entries of~\smash{$\mathcal{K}^{(\h)}(u)$} is expected to be of the form~\eqref{delwK} up to appropriate normalization.
\begin{lem}\label{coacK}
The evaluated coaction
 $\delta_w\colon \mathcal{A}_q \rightarrow \mathcal{A}_q \otimes \Uq \big[w^{\pm1}\big]$ is such that, recall~\eqref{eq:conv-square},
 \begin{equation}\label{coact-h}
 (\delta_{w} \otimes \id)\bigl({\mathcal K}^{(\h)}(u)\bigr)= \frac{U^{-1}}{q+q^{-1}}\bigl({\mathcal
 L}^{(\h)}(u/w)\bigr)_{[\mathsf 2]} \bigl({\mathcal
 K}^{(\h)}(u)\bigr)_{[\mathsf 1]} \bigl({\mathcal L}^{(\h)}(u
 w)\bigr)_{[\mathsf 2]}.
 \end{equation}
\end{lem}
\begin{proof}
 Assume the evaluated coaction takes the form
 \begin{equation}
 (\delta_{w} \otimes \id)\bigl({\mathcal K}^{(\h)}(u)\bigr)=f(u)\bigl({\mathcal
 L}^{(\h)}(u/w)\bigr)_{[\mathsf 2]} \bigl({\mathcal
 K}^{(\h)}(u)\bigr)_{[\mathsf 1]} \bigl({\mathcal L}^{(\h)}(u
 w)\bigr)_{[\mathsf 2]} ,\label{deltp}
 \end{equation}
 where $f(u)$ is assumed to be invertible central in ${\mathcal A}_q\big[\big[u^{-1}\big]\big] \otimes \Uq$. We show that $\delta_w$ is indeed an algebra homomorphism for a certain choice of $f(u)$. It is easily checked using the Yang--Baxter algebra \eqref{univRLL} satisfied by \smash{$\mathcal{L}^{(\h)}(u)$} that \eqref{deltp} solves the reflection equation \eqref{RE} with the substitution \smash{${\mathcal
 K}^{(\h)}(u) \rightarrow (\delta_{w} \otimes \id)\bigl({\mathcal K}^{(\h)}(u)\bigr)$}.
 Then, we fix the function $f(u)$ as follows. We compare the left-hand side of~\eqref{deltp} using~\eqref{K-Aq} to the right-hand side of~\eqref{deltp} that is computed using \smash{$\mathcal{L}^{(\h)}(u)$} given by~\eqref{Laxh}.
 Consider the matrix entry~$(2,1)$ of~\eqref{deltp}. It reads
 \begin{gather*}
 \frac{1}{k_+\bigl(q+q^{-1}\bigr)} \delta_w(\cG_-(u)) + \frac{k_- \bigl(q+q^{-1}\bigr)}{q-q^{-1}} \delta_w(1) \\
 \qquad= \frac{f(u)}{k_+\bigl(q+q^{-1}\bigr)} \left[\frac{k_+}{k_-}\bigl(q-q^{-1}\bigr)^2{\cG}_+(u)\otimes E^2
 - {\cG}_{-}(u)\otimes\bigl(w^{-2}K^{-1}+w^{2}K\bigr)\right.
 \\
\qquad\quad {} +U\bigl(q+q^{-1}\bigr){\cG}_{-}(u) \otimes 1+ \bigl(q+q^{-1}\bigr)\bigl(q^2-q^{-2}\bigr)\\
\qquad\quad \quad {}\times
 \left(k_+ w q^{ \h}\bigl(U{\cW}_{+}(u)-
 {\cW}_{-}(u)\bigr)\otimes E K^{\h} + \frac{k_+}{wq^{\h}}\bigl(U {\cW}_{-}(u) - {\cW}_{+}(u)\bigr)\otimes E K^{-\frac{1}{2}}
 \right)\\
 \left. \qquad\quad {}+ \frac{k_+k_-\bigl(q+q^{-1}\bigr)^2}{\bigl(q-q^{-1}\bigr)} 1 \otimes
 \left( \frac{k_+}{k_-}\bigl(q-q^{-1}\bigr)^2 E^2 -
 \bigl(w^{-2}K^{-1}+w^{2}K\bigr) \right) \right] \\
\qquad\quad{}
 + \frac{k_-\bigl(q+q^{-1}\bigr)^2}{q-q^{-1}} U f(u) (1\otimes 1).
 \end{gather*}
 By definition $\delta_w(1) = 1 \otimes 1$, so the above equation fixes $f(u)= U^{-1}/\bigl(q+q^{-1}\bigr)$, and then ${\delta_w(\mathcal{G}_-(u))}$ is also fixed. The same result for $f(u)$ follows from the matrix entry $(1,2)$. From the matrix entries $(1,1)$ and $(2,2)$ of~\eqref{deltp} one finds
 \begin{align*}
 \delta_w(\mathcal{W}_+(u))={}& f(u) \big\lbrack {\cW}_-(u)\otimes\left(\bigl(q-q^{-1}\bigr)^2EF
 -q\bigl(K-K^{- 1}\bigr) \right)
 - \bigl(w^2+w^{-2}\bigr) {\cW}_{+}(u) \otimes 1
 \\
 & + \frac{\bigl(q-q^{-1}\bigr) }{k_+k_-\bigl(q+q^{-1}\bigr)}
 \bigl(k_+q^{\h} {\cG}_{+}(u)\otimes \bigl(w^{- 1}EK^{\h}\bigr)
 +k_-q^{-\h}{\cG}_{-}(u)\otimes \bigl(w FK^{\h}\bigr)\bigr) \\
 & + \bigl(q+q^{-1}\bigr)\bigl( 1 \otimes \bigl(k_+q^{\h}w^{-
 1}EK^{\h}\bigr)+1\otimes\bigl(k_-q^{-\h} w F K^{
 \h}\bigr) \bigr)\\
 &+ U\bigl(q+q^{-1}\bigr){\cW}_{+}(u) \otimes K \big\rbrack.
 \end{align*}
 Then, inserting the power series~\eqref{c1} and~\eqref{c2}, one gets
 \begin{equation} \nonumber
 \delta_w(\normalfont{\tW}_0)= \underbrace{f(u) U\bigl(q+q^{-1}\bigr)}_{=1} \big \lbrack 1 \otimes \bigl(k_+ q^{\h} w^{-1} EK^{\h} + k_- q^{-\h} w F K^{\h}\bigr) + \normalfont{\tW}_0 \otimes K \big \rbrack.
 \end{equation}
 This corresponds to the evaluation of the coaction $\delta(\tW_0)$ given in~\eqref{coW0}. Similarly, from the analysis of $\delta_w(\mathcal{W}_-(u))$, we obtain the evaluation of $\delta(\tW_1)$ given in~\eqref{coW1}.
\end{proof}

The evaluated coaction of the generating functions \eqref{c1} and \eqref{c2} is readily extracted from \eqref{coact-h}
\begin{align}
\delta_w(\cW_\pm(u))={}&
 \frac{U^{-1}}{q+q^{-1}}\left[{\cW}_\mp(u)\otimes\bigl(\bigl(q-q^{-1}\bigr)^2S_\pm S_\mp
 -q\bigl(K^{\pm 1}-K^{\mp 1}\bigr) \bigr)
\right.
 \nonumber\\
 &- \bigl(w^2+w^{-2}\bigr) {\cW}_{\pm}(u) \otimes 1+ \frac{\bigl(q-q^{-1}\bigr) }{k_+k_-\bigl(q+q^{-1}\bigr)}
 \bigl(k_+q^{\pm \h} {\cG}_{+}(u)\otimes \bigl(w^{\mp 1}S_+K^{\pm \h}\bigr)\nonumber \\
 &\left.{}
 +k_-q^{\mp \h}{\cG}_{-}(u)\otimes \bigl(w^{\pm
 1}S_-K^{\pm\h}\bigr)\bigr)\right] + U^{-1}\bigl( 1 \otimes \bigl(k_+q^{\pm \h}w^{\mp
 1}S_+K^{\pm \h}\bigr)\nonumber \\
 &+1\otimes\bigl(k_-q^{\mp \h} w^{\pm 1}S_- K^{\pm
 \h}\bigr) \bigr)+ {\cW}_{\pm}(u) \otimes K^{\pm1}, \label{dWP}\\
 \delta_w(\cG_\pm(u))={}&\frac{k_\mp}{k_\pm}\frac{\bigl(q-q^{-1}\bigr)^2}{\bigl(q+q^{-1}\bigr)}
 U^{-1}{\cG}_\mp(u)\otimes S_\mp^2
 -
 \frac{U^{-1}}{q+q^{-1}}{\cG}_{\pm}(u)\otimes\bigl(w^{-2}K^{\pm1}+w^{2}K^{\mp1}\bigr)
 \nonumber\\
 &+{\cG}_{\pm}(u) \otimes 1+ \bigl(q^2-q^{-2}\bigr)\bigl(k_\mp q^{\mp \h}\bigl({\cW}_{+}(u)-
 U^{-1}{\cW}_{-}(u)\bigr)\otimes \bigl( w^{\mp 1} S_\mp K^{\h}\bigr)
 \nonumber\\
 & +k_\mp q^{\pm
 \h}\bigl({\cW}_{-}(u)-U^{-1}{\cW}_{+}(u)\bigr)\otimes \bigl(w^{\pm 1}S_\mp K^{-\h}\bigr)
 \bigr)
\label{dGP}\\
 &+ \frac{k_+k_-\bigl(q+q^{-1}\bigr)U^{-1}}{\bigl(q-q^{-1}\bigr)} 1 \otimes
 \left( \frac{k_\mp}{k_\pm}\bigl(q-q^{-1}\bigr)^2S_\mp^2 -
 \bigl(w^{-2}K^{\pm1}+w^{2}K^{\mp1}\bigr) \right), \nonumber
\end{align}
where we used the shorthand notation $S_+ \equiv E$, $S_- \equiv F$. We note that these expressions were first obtained in \cite[Proposition~2.2]{BSh1}.\footnote{We corrected typos in \cite{BSh1}, a prefactor was missing.} Expanding~\eqref{dWP},~\eqref{dGP} as power series in
$U^{-1}$, it is straightforward to prove Proposition~\ref{prop:coact}.

Now, using \eqref{dWP}, \eqref{dGP} we can compute the evaluated coaction of the quantum determinant~$\Gamma(u)$ from~\eqref{gamma}
\begin{equation} \label{deltaGam}
 \delta_w(\Gamma(u)) = \frac{1}{\bigl(u^2q+u^{-2}q^{-1}\bigr)\bigl(u^2q^3+u^{-2}q^{-3}\bigr)} \Gamma(u) \otimes \gamma(u/w) \gamma (u w),
\end{equation}
where $\gamma(u)$ is given in~\eqref{eq:qdet}. Here we used the ordering relations of $\mathcal{A}_q$ in Lemma~\ref{rel-PBW-Aq} and the PBW basis of $\Uq$
given by the monomials $\big\{ E^r K^{\pm\frac{s}{2}}F^t \mid r,s,t \in \mathbb{N} \big\}$, see, for instance,~\cite[Chapter~3]{Klimyk}.

The following result is a natural generalization of Lemma \ref{coacK}.

\begin{prop} \label{propeqK2} The evaluated coaction of $\mathcal{K}^{(j)}(u)$ for $j \in \h \mathbb{N}_+$ is given by
 \begin{align}
 (\delta_{w} \otimes \id) \bigl({\mathcal K}^{(j)}(u)\bigr) ={}& \left (\prod_{p=1}^{2j} \frac{U^{-1}}{q+q^{-1}}\big\rvert_{u=u q^{j+\h-p}} \right ) \nonumber
\\
&\times \bigl( {\mathcal L}^{(j)}(u/w) \bigr)_{[\mathsf 2]} \bigl( {\mathcal K}^{(j)}(u)\bigr)_{[\mathsf 1]} \bigl( {\mathcal L}^{(j)}(u w) \bigr)_{[\mathsf 2]}.\label{coactj1}
 \end{align}
\end{prop}
\begin{proof}
 An induction argument is used. For $j=\h$, the relation~\eqref{coactj1} coincides with \eqref{coact-h}. Now, suppose $j\geq1$. For convenience, we omit the notation $[\mathsf 1]$, $[\mathsf 2]$. Expand the left-hand side of~\eqref{coactj1} using the expression of the fused K-operator~\eqref{fusedunormK} and the evaluated coaction~\eqref{coactj1}. It follows
 \begin{align}\nonumber
 \delta_{w} \bigl({\mathcal K}^{(j)}(u)\bigr)={}& \mathcal{F}^{(j)}_\fu \big\lbrack
 \delta_{w} \bigl({\mathcal K}_1^{(\h)}\bigl(uq^{-j+\h}\bigr)\bigr) \rbrack R^{(\h,j-\h)}_{12}\bigl(u^2q^{-j+1}\bigr)\big\lbrack
 \delta_{w} \bigl({\mathcal K}_2^{(j-\h)}\bigl(uq^{\h}\bigr)\bigr) \rbrack \mathcal{E}^{(j)}_\fu\\ \nonumber
={}& \left (\prod_{p=1}^{2j} \frac{U^{-1}}{q+q^{-1}}\big\rvert_{u=u q^{j+\h-p}} \right ) \mathcal{F}^{(j)}_\fu \mathcal{L}_1^{(\h)}\bigl(u w^{-1} q^{-j+\h}\bigr){\mathcal K}_1^{(\h)}\bigl(u q^{-j+\h}\bigr) \\ \nonumber
 &\times \mathcal{L}_1^{(\h)}\bigl(u w q^{-j+\h}\bigr)R_{12}^{(\h,j-\h)}\bigl(u^2 q^{-j+1}\bigr)\mathcal{L}_2^{(j-\h)}\bigl(u w^{-1} q^{\h}\bigr) {\mathcal K}_2^{(j-\h)}\bigl(uq^{\h}\bigr) \\ \label{step1}
 &\times \mathcal{L}_2^{(j-\h)}\bigl(u w q^{\h}\bigr) \mathcal{E}^{(j)}_\fu,
 \end{align} %
 where we used $\delta_w$ instead of $\delta_w \otimes \id$ for convenience. Multiplying on the left and on the right~\eqref{RLL} by \smash{$\bigl(\mathcal{L}_2^{(j)}(v)\bigr)^{-1}$} and using~\eqref{inverse-Lax}, we get
 \begin{gather}
 \label{LRinvL}
 \mathcal{L}_1^{(\h)}(u) R^{(\h,j-\h)}(u/v)\mathcal{L}_2^{(j-\h)}\bigl(v^{-1}\bigr)=\mathcal{L}_2^{(j-\h)}\bigl(v^{-1}\bigr)R^{(\h,j-\h)}(u/v)\mathcal{L}_1^{(\h)}(u).
\end{gather}
 Using~\eqref{LRinvL} for \smash{$u \rightarrow uw q^{-j+\h}$}, \smash{$v\rightarrow u^{-1}w q^{-\h}$} and the commutation \smash{$\big[{\mathcal K}^{(j_1)}_1(u),\mathcal{L}^{(j_2)}_2(v)\big]=0$}, \eqref{step1} becomes
 \begin{gather}
 \qquad= \left (\prod_{p=1}^{2j} \frac{U^{-1}}{q+q^{-1}}\big\rvert_{u=u q^{j+\h-p}} \right ) \mathcal{F}^{(j)}_\fu\mathcal{L}_1^{(\h)}\bigl(uw^{-1} q^{-j+\h}\bigr)\mathcal{L}_2^{(j-\h)}\bigl(u w^{-1} q^{\h}\bigr) {\mathcal K}_1^{(\h)}\bigl(u q^{-j+\h}\bigr) \nonumber\\
 \phantom{\qquad=}{}\times R_{12}^{(\h,j-\h)}\bigl(u^2q^{-j+1}\bigr){\mathcal K}_2^{(j-\h)}\bigl(u q^{\h}\bigr) \mathcal{L}_1^{(\h)}\bigl(u wq^{-j+\h}\bigr) \mathcal{L}_2^{(j-\h)}\bigl(u wq^{\h}\bigr) \mathcal{E}^{(j)}_\fu.\label{find}
 \end{gather}
 On the other hand, inserting \eqref{v2fused-L-uq}, \eqref{fusedunormK}, in the right-hand side of \eqref{coactj1} one gets
 \begin{align*}
\delta_{w} ({\mathcal K}^{(j)}(u))= {}&\left (\prod_{p=1}^{2j} \frac{U^{-1}}{q+q^{-1}}\big\rvert_{u=u q^{j+\h-p}} \right ) \mathcal{F}^{(j)}_\fu \mathcal{L}_1^{(\h)}\bigl(u w^{-1} q^{-j+\h}\bigr) \mathcal{L}_2^{(j-\h)}\bigl(u w^{-1}q^{\h}\bigr) \mathcal{E}^{(j)}_\fu\\ \nonumber
 & \times \mathcal{F}^{(j)}_\fu {\mathcal K}_1^{(\h)}\bigl(u q^{-j+\h}\bigr) R_{12}^{(\h,j-\h)}\bigl(u^2q^{-j+1}\bigr){\mathcal K}_2^{(j-\h)}\bigl(u q^{\h}\bigr)\mathcal{E}^{(j)}_\fu \mathcal{F}^{(j)}_\fu\\ \nonumber
 & \times \mathcal{L}_1^{(\h)}\bigl(u wq^{-j+\h}\bigr) \mathcal{L}_2^{(j-\h)}\bigl(u wq^{\h}\bigr) \mathcal{E}^{(j)}_\fu.
 \end{align*}
 Now, in order to remove the products \smash{$\mathcal{E}^{(j)}_\fu\mathcal{F}^{(j)}_\fu$}, multiply first the expression above from the right by \smash{$\mathcal{H}^{(j)}_\fu \big\lbrack \mathcal{H}^{(j)}_\fu \big\rbrack^{-1}$}. Then, using the relations~\eqref{EHF1}--\eqref{EHF3},~\eqref{annexRK} and
 \[
 R^{(\h,j-\h)}(v/u)\mathcal{L}_2^{(j-\h)}(v) \mathcal{L}_1^{(\h)}(u)=\mathcal{L}_1^{(\h)}(u)\mathcal{L}_2^{(j-\h)}(v)R^{(\h,j-\h)}(v/u),
 \]
 which is obtained from the RLL equation, after simplifications the expression agrees with equation~\eqref{find}.
\end{proof}

\subsection{Twisted intertwining relations for fused K-operators} \label{sec:twisted-rel-fusedK}
In the next section, we will also need the twisted intertwining relations satisfied by the fused K-operators. For the fundamental K-operator \eqref{K-Aq}, the twisted intertwining relations have been given in~\cite[Proposition~4.2]{BSh1}. This result is now extended to higher values of $j$.
\begin{prop}\label{propeqK1}
 The following relation holds for any $j \in \h \mathbb{N}_+$ and all $b\in \mathcal{A}_q$
 \begin{equation} \label{coactRE}
 {\mathcal K}^{(j)}(v) \bigl(\id \otimes \pi^{j} \bigr)\big[\delta_{v^{-1}}(b)\big] = \bigl(\id \otimes \pi^{j} \bigr)[\delta_{v}(b)]{\mathcal K}^{(j)}(v).
 \end{equation}
\end{prop}
\begin{proof}
 Lemma \ref{coacK} implies that
 \begin{equation}\label{coact-ev-j}
 \bigl(\id \otimes \pi^{j}\otimes \id\bigr) \big[(\delta_v \otimes \id) \bigl({\mathcal K}^{(\h)}(u)\bigr)\big] = \frac{U^{-1}}{q+q^{-1}} R^{(j,\h)}(u/v) {\mathcal K}_2^{(\h)}(u)R^{(j,\h)}(u v).
 \end{equation}
 We then notice that ${\mathcal K}^{(j)}(v)$ satisfies the equation
 \begin{equation}\label{refl-eq-new}
 {\mathcal K}^{(j)}_1(v)R^{(j,\h)}(u v) {\mathcal K}_2^{(\h)}(u) R^{(j,\h)}(u/v) =
 R^{(j,\h)}(u/v) {\mathcal K}_2^{(\h)}(u) R^{(j,\h)}(u v) {\mathcal K}^{(j)}_1(v).
 \end{equation}
 This version of reflection equation follows from the standard reflection equation~\eqref{REj1j2} for $j_1=j$, $j_2=\h$ and $u_1=v$, $u_2=u$. Indeed, we multiply~\eqref{REj1j2} on the left and the right by \smash{$R^{(j_1,j_2)}(u_2/u_1)$} and using~\eqref{inverse-R} we obtain~\eqref{refl-eq-new}.

 Now using~\eqref{coact-ev-j}, the equation~\eqref{refl-eq-new} can be rewritten as
 \begin{gather}
 {\mathcal K}_1^{(j)}(v) \bigl(\id \otimes \pi^{j} \otimes \id\bigr) \big[ \bigl(\delta_{v^{-1}} \otimes \id \bigr)\bigl({\mathcal K}^{(\h)}(u)\bigr)\big]\nonumber \\
\qquad= \bigl(\id \otimes \pi^{j} \otimes \id\bigr) \big[\bigl(\delta_v \otimes \id\bigr) \bigl({\mathcal K}^{(\h)}(u)\bigr)\big] {\mathcal K}_1^{(j)}(v).\label{eq:intKpr}
 \end{gather}
 This equation can be thought as an equation in $\mathcal{A}_q\big[\big[u^{-1}\big]\big](\bigl(v^{-1}\bigr))\otimes \End\bigl(\mathbb{C}^{2j+1}\bigr)\otimes \End\bigl(\mathbb{C}^2\bigr)$.
 Denote the entries of the fundamental K-operator by \smash{${\mathcal K}_{m n}^{(\h)}(u) \in {\mathcal A}_q\big[\big[u^{-1}\big]\big]$}, $m,n=1,2$. Now, considering~\eqref{eq:intKpr} as four equations in $\mathcal{A}_q\big[\big[u^{-1}\big]\big]\bigl(\bigl(v^{-1}\bigr)\bigr)\otimes \End\bigl(\mathbb{C}^{2j+1}\bigr)$, i.e., taking the matrix elements of $\End\bigl(\mathbb{C}^2\bigr)$, it yields
 \begin{align*}
 {\mathcal K}^{(j)}(v)\bigl( \bigl( \id \otimes \pi^{j}\bigr) \big[ \delta_{v^{-1}}\bigl( {\mathcal K}_{m n}^{(\h)}(u)\bigr)\big]\bigr) = \bigl( \bigl(\id \otimes \pi^{j}\bigr) \big[ \delta_{v}\bigl( {\mathcal K}_{m n}^{(\h)}(u)\bigr)\big] \bigr) {\mathcal K}^{(j)}(v).
 \end{align*}
 Inserting the entries according to \eqref{K-Aq}, extracting the independent relations and using \eqref{c1}, \eqref{c2}, this implies~\eqref{coactRE}.
\end{proof}

\section[Fused K-operators and the universal K-matrix for A\_q]{Fused K-operators and the universal K-matrix for $\boldsymbol{{\mathcal A}_q}$}\label{sec6}

In this section, we assume there exists a universal K-matrix $\mathfrak{K}$ for ${\mathcal A}_q$. We are interested in the precise relationship between the fused K-operators $\mathcal{K}^{(j)}(u)$ constructed in the previous section using~\eqref{fusedunormK} and the spin-$j$ K-operators defined in~\eqref{evalK} through the evaluation of the universal K-matrix.
Motivated by Lemma \ref{lem-muj} relating spin-$j$ L-operators~\eqref{evalL} and fused L-operators~\eqref{cal-fused-L-uq}, and by supporting evidence discussed below, we propose the following conjecture.

\begin{conj} \label{conj1} For $j\in \h \mathbb{N}$, we have
 \begin{equation}\label{evalKj}
 {\bf K}^{(j)}(u) = \nu^{(j)}(u) {\mathcal K}^{(j)}(u) \in {\mathcal A}_q \big[\big[u^{-1}\big]\big]\otimes \End\bigl(\mathbb{C}^{2j+1}\bigr),
 \end{equation}
 where $\mathcal{K}^{(j)}(u)$ is defined in~\eqref{fusedunormK} with
 \begin{align} \label{exp-nujh}
 \nu^{(j)}(u) &= \left ( \prod_{m=0}^{2j-1} \nu\bigl(uq^{j-\h-m}\bigr) \right) \left ( \prod_{k=0}^{2j-2} \prod_{\ell=0}^{2j-k-2} \pi^{\h} \bigl(\mu \bigl(u^2q^{2j-2-2k-\ell}\bigr)\bigr) \right ).
 \end{align}
 Here \smash{$\pi^{\h}(\mu(u))$} is given by~\eqref{pimu} and \smash{$\nu(u)\equiv \nu^{(\h)}(u)$} is an invertible central element in \smash{${\mathcal A}_q\big[\big[u^{-1}\big]\big]$},
 defined by the functional relation
 \begin{equation} \label{funct-nu}
 \pi^{\h}\bigl(\mu\bigl(u^2 q\bigr)\bigr) \nu(u) \nu(u q) \Gamma(u) = 1,
 \end{equation}
 where $\Gamma(u)$ is given in~\eqref{gamma},
 and has the evaluated coaction
 \begin{equation}
 \delta_w(\nu(u)) = \bigl(u^2q+u^{-2}q^{-1}\bigr)\nu(u)\otimes\mu(u/w) \mu(u w).\label{dnuw}
 \end{equation}
\end{conj}

We note that a solution $\nu(u)$ to~\eqref{funct-nu} can be uniquely (up to a sign of its constant term) determined by recursion provided in the next lemma. Here, we assume that $\nu(u)$ is invertible and generates a commutative subalgebra in ${\mathcal A}_q\bigl(\bigl(u^{-1}\bigr)\bigr)$.
In what follows, we use the following notation
\begin{equation}\label{eq:Fu}
 F(u) := \pi^{\h}\bigl(\mu\bigl(u^2 q\bigr)\bigr) \Gamma(u).
\end{equation}
We observe that only even modes of $F(u)$ are non-zero, i.e.,
\begin{gather}
F(u) = \sum_{k=0}^\infty F_k u^{-2k}.\label{eq:modeF}
\end{gather}
Recall the central elements from Example~\ref{ex:delta}, in terms of which the first modes of $F(u)$ read
\begin{gather}
F_0= -\frac{\rho q^{\h}}{\bigl(q-q^{-1}\bigr)^2},\qquad F_1 = \frac{q^{\h}c_1}{2\bigl(q-q^{-1}\bigr)}\Delta_1,\qquad F_2 = \frac{q^{\h}c_2}{2\bigl(q-q^{-1}\bigr)}\Delta_2 - \frac{\rho q^{\h}}{\bigl(q^4-1\bigr)}, \label{eq:Fi}
\end{gather}
with $c_k$ defined in~\eqref{eq:ck}.

\begin{lem}\label{lem:nu}
 The generating function \smash{$\nu(u)=\nu^{(\h)}(u)$} from~\eqref{evalKj} is central in ${\mathcal A}_q\big[\big[u^{-1}\big]\big]$ and takes the following form:
\begin{gather}
 \nu(u) = \sum_{k=0}^\infty \nu_k u^{-2k},\label{eq:nupow}
\end{gather}
 where $\nu_0$ is fixed up to a sign by
 \begin{equation}\label{eq:nu0}
 \nu_0^2 = - \rho^{-1}q^{-\h}\bigl(q-q^{-1}\bigr)^2,
 \end{equation}
 and all $\nu_k$ with $k\geq1$ are determined recursively by
\begin{gather}\label{eq:nuk}
 \nu_{k} = - \frac{1}{F_0}\sum_{\ell=0}^{k-1}\frac{1-q^{-2(k+\ell)}}{1-q^{-4k}}F_{k-\ell} \nu_{\ell}.
\end{gather}
\end{lem}

\begin{proof}
 With the notation~\eqref{eq:Fu}, the functional relation~\eqref{funct-nu} takes the form
 \beq\label{eq:Fnunu-1}
 F(u) \nu(u) \nu(uq) = 1.
 \eeq
 Replacing $u \to uq$, we get from it another equation
 \beq\label{eq:Fnunu-2}
 F(uq) \nu(uq) \nu\bigl(uq^2\bigr) = 1.
 \eeq
 Comparing both the equations~\eqref{eq:Fnunu-1} and~\eqref{eq:Fnunu-2}, and using that by assumption $[\nu(u),\nu(uq)]=0$ and $\nu(u)$ is invertible, we get
\begin{gather}
 F(uq) \nu\bigl(uq^2\bigr) = F(u) \nu(u).\label{eq:Fnu}
\end{gather}

 Assuming $\nu(u) \in {\mathcal A}_q\bigl(\bigl(u^{-1}\bigr)\bigr)$ with an expansion $\nu(u) = \sum_{k=-M}^\infty \nu'_k u^{-k}$ for some positive $M$ and using~\eqref{eq:Fnu} with~\eqref{eq:modeF} we conclude by induction that $M$ has to be zero.
 We can then assume that $\nu(u) = \sum_{k=0}^\infty \nu'_k u^{-k}$. Inserting (\ref{eq:modeF}) in the above equation~\eqref{eq:Fnu}, it is equivalent~to
 \begin{gather*}
 \sum_{p=0}^\infty u^{-p} \sum_{k=0}^{\lfloor \frac{p}{2} \rfloor} \bigl(q^{-2p+2k} -1\bigr)F_k \nu'_{p-2k}=0 \\
 \qquad {} \Leftrightarrow \ \sum_{k=0}^{\lfloor \frac{p}{2} \rfloor} \bigl(q^{-2p+2k} -1\bigr)F_k \nu'_{p-2k}=0,\qquad p=0,1,2,\ldots\,.
 \end{gather*}
 Let $n\in{\mathbb N}$. For $p=2n+1$, one gets
 \begin{gather*}
 \nu'_{2n+1} = - \frac{1}{F_0}\sum_{k=1}^{n}\left( \frac{1-q^{-2(2n+1) + 2k}}{1-q^{-2(2n+1)}}\right)F_k \nu'_{2n+1-2k}.
 \end{gather*}
 For $n=0$, it gives $\nu'_1=0$ which implies by induction that all modes $\nu'_{2n+1}=0$.
 Thus, the power series $\nu(u)$ is of the form (\ref{eq:nupow}). From~\eqref{eq:Fnunu-1} we readily see that $\nu_0$ is fixed as in~\eqref{eq:nu0}. Then, the recurrence relation that follows from (\ref{eq:Fnu}) with (\ref{eq:nupow}) and (\ref{eq:modeF}) is now given by
 \[
 \sum_{l=0}^p \bigl(q^{-2(p+l)} - 1\bigr) F_{p-l} \nu_l = 0,
 \]
 which determines $\nu_l$ as in~\eqref{eq:nuk}.
\end{proof}

\begin{Example}\label{ex:nu}
 Together with~\eqref{eq:Fi} and \eqref{delta1} and \eqref{delta2}, we get explicit expression of the first few modes of $\nu(u)$ developed in~\eqref{eq:nupow} in terms of the alternating generators of ${\mathcal A}_q$
 \begin{gather}
 \nu_1 = \frac{\bigl(q-q^{-1}\bigr)^2 q^{\h} \nu_0}{\rho \bigl(q+q^{-1}\bigr)}F_1,\nonumber\\
 \nu_2 = \frac{\bigl(q-q^{-1}\bigr)^2q^{3/2}\nu_0}{\rho\bigl(q^2+q^{-2}\bigr)}\left(F_2 + q^{-\h}\frac{\bigl(q-q^{-1}\bigr) \bigl(q^3-q^{-3}\bigr)}{\rho\bigl(q+q^{-1}\bigr)^2}F_1^2 \right),\label{eq:nu-ex}
 \end{gather}
 where $\nu_0$ is fixed as a square root of~\eqref{eq:nu0}.
\end{Example}

Let us now show that \smash{${\mathbf{K}}^{(j)}(u) \in \mathcal{A}_q\big[\big[u^{-1}\big]\big]\otimes \End\bigl(\mathbb{C}^{2j+1}\bigr)$}.
First, we note that the leading term of \smash{$\nu^{(j)}(u)$} given in~\eqref{exp-nujh} is at \smash{$u^{-4j^2+2j}$}, using \eqref{eq:nupow} and~\eqref{pimu}. Next, from the definition of \smash{$\mathcal{K}^{(j)}(u)$} in~\eqref{fusedunormK}, we find that its leading term is at \smash{$u^{4j^2-2j}$}, using the fact that the leading terms of \smash{$\mathcal{K}^{(\h)}(u)$} and \smash{$R^{(\h,j)}(u)$} from~\eqref{R-Rqg} are respectively at $u^0$ and $u^{2j}$.
Finally, we conclude that \smash{$\nu^{(j)}(u) \mathcal{K}^{(j)}(u)$} is in $\mathcal{A}_q\big[\big[u^{-1}\big]\big] \otimes \End\bigl(\mathbb{C}^{2j+1}\bigr)$ with its leading term at $u^0$.

Below, we present supporting evidence for Conjecture~\ref{conj1}, and then we derive from this conjecture certain properties of the fused K-operators for $j\geq0$.

\subsection{Supporting evidence}

For the clarity of the presentation, let us define
\begin{equation}
 \tilde{{\bf K}}^{(j)}(u) = \nu^{(j)}(u) {\mathcal K}^{(j)}(u) \qquad \mbox{for}\ j\in \h{\mathbb N},\label{evalKtildej}
\end{equation}
where we assume \smash{$\nu^{(\h)}(u)$} is an invertible central element in ${\mathcal A}_q\big[\big[u^{-1}\big]\big]$.
Importantly, it is not assumed that \smash{$\tilde{{\bf K}}^{(j)}(u)$} is obtained from the evaluation of a universal K-matrix.

We provide supporting evidence for Conjecture~\ref{conj1}.
We show that $\tilde{{\bf K}}^{(j)}(u)$ for all $j$ satisfy the following systems of equations
\begin{gather} \label{eqK1} \tag{K1$'$}
 \tilde{{\bf K}}^{(j)}(v) \bigl(\id \otimes \pi^{j} \bigr)\big[\delta_{v^{-1}}(b)\big] = \bigl(\id \otimes \pi^{j} \bigr)[\delta_{v}(b)]\tilde{{\bf K}}^{(j)}(v), \\
 \label{eqK2} \tag{K2$'$}
 (\delta_w \otimes \id) \bigl( \tilde{{\bf K}}^{(j)}(u) \bigr)= \bigl( {\bf L}^{(j)}(u/w) \bigr)_{\textsf{[2]}} \bigl( \tilde{{\bf K}}^{(j)}(u) \bigr)_{\textsf{[1]}} \bigl( {\bf L}^{(j)}(u w) \bigr)_{\textsf{[2]}}, \\
 \tilde{\bf K}^{(j)}(u) = \mathcal{F}^{(j)}_\fu \tilde{\bf K}_1^{(\h)}\bigl(u q^{-j+\h}\bigr) \mathcal{R}^{(\h,j-\h)}\bigl(u^2 q^{-j+1}\bigr) \tilde{\bf K}_2^{(j-\h)}\bigl(u q^{\h}\bigr) \mathcal{E}^{(j)}_\fu, \label{tildeKrel1}\tag{K3$'$}
\end{gather}
where $\mathcal{R}^{(\h,j)}(u)$ is given in~\eqref{fused-R-uq}, if and only if~\eqref{exp-nujh} and~\eqref{dnuw} hold.
Here, \eqref{eqK1}--\eqref{tildeKrel1} are direct analogs of~\eqref{evalbfK1},~\eqref{delwK},~\eqref{fusedevalKv2}, respectively.
We will also show that $\tilde{{\bf K}}^{(j)}(u)$ satisfies the reflection equation~\eqref{evalpsi} where ${\bf K}^{(j)}(u)$ is replaced by $\tilde{\bf K}^{(j)}(u)$.

\subsubsection[Fusion relation (K3'), intertwining relations (K1') and evaluated coaction (K2')]{Fusion relation~(\ref{tildeKrel1}), intertwining relations~(\ref{eqK1})\\ and evaluated coaction~(\ref{eqK2})}
 We assume that $\nu(u)$ is an invertible central element in ${\mathcal A}_q\big[\big[u^{-1}\big]\big]$.
\begin{lem}\label{lem:nu-fact} The K-operators
 \smash{$\tilde{\bf K}^{(j)}(u)$} for all $j \in \h \mathbb{N}$ satisfy the fusion relation~\eqref{tildeKrel1}
 if and only if \smash{$\nu^{(j)}(u)$} takes the form \eqref{exp-nujh}.
\end{lem}
\begin{proof}
 Note first that~\eqref{fusedunormK} contains \smash{$R^{(\h,j)}(u)$} while~\eqref{tildeKrel1} contains \smash{$\mathcal{R}^{(\h,j)}(u)$}, and they are related~by
 \begin{equation} \label{calRsR}
 \mathcal{R}^{(\h,j)}(u)= \left ( \prod_{k=0}^{2j-1} \pi^{\h} \bigl(\mu\bigl(u q^{-j+\h+k}\bigr) \bigr) \right ) R^{(\h,j)}(u)
 \end{equation}
 due to~\eqref{evalRhj} and~\eqref{exp-mujh}.
 Then, inserting~\eqref{evalKtildej} in~\eqref{tildeKrel1} and using~\eqref{fusedunormK} and~\eqref{calRsR},
 the assumption that $\nu(u)$ is an invertible central element in $\mathcal{A}_q\big[\big[u^{-1}\big]\big]$, and that $\mathcal{K}^{(j)}(u)$ is invertible, see Remark~\ref{rem-invertKj}, the resulting recursion relation on \smash{$\nu^{(j)}(u)$} takes the following form:
 \[
 \nu^{(j + \h)}(u) = \nu\bigl(u q^{-j}\bigr) \nu^{(j)}\bigl(u q^{\h}\bigr) \pi^\h \bigl( \mu^{(j)}\bigl(u^2 q^{-j+\h}\bigr) \bigr),
 \]
 with $\mu^{(j)}(u)$ defined in~\eqref{exp-mujh}. This recursion gives
 the solution~\eqref{exp-nujh}.
\end{proof}

In what follows, we will assume that~\eqref{tildeKrel1} holds, so in particular all $\nu^{(j)}(u)$ are central. Then, using Proposition~\ref{propeqK1}, the twisted intertwining relation for $\tilde{{\bf K}}^{(j)}(u)$ is immediate.

\begin{lem}
 The K-operators $\tilde{\mathbf{K}}^{(j)}(u)$ for all $j \in \h \mathbb{N}$ satisfy~\eqref{eqK1}.
\end{lem}

We now show that the evaluated coaction~\eqref{eqK2} holds for $\tilde{\bf K}^{(j)}(u)$.
\begin{lem}The \smash{$\tilde{\bf K}^{(j)}(u)$} for all $j \in \h \mathbb{N}$ satisfy~\eqref{eqK2}
 if and only if $\delta_w(\nu(u))$ is given by~\eqref{dnuw}.
\end{lem}
\begin{proof}
 We first consider the case $j=\h$. Inserting \eqref{evalKtildej} in~\eqref{eqK2} for $j=\h$, using \eqref{coact-h}, and the invertibility of \smash{$\mathcal{L}^{(\h)}(u)$}, \smash{$\mathcal{K}^{(\h)}(u)$},
 the resulting equation is equivalent to \eqref{dnuw}.
 It remains to check that given the evaluated coaction~\eqref{dnuw},~\eqref{eqK2} holds for higher $j$. Now, consider $j\geq 1$. On~one hand, inserting~\eqref{evalKtildej} in~\eqref{eqK2}, the left-hand side reads
 \begin{gather}
 \delta_w \bigl(\nu^{(j)}(u) \bigr) \delta_w(\mathcal{K}^{(j)}(u))\nonumber \\
\qquad= \left ( \prod_{m=0}^{2j-1} \delta_w\bigl(\nu\bigl(uq^{j-\h-m}\bigr)\bigr) \right) \left ( \prod_{k=0}^{2j-2} \prod_{\ell=0}^{2j-k-2} \pi^{\h} \bigl(\mu \bigl(u^2q^{2j-2-2k-\ell}\bigr)\bigr) \right ) \nonumber\\
 \phantom{\qquad=}{} \times \left (\prod_{p=1}^{2j} \frac{U^{-1}}{q+q^{-1}} \big\rvert_{u=u q^{j+\h-p}} \right ) \bigl( {\mathcal L}^{(j)}(u/w) \bigr)_{[\mathsf 2]} \bigl( {\mathcal K}^{(j)}(u)\bigr)_{[\mathsf 1]} \bigl( {\mathcal L}^{(j)}(u w) \bigr)_{[\mathsf 2]},\label{lhsK2}
 \end{gather}
 where we used~\eqref{coactj1},~\eqref{exp-nujh} and the fact that $\delta_w$ in an algebra homomorphism. On the other hand, the right-hand side of~\eqref{eqK2} is
 \begin{equation} \label{rhsK2}
 \nu^{(j)}(u) \otimes \mu^{(j)}(u/w) \mu^{(j)}(uw) \bigl( {\mathcal L}^{(j)}(u/w) \bigr)_{[\mathsf 2]} \bigl( {\mathcal K}^{(j)}(u)\bigr)_{[\mathsf 1]} \bigl( {\mathcal L}^{(j)}(u w) \bigr)_{[\mathsf 2]}.
 \end{equation}
 Then, replacing $\mu^{(j)}(u)$ and $\nu^{(j)}(u)$ in~\eqref{rhsK2} by~\eqref{exp-mujh},~\eqref{exp-nujh} respectively, and using~\eqref{dnuw} in~\eqref{lhsK2}, we find that~\eqref{lhsK2} and~\eqref{rhsK2} are equal.
\end{proof}

Finally, assuming~\eqref{tildeKrel1} holds, so that $\nu^{(j)}(u)$ are central by Lemma~\ref{lem:nu-fact}, we show that the fused K-operators $\tilde{\bf K}^{(j)}(u)$ satisfy the reflection equation.
\begin{lem}\label{REtildeK}
 The K-operators $\tilde{\bf K}^{(j)}(u)$ satisfy the reflection equation~\eqref{evalpsi} where ${\bf K}^{(j)}(u)$ is replaced by $\tilde{\bf K}^{(j)}(u)$ for any $j_1,j_2 \in \h \mathbb{N}_+$.
\end{lem}

\begin{proof}
 By Theorem~\ref{prop:fusedRE}, the fused K-operators $\mathcal{K}^{(j)}(u)$ satisfy the equation~\eqref{REj1j2}. Then, multiplying this equation by $\nu^{(j_1)}(u) \nu^{(j_2)}(v)$ and using the fact that they are central, we obtain~\eqref{evalpsi} where ${\bf K}^{(j)}(u)$ is replaced by $\tilde{\bf K}^{(j)}(u)$.
\end{proof}

\subsubsection[Functional relation on nu(u)]{ Functional relation on $\boldsymbol{\nu(u)}$}
We have seen in Lemma~\ref{lem:nu-fact} that the relation~\eqref{tildeKrel1} fixes the normalization factor $\nu^{(j)}(u)$ as~\eqref{exp-nujh}. Here we show that the analog of the reduction relation~\eqref{fusedevalbarKv2} for $\tilde{\bf K}^{(j)}(u)$
leads to the functional relation~\eqref{funct-nu}.
Recall the functional relation on $\mu(u)$ in~\eqref{eq:mu} was obtained by comparing the fusion relation with the reduction relation satisfied by the spin-$j$ L-operators, see Proposition~\ref{lem-mu}. We proceed similarly for $\tilde{\bf K}^{(j)}(u)$.

\begin{prop} The K-operators $\tilde{\bf K}^{(j)}(u)$ satisfy~\eqref{tildeKrel1} and
 \begin{equation}
 \tilde{\bf K}^{(j-\h)}(u)= \bar{\mathcal{F}}_\fu^{(j-\h)} \tilde{\bf K}_1^{(\h)}\bigl(u q^{j+1}\bigr) \mathcal{R}^{(\h,j)}\bigl(u^2 q^{j+\tha}\bigr)\tilde{\bf K}_2^{(j)}\bigl(u q^{\h}\bigr) \bar{\mathcal{E}}_\fu^{(j-\h)}, \label{tildeKrel2}\tag{K3$''$}
 \end{equation}
 for $j=1$ if and only if $\nu(u)$ satisfies the functional relation~\eqref{funct-nu}.
\end{prop}

\begin{proof}
 The equation (\ref{tildeKrel2}) for $j=1$ in terms of fused K-operators reads as
 \begin{align}
 \mathcal{K}^{(\h)}(u) ={}& \nu(u q) \nu \bigl(u q^2\bigr) \pi^{\h}\bigl(\mu\bigl(u^2 q\bigr) \mu\bigl(u^2 q^2\bigr) \mu\bigl(u^2q^3\bigr)\bigr)\nonumber\\
 &\times \bar{\mathcal{F}}^{(\h)}_\fu \mathcal{K}_1^{(\h)}\bigl(u q^2\bigr) R^{(\h,1)}\bigl(u^2 q^{\frac{5}{2}}\bigr) \mathcal{K}_2^{(1)}\bigl(u q^\h\bigr) \bar{\mathcal{E}}_\fu^{(\h)},\label{step1-pr}
 \end{align}
 where we used \smash{$\tilde{{\bf K}}^{(\h)}(u) = \nu(u) {\mathcal K}^{(\h)}(u)$}, the factorized form~\eqref{exp-nujh} for $\nu^{(1)}(u)$ due to Lemma~\ref{lem:nu-fact} and we used \smash{$\mathcal{R}^{(\h,1)}(u)=\pi^\h(\mu\bigl(u q^\h\bigr)\mu\bigl(u q^{-\h}\bigr)) R^{(\h,1)}(u)$}, recall~\eqref{calRsR}. From the relation satisfied by $\mu(u)$ given in~\eqref{eq:mu} and using $\pi^{j}(C) = q^{2j+1}+q^{-2j-1}$, one gets
 \begin{align}
 \label{eval-mumu}
 \pi^{\h}(\mu(u) \mu(u q)) = \frac{1}{c(u)c\bigl(uq^2\bigr)},
 \end{align}
 where $c(u)$ is given in~\eqref{eq:cu} and by our standard convention the right-hand side is developed as power series in $u^{-1}$. Then, the equation~\eqref{step1-pr} becomes
\begin{align} \label{gen-nunu}
 \mathcal{K}^{(\h)}(u) &=\frac{\nu(u q) \nu \bigl(u q^2\bigr) \pi^{\h}\bigl( \mu\bigl(u^2 q^3\bigr) \bigr) }{c\bigl(u^2q\bigr)c\bigl(u^2q^3\bigr)}\bar{\mathcal{F}}^{(\h)}_\fu \mathcal{K}_1^{(\h)}\bigl(u q^2\bigr) R^{(\h,1)}\bigl(u^2 q^{\frac{5}{2}}\bigr) \mathcal{K}_2^{(1)}\bigl(u q^\h\bigr) \bar{\mathcal{E}}_\fu^{(\h)}.
 \end{align}
 The right-hand side of~\eqref{gen-nunu} is now computed using the expressions for \smash{$\mathcal{K}^{(\h)}(u)$}, $\mathcal{K}^{(1)}(u)$ given respectively in~\eqref{K-Aq}, \eqref{expK-spin1}, the fused R-matrix \eqref{mat:Rh1} and \smash{$\bar{\mathcal{E}}_\fu^{(\h)}$}, \smash{$\bar{\mathcal{F}}_\fu^{(\h)}$} given by
 \begin{equation*}
 \bar{\mathcal{E}}_\fu^{(\h)}=\begin{pmatrix}
 0 & 0 \\
 1 & 0\\
 0 & \sqrt{[2]_q}\\
 -\sqrt{[2]_q} &0 \\
 0& -1 \\
 0 &0
 \end{pmatrix}, \qquad \bar{\mathcal{F}}_\fu^{(\h)}=\begin{pmatrix}
 0& \frac{1}{1+[2]_q} & 0 & - \frac{\sqrt{[2]_q}}{1+[2]_q} &0 &0 \\
 0&0 &\frac{\sqrt{[2]_q}}{1+[2]_q} &0 &-\frac{1}{1+[2]_q} &0
 \end{pmatrix}.
 \end{equation*}
 In terms of the quantum determinant \eqref{gamma}, one finds
 \begin{equation}\label{redj1}
 \bar{\mathcal{F}}_{\fu}^{(\h)} {\mathcal K}_1^{(\h)}\bigl(u q^2\bigr) R^{\left (\h,1 \right)}\bigl(u^2 q^{\frac{5}{2}}\bigr) {\mathcal K}_2^{(1)}\bigl(u q^\h\bigr) \bar{\mathcal{E}}_\fu^{(\h)} = c\bigl(u^2 q\bigr) c\bigl(u^2q^3\bigr) \Gamma(u q) {\mathcal K}^{(\h)}(u).
 \end{equation}
 This relation is obtained by applying the ordering relations for $\mathcal{A}_q$ given in Appendix~\ref{apB}.
 Inserting this expression in \eqref{gen-nunu} and multiplying by \smash{$\big\lbrack{\mathcal K}^{(\h)}(u)\big\rbrack^{-1}$} given in~\eqref{rem-invKh} the relation~\eqref{funct-nu} follows.\looseness=1
\end{proof}

\begin{rem}
 As a consistency check, we observe the following:
 \begin{enumerate}\itemsep=0pt
 \item[(1)]
 The evaluated coaction in~\eqref{dnuw} respects the functional relation~\eqref{funct-nu} on $\nu(u)$. Indeed, using (\ref{deltaGam}) and the functional relations~\eqref{eq:mu} and~\eqref{funct-nu}, we obtain
 \begin{equation*}
 \delta_w( \nu(u) ) \delta_w(\nu(u q)) \delta_w(\Gamma(u)) \pi^{\h}\bigl(\mu\bigl(u^2 q\bigr)\bigr) = 1 \otimes 1.
 \end{equation*}

 \item[(2)]
 The expression for $\delta_w(\nu(u))$ in~\eqref{dnuw} agrees with the direct calculation of $\delta_w(\nu_1)$.
 To see it, we first recall the expression of $\nu_1$ in~\eqref{eq:nu-ex} together with~\eqref{eq:Fi} and~\eqref{delta1}. With the expressions from Proposition~\ref{prop:coact}, we need to calculate only
 the evaluated coaction on $\tilde{\tG}_1$. This is the 1st mode of the current $\cG_-(u)$, recall~\eqref{c2}, and thus $\delta_w\big(\tilde{\tG}_1\big)$ is extracted from the 1st mode of~\eqref{dGP}
 \begin{align*}
 \delta_w ( \normalfont{\tilde{\tG}}_1)= {}&\frac{k_+ k_- \bigl(q+q^{-1}\bigr)}{q-q^{-1}} 1 \otimes \left ( \frac{k_+}{k_-} \bigl(q-q^{-1}\bigr)^2 E^2 - \bigl(w^{-2} K^{-1} + w^2 K\bigr) \right )
 \\
 &+ \normalfont{\tilde \tG}_1 \otimes 1 + \bigl(q^2-q^{-2}\bigr) k_+ \bigl( q^{\h} {\tW}_0 \otimes \bigl(w E K^\h\bigr) + q^{-\h} {\tW}_1 \otimes \bigl(w^{-1} E K^{-\h}\bigr) \bigr).
 \end{align*}
 Combining all these expressions of $\delta_w$ together as in~\eqref{delta1}, using that $\delta_w$ is an algebra map for any $w$, and simplifying, we obtain
 \[
 \delta_w (\Delta_1) = \Delta_1 \otimes 1 - \frac{2 k_+ k_- \bigl(w^2 +w^{-2}\bigr) }{q-q^{-1}} 1 \otimes C,\]
 where $C$ is the Casimir element~\eqref{Cas-Uqsl2}.
 Using~\eqref{eq:nu-ex} together with the expression of $F_1$ in~\eqref{eq:Fi}, we get the final result
 \[
 \delta_w(\nu_1) = \nu_1 \otimes 1 + \frac{q^{-1} \nu_0}{q+q^{-1}} \bigl(w^2+w^{-2}\bigr) \otimes C,\]
 which indeed agrees with the term at $u^{-2}$ in the expansion of~\eqref{dnuw}.
 \end{enumerate}
\end{rem}

In summary, the evidence supporting Conjecture~\ref{conj1} can be summarized in Figure~\ref{fig:conj}.

\begin{figure}[!ht]\centering
 \begin{tikzpicture}[scale=1.1, every node/.style={scale=0.8}]
 \draw (-2,0);
 \draw (0,0) node[]{$\mathfrak{K}$};
 \draw (0,-0.5) node[]{\eqref{univK1}};
 \draw (0,-1) node[]{\eqref{univK2}};
 \draw (0,-1.5) node[]{\eqref{univK3}};

 \draw[->] (0.5,-0.5) -- (2.5,-0.5) node[midway,above]{evaluation};
 \draw[->] (0.5,-1) -- (2.5,-1) node[midway,above]{evaluation};
 \draw[->] (0.5,-1.5) -- (2.5,-1.5) node[midway,above]{evaluation};

 \draw (3.2,0) node[]{${\bf K}^{(j)}(u)$};

 \draw (3.2,-0.5) node[]{\eqref{evalbfK1}};
 \draw (3.2,-1) node[]{\eqref{delwK}};
 \draw (3.2,-1.5) node[]{\eqref{fusedevalKv2}};

 \draw [<->] (3.9,-0.5) -- (6.2,-0.5) node[midway,above] {\textcolor{black}{Conjecture \ref{conj1}}};
 \draw (7,0) node[]{ $\tilde{\bf K}^{(j)}(u)$};

 \draw [<->] (3.9,-1) -- (6.2,-1) node[midway,above] {\textcolor{black}{Conjecture \ref{conj1}}};

 \draw [<->] (3.9,-1.5) -- (6.2,-1.5) node[midway,above]{\textcolor{black}{Conjecture \ref{conj1}}};

 \draw (7,-0.5) node[]{\eqref{eqK1}};
 \draw (7,-1) node[]{\eqref{eqK2}};
 \draw (7,-1.5) node[]{\eqref{tildeKrel1}};

 \draw (0,-2.25) node[]{(URE)};
 \draw[->] (0.5,-2.25) -- (2.5,-2.25) node[midway,above]{evaluation};
 \draw (3.2,-2.25) node[]{(RE)};
 \draw [<->] (3.9,-2.25) -- (6.2,-2.25) node[midway,above] {Conjecture \ref{conj1}};
 \draw (7,-2.25) node[]{(RE)};

 \draw[line width=0.25mm] (-0.3,-0.5) -- (-0.5,-0.5) -- (-0.5,-1.5) -- (-0.3,-1.5);
 \draw[line width=0.25mm] (-0.3,-1) -- (-0.7,-1) -- (-0.7,-2.25);
 \draw[line width=0.25mm,->] (-0.7,-2.25) --(-0.45,-2.25);

 \node[black,dashed,
 draw = red,
 minimum width = 1.5cm,
 minimum height = 2.9cm] (r) at (7,-1.35) {};
 \draw[color=red] (7.8,-0.5) node[]{P};
 \draw[color=red] (7.8,-0.8) node[]{r};
 \draw[color=red] (7.8,-1.1) node[]{o};
 \draw[color=red] (7.8,-1.4) node[]{v};
 \draw[color=red] (7.8,-1.7) node[]{e};
 \draw[color=red] (7.8,-2) node[]{n};
 \end{tikzpicture}
 \caption{Supporting evidence for Conjecture~\ref{conj1}, where the double-sided arrows connecting two equations signify their equality under the assumption that Conjecture~\ref{conj1} is true.}
 \label{fig:conj}
\end{figure}
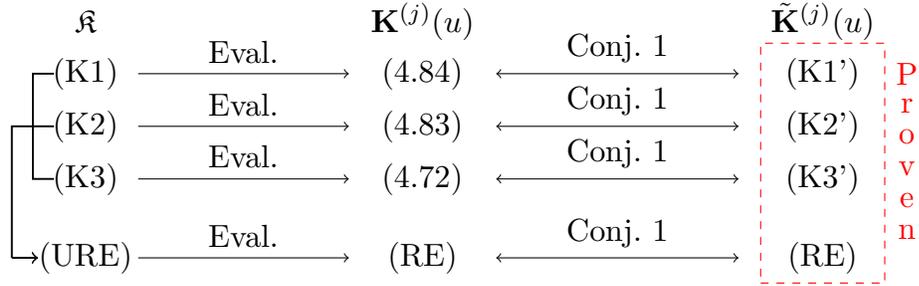

\subsection{Coaction} \label{sec:subcoac}
Motivated by the discussion in Section~\ref{sec:comod} and Proposition~\ref{prop:coac},
we propose a right coaction for the components of $\tilde{{\bf K}}^{(\h)}(u)$
\begin{equation} \label{deltatildeK}
 (\delta \otimes \id)\bigl( \tilde{\bf K}^{(\h)}(u)\bigr)= \bigl( \big\lbrack {\bf L}^{-}\bigl(u^{-1}\bigr) \big\rbrack^{-1} \bigr)_{[2]} \bigl( \tilde{\bf K}^{(\h)}(u) \bigr)_{[1]} \bigl( {\bf L}^+(u) \bigr)_{[2]},
\end{equation}
where ${\bf L}^\pm(u)$ are defined in~\eqref{defLpm} and are calculated in~\eqref{app:expLP} and~\eqref{app:expLM}, respectively.
First of all, we show that the proposed coaction \eqref{deltatildeK} indeed respects the relations satisfied by the components of \smash{$\tilde{\bf K}^{(\h)}(u)$}.
Recall that, due to Lemmas~\ref{REtildeK} and~\ref{lem-Rhj}, these relations are
\begin{equation}
 R^{(\h,\h)}(u/v) \tilde{\bf K}^{(\h)}_1(u) R^{(\h,\h)}(u v) \tilde{\bf K}^{(\h)}_2(v) = \tilde{\bf K}^{(\h)}_2(v) R^{(\h,\h)}(u v) \tilde{\bf K}^{(\h)}_1(u) R^{(\h,\h)}(u/v).\label{KpmRE}
\end{equation}
\begin{lem}\label{lem-dress-RE}

 The following element from \smash{$({\mathcal A}_q \otimes \Loop)\big[\big[u^{-1}\big]\big]\otimes \End\bigl(\mathbb{C}^2\bigr)$} and defined by
 \begin{equation}\label{dress-K}
 \tilde{\bf K}^{(-,+)}(u)= \bigl( \lbrack {\bf L}^{-}\bigl(u^{-1}\bigr) \rbrack^{-1} \bigr)_{[2]} \bigl( \tilde{\bf K}^{(\h)}(u) \bigr)_{[1]} \bigl( {\bf L}^+(u) \bigr)_{[2]}
 \end{equation}
 satisfies the reflection equation \eqref{KpmRE} where \smash{$\tilde{\bf K}^{(\h)}(u)$} is replaced by $ \tilde{\bf K}^{(-,+)}(u)$.
\end{lem}
\begin{proof}
 We first substitute the K-operators in~\eqref{dress-K} into the reflection equation~\eqref{KpmRE}. Then, we multiply both sides from the left by $ {\bf L}_1^-\bigl(u^{-1}\bigr) {\bf L}_2^-\bigl(v^{-1}\bigr)$ and from the right by $\smash{ \big\lbrack {\bf L}_2^{+}(v) \big\rbrack^{-1}}\allowbreak\times \smash{\big\lbrack {\bf L}_1^{+}(u) \big\rbrack^{-1}}$. One has for the left-hand side of the resulting equation
 \begin{gather}
 {\bf L}_1^{-}\bigl(u^{-1}\bigr) \underline{{\bf L}_2^{-}\bigl(v^{-1}\bigr) R^{(\h,\h)}(u/v) \lbrack {\bf L}_1^{-}\bigl(u^{-1}\bigr) \rbrack^{-1} } \tilde{\bf K}^{(\h)}_1(u) {\bf L}_1^{+}(u) R^{(\h,\h)}(u v) \lbrack {\bf L}^{-}_2\bigl(v^{-1}\bigr) \rbrack^{-1}\nonumber\\
\qquad {} \times \underline{\tilde{\bf K}_2^{(\h)}(v) \lbrack {\bf L}_1^{+}(u) \rbrack ^{-1} } \nonumber\\ \nonumber
 \phantom{\qquad \times}{} =R^{(\h,\h)}(u/v) \underline{ {\bf L}_2^{-}\bigl(v^{-1}\bigr) \tilde{\bf K}^{(\h)}_1(u) } \underline{{\bf L}_1^{+}(u) R^{(\h,\h)}(u v) \lbrack {\bf L}^{-}_2\bigl(v^{-1}\bigr) \rbrack^{-1} } \lbrack {\bf L}_1^{+}(u) \rbrack ^{-1} \tilde{\bf K}_2^{(\h)}(v) \\ \label{lhs-dress-RE}
 \phantom{\qquad \times}{} = R^{(\h,\h)}(u/v) \tilde{\bf K}^{(\h)}_1(u) R^{(\h,\h)}(u v) \tilde{\bf K}^{(\h)}_2(v),
 \end{gather}
 where we underlined the steps of calculation that correspond either to the commutation relations between L- and K-operators associated with different auxiliary spaces or to the use of variations of~\eqref{RLpm1},~\eqref{RLpm2}. For instance, in the first line we use
 \[
 {\bf L}_2^{-}\bigl(v^{-1}\bigr) R^{(\h,\h)}(u/v) \lbrack {\bf L}_1^{-}\bigl(u^{-1}\bigr) \rbrack^{-1}
 = \lbrack {\bf L}_1^{-}\bigl(u^{-1}\bigr) \rbrack^{-1} R^{(\h,\h)}(u/v){\bf L}_2^{-}\bigl(v^{-1}\bigr),
 \]
 which is obtained by multiplying the equation~\eqref{RLpm1}, with the all-minus choice, from the left by
 \smash{$\lbrack {\bf L}^-_1(u) \rbrack^{-1} R^{(\h,\h)}(v/u)$} and from the right by \smash{$R^{(\h,\h)}(v/u) \lbrack {\bf L}^-_1(u) \rbrack^{-1}$}, then using the unitarity property of the R-matrix given in~\eqref{inverse-R},
 and finally, substituting $u\rightarrow u^{-1}$, $v\rightarrow v^{-1}$.

 On the other hand, the right-hand side of the equation obtained in the first paragraph of this proof takes the following form:
 \begin{gather}
 \underline{{\bf L}_1^{-}\bigl(u^{-1}\bigr) \tilde{\bf K}^{(\h)}_2(v) } {\bf L}^{+}_2(v) R^{(\h,\h)}(u v) \big\lbrack {\bf L}^{-}_1\bigl(u^{-1}\bigr) \big\rbrack^{-1} \tilde{\bf K}^{(\h)}_1(u) \nonumber\\
 \qquad{}\times\underline{ {\bf L}_1^{+}(u) R^{(\h,\h)}(u/v) \lbrack {\bf L}_2^+(v) \rbrack^{-1} } \big\lbrack {\bf L}_1^+(u) \big\rbrack^{-1}\nonumber\\ \nonumber
 \phantom{\qquad\times}{}= \tilde{\bf K}^{(\h)}_2(v) {\bf L}_1^{-}\bigl(u^{-1}\bigr) \underline{ {\bf L}^{+}_2(v) R^{(\h,\h)}(u v) \big\lbrack {\bf L}^{-}_1\bigl(u^{-1}\bigr) \big\rbrack^{-1} } \underline{ \tilde{\bf K}^{(\h)}_1(u)\big\lbrack {\bf L}_2^+(v) \big\rbrack^{-1} } R^{(\h,\h)}(u/v) \\ \label{rhs-dress-RE}
 \phantom{\qquad\times}{}=\tilde{\bf K}^{(\h)}_2(v) R^{(\h,\h)}(u v) \tilde{\bf K}^{(\h)}_1(u) R^{(\h,\h)}(u/v).
 \end{gather}
 Finally, comparing~\eqref{lhs-dress-RE} with~\eqref{rhs-dress-RE}, one gets the reflection equation~\eqref{KpmRE} that was proven in Lemma~\ref{REtildeK}.
\end{proof}

We finally show that $\delta$ defined in~\eqref{deltatildeK} is coassociative and counital, see~\eqref{def-coideal}.
Firstly, we check the coassociativity
 \begin{gather*} \nonumber
 (\delta \otimes \id \otimes \id) \circ (\delta \otimes \id) \bigl( \tilde{\bf K}^{(\h)}(u)\bigr)\\
\qquad= (\delta \otimes \id \otimes \id)\bigl( \bigl( \lbrack {\bf L}^{-}\bigl(u^{-1}\bigr) \rbrack^{-1} \bigr)_{[2]} \bigl( \tilde{\bf K}^{(\h)}(u) \bigr)_{[1]} \bigl({\bf L}^+(u) \bigr)_{[2]} \bigr) \\ \nonumber
 \qquad= \bigl( \lbrack {\bf L}^{-}\bigl(u^{-1}\bigr) \rbrack^{-1} \bigr)_{[3]} \bigl( \lbrack {\bf L}^{-}\bigl(u^{-1}\bigr) \rbrack^{-1} \bigr)_{[2]} \bigl( \tilde{\bf K}^{(\h)}(u) \bigr)_{[1]} \bigl({\bf L}^+(u) \bigr)_{[2]} \bigl({\bf L}^+(u) \bigr)_{[3]} \\
 \qquad=( \id \otimes \Delta \otimes \id) \circ (\delta \otimes \id) \bigl( \tilde{\bf K}^{(\h)}(u)\bigr),
\end{gather*}
where the coproduct is given in~\eqref{coprodLpm} and we used $(\Delta \otimes \id)\bigl(\lbrack {\bf L}^\mp(u) \rbrack^{-1} \bigr) = \left ( {\bf L}^\mp(u) \right)^{-1}_{[2]}\left ( {\bf L}^\mp(u) \right)^{-1}_{[1]}$.
Secondly, the condition with the counit is checked using~\eqref{epsLpm}
\begin{gather*}
 (\id \otimes \epsilon \otimes \id) \circ ( \delta \otimes \id) \bigl( \tilde{\bf K}^{(\h)}(u) \bigr) \\
\qquad= ( \id \otimes \epsilon \otimes \id) \circ \bigl( \lbrack {\bf L}^{\mp}\bigl(u^{-1}\bigr) \rbrack^{-1} \bigr)_{[2]} \bigl( \tilde{\bf K}^{(\h)}(u) \bigr)_{[1]} \bigl({\bf L}^\pm(u) \bigr)_{[2]}
 = \tilde{\bf K}^{(\h)}(u).
\end{gather*}

Let us finally check that the coaction proposed in~\eqref{deltatildeK} reproduces the `standard' $q$-Onsager coaction~\eqref{coW0}--\eqref{coW1} for the first two generators $\tW_0$ and $\tW_1$. Recall from~\eqref{evalKtildej} that $\smash{\tilde{{\bf K}}^{(\h)}(u) }= \smash{\nu(u) {\mathcal K}^{(\h)}(u)}$, with $\nu(u)$ calculated in Lemma~\ref{lem:nu} and \smash{${\mathcal K}^{(\h)}(u)$} defined in~\eqref{K-Aq}. With the Ding--Frenkel L-operators computed in~\eqref{app:expLP} and~\eqref{app:expLM},
we compare leading terms at $u^{-1}$ of the matrix entries $(1,1)$ and $(2,2)$ of both sides of~\eqref{deltatildeK} which gives
\begin{gather*}
 \delta({\normalfont \tW}_0) = 1 \otimes \bigl( k_+ q^{\h} E_1 K_1^{\h} + k_- q^{-\h} F_1 K_1^{\h} \bigr)+ {\normalfont \tW}_0 \otimes K_1, \\
 \delta({\normalfont \tW}_1)= 1 \otimes \bigl( k_+ q^{-\h} F_0 K_0^{\h} + k_- q^\h E_0 K_0^{\h} \bigr)+ {\normalfont \tW}_1 \otimes K_0.
\end{gather*}
These formulas indeed agree with~\eqref{coW0} and~\eqref{coW1}, respectively.

\subsection{Comments}
Based on the supporting evidence given in the previous subsection, we believe Conjecture~{\rm\ref{conj1}} is correct.
Some straightforward consequences are now pointed out. Firstly, some relations among the fused K-operators \eqref{fusedunormK} are derived. They generalize the relation~\eqref{redj1}.
\begin{prop}\label{prop:evalbar}
 Assume Conjecture {\rm\ref{conj1}}. Then, the following relations hold for any $j \in \h \mathbb{N}_+$:
 \begin{gather}
 \bar{\mathcal{F}}^{(j-\h)}_\fu \mathcal{K}_1^{(\h)}\bigl(u q^{j+1}\bigr) R^{(\h,j)}\bigl(u^2 q^{j+\tha}\bigr) \mathcal{K}^{(j)}_2\bigl(u q^{\h}\bigr) \bar{\mathcal{E}}_\fu^{(j-\h)} \nonumber \\
 \qquad =\left ( \prod_{k=0}^{2j-2} c\bigl(u^2 q^{2j-1-k}\bigr) c\bigl(u^2q^{2j+1-k}\bigr) \right ) \Gamma\bigl(u q^j\bigr) \mathcal{K}^{(j-\h)}(u),\label{eq:TT}
\\
 \bar{\mathcal{F}}^{(j-\h)}_\fu \mathcal{K}_2^{(j)}\bigl(u q^{-\h}\bigr) R^{(\h,j)}\bigl(u^2 q^{-j-\tha}\bigr) \mathcal{K}_1^{(\h)}\bigl(u q^{-j-1}\bigr) \bar{\mathcal{E}}_\fu^{(j-\h)} \nonumber\\
 \qquad =\left (\prod_{k=0}^{2j-2} c\bigl(u^2 q^{-2j+2+k}\bigr) c\bigl(u^2q^{-2j+k}\bigr) \right ) \Gamma\bigl(u q^{-j-1}\bigr) \mathcal{K}^{(j-\h)}(u),\label{TT2}
 \end{gather}
 where \smash{$\bar{\mathcal{E}}^{(j-\h)}$} is fixed by Lemma~{\rm\ref{lem-barE}} and \smash{$\bar{\mathcal{F}}^{(j-\h)}$} is given in~\eqref{exprbF} with~\eqref{coefbar}.
\end{prop}
\begin{proof}
 From Remark~\ref{rem:oppcop}, the intertwining property with $\Delta^{\rm op}$ reads
 \begin{equation} \label{int-op}
 \bar{\mathcal{E}}^{(j-\h)} \bigl(\pi_u^{j-\h}\bigr) (x) = \bigl(\pi_{u q^{-j-1}}^{\h} \otimes \pi_{u q^{-\h}}^j\bigr) (\Delta^{\rm op}(x)) \bar{\mathcal{E}}^{(j-\h)} \qquad \forall x \in \Loop.
 \end{equation}
 Now, express the left-hand side of~\eqref{eq:TT} in terms of K-operators and R-matrices. It reads
 \begin{gather} \nonumber
 \frac{\bar{\mathcal{F}}_\fu^{(j-\h)} {\bf K}_1^{(\h)}\bigl(u q^{j+1}\bigr) \mathcal{R}^{(\h,j)}\bigl(u^2q^{j+\tha}\bigr) {\bf K}_2^{(j)}\bigl(u q^\h\bigr) \bar{\mathcal{E}}_\fu^{(j-\h)}}{\nu\bigl(u q^{j+1}\bigr) \nu^{(j)}\bigl(u q^\h\bigr) \pi^\h\bigl(\mu^{(j)}\bigl(u^2q^{j+\tha}\bigr)\bigr) } \\ \nonumber
 \qquad= \frac{\bigl(\id \otimes \bar{\mathcal{F}}^{(j-\h)}\bigr) \big[ \bigl(\id \otimes \pi^\h_{u^{-1}q^{-j-1}} \otimes \pi^{j}_{u^{-1} q^{-\h}}\bigr) (\id \otimes \Delta^{\rm op}) (\mathfrak{K}) \big] \bigl(\id \otimes \bar{\mathcal{E}}^{(j-\h)}\bigr)}{\nu\bigl(u q^{j+1}\bigr) \nu^{(j)}\bigl(u q^\h\bigr) \pi^\h\bigl(\mu^{(j)}\bigl(u^2q^{j+\tha}\bigr)\bigr) } \\
 \label{pr-tt1}
 \qquad \overset{\eqref{int-op}}{=} \frac{ \nu^{(j-\h)}(u) }{\nu\bigl(u q^{j+1}\bigr) \nu^{(j)}\bigl(u q^\h\bigr) \pi^\h\bigl(\mu^{(j)}\bigl(u^2q^{j+\tha}\bigr)\bigr) } \mathcal{K}^{(j-\h)}(u).
 \end{gather}
 Then, simplifying the normalization factors and using~\eqref{funct-nu}, we get
 \begin{gather*}
\big [ \nu\bigl(u q^j\bigr) \nu(u q^{j-1}) \pi^\h( \mu^{(j-\h)}\bigl(u^2q^j\bigr)\mu^{(j)}\bigl(u^2q^{j+\tha}\bigr))\big]^{-1} \mathcal{K}^{(j-\h)}(u)\\
 \qquad= \left[ \prod_{k=0}^{2j-2} \pi^\h( \mu\bigl(u^2q^{2j-k-1}\bigr)\mu(u^2q^{2j-k})) \right]^{-1} \Gamma\bigl(u q^j\bigr) \mathcal{K}^{(j-\h)}(u).
 \end{gather*}
 Finally, using~\eqref{eval-mumu}, the equation~\eqref{eq:TT} follows. The relation~\eqref{TT2} is obtained similarly.
\end{proof}

Secondly, we analyze the spin-$0$ K-operator ${\bf K}^{(0)}(u)$ and the analog of the quantum determinant~\eqref{gammaform} for the spin-$\h$ K-operator \smash{${\bf K}^{(\h)}(u)$}.
\begin{prop}
 Assume Conjecture {\rm\ref{conj1}}, then ${\bf K}^{(0)}(u)=1$. Furthermore, the \textit{normalized} quantum determinant of the K-operator ${\bf K}^{(\h)}(u)$ is equally~$1$,
 \begin{equation}\label{qdet-K-bold}
 \normalfont{\text{tr}}_{12} \bigl( \mathcal{P}_{12}^- {\bf K}_1^{(\h)}(u) \mathcal{R}^{(\frac{1}{2},\frac{1}{2})}\bigl(qu^2\bigr) {\bf K}_2^{(\h)}(u q) \bigr) = 1.
 \end{equation}
\end{prop}

\begin{proof} Specializing (\ref{tildeKrel2}) to $j=\h$, we get
 \begin{align}
 {\bf K}^{(0)}(u) &= \bar{\mathcal{F}}_\fu^{(0)} {\mathbf{K}}_1^{(\h)}\bigl(u q^\tha\bigr) \mathcal{R}^{(\h,\h)}\bigl(u^2q^2\bigr) {\mathbf{K}}_2^{(\h)}\bigl(u q^\h\bigr) \bar{\mathcal{E}}_\fu^{(0)}\nonumber \\
 &= \nu\bigl(u q^\tha\bigr) \nu \bigl(uq^\h\bigr) \pi^\h\bigl(\mu\bigl(u^2q^2\bigr)\bigr) \bar{\mathcal{F}}_{\fu}^{(0)} \mathcal{K}_1^{(\h)}\bigl(u q^\tha\bigr) R^{(\h,\h)}\bigl(u^2 q^2\bigr) \mathcal{K}^{(\h)}_2\bigl(u q^\h\bigr) \bar{\mathcal{E}}_\fu^{(0)},\label{propK0p1}
 \end{align}
 where $\bar{\mathcal{E}}^{(0)}$, $\bar{\mathcal{F}}^{(0)}$ are given by
\smash{$
 \bar{\mathcal{E}}^{(0)} =
 \begin{pmatrix}
 0 &
 1&
 -1&
 0
 \end{pmatrix}^t$},
\smash{$\bar{\mathcal{F}}^{(0)}=
 \begin{pmatrix}
 0 & \frac{1}{2} & -\frac{1}{2} & 0
 \end{pmatrix}$}.
 Then, noticing that for any two-by-two matrix $A$, one has the property
 \begin{equation}\label{EF-trace}
 \bar{\mathcal{F}}_{\fu}^{(0)} A \bar{\mathcal{E}}_{\fu}^{(0)} = \mathsf{tr}_{12} ( \mathcal{P}_{12}^- A),
 \end{equation}
 it follows from~\eqref{gammaform} that
 \smash{$
 \Gamma(u)= \bar{\mathcal{F}}_{\fu}^{(0)} \mathcal{K}_1^{(\h)}(u q) R^{(\h,\h)}\bigl(u^2 q\bigr) \mathcal{K}^{(\h)}_2(u) \bar{\mathcal{E}}_\fu^{(0)}$}.
 Therefore, the right-hand side of~\eqref{propK0p1} becomes
 ${\bf K}^{(0)}(u) = \nu\bigl(u q^\tha\bigr) \nu \bigl(uq^\h\bigr) \pi^\h\bigl(\mu\bigl(u^2q^2\bigr)\bigr) \Gamma\bigl(u q^\h\bigr) =1$,
 where we used the functional relation~\eqref{funct-nu}.
 We finally note that the quantum determinant in the left-hand side of~\eqref{qdet-K-bold} is ${\bf K}^{(0)}(u)$, due to~\eqref{propK0p1} and~\eqref{EF-trace}, and so it equals 1.
\end{proof}

\begin{prop}
 Assume Conjecture~{\rm\ref{conj1}}, then \smash{$\widehat{\mathcal{K}}^{(j)}(u)$} from~\eqref{invertKj} is equal to the fused K-operator \smash{$\mathcal{K}^{(j)}(u)$} defined in~\eqref{fusedunormK}.
\end{prop}

\begin{proof}
 Recall the K-operators \smash{${\bf K}^{(j+\h)}(u)$} can be written either as~\eqref{fusedevalK} or as~\eqref{fusedevalKv2}. Using~\eqref{evalKj} with~\eqref{exp-nujh} and the invertibility of $\nu^{(j)}(u)$, we show by induction (recall that \smash{$\widehat{\mathcal{K}}^{(\h)}(u)=\mathcal{K}^{(\h)}(u)$}) that \smash{$\mathcal{K}^{(j)}(u)$} equals \smash{$\widehat{\mathcal{K}}^{(j)}(u)$}.
\end{proof}

\section{Summary and outlook} \label{sec8}

To briefly summarize our main results, we provided a new set of K-operator solutions to the spectral parameter dependent reflection equation~\eqref{genREAj} in terms of generating functions of the centrally extended $q$-Onsager algebra ${\mathcal A}_q$.
The central formula of this work is the recursion~\eqref{fusedunormK} for the fused K-operators of arbitrary spin $j\in\frac{1}{2}\mathbb{N}$ as well as Theorem~\ref{prop:fusedRE} on the reflection equation they satisfy.
We also gave formulas for the fused R-matrices and the fused K-operators in~\eqref{dvpR}, \eqref{dvpK}, whose expressions contain only the fundamental R-matrix and K-operator. These results were established within a general framework of universal K-matrices that we developed in Section~\ref{sec:universal-K}, extending the previously known approaches (discussed in the introduction). In particular, the central formula~\eqref{fusedunormK} is based on the results in Proposition~\ref{propfusK} and in Remark~\ref{Dop-K}. We also provided in Section~\ref{sec5.3} a few explicit examples of the fused K-operators (for spins $j=1$ and $j=\tha$) in terms of generating functions of~$\mathcal{A}_q$.

As the existence of a universal K-matrix (for our choice of algebras $H=\Loop$ and $B={\mathcal A}_q$ and the compatible twists) is still an open fundamental question, we have investigated whether the fused K-operators~\eqref{fusedunormK} satisfy the (evaluated version of) universal K-matrix axioms \mbox{\eqref{univK1}--\eqref{univK3}} which is resulted in Conjecture~\ref{conj1}. One of the key problems here is to understand better the central element $\nu(u)\in {\mathcal A}_q\big[\big[u^{-1}\big]\big]$, in particular to derive its coaction $\delta(\nu(u))$ so that it reproduces the evaluated coaction in~\eqref{dnuw}.

In the case of coideal subalgebras of $\Loop$ like the $q$-Onsager algebra $B=O_q$, an explicit formula for our universal K-matrix in a completion of $O_q \otimes \Loop$ could be deduced from the expression~\eqref{eps-K2} provided the one-component universal K-matrix $\mathcal{K}$ is known, see also~\cite{AV25}. Indeed, there are existence results for $\mathcal{K}$ for certain family of twist pairs~\cite{AV20}.
Unfortunately, no explicit expression for $\mathcal{K}$ associated to $O_q$ is known, and thus no explicit form of K-operators. Also, we make a choice of twist pair $(\psi,J)$ that can not fit into the family of twist pairs used in~\cite[Section~9.5]{AV20}.
This is due to the fact that both choices differ by an outer automorphism of~$\Loop$ corresponding to the automorphism of the affine Dynkin diagram. We make our choice of the twist pair in order to obtain the standard reflection equation from the universal one~\eqref{psiRE}, in the principal gradation, and eventually to relate the corresponding K-operators to the \textit{standard in literature} spin-chain transfer matrices, as it is studied~in~\cite{LBG}.\looseness=1

As the algebra ${\mathcal A}_q$ is a central extension of $O_q$, fixing values of its center provides a surjective algebra map $\Psi\colon {\mathcal A}_q \to O_q$. Applying $\Psi$ to the entries of our fused K-operators ${\mathcal K}^{(j)}(u)$, we~thus expect to get spin-$j$ K-operators for $O_q$, in the sense that they satisfy the relations~\eqref{coacDK} and~\eqref{evalbfK1}, which will be discussed in more details elsewhere.

In the literature on quantum integrable systems, K-operators and their
images in the tensor product (or spin-chain) representations of the algebra~${\mathcal A}_q$~-- known as Sklyanin's operators~-- are the
basic building elements for the construction of mutually commuting
quantities, for instance the Hamiltonian of open spin chains with
integrable boundary conditions \cite{Skly88}. For the quotient of~${\mathcal A}_q$ known as the $q$-Onsager algebra,
the fundamental K-operator~\eqref{K-Aq} is the essential ingredient in the open XXZ spin-$\h$ chain with generic boundary
conditions \cite{BK07}. For the generic diagonal boundary conditions in this spin chain, the fundamental K-operator~\eqref{K-Aq} generates another quotient\footnote{More precisely, it is a degenerate specialization at $\rho\to 0$ of the $q$-Onsager quotient.} of ${\mathcal A}_q$ known as the augmented $q$-Onsager algebra~\cite[Section~2]{BB12}.
For all these spin-$\h$ models, the transfer matrix is the image in the spin-chain representation of a~generating function~\smash{$\bt^{(\h)}(u)$} built from the K-operator~\eqref{K-Aq} and a dual solution of the reflection equation
for a spin-$\h$ auxiliary space. Importantly, \smash{$\bt^{(\h)}(u)$} reads as a~{\it linear} combination of some fundamental generators \smash{$\{{\mathcal I}_{2k+1}|k\in{\mathbb N}\}$} of a commutative subalgebra of~${\mathcal A}_q$. Therefore, in this approach the
diagonalization of the transfer matrix reduces to the diagonalization of
the image of this commutative subalgebra.\looseness=1

The fused K-operators of spin-$j$ constructed in this paper open a route to the representation-independent analysis of related integrable models beyond the case of the fundamental spin-$\h$ auxiliary space, for instance, of the open XXZ spin-$j$ chain with generic integrable boundary conditions.
And the following problems can be addressed here:
\begin{itemize}\itemsep=0pt
\item[(1)]
Firstly,
it is natural to ask about relations between any local\footnote{In the case of integrable spin chains,
an operator is said to be local whenever it is a product of a finite number of spin matrices, or a linear combination of such products.} or
non-local mutually commuting quantities of quantum spin chains and the generators $\{{\mathcal I}_{2k+1}\mid k\in{\mathbb N}\}$ of the commutative subalgebra in~${\mathcal A}_q$. For
instance, in the spin-$\h$ case the differentiation of the transfer matrix leads to the expression of the Hamiltonian in terms of the operators~${\mathcal I}_{2k+1}$'s~\cite[equation~(39)]{BK07}. We thus also expect that the transfer matrix for the models based on the auxiliary space of arbitrary spin-$j$ admits a unified formulation
as the image of a generating function $\bt^{(j)}(u)$ in the commutative
subalgebra of ${\mathcal A}_q\big[\big[u^{-1}\big]\big]$. In~\cite{LBG}, the
structure of the universal transfer matrices \smash{$\bt^{(j)}(u)$} for higher spin-$j$ auxiliary space
representation is studied in details. In particular, the so called TT-relations -- a~recursion on~$\bt^{(j)}(u)$ given in~\eqref{intro:TT} --
are derived at the algebraic level, independently of a~representation chosen. Generalizing the spin-$\h$ case, it is shown
that $\bt^{(j)}(u)$ is a~power series in~$u^{-1}$ with coefficients being polynomials of degree $2j$ in the generators $\{{\mathcal
I}_{2k+1}\mid k\in{\mathbb N}\}$, and that the universal $\bt^{(j)}(u)$'s indeed agree with the `physical' transfer matrices on the spin chains.\looseness=1

\item[(2)] Secondly, given those transfer matrices generated from various images of
${\mathcal A}_q$, the problem of cha\-rac\-te\-ri\-zing their spectral properties --
leading to the eigenstates and eigenvalues of the Hamiltonian -- is consequently
reduced to the diagonalization of the images of~$\{{\mathcal
I}_{2k+1}|k\in{\mathbb N}\}$. For the simplest example of the quotient of
${\mathcal A}_q$ known as the Askey--Wilson algebra, for irreducible finite-dimensional representations the problem is solved in~\cite{BP19}, combining the theory
of Leonard pairs and the so-called modified algebraic Bethe ansatz~\cite{B14,BC13,BP14}. A similar analysis for ${\mathcal A}_q$ remains to be
done.

\item[(3)] The K-operators are also used to construct the Baxter's $Q$-operator for diagonal boundary conditions~\cite{BT17,VW20} and triangular boundary conditions~\cite{T19, T20}.
The construction of a~universal $Q$-operator
for ${\mathcal A}_q$ and corresponding TQ-relations may be also addressed. In
that case, it would be desirable to construct the analogue of the fused
K-operator for~${j\rightarrow \infty}$ as suggested in \cite{YNZ06}, see also \cite{T20,VW20}.

\item[(4)] Finally, it is very desirable to construct K-operators of arbitrary complex spins, i.e., associated to $\Uq$ Verma modules of complex weights, as it would give essentially the corresponding universal K-matrix. Such integrable quantum spin-chains with integrable boundary conditions based on the Verma modules were recently introduced~\cite{CGS}, and it was shown~\cite{CGJS} a deep connection to the $q$-Onsager algebra via the common XXZ spin chain spectrum with the integrable non-diagonal boundary conditions.
\end{itemize}

\appendix

 \section{The universal R-matrix} \label{appC}
 In this appendix, we compute evaluations of the universal R-matrix.
 Firstly, we recall the construction of the universal R-matrix of Khoroshkin--Tolstoy~\cite{Tolstoy1991} for the Hopf algebra ${H = \Loop}$ in terms of root vectors. Secondly, evaluations of the universal R-matrix are considered. In particular, we give expressions of the Ding--Frenkel L-operators $\mathbf{L}^+(u)$ and \smash{$\bigl\lbrack \mathbf{L}^-\bigl(u^{-1}\bigr)\bigr\rbrack^{-1}$}, as~defined in~\eqref{defLpm}. Finally, the spin-$\h$ L-operator \smash{${\bf L}^{(\h)}(u)$} is computed by evaluating ${\bf L}^+(u)$.

 \subsection{Root vectors}
 Let us first recall the definition of the root vectors of $\Loop$. We adapt the construction in~\cite{Boos2012} to our choice of coproduct, recall the relation~\eqref{DeltaTK}.
 Let us set
 \[
 e_\alpha = E_1 K_1^{-\h}, \qquad e_{\delta-\alpha} = E_0 K_0^{-\h}, \qquad f_\alpha = K_1^\h F_1, \qquad f_{\delta-\alpha}=K_0^\h F_0.
 \]
 The other root vectors are defined by the recursion relations
 \begin{alignat}{3}
 & e'_{k\delta}= q^{-1} [e_{\alpha+(k-1)\delta}, e_{\delta-\alpha}]_{q},\qquad&&
 f'_{k \delta} = q [f_{\delta-\alpha},f_{\alpha+(k-1)\delta}]_{q^{-1}},&\nonumber\\
 & e_{\alpha+k \delta}=[2]_q^{-1} [e_{\alpha +(k-1)\delta},e'_\delta],\qquad&&
 f_{\alpha+k\delta} = [2]_q^{-1} [f'_\delta, f_{\alpha +(k-1) \delta} ],\qquad k \in \mathbb{N}_+,& \nonumber\\
 & e_{\delta-\alpha+k \delta} = [2]_q^{-1} [ e'_\delta,e_{\delta-\alpha+(k-1)\delta}],\qquad&&
 f_{\delta-\alpha+k\delta} = [2]_q^{-1} [f_{\delta-\alpha+(k-1)\delta}, f'_\delta].&\label{app:rootvect}
 \end{alignat}
 The root vectors $e_{k \delta}$, $f_{k\delta}$ are defined via the generating functions
 \begin{gather*}
 \lambdaQ \sum_{k=1}^{\infty} e_{k\delta} z^{-k} = \log \left( 1 + \lambdaQ \sum_{k=1}^{\infty} e'_{k\delta} z^{-k} \right), \\
 - \lambdaQ \sum_{k=1}^{\infty} f_{k\delta} z^{-k} = \log \left( 1- \lambdaQ \sum_{k=1}^{\infty} f'_{k\delta} z^{-k} \right).
 \end{gather*}

 \subsection{Khoroshkin--Tolstoy construction} \label{app:C2}
 Let $\{ \alpha + k\delta\}_{k=0}^\infty \cup \{ k\delta\}_{k=0}^\infty \cup \{ \delta - \alpha + k \delta \}_{k=0}^\infty$ be the positive root system of $\widehat{\mathfrak{sl}}_2$. We choose the root ordering as
 \[\label{eq:rootord}
 \alpha, \alpha + \delta, \dots, \alpha + k\delta, \dots, \delta, 2 \delta, \dots, \ell \delta, \dots, \dots, (\delta-\alpha)+ m\delta, \dots, (\delta-\alpha) + \delta, \delta-\alpha,
 \]
 for any $k, \ell, m\in \mathbb{N}$.
 Then, the universal R-matrix obtained by Khoroshkin and Tolstoy takes the following factorized form
 \begin{equation} \label{app:univRform}
 \mathcal{R}=\mathcal{R}^{+} \mathcal{R}^{0} \mathcal{R}^{-} q^{\h h_1 \otimes h_1},
 \end{equation}
 where\footnote{Here we use the notation
 \smash{${\prod^{\overleftarrow{n}}_{k=0}} a(k)= a(n) a(n-1) \cdots a(0)$} and \smash{${\prod^{\overrightarrow{n}}_{k=0}} a(k)= a(0)a(1) \cdots a(n)$}, for any function~$a(n)$.}
 \begin{gather} \label{app:eqRP}
 \mathcal{R}^{+}=\prod^{\overrightarrow{\infty}}_{k=0} \exp_{q^{-2}} \bigl( \lambdaQ e_{\alpha+k\delta} \otimes f_{\alpha + k\delta} \bigr),\\ \label{app:eqR0}
 \mathcal{R}^{0}=\exp \left(\lambdaQ \sum_{k=1}^\infty \frac{k}{[2k]_q}e_{k\delta} \otimes f_{k\delta} \right),\\ \label{app:eqRM}
 \mathcal{R}^{-}=\prod^{\overleftarrow{\infty}}_{k=0} \exp_{q^{-2}} \left ( \lambdaQ e_{\delta-\alpha+k\delta} \otimes f_{\delta-\alpha + k\delta} \right),
 \end{gather}
 with $q^{h_1} \equiv K_1$ and the $q$-exponential is
 \[
 \exp_q(x)=1+\sum_{k=1}^{\infty} \frac{x^k}{(k)_q!}, \qquad (k)_q! = (1)_q (2)_q \cdots (k)_q, \qquad (k)_q = \frac{q^k-1}{q-1}.
 \]
 We also notice that $[e_{k\delta},e_{\ell\delta}]=0=[f_{k\delta},f_{\ell\delta}]$ for any $k$, $\ell$, and so the exponent in~\eqref{app:eqR0} can be also written in the form of semi-infinite product of exponents involving only one term $\sim$ $e_{k\delta} \otimes f_{k\delta}$.

 \subsection{Evaluation of the universal R-matrix}\label{sec:eval-R}
 In the previous subsection, the explicit form of the universal R-matrix was recalled. It is expressed as a product of $q$-exponentials with root vectors in the arguments. Now, we evaluate the second tensor product component of the universal R-matrix by taking its image under the formal fundamental evaluation representation as introduced in Section~\ref{sec:formal-ev}.
 \subsubsection{Evaluation of the root vectors}
 The action of the evaluation map defined in~\eqref{eval} on the first root vectors gives
 \begin{alignat}{4}
 & \mathsf{ev}_u(e_{\delta-\alpha})= u^{-1} F K^\h,\qquad&& \mathsf{ev}_u(e_\alpha) = u^{-1} E K^{-\h}, \qquad&& \mathsf{ev}_u \bigl(K_0^\h\bigr) = K^{-\h},&\nonumber\\
 & \mathsf{ev}_u(f_{\delta-\alpha})=u q^{-1} E K^{-\h},\qquad&&
 \mathsf{ev}_u(f_\alpha) = u q^{-1} F K^\h, \qquad&&\mathsf{ev}_u\bigl(K_1^\h\bigr)=K^\h.&\label{app:affev}
 \end{alignat}
 The image of the other root vectors of $\Loop$ in~\eqref{app:rootvect} under the evaluation map are obtained by induction similarly to~\cite[Section~4.4]{Boos2012}. They are given for $k \in \mathbb{N}$ by
 \begin{alignat*}{3} 
 & \mathsf{ev}_u(e_{\alpha+k \delta}) = (-1)^k u^{-2k-1} q^{-k} E K^{-k-\h}, \qquad&&
 \mathsf{ev}_u(e_{\delta-\alpha+k \delta})= (-1)^k u^{-2k-1} q^k F K^{-k+\h},&\\
 & \mathsf{ev}_u(f_{\alpha+k\delta})= (-1)^k u^{2k+1} q^{-k-1} F K^{k+\h},\qquad&&
 \mathsf{ev}_u(f_{\delta-\alpha+k \delta}) = (-1)^k u^{2k +1} q^{k-1} E K^{k-\h},&
 \end{alignat*}
 and for $ k \in \mathbb{N}_+$
 \begin{gather}
 \mathsf{ev}_u(e'_{k \delta}) = (-1)^{k-1}u^{-2k} [E,F]_{q^{k}} K^{-k+1},\nonumber\\
 \mathsf{ev}_u(e_{k \delta})= \frac{(-1)^{k-1} u^{-2k}}{\bigl(q-q^{-1}\bigr) k} \bigl( C_k - \bigl(q^k+q^{-k}\bigr)K^{-k}\bigr),\nonumber\\
 \mathsf{ev}_u(f'_{k \delta})= (-1)^{k-1} u^{2k} [E,F]_{q^{-k}} K^{k-1},\nonumber\\
 \mathsf{ev}_u(f_{k\delta})=\frac{(-1)^{k} u^{2k}}{\bigl(q-q^{-1}\bigr) k} \bigl( C_k - \bigl(q^k+q^{-k}\bigr)K^{k}\bigr),\label{app:rootvectevP2}
 \end{gather}
 where the elements $C_k$ are defined by the generating function
 \begin{equation} \label{genCK}
\sum_{k=1}^{\infty} (-1)^{k-1} C_k \frac{z^{-k}}{k}=\log\bigl(1+C z^{-1} + z^{-2}\bigr), \qquad z \in \mathbb{C},
 \end{equation}
 and where $C$ is the central element of $\Uq$ given in~\eqref{Cas-Uqsl2}. For instance by expanding~\eqref{genCK} we get the first elements of $C_k$
 $
 C_1=C$, $ C_2= C^2-2$, $ C_3=C^3-3C$, $ C_4= C^4-4C^2+2$.
 Recall~$E_{ab}$ is the matrix with zero everywhere except 1 in the entry $(a,b)$. The matrix multiplication obeys~$
 E_{ab} E_{cd} = \delta_{b,c} E_{ad}$.
 In this notation, the spin-$\h$ finite-dimensional representation of~$\Uq$ reads
\smash{$
 \pi^{\h}(K^m)=q^m E_{11} + q^{-m} E_{22}$}, \smash{$ \pi^{\h} (E) = E_{12}$}, \smash{$ \pi^{\h} (F) = E_{21}$},
 and the central elements $C_k$ become, for all $k\geq 1$,
\smash{$
 \pi^{\h}(C_k) = (q^{2k}+q^{-2k}) \mathbb{I}_{2}$}.

 In order to obtain Ding--Frenkel L-operators and R-matrices from the universal R-matrix, one also needs the image of the root vectors under the representation map \smash{$\pi_u^{\h}\colon \Loop \rightarrow \End\bigl(\mathbb{C}^2\bigr)$}, which is
 \begin{gather}
 \pi_u^{\h}(e_{\alpha+k \delta})= (-1)^k u^{-2k -1}q^{\h} E_{12},\qquad
 \pi_u^{\h}(e_{\delta-\alpha+k \delta})=(-1)^k u^{-2k-1} q^{\h} E_{21},\nonumber\\
 \pi_u^{\h}(f_{\alpha+k\delta})= (-1)^k u^{2k+1} q^{-\h} E_{21},\qquad
 \pi_u^{\h}(f_{\delta-\alpha+k \delta})= (-1)^k u^{2k+1} q^{-\h} E_{12},\nonumber\\
 \pi_u^{\h}(e'_{k \delta})= (-1)^{k-1} u^{-2k} \bigl(q E_{11} - q^{-1} E_{22}\bigr),\nonumber\\
 \pi_u^{\h}(e_{k \delta})= (-1)^{k-1} u^{-2k}\frac{[k]_q}{k} \bigl( q^k E_{11} - q^{-k}E_{22}\bigr),\nonumber\\
 \pi_u^{\h}(f'_{k \delta})= (-1)^{k-1} u^{2k} \bigl(q^{-1} E_{11} -q E_{22}\bigr),\nonumber\\
 \pi_u^{\h}(f_{k\delta})= (-1)^{k-1} u^{2k}\frac{[k]_q}{k} \bigl( q^{-k} E_{11} - q^k E_{22}\bigr). \label{app:piroot}
 \end{gather}
 Recall that ${\bf L}^\pm(u)$ are defined in~\eqref{defLpm}.
 We now compute explicitly $\mathbf{L}^+(u)$ and $\big\lbrack \mathbf{L}^-\bigl(u^{-1}\bigr) \big\rbrack^{-1}$.

 \subsubsection[L\^+(u)]{$\boldsymbol{\mathbf{L}^+(u)}$}\label{app-subsec:Lp}
 Recall the factorized form of the universal R-matrix~\eqref{app:univRform}. We now compute the image of $\mathfrak{R}^\pm$, $\mathfrak{R}^0$, \smash{$q^{\h h_1 \otimes h_1}$} under the action of \smash{$\bigl(\id \otimes \pi^{\h}_{u^{-1}}\bigr)$}.
 From~\eqref{app:eqRP} and with~\eqref{app:piroot}, we get
 \begin{align}
 \bigl(\id \otimes \pi_{u^{-1}}^\h\bigr)\bigl( \mathfrak{R}^+\bigr)& = \prod^{\overrightarrow{\infty}}_{k=0} \exp_{q^{-2}} \bigl( \bigl(-u^{-2}\bigr)^k u^{-1} \lambdaQ q^{-\h} e_{\alpha + k\delta} \otimes E_{21} \bigr)\nonumber\\
 &= 1 \otimes \mathbb{I}_2 + e^+(u) \otimes E_{21}, \label{app:LPRP}
 \end{align}
 where
 \begin{equation} \label{app:defeP}
 e^+(u) = \lambdaQ q^{-\h}u^{-1} \left (\sum_{k=0}^\infty \bigl(-u^{-2}\bigr)^k e_{\alpha + k \delta} \right ).
 \end{equation}
 Similarly, from~\eqref{app:eqRM} and with~\eqref{app:piroot} we have
 \begin{align} \nonumber
 \bigl(\id \otimes \pi_{u^{-1}}^\h\bigr)( \mathfrak{R}^-) &= \prod^{\overleftarrow{\infty}}_{k=0} \exp_{q^{-2}} \bigl( \bigl(-u^{-2}\bigr)^k u^{-1} \lambdaQ q^{-\h} e_{\delta-\alpha+ k \delta} \otimes E_{12} \bigr) \\ \label{app:LPRM}
 &= 1 \otimes \mathbb{I}_2 + f^+(u) \otimes E_{12},
 \end{align}
 where
 \begin{equation} \label{app:deffP}
 f^+(u) = \lambdaQ q^{-\h} u^{-1} \left (\sum_{k=0}^\infty \bigl(-u^{-2}\bigr)^k e_{\delta-\alpha + k \delta} \right).
 \end{equation}
 A straightforward calculation from~\eqref{app:eqR0} and using~\eqref{app:piroot} yields
 \begin{gather} \label{app:LPR0}
 \bigl(\id \otimes \pi_{u^{-1}}^\h \bigr) \bigl( \mathfrak{R}^0\bigr) = k^+(u) \otimes E_{11} + \tilde{k}^+(u)\otimes E_{22},
\\
 k^+(u)= \exp \left( -\lambdaQ \sum_{k=1}^\infty \frac{\bigl(-u^{-2} q^{-1}\bigr)^k}{q^k+q^{-k}} e_{k\delta} \right), \\
 \tilde{k}^+(u) = \exp \left( \lambdaQ \sum_{k=1}^\infty \frac{\bigl(-u^{-2} q\bigr)^k}{q^k+q^{-k}} e_{k\delta} \right). \label{app:defkP}
 \end{gather}
 Then, using \smash{$\pi^\h_u(h_1) = E_{11} - E_{22}$}, we get
 \begin{align} \label{app:LPqK}
 \bigl(\id \otimes \pi_{u^{-1}}^\h \bigr) \bigl( q^{\h h_1 \otimes h_1}\bigr) = K_1^\h \otimes E_{11} + K_1^{-\h} \otimes E_{22}.
 \end{align}
 Finally, combining~\eqref{app:LPRP}--\eqref{app:LPqK}, we get
 \begin{equation} \label{app:expLP}
 {\bf L}^+(u) =
 \begin{pmatrix}
 k^+(u)K_1^\h & k^+(u) f^+(u) K_1^{-\h} \\
 e^+(u) k^+(u) K_1^\h & \tilde{k}^+(u)K_1^{-\h} +e^+(u) k^+(u) f^+(u) K_1^{-\h} \\
 \end{pmatrix}.
 \end{equation}
 Note that from the definition of $e^+(u)$, $f^+(u)$, $k^+(u)$, $\tilde{k}^+(u)$ in~\eqref{app:defeP},~\eqref{app:deffP},~\eqref{app:defkP} it is easy to see that ${\bf L}^+(u)$ is a formal power series in $u^{-1}$, i.e., ${\bf L}^{+}(u)$ is in $\Loop\big[\big[ u^{-1}\big]\big] \otimes \End\bigl(\mathbb{C}^2\bigr)$.

 \subsubsection[L\^-(u\^{}\{-1\})]{$\boldsymbol{\big\lbrack \mathbf{L}^-\bigl(u^{-1}\bigr) \big\rbrack^{-1}}$}
 Consider $\mathfrak{p}\circ\mathfrak{R}^\pm = \mathfrak{R}^\pm_{21}$, $\mathfrak{p} \circ \mathfrak{R}^0=\mathfrak{R}^0_{21}$. We now compute their image under the action of~\smash{$\bigl(\id \otimes \pi^{\h}_{u^{-1}}\bigr)$} to obtain the expression of \smash{$\big\lbrack \mathbf{L}^-\bigl(u^{-1}\bigr) \big\rbrack^{-1}$} defined in~\eqref{defLpm}.
 First, it follows from~\eqref{app:eqRP} and~\eqref{app:piroot} that
 \begin{align}
 \bigl(\id \otimes \pi_{u}^{\h}\bigr) \bigl(\mathfrak{R}_{21}^+\bigr) &= \prod^{\overrightarrow{\infty}}_{k=0} \exp_{q^{-2}} \bigl( \bigl(-u^{-2}\bigr)^k u^{-1} \lambdaQ q^{\h} f_{\alpha + k\delta} \otimes E_{12} \bigr)\nonumber\\
& = 1 \otimes \mathbb{I}_2 + f^-(u) \otimes E_{12},\label{app:LMRP}
 \end{align}
 where
 \begin{equation} \label{app:deffM}
 f^-(u)= \lambdaQ q^{\h} u^{-1} \left ( \sum_{k=0}^\infty \bigl(-u^{-2}\bigr)^k f_{\alpha + k \delta} \right).
 \end{equation}
 Similarly, from~\eqref{app:eqRM} and using~\eqref{app:piroot}
 \begin{align}
 \bigl(\id \otimes \pi_{u}^{\h}\bigr) (\mathfrak{R}_{21}^-)& = \prod^{\overleftarrow{\infty}}_{k=0} \exp_{q^{-2}} \bigl( \bigl(-u^{-2}\bigr)^k u^{-1} \lambdaQ q^{\h} f_{\delta-\alpha+ k \delta} \otimes E_{21} \bigr)\nonumber\\
 & = 1 \otimes \mathbb{I}_2 + e^-(u) \otimes E_{21}, \label{app:LMRM}
 \end{align}
 where
 \begin{equation} \label{app:defeM}
 e^-(u)= \lambdaQ q^{\h} u^{-1} \left (\sum_{k=0}^\infty \bigl(-u^{-2}\bigr)^k f_{\delta-\alpha + k \delta} \right).
 \end{equation}
 Then from~\eqref{app:eqRM} and using~\eqref{app:piroot}, we get
 \begin{gather}
 \bigl(\id \otimes \pi_{u}^\h \bigr) \bigl( \mathfrak{R}_{21}^0\bigr) = k^-(u) \otimes E_{11} + \tilde{k}^-(u) \otimes E_{22},\label{app:LMR0}
\\
 k^-(u)= \exp\left ( - \lambdaQ \sum_{k=1}^{\infty} \frac{\bigl(-u^{-2}q\bigr)^k}{q^k+q^{-k}} f_{k\delta} \right ), \nonumber\\ \tilde{k}^-(u) = \exp\left ( \lambdaQ \sum_{k=1}^{\infty} \frac{\bigl(-u^{-2}q^{-1}\bigr)^k}{q^k+q^{-k}} f_{k\delta} \right ).\nonumber
 \end{gather}
 Finally, combining~\eqref{app:LMRP}--\eqref{app:LMR0} and~\eqref{app:LPqK}, we get
 \begin{equation} \label{app:expLM}
 \lbrack \mathbf{L}^-\bigl(u^{-1}\bigr) \rbrack^{-1}=
 \begin{pmatrix}
 k^-(u) K_1^\h+ f^-(u) \tilde{k}^-(u) e^-(u) K_1^\h & f^-(u) \tilde{k}^-(u) K_1^{-\h} \\
 \tilde{k}^-(u) e^-(u) K_1^\h & \tilde{k}^-(u) K_1^{-\h}\\
 \end{pmatrix}.
 \end{equation}
 Note that from the definition of $e^-(u)$, $f^-(u)$, $k^-(u)$, $\tilde{k}^-(u)$ in~\eqref{app:defeM},~\eqref{app:deffM},~\eqref{app:LMR0} it is easy to see that $\big\lbrack{\bf L}^-\bigl(u^{-1}\bigr) \big\rbrack^{-1}$ is a formal power series in $u^{-1}$, i.e., $\big\lbrack{\bf L}^-\bigl(u^{-1}\bigr) \big\rbrack^{-1}$ is in $\Loop\big[\big[ u^{-1}\big]\big] \otimes \End\bigl(\mathbb{C}^2\bigr) $.

 \subsection[The spin-1/2 L-operator L\^{}(1/2)(u)]{The spin-$\boldsymbol{\h}$ L-operator $\boldsymbol{\mathbf{L}^{(\h)}(u)}$}
 \label{appC4}
 We now compute the spin-$\h$ L-operator ${\bf L}^{(\h)}(u)$ defined in~\eqref{evalL}. It is obtained by taking the image of ${\bf L}^+(u)$ under the evaluation with $(\mathsf{ev}_v \otimes \id)$.

 Recall the expression of ${\bf L}^+(u)$ in~\eqref{app:expLP}. The spin-$\h$ L-operator is then obtained by evaluating $e^+(u)$, $f^+(u)$, $k^+(u)$, \smash{$\tilde{k}^+(u)$} defined in~\eqref{app:defeP}, \eqref{app:deffP}, \eqref{app:defkP}.
 Let us first introduce the function~\cite{Boos2012}
 \begin{equation} \label{app:Lambda}
 \Lambda(u) = \sum_{k=1}^\infty \frac{C_k }{\bigl(q^k+q^{-k}\bigr)} \frac{u^k}{k},
 \end{equation}
 where the central elements $C_k$ are defined by~\eqref{genCK}. Note that it satisfies
 \begin{equation} \label{app:Lambeq}
 \Lambda(u q) + \Lambda\bigl( u q^{-1}\bigr) =- \log\bigl(1-Cu+u^2\bigr).
 \end{equation}
 From the evaluated root vectors~\eqref{app:affev} and~\eqref{app:rootvectevP2}, we get
 \begin{gather*}
 \mathsf{ev}_v\bigl(e^+(u)\bigr)= \lambdaQ q^{-\h} u^{-1} v^{-1} E K^{-\h} \bigl( 1 - u^{-2} v^{-2} q^{-1} K^{-1} \bigr)^{-1}, \\
 \mathsf{ev}_v\bigl(f^+(u)\bigr)=\lambdaQ q^{-\h} u^{-1} v^{-1} F K^{\h} \bigl( 1 - u^{-2} v^{-2} q K^{-1} \bigr)^{-1}, \\
 \mathsf{ev}_v\bigl(k^+(u)\bigr) = e^{ \Lambda(u^{-2} v^{-2} q^{-1})} \bigl( 1 - u^{-2} v^{-2} q^{-1} K^{-1} \bigr), \\
 \mathsf{ev}_v\bigl(\tilde{k}^+(u)\bigr) = e^{-\Lambda(u^{-2} v^{-2} q) } \bigl( 1 - u^{-2} v^{-2} q K^{-1} \bigr)^{-1}.
 \end{gather*}
 For instance, let us now compute the evaluation of the matrix entry $(2,2)$ of ${\bf L}^+(u)$ in~\eqref{app:expLP}
 \begin{align*}
 \mathsf{ev}_v \bigl(\bigl({\bf L}^+(u)\bigr)_{22}\bigr)={}& \mathsf{ev}_v \bigl(\tilde{k}^+(u)K_1^{-\h} +e^+(u) k^+(u) f^+(u) K_1^{-\h} \bigr) \\
={}& \bigl( e^{-\Lambda(u^{-2}v^{-2} q)} + e^{\Lambda(u^{-2}v^{-2} q^{-1})} u^{-2}v^{-2} \lambdaQ^2 E F \bigr)\\
 &\times \bigl( 1 - u^{-2}v^{-2} q K^{-1} \bigr)^{-1} K^{-\h} \\
={}& e^{\Lambda(u^{-2}v^{-2} q^{-1})} \bigl( 1 + u^{-4}v^{-4} - u^{-2}v^{-2} ( C - \lambdaQ^2 E F) \bigr) \\
 &\times\bigl( 1 - u^{-2}v^{-2} q K^{-1} \bigr)^{-1} K^{-\h},
 \end{align*}
 where we used~\eqref{app:Lambeq} on the third line, and where $C$ is defined in~\eqref{Cas-Uqsl2}. Then, we obtain \[\mathsf{ev}_v\bigl(\bigl({\bf L}^+(u)\bigr)_{22}\bigr)
 = e^{\Lambda(u^{-2}v^{-2} q^{-1})}\bigl( K^{-\h} - u^{-2}v^{-2} q^{-1} K^{\h} \bigr).\]
 Computing the other entries, we have
 \begin{gather} \label{app:evalLP} {\bf L}^{(\h)}(u v) =
 (\mathsf{ev}_v \otimes \id )\bigl( {\bf L}^+(u)\bigr) = \mu(u v)
 \mathcal{L}^{(\h)}(u v),
 \end{gather}
 where $\mu(u)$ and $\mathcal{L}^{(\h)}(u)$ are given respectively in~\eqref{expMU} and~\eqref{Laxh}.

 Similarly, evaluating the Ding--Frenkel L-operator in~\eqref{app:expLM}, we have
 \begin{gather} 
 (\mathsf{ev}_v \otimes \id )\bigl( \lbrack{\bf L}^-\bigl(u^{-1}\bigr)\rbrack^{-1}\bigr) = \mu(u/v)
 \mathcal{L}^{(\h)}(u/v) = {\bf L}^{(\h)}(u/v).
 \end{gather}

 \section[Ordering relations for A\_q]{Ordering relations for $\boldsymbol{\mathcal{A}_q}$}\label{apB}
 \begin{lem}\label{rel-PBW-Aq} The following relations hold in ${\mathcal A}_q\big[\big[u^{-1}\big]\big]$:
 \begin{gather}
 \cG_-(v) \cG_+(u) = \cG_+(u) \cG_-(v)+\rho\bigl(q^2-q^{-2}\bigr) \left( \cW_+(u) \cW_+(v) - \cW_-(u) \cW_-(v)\right.\nonumber\\
 \phantom{ \cG_-(v) \cG_+(u) =}{}\left.+\frac{1-UV}{U-V} \bigl( \cW_+(u) \cW_-(v) - \cW_+(v) \cW_-(u)\bigr)
 \right), \label{R1-Aq}\\
 \cW_-(v)\cW_+(u)=\cW_+(u) \cW_-(v) + \frac{1}{V-U} \left ( \frac{\bigl(q-q^{-1}\bigr)}{\rho\bigl(q+q^{-1}\bigr)} \bigl(\cG_+(u) \cG_-(v) - \cG_+(v) \cG_-(u) \bigr)\right.\nonumber\\
 \phantom{ \cW_-(v)\cW_+(u)=}{}\left.+\frac{1}{q+q^{-1}} \bigl( \cG_+(u)-\cG_-(u)+\cG_-(v)-\cG_+(v) \bigr) \right ), \label{R2-Aq}
\\
 \cW_-(v) \cG_+(u) = \frac{q}{U-V} \bigl( \bigl(U q^{-1} - V q\bigr) \cG_+(u) \cW_-(v)\nonumber\\
 \phantom{\cW_-(v) \cG_+(u) =}{} - \bigl(q-q^{-1}\bigr) \bigl( \cW_+(u) \cG_+(v)
 - \cW_+(v)\cG_+(u)- U \cG_+(v)\cW_-(u) \bigr)\nonumber\\
 \phantom{\cW_-(v) \cG_+(u) =}{}+ \rho \bigl(U \cW_-(u) - V \cW_-(v) - \cW_+(u) + \cW_+(v) \bigr) \bigr), \label{R3-Aq}\\
 \cW_-(v)\cG_-(u)= \frac{1}{q(V-U)}\bigl( \bigl( V q^{-1}-U q\bigr) \cG_-(u) \cW_-(v) \nonumber\\
 \phantom{ \cW_-(v)\cG_-(u)=}{} -\bigl(q-q^{-1}\bigr)\bigl(\cW_+(u)\cG_-(v) - \cW_+(v)\cG_-(u)
 - U \cG_-(v) \cW_-(u) \bigr) \nonumber\\
 \phantom{ \cW_-(v)\cG_-(u)=}{}+ \rho \bigl( U \cW_-(u)-V\cW_-(v)-\cW_+(u)+\cW_+(v) \bigr) \bigr),\label{R4-Aq}
\\
 \nonumber
 \cG_+(v)\cW_+(u) = \frac{q}{U-V} \bigl( \bigl( U q - Vq^{-1} \bigr) \cW_+(u) \cG_+(v)\\
 \phantom{\cG_+(v)\cW_+(u) =}{}-\bigl(q-q^{-1}\bigr)\bigl(\cG_+(v)\cW_-(u)- \cG_+(u)\cW_-(v)
 + V\cW_+(v)\cG_+(u) \bigr)\nonumber\\
 \phantom{\cG_+(v)\cW_+(u) =}{}+ \rho \bigl( U \cW_+(u) - V \cW_+(v)-\cW_-(u)+\cW_-(v) \bigr) \bigr), \\ \nonumber
 \cG_-(v) \cW_+(u)=\frac{1}{q(V-U)} \bigl( \bigl(Vq-Uq^{-1}\bigr)\cW_+(u) \cG_-(v)\\
 \phantom{\cG_-(v) \cW_+(u)=}{}-\bigl(q-q^{-1}\bigr) \bigl( \cG_-(v) \cW_-(u) - \cG_-(u)\cW_-(v)
 +V \cW_+(v)\cG_-(u) \bigr)\nonumber\\
 \phantom{\cG_-(v) \cW_+(u)=}{}+ \rho \bigl( U \cW_+(u) - V \cW_+(v)-\cW_-(u)+\cW_-(v) \bigr) \bigr).\label{R6-Aq}
 \end{gather}
 \end{lem}

 \begin{proof} The first two ordering relations (\ref{R1-Aq}), (\ref{R2-Aq}), are obtained directly from~\eqref{eq2-Aq} and~\eqref{eq5-Aq}. The third relation (\ref{R3-Aq}) follows from~\eqref{eq3-Aq} and~\eqref{eq8-Aq} by replacing the element which is not in the chosen order. The relations~\eqref{R4-Aq}--\eqref{R6-Aq} are derived similarly.
 \end{proof}

 \section[Proof of Theorem 5.6]{Proof of Theorem \ref{prop:fusedRE}} \label{apD}
 For $(j_1,j_2)=\bigl(\h,\h\bigr)$, the reflection equation~\eqref{REj1j2} holds for $\mathcal{K}^{(\h)}(u)$ due to~\cite{BSh1}.
 The proof is divided into three parts:
 We first show the case $(j_1,j_2)=\bigl(\h,j+\h\bigr)$, assuming the equation~\eqref{REj1j2} holds for $(j_1,j_2)=\bigl(\h,j\bigr)$. Then, we prove the case $\bigl(j+\h,\h\bigr)$, assuming~\eqref{REj1j2} holds for~${(j_1,j_2)=\bigl(j,\h\bigr)}$.
 Finally, the generic case ($j_1$, $j_2$) is proven considering two distinct cases~${j_1 \geq j_2}$ and~${j_1 \leq j_2}$, using
 that the reflection equation~\eqref{REj1j2} holds for $(j_1,j_2)=\bigl(j+\h,\h\bigr)$ and $\bigl(\h,j+\h\bigr)$.

 First of all, the following Yang--Baxter equation holds for any $j_1, j_2 \in \h \mathbb{N}$:
 \begin{gather}
 R_{12}^{(j_1,j_2)}(u_1/u_2) R_{13}^{(j_1,j_3)}(u_1/u_3) R_{23}^{(j_2,j_3)} (u_2/u_3)\nonumber \\
 \qquad= R_{23}^{(j_2,j_3)} (u_2/u_3) R_{13}^{(j_1,j_3)}(u_1/u_3) R_{12}^{(j_1,j_2)}(u_1/u_2).\label{YBstraight}
 \end{gather}
 This is due to the fact that $\mathcal{R}^{(j_1,j_2)}(u)$ is proportional to $R^{(j_1,j_2)}(u)$, see~\eqref{evalRj1j2}, and it satisfies the Yang--Baxter equation in~\eqref{YBj1j2}.
 In what follows, we use the shorthand notation
 \begin{alignat}{3}
 & R_{k\ell}=R_{k\ell}^{(j_k,j_\ell)}(u_k/u_\ell), \qquad && \bar{R}_{k\ell}=R_{k\ell}^{(j_k,j_\ell)}(u_\ell/u_k), & \nonumber\\ &\hat{R}_{k\ell}=R_{k\ell}^{(j_k,j_\ell)}(u_k u_\ell), \qquad&& {\mathcal K}_\ell= {\mathcal K}_\ell^{(j_\ell)}(u_\ell),&\label{Not-lem}
 \end{alignat}
 and denote by \smash{$R_{\langle \ell \ell+1 \rangle k}^{(j+\h,j_k)}(u)$}, \smash{$R_{k \langle \ell \ell+1 \rangle }^{(j_k,j+\h)}(u)$}, \smash{${\mathcal K}^{(j+\h)}_{\langle \ell \ell+1 \rangle}(u)$} the resulting R-matrices and K-operators obtained by fusing the space at position $\ell$ with the one at position $\ell+1$, for $k, \ell \in \mathbb{N}_+$.

 According to the notation~\eqref{Not-lem}, the Yang--Baxter equation~\eqref{YBstraight} reads
 \begin{gather}
 R_{12} R_{13} R_{23} = R_{23} R_{13}R_{12}, \qquad \text{or} \qquad R_{12} \hat{R}_{13} \hat{R}_{23} = \hat{R}_{23} \hat{R}_{13} R_{12}, \nonumber \\
 \text{or} \qquad \hat{R}_{12} \hat{R}_{13} \bar{R}_{23} = \bar{R}_{23} \hat{R}_{13} \hat{R}_{12},\label{YB-Not}
 \end{gather}
 where the first two relations in~\eqref{YB-Not} are related by $u_3 \rightarrow u_3^{-1}$, and the last two relations in~\eqref{YB-Not} are related by $u_2 \rightarrow u_2^{-1}$. In the following, we will need relations derived from the Yang--Baxter equation
 \begin{equation}
 \label{annexR}
 \hat{R}_{13} \hat{R}_{12} R_{23} = R_{23} \hat{R}_{12} \hat{R}_{13}, \qquad R_{13} \hat{R}_{12} \hat{R}_{23} = \hat{R}_{23} \hat{R}_{12} R_{13}.
 \end{equation}
 The first equality
 is obtained from the first relation in~\eqref{YB-Not} by multiplying both sides on the left and on the right by $\bar{R}_{23}$, using the unitarity property of the R-matrix in~\eqref{inverse-R}, that is~$R_{23} \bar{R}_{23} \propto \mathbb{I}$, and substituting $u_2 \rightarrow u_2^{-1}$, $u_3 \rightarrow u_3^{-1}$. The two equations in~\eqref{annexR} are related by~${u_3 \rightarrow u_3^{-1}}$.

 First, we give a relation derived from the reflection equation that will be used later.
 \begin{lem}\label{lem:B1} Assuming the reflection equation holds for any $j_1$, $j_2$, $j_3$, then
 \begin{equation}
 \label{annexRK}
 \mathcal{K}_1 \hat{R}_{12} \mathcal{K}_2 \bar{R}_{12} = \bar{R}_{12} \mathcal{K}_2 \hat{R}_{12} \mathcal{K}_1.
 \end{equation}
 \end{lem}
 \begin{proof}
 According to the notation~\eqref{Not-lem}, the reflection equation~\eqref{REj1j2} reads
 \begin{equation} \label{RK-Not}
 R_{12} \mathcal{K}_1 \hat{R}_{12} \mathcal{K}_2 = \mathcal{K}_2 \hat{R}_{12} \mathcal{K}_1 R_{12}.
 \end{equation}
 The equation~\eqref{annexRK} is obtained by multiplying on both sides of~\eqref{RK-Not} on the left and on the right by $\bar{R}_{12}$ and using the unitarity property~\eqref{inverse-R}.
 \end{proof}

 In the following, we will need various expressions of the fused R-matrices.
 \begin{lem} \label{lem:exps-fused-R}
 The fused R-matrices defined in~\eqref{fusRstj1j2} satisfy the following relations:
 \begin{align} \label{fusRv1}
 R^{(j_1,j_2)}(u) &= \mathcal{F}^{(j_1)}_\fu R_{23}^{(j_1-\h,j_2)}\bigl(u q^{-\h}\bigr) R_{13}^{(\h,j_2)} \bigl(u q^{j_1-\h}\bigr)\mathcal{E}^{(j_1)}_\fu \\ \label{fusRv2}
 &= \mathcal{F}^{(j_2)}_{\langle 23 \rangle } R_{12}^{(j_1,\h)}\bigl(uq^{-j_2+\h}\bigr) R_{13}^{(j_1,j_2-\h)}
 \bigl(u q^{\h}\bigr)\mathcal{E}^{(j_2)}_{\langle 23 \rangle } \\ \label{fusRv3}
 & =\mathcal{F}^{(j_2)}_{\langle 23 \rangle} R_{13}^{(j_1,j_2-\h)}\bigl(u q^{-\h}\bigr) R_{12}^{(j_1,\h)}\bigl(u q^{j_2-\h}\bigr) \mathcal{E}^{(j_2)}_{\langle 23 \rangle}.
 \end{align}
 \end{lem}
 \begin{proof}
 First, note that the right-hand side of~\eqref{fusRv1}--\eqref{fusRv2} are directly derived from the expressions for the R-matrices in~\eqref{v2fused-R-uq12} and \eqref{v2fused-R-uq}, while the right-hand side of~\eqref{fusRv3} is a generalization\footnote{The relation~\eqref{fused-R-uq} is given for $j_1=\h$ but it can be generalized to any $j_1$ by computing $\mathcal{R}^{(j_1,j_2)}(u)=\smash{\bigl(\pi^{j_1} \otimes \id\bigr)\bigl({\bf L}^{(j_2)}(u)\bigr)}$ with ${\bf L}^{(j_2)}(u)$ in~\eqref{fused-L-uq}.} of~\eqref{fused-R-uq}. Recall that $R^{(j_1,j_2)}(u)$ is proportional to $\mathcal{R}^{(j_1,j_2)}(u)$, see~\eqref{evalRj1j2} with~\eqref{exp-f}, and using~\eqref{exp-mujh} one finds that the coefficient of proportionality reads
 \begin{equation} \label{gj1j2}
 f^{(j_1,j_2)}(u) = \prod_{k=0}^{2j_1-1} \prod_{\ell=0}^{2j_2-1} \pi^{\h}\bigl(\mu\bigl(u q^{j_1+j_2-1-k-\ell}\bigr)\bigr).
 \end{equation}
 Thus, we have 'smash{$R^{(j_1,j_2)}(u) = \mathcal{R}^{(j_1,j_2)}(u)/f^{(j_1,j_2)}(u)$}.
 Then, inserting in the right-hand side of the latter relation the corresponding R-matrices, the equations~\eqref{fusRv1}--\eqref{fusRv3} are obtained provided the following equalities hold:
 \begin{align*}
 \frac{f^{(j_1-\h,j_2)}\bigl(u q^{-\h}\bigr) f^{(\h,j_2)}\bigl(u q^{j_1-\h}\bigr)}{f^{(j_1,j_2)}(u)} &= \frac{f^{(j_1,\h)}\bigl(u q^{-j_2+\h}\bigr) f^{(j_1,j_2-\h)}\bigl(u q^\h\bigr)}{f^{(j_1,j_2)}(u)}\\
 & = \frac{f^{(j_1,j_2-\h)}\bigl(u q^{-\h}\bigr) f^{(j_1,\h)}\bigl(uq^{j_2-\h}\bigr)}{f^{(j_1,j_2)}(u)} = 1.
 \end{align*}
 Finally, it is easily checked that they indeed hold using~\eqref{gj1j2}.
 \end{proof}

 We now show by induction that the reflection equation~\eqref{REj1j2} holds for $(j_1,j_2)=\bigl(\h,j+\h\bigr)$ using the R-matrix decomposition in~\eqref{decompR} and Corollary~\ref{corEHF} together with the second result in Lemma~\ref{lem-Rhj}.
 \begin{lem}\label{hjh}
 The following relation holds for $j \in \frac{1}{2} \mathbb{N}$:
 \begin{gather}
 R_{1 \langle 23 \rangle}^{(\h,j+\h)}(u/v){\mathcal K}_1^{(\h)}(u)R_{1\langle 23 \rangle}^{(\h,j+\h)}(uv){\mathcal K}_{\langle 23 \rangle}^{(j+\h)}(v)\nonumber\\
 \qquad=
 {\mathcal K}_{\langle 23 \rangle}^{(j+\h)}(v)R_{1\langle 23 \rangle}^{(\h,j+\h)}(uv){\mathcal K}_1^{(\h)}(u)R_{1 \langle 23 \rangle}^{(\h,j+\h)}(u/v).\label{proofRK}
 \end{gather}
 \end{lem}
 \begin{proof}
 We proceed by induction. Recall the reflection equation~\eqref{REj1j2} holds for $(j_1,j_2)=\bigl(\h,\h\bigr)$ due to \cite{BSh1}. Assume that~\eqref{REj1j2} holds for $\bigl(\h,j\bigr)$ with a fixed value of $j \in \frac{1}{2} \mathbb{N}_+$. Consider the left-hand side of~\eqref{proofRK} and multiply on the right by
 \[
 \mathbb{I}_{2(2j+2)}=\mathcal{H}^{(j+\h)}_{\langle 23\rangle} \big\lbrack \mathcal{H}^{(j+\h)}_{\langle 23\rangle} \big\rbrack^{-1}.
\] Then, using the relations~\eqref{fusRv3},~\eqref{fusRv2} and the fused K-operator given in~\eqref{fusedunormK}, we get (with the necessary steps of the calculation underlined)
 \begin{gather} \nonumber
 R_{1 \langle 23 \rangle}^{(\h,j+\h)}(u/v){\mathcal K}_1^{(\h)}(u)R_{1\langle 23 \rangle}^{(\h,j+\h)}(uv){\mathcal K}_{\langle 23 \rangle}^{(j+\h)}(v) \mathcal{H}_{\langle 23 \rangle}^{(j+\h)}\big\lbrack \mathcal{H}_{\langle 23 \rangle}^{(j+\h)} \big\rbrack^{-1}\\
 \qquad=\mathcal{F}_{\langle 23 \rangle} R_{13}R_{12} \underline{\CE_{\langle 23 \rangle}{\mathcal K}_1} \CF_{\langle 23 \rangle} \hat{R}_{12} \hat{R}_{13} \CE_{\langle 23 \rangle } \CF_{\langle 23\rangle} {\mathcal K}_2 \hat{R}_{23} {\mathcal K}_3\underline{ \CE_{\langle 23 \rangle}\mathcal{H}_{\langle 23 \rangle}} \lbrack \mathcal{H}_{\langle 23 \rangle} \rbrack^{-1},\label{hjh-s1}
 \end{gather}
 where we fix in the notations from~\eqref{Not-lem}
 \begin{equation}\label{hjh-fix}
 j_1=j_2={\textstyle \h}, \qquad j_3=j, \qquad u_1=u, \qquad u_2=v q^{-j}, \qquad u_3=vq^{\h}.
 \end{equation}
 First, we remove the products $\mathcal{E}_{\langle 23 \rangle} \mathcal{F}_{\langle 23 \rangle}$. Recall that \smash{$R^{(\h,j)}(u)$} from~\eqref{fusRstraight} agrees with~\eqref{R-Rqg}, see Lemma~\ref{lem-Rhj}, then using~\eqref{EHF1} we find that~\eqref{hjh-s1} equals
 \begin{gather*}
\mathcal{F}_{\langle 23 \rangle} R_{13}R_{12}{\mathcal K}_1 \CE_{\langle 23 \rangle} \CF_{\langle 23 \rangle} \hat{R}_{12} \hat{R}_{13} \CE_{\langle 23 \rangle } \CF_{\langle 23\rangle} \underline{ {\mathcal K}_2 \hat{R}_{23} {\mathcal K}_3 \bar{R}_{23} } \CE_{\langle 23 \rangle} \lbrack \mathcal{H}_{\langle 23 \rangle} \rbrack^{-1}\\
 \qquad\overset{\eqref{annexRK}}{=}\mathcal{F}_{\langle 23 \rangle} R_{13}R_{12}{\mathcal K}_1 \CE_{\langle 23 \rangle} \CF_{\langle 23 \rangle} \hat{R}_{12} \hat{R}_{13} \underline{ \CE_{\langle 23 \rangle } \CF_{\langle 23\rangle} \bar{R}_{23} } {\mathcal K}_3 \hat{R}_{23} {\mathcal K}_2 \CE_{\langle 23 \rangle} \lbrack \mathcal{H}_{\langle 23 \rangle} \rbrack^{-1}\\
 \qquad\overset{\eqref{EHF3}}{=}\mathcal{F}_{\langle 23 \rangle} R_{13}R_{12}{\mathcal K}_1 \CE_{\langle 23 \rangle} \CF_{\langle 23 \rangle} \underline{ \hat{R}_{12} \hat{R}_{13}\bar{R}_{23} } {\mathcal K}_3 \hat{R}_{23} {\mathcal K}_2 \CE_{\langle 23 \rangle} \lbrack \mathcal{H}_{\langle 23 \rangle} \rbrack^{-1}\\
 \qquad\overset{\eqref{YB-Not}}{=}\mathcal{F}_{\langle 23 \rangle} R_{13}R_{12}{\mathcal K}_1 \underline{\CE_{\langle 23 \rangle} \CF_{\langle 23 \rangle} \bar{R}_{23} } \hat{R}_{13}\hat{R}_{12} {\mathcal K}_3 \hat{R}_{23} {\mathcal K}_2 \CE_{\langle 23 \rangle} \lbrack \mathcal{H}_{\langle 23 \rangle} \rbrack^{-1}\\
 \qquad\overset{\eqref{EHF3}}{=}\mathcal{F}_{\langle 23 \rangle} R_{13}R_{12}{\mathcal K}_1 \underline{\bar{R}_{23} \hat{R}_{13}\hat{R}_{12} {\mathcal K}_3 \hat{R}_{23} {\mathcal K}_2 \CE_{\langle 23 \rangle} \lbrack \mathcal{H}_{\langle 23 \rangle} \rbrack^{-1}} \\
\qquad = \mathcal{F}_{\langle 23 \rangle} R_{13}R_{12} {\mathcal K}_1 \hat{R}_{12} \underline{\hat{R}_{13} {\mathcal K}_2} \hat{R}_{23} {\mathcal K}_3 \CE_{\langle 23 \rangle},
 \end{gather*}
 where we used~\eqref{YB-Not},~\eqref{annexRK} and~\eqref{EHF1} to get the last line.
 Secondly, rearranging the different R-matrices and K-operators using $[R_{23},{\mathcal K}_1]=[R_{13},{\mathcal K}_2] = \big[ \hat{R}_{13}, {\mathcal K}_2 \big] =\big[ \hat{R}_{12}, {\mathcal K}_3 \big]=0$ and~\eqref{annexR},~\eqref{YB-Not},~\eqref{RK-Not}, we get
 \begin{gather*}
 = \mathcal{F}_{\langle 23 \rangle} R_{13} \underline{R_{12} {\mathcal K}_1 \hat{R}_{12}
 {\mathcal K}_2} \hat{R}_{13} \hat{R}_{23} {\mathcal K}_3 \CE_{\langle 23 \rangle}
\overset{\eqref{RK-Not}}{=} \mathcal{F}_{\langle 23 \rangle} \underline{ R_{13}{\mathcal K}_2 } \hat{R}_{12}{\mathcal K}_1 \underline{R_{12}
 \hat{R}_{13} \hat{R}_{23}} {\mathcal K}_3 \CE_{\langle 23 \rangle} \\
 \overset{\eqref{YB-Not}}{=} \mathcal{F}_{\langle 23 \rangle} {\mathcal K}_2 \underline{ R_{13} \hat{R}_{12}{\mathcal K}_1 \hat{R}_{23}}
 \hat{R}_{13}\underline{ R_{12} {\mathcal K}_3 }\CE_{\langle 23 \rangle}
\overset{\eqref{annexR}}{=} \mathcal{F}_{\langle 23 \rangle} {\mathcal K}_2 \hat{R}_{23} \hat{R}_{12} \underline{R_{13} {\mathcal K}_1
 \hat{R}_{13} {\mathcal K}_3 } R_{12} \CE_{\langle 23 \rangle} \\
 \overset{\eqref{RK-Not}}{=} \mathcal{F}_{\langle 23 \rangle} {\mathcal K}_2 \hat{R}_{23}\underline{ \hat{R}_{12}{\mathcal K}_3 } \hat{R}_{13} {\mathcal K}_1
 R_{13} R_{12} \CE_{\langle 23 \rangle}
 = \mathcal{F}_{\langle 23 \rangle} {\mathcal K}_2 \hat{R}_{23}{\mathcal K}_3 \hat{R}_{12} \hat{R}_{13} {\mathcal K}_1
 R_{13} R_{12} \CE_{\langle 23 \rangle}\\
 = \mathcal{F}_{\langle 23 \rangle} {\mathcal K}_2 \hat{R}_{23}{\mathcal K}_3 \CE_{\langle 23 \rangle} \mathcal{F}_{\langle 23 \rangle} \hat{R}_{12} \hat{R}_{13}\CE_{\langle 23 \rangle} {\mathcal K}_1 \mathcal{F}_{\langle 23 \rangle}
 R_{13} R_{12} \CE_{\langle 23 \rangle},
 \end{gather*}
 where we put back the products $\mathcal{E}_{\langle 23 \rangle}\mathcal{F}_{\langle 23 \rangle}$ by inserting \smash{$\mathbb{I}_{2(2j+2)}=\mathcal{H}_{\langle 23 \rangle}^{(j+\h)} \big\lbrack \mathcal{H}_{\langle 23\rangle}^{(j+\h)}\big\rbrack^{-1}$} and using~\eqref{annexR}, \eqref{annexRK},~\eqref{EHF1}--\eqref{EHF3}. Finally, identifying the fused R-matrices and the fused K-operator, the equation~\eqref{proofRK} follows.
 \end{proof}

 Then, we show by induction that the reflection equation~\eqref{REj1j2} holds for $(j_1,j_2)=\bigl(j+\h,\h\bigr)$.

 \begin{lem}\label{jhh}The following relation holds for $j \in \frac{1}{2} \mathbb{N}$:
 \begin{gather}
 R_{\fu 3}^{(j+\h,\h)}(u/v) {\mathcal K}_\fu^{(j+\h)}(u) R_{\fu 3}^{(j+\h,\h)}(u v) {\mathcal K}_3^{(\h)}(v)\nonumber \\
 \qquad={\mathcal K}_3^{(\h)}(v)R_{\fu 3}^{(j+\h,\h)}(u v){\mathcal K}_\fu^{(j+\h)}(u)R_{\fu 3}^{(j+\h,\h)}(u/v). \label{eqjhh}
 \end{gather}
 \end{lem}
 \begin{proof}We proceed by induction similarly to the proof of Lemma~\ref{hjh}. Recall the reflection equation~\eqref{REj1j2} holds for $(j_1,j_2)=\bigl(\h,\h\bigr)$ due to \cite{BSh1}. Assume that~\eqref{REj1j2} holds for $\bigl(j,\h\bigr)$ with a fixed value of $j \in \frac{1}{2} \mathbb{N}_+$. Consider the left-hand side of~\eqref{eqjhh} and multiply on the right by
\[
 \mathbb{I}_{2(2j+2)}=\mathcal{H}^{(j+\h)}_{\langle 12\rangle} \big\lbrack \mathcal{H}^{(j+\h)}_{\langle 12\rangle} \big\rbrack^{-1}.
\]
Then, using the fused R-matrix~\eqref{fusRstj1j2} and the fused K-operator given in~\eqref{fusedunormK}, we get (with the necessary steps of the calculation underlined)
 \begin{gather}
 R_{\fu 3}^{(j+\h,\h)}(u/v) {\mathcal K}_\fu^{(j+\h)}(u) R_{\fu 3}^{(j+\h,\h)}(u v) {\mathcal K}_3^{(\h)}(v)\mathcal{H}_\fu^{(j+\h)} \big\lbrack \mathcal{H}^{(j+\h)}_\fu \big\rbrack^{-1} \nonumber\\
 \qquad=\mathcal{F}_{\fu} R_{13} R_{23}\mathcal{E}_\fu\mathcal{F}_\fu {\mathcal K}_1 \hat{R}_{12} {\mathcal K}_2 \mathcal{E}_\fu\mathcal{F}_\fu \hat{R}_{13} \hat{R}_{23}\mathcal{E}_\fu \mathcal{H}_\fu {\mathcal K}_3 \lbrack \mathcal{H}_\fu \rbrack^{-1},\label{jhh-s1}
 \end{gather}
 where we fix in the notations from~\eqref{Not-lem}
 \begin{equation} \label{jhh-fix}
 j_1=j_3= {\textstyle \h}, \qquad j_2=j, \qquad u_1= uq^{-j},\qquad u_2= u q^{\h}, \qquad u_3=v.
 \end{equation}
 Firstly, similarly to Lemma~\ref{hjh}, we remove the products $\mathcal{E}_\fu\mathcal{F}_\fu$ using~\eqref{EHF1}--\eqref{EHF3},~\eqref{YB-Not} and~\eqref{RK-Not}.
 Then, we find that~\eqref{jhh-s1} equals
 \begin{align} \label{jhh-s2}
 \mathcal{F}_\fu R_{13} R_{23} {\mathcal K}_1 \hat{R}_{12} \underline{{\mathcal K}_2 \hat{R}_{13}} \hat{R}_{23} {\mathcal K}_3 \mathcal{E}_\fu.
 \end{align}
 Secondly, we rearrange the R-matrices and K-operators in~\eqref{jhh-s2} as follows:
 \begin{gather*}
 \mathcal{F}_\fu R_{13} {\mathcal K}_1 \underline{R_{23} \hat{R}_{12} \hat{R}_{13}} {\mathcal K}_2 \hat{R}_{23} {\mathcal K}_3 \mathcal{E}_\fu
 \overset{\eqref{annexR}}{=}\mathcal{F}_\fu R_{13} {\mathcal K}_1 \hat{R}_{13}\hat{R}_{12} \underline{R_{23} {\mathcal K}_2 \hat{R}_{23} {\mathcal K}_3} \mathcal{E}_\fu \\
 \qquad\overset{\eqref{RK-Not}}{=}\mathcal{F}_\fu\underline{R_{13} {\mathcal K}_1 \hat{R}_{13} \hat{R}_{12} {\mathcal K}_3}\hat{R}_{23} {\mathcal K}_2R_{23} \mathcal{E}_\fu
 \overset{\eqref{RK-Not}}{=}\underline{\mathcal{F}_\fu{\mathcal K}_3} \hat{R}_{13} {\mathcal K}_1 \underline{ R_{13} \hat{R}_{12} \hat{R}_{23}} {\mathcal K}_2R_{23} \mathcal{E}_\fu \\
 \qquad \overset{\eqref{annexR}}{=}{\mathcal K}_3 \mathcal{F}_\fu \hat{R}_{13} \underline{ {\mathcal K}_1 \hat{R}_{23}}\hat{R}_{12} \underline{R_{13} {\mathcal K}_2}R_{23} \mathcal{E}_\fu
 = {\mathcal K}_3 \mathcal{F}_\fu \hat{R}_{13} \hat{R}_{23}{\mathcal K}_1 \hat{R}_{12} {\mathcal K}_2R_{13} R_{23} \mathcal{E}_\fu \\
 \qquad = {\mathcal K}_3 \mathcal{F}_\fu \hat{R}_{13} \hat{R}_{23}\mathcal{E}_\fu\mathcal{F}_\fu {\mathcal K}_1 \hat{R}_{12} {\mathcal K}_2\mathcal{E}_\fu\mathcal{F}_\fu R_{13} R_{23} \mathcal{E}_\fu,
 \end{gather*}
 where we put back the products $\mathcal{E}_{\fu}\mathcal{F}_\fu$ inserting \smash{$\mathbb{I}_{2(2j+2)} = \mathcal{H}^{(j+\h)}_{\langle 12\rangle} \big\lbrack \mathcal{H}^{(j+\h)}_{\langle 12\rangle} \big\rbrack^{-1}$} and using \eqref{EHF1}--\eqref{EHF3},~\eqref{annexR} and \eqref{annexRK}. Finally, identifying the fused R-matrices and the fused K-operator, the equation~\eqref{eqjhh} follows.
 \end{proof}

 \begin{proof}[Proof of Theorem~\ref{prop:fusedRE}]
 We are now ready to show that the reflection equation~\eqref{REj1j2} is satisfied by the fused K-operator~\eqref{fusedunormK} for all $j_1, j_2 \in \h \mathbb{N}_+$.
 We consider separately~\eqref{REj1j2} for two distinct cases $j_1 \geq j_2$ and $j_1 \leq j_2$
 \begin{itemize}\itemsep=0pt
 \item[(i)] $(j_1,j_2)=\bigl(j+\h,k+\h\bigr)$ with $0\leq k \leq j$,
 \item[(ii)] $(j_1,j_2)=\bigl(\ell+\h,j+\h\bigr)$ with $0\leq \ell \leq j$.
 \end{itemize}

The first case is shown by induction on $k$. For $k=0$, the equation~\eqref{REj1j2} holds for $(j_1,j_2)=\bigl(j+\h,\h\bigr)$ due to Lemma~\ref{jhh}. Now, assume~\eqref{REj1j2} holds for $(j_1,j_2)=\bigl(j+\h,k\bigr)$ with a fixed value of $k \leq j$. The case $(i)$ is shown similarly to the proof of Lemma~\ref{hjh}.
 Indeed, fix in~\eqref{hjh-s1} $j_1=j+\h$, $j_2=\h$, $j_3=k$ and $u_1=u$, $u_2=vq^{-k}$, \smash{$u_3=vq^{\h}$} (instead of~\eqref{hjh-fix}). Then, the rest of the proof is the same as for Lemma~\ref{hjh}, using now~\eqref{proofRK},~\eqref{eqjhh},~\eqref{evalRj1j2} and our assumption.

The second case is also shown by induction on $\ell$. For $\ell=0$, the equation~\eqref{REj1j2} holds for~${(j_1,j_2)=\bigl(\h,j+\h\bigr)}$ due to Lemma~\ref{hjh}. Assuming that~\eqref{REj1j2} holds for $(j_1,j_2)=\bigl(\ell, j+\h\bigr)$ with a fixed value of $\ell \leq j$, then the proof of the case $(ii)$ is similar to Lemma~\ref{jhh}. Fix in~\eqref{jhh-s1} $j_1=\h$, $j_2=\ell$, $j_3=j+\h$ and $u_1=uq^{-\ell}$, \smash{$u_2=uq^{\h}$}, $u_3=v$ (instead of~\eqref{jhh-fix}). Then, the rest of the proof is the same as for Lemma~\ref{jhh}, using now~\eqref{proofRK},~\eqref{eqjhh},~\eqref{evalRj1j2} and our assumption.

This concludes the proof of Theorem \ref{prop:fusedRE}.
 \end{proof}

\subsection*{Acknowledgements}
We thank A.~Appel, P.~Terwilliger, and B.~Vlaar for stimulating discussions. We thank anonymous referees for comments and suggestions on the manuscript.
P.B. and A.M.G. are supported by C.N.R.S. The work of A.M.G. was also partially supported by the ANR grant JCJC ANR-18-CE40-0001.

\pdfbookmark[1]{References}{ref}
\LastPageEnding


\begin{thebibliography}{99}
\footnotesize\itemsep=0pt

\bibitem{AV20}
Appel A., Vlaar B., Universal {K}-matrices for quantum {K}ac--{M}oody algebras,
 \href{https://doi.org/10.1090/ert/623}{\textit{Represent. Theory}}
 \textbf{26} (2022), 764--824,
 \href{http://arxiv.org/abs/2007.09218}{arXiv:2007.09218}.

\bibitem{AV25}
Appel A., Vlaar B., Tensor {$K$}-matrices for quantum symmetric pairs,
 \href{https://doi.org/10.1007/s00220-025-05241-5}{\textit{Comm. Math. Phys.}}
 \textbf{406} (2025), 100, 57~pages,
 \href{http://arxiv.org/abs/2402.16676}{arXiv:2402.16676}.

\bibitem{AV22}
Appel A., Vlaar B., Trigonometric {$K$}-matrices for finite-dimensional
 representations of quantum affine algebras,
 \href{https://doi.org/10.4171/JEMS/1686}{\textit{J.~Eur. Math. Soc. (JEMS)}},
 {t}o appear, \href{http://arxiv.org/abs/2203.16503}{arXiv:2203.16503}.

\bibitem{BKo19}
Balagovi\'c M., Kolb S., Universal {$K$}-matrix for quantum symmetric pairs,
 \href{https://doi.org/10.1515/crelle-2016-0012}{\textit{J.~Reine Angew.
 Math.}} \textbf{747} (2019), 299--353,
 \href{http://arxiv.org/abs/1507.06276}{arXiv:1507.06276}.

\bibitem{BW13}
Bao H., Wang W., A new approach to {K}azhdan--{L}usztig theory of type~{$B$}
 via quantum symmetric pairs, \textit{Ast\'erisque} \textbf{402} (2018),
 vii+134~pages, \href{http://arxiv.org/abs/1310.0103}{arXiv:1310.0103}.

\bibitem{Ba05}
Baseilhac P., Deformed {D}olan--{G}rady relations in quantum integrable models,
 \href{https://doi.org/10.1016/j.nuclphysb.2004.12.016}{\textit{Nuclear
 Phys.~B}} \textbf{709} (2005), 491--521,
 \href{http://arxiv.org/abs/hep-th/0404149}{arXiv:hep-th/0404149}.

\bibitem{BB12}
Baseilhac P., Belliard S., The half-infinite~{XXZ} chain in {O}nsager's
 approach,
 \href{https://doi.org/10.1016/j.nuclphysb.2013.05.003}{\textit{Nuclear
 Phys.~B}} \textbf{873} (2013), 550--584,
 \href{http://arxiv.org/abs/1211.6304}{arXiv:1211.6304}.

\bibitem{BB16}
Baseilhac P., Belliard S., Non-{A}belian symmetries of the half-infinite~{XXZ}
 spin chain,
 \href{https://doi.org/10.1016/j.nuclphysb.2017.01.012}{\textit{Nuclear
 Phys.~B}} \textbf{916} (2017), 373--385,
 \href{http://arxiv.org/abs/1611.05390}{arXiv:1611.05390}.

\bibitem{BasBel}
Baseilhac P., Belliard S., An attractive basis for the $q$-{O}nsager algebra,
 \href{http://arxiv.org/abs/1704.02950}{arXiv:1704.02950}.

\bibitem{LBG}
Baseilhac P., Gainutdinov A.M., Lemarthe G., Universal {TT}-and {TQ}-relations
 via centrally extended $q$-{O}nsager algebra,
 \href{http://arxiv.org/abs/2511.15876}{arXiv:2511.15876}.

\bibitem{BK03}
Baseilhac P., Koizumi K., Sine-{G}ordon quantum field theory on the half-line
 with quantum boundary degrees of freedom,
 \href{https://doi.org/10.1016/S0550-3213(02)00980-X}{\textit{Nuclear
 Phys.~B}} \textbf{649} (2003), 491--510,
 \href{http://arxiv.org/abs/hep-th/0208005}{arXiv:hep-th/0208005}.

\bibitem{BK07}
Baseilhac P., Koizumi K., A deformed analogue of {O}nsager's symmetry in
 the~{$XXZ$} open spin chain,
 \href{https://doi.org/10.1088/1742-5468/2005/10/p10005}{\textit{J.~Stat.
 Mech. Theory Exp.}} \textbf{2005} (2005), P10005, 15~pages,
 \href{http://arxiv.org/abs/hep-th/0507053}{arXiv:hep-th/0507053}.

\bibitem{BK05}
Baseilhac P., Koizumi K., A new (in)finite-dimensional algebra for quantum
 integrable models,
 \href{https://doi.org/10.1016/j.nuclphysb.2005.05.021}{\textit{Nuclear
 Phys.~B}} \textbf{720} (2005), 325--347,
 \href{http://arxiv.org/abs/math-ph/0503036}{arXiv:math-ph/0503036}.

\bibitem{BK17}
Baseilhac P., Kolb S., Braid group action and root vectors for the
 {$q$}-{O}nsager algebra,
 \href{https://doi.org/10.1007/s00031-020-09555-7}{\textit{Transform. Groups}}
 \textbf{25} (2020), 363--389,
 \href{http://arxiv.org/abs/1706.08747}{arXiv:1706.08747}.

\bibitem{BP19}
Baseilhac P., Pimenta R.A., Diagonalization of the {H}eun--{A}skey--{W}ilson
 operator, {L}eonard pairs and the algebraic {B}ethe ansatz,
 \href{https://doi.org/10.1016/j.nuclphysb.2019.114824}{\textit{Nuclear
 Phys.~B}} \textbf{949} (2019), 114824, 66~pages,
 \href{http://arxiv.org/abs/1909.02464}{arXiv:1909.02464}.

\bibitem{BSh1}
Baseilhac P., Shigechi K., A new current algebra and the reflection equation,
 \href{https://doi.org/10.1007/s11005-010-0380-x}{\textit{Lett. Math. Phys.}}
 \textbf{92} (2010), 47--65,
 \href{http://arxiv.org/abs/0906.1482}{arXiv:0906.1482}.

\bibitem{BT17}
Baseilhac P., Tsuboi Z., Asymptotic representations of augmented
 {$q$}-{O}nsager algebra and boundary {$K$}-operators related to {B}axter
 {$Q$}-operators,
 \href{https://doi.org/10.1016/j.nuclphysb.2018.02.017}{\textit{Nuclear
 Phys.~B}} \textbf{929} (2018), 397--437,
 \href{http://arxiv.org/abs/1707.04574}{arXiv:1707.04574}.

\bibitem{Ba72}
Baxter R.J., Partition function of the eight-vertex lattice model,
 \href{https://doi.org/10.1016/0003-4916(72)90335-1}{\textit{Ann. Physics}}
 \textbf{70} (1972), 193--228.

\bibitem{Be2015}
Beisert N., de~Leeuw M., Nag P., Fusion for the one-dimensional {H}ubbard
 model,
 \href{https://doi.org/10.1088/1751-8113/48/32/324002}{\textit{J.~Phys.~A}}
 \textbf{48} (2015), 324002, 17~pages,
 \href{http://arxiv.org/abs/1503.04838}{arXiv:1503.04838}.

\bibitem{B14}
Belliard S., Modified algebraic {B}ethe ansatz for~{XXZ} chain on the
 segment~-- {I}: {T}riangular cases,
 \href{https://doi.org/10.1016/j.nuclphysb.2015.01.003}{\textit{Nuclear
 Phys.~B}} \textbf{892} (2015), 1--20,
 \href{http://arxiv.org/abs/1408.4840}{arXiv:1408.4840}.

\bibitem{BC13}
Belliard S., Cramp\'e N., Heisenberg~{XXX} model with general boundaries:
 eigenvectors from algebraic {B}ethe ansatz,
 \href{https://doi.org/10.3842/SIGMA.2013.072}{\textit{SIGMA}} \textbf{9}
 (2013), 072, 12~pages,
 \href{http://arxiv.org/abs/1309.6165}{arXiv:1309.6165}.

\bibitem{BF11}
Belliard S., Fomin V., Generalized {$q$}-{O}nsager algebras and dynamical
 {$K$}-matrices,
 \href{https://doi.org/10.1088/1751-8113/45/2/025201}{\textit{J.~Phys.~A}}
 \textbf{45} (2012), 025201, 17~pages,
 \href{http://arxiv.org/abs/1106.1317}{arXiv:1106.1317}.

\bibitem{BP14}
Belliard S., Pimenta R.A., Modified algebraic {B}ethe ansatz for~{XXZ} chain on
 the segment~-- {II}~-- {G}eneral cases,
 \href{https://doi.org/10.1016/j.nuclphysb.2015.03.016}{\textit{Nuclear
 Phys.~B}} \textbf{894} (2015), 527--552,
 \href{http://arxiv.org/abs/1412.7511}{arXiv:1412.7511}.

\bibitem{Boos2010}
Boos H., G\"ohmann F., Kl\"umper A., Nirov K.S., Razumov A.V., Exercises with
 the universal {$R$}-matrix,
 \href{https://doi.org/10.1088/1751-8113/43/41/415208}{\textit{J.~Phys.~A}}
 \textbf{43} (2010), 415208, 35~pages,
 \href{http://arxiv.org/abs/1004.5342}{arXiv:1004.5342}.

\bibitem{Boos2012}
Boos H., G\"ohmann F., Kl\"umper A., Nirov K.S., Razumov A.V., Universal
 {$R$}-matrix and functional relations,
 \href{https://doi.org/10.1142/S0129055X14300052}{\textit{Rev. Math. Phys.}}
 \textbf{26} (2014), 1430005, 66~pages,
 \href{http://arxiv.org/abs/1205.1631}{arXiv:1205.1631}.

\bibitem{gradation-paper}
Bracken A.J., Gould M.D., Zhang Y.-Z., Delius G.W.,
 \href{https://doi.org/10.1142/S0217979294001585}{\textit{Internat.~J. Modern
 Phys.~B}} \textbf{8} (1994), 3679--3691,
 \href{http://arxiv.org/abs/hep-th/9310183}{arXiv:hep-th/9310183}.

\bibitem{CYSW14}
Cao J., Yang W.-L., Shi K., Wang Y., Exact solution of the alternating {XXZ}
 spin chain with generic non-diagonal boundaries,
 \href{https://doi.org/10.1016/j.aop.2015.06.009}{\textit{Ann. Physics}}
 \textbf{361} (2015), 91--106,
 \href{http://arxiv.org/abs/1409.3646}{arXiv:1409.3646}.

\bibitem{chari}
Chari V., Greenstein J., Quantum loop modules,
 \href{https://doi.org/10.1090/S1088-4165-03-00168-7}{\textit{Represent.
 Theory}} \textbf{7} (2003), 56--80,
 \href{http://arxiv.org/abs/math.QA/0206306}{arXiv:math.QA/0206306}.

\bibitem{CP}
Chari V., Pressley A., Quantum affine algebras,
 \href{https://doi.org/10.1007/BF02102063}{\textit{Comm. Math. Phys.}}
 \textbf{142} (1991), 261--283.

\bibitem{CP95}
Chari V., Pressley A., A guide to quantum groups, Cambridge University Press,
 Cambridge, 1995.

\bibitem{Cher84}
Cherednik I.V., Factorizing particles on a half line, and root systems,
 \href{https://doi.org/10.1007/BF01038545}{\textit{Theoret. and Math. Phys.}}
 \textbf{61} (1984), 977--983.

\bibitem{CGJS}
Chernyak D., Gainutdinov A.M., Jacobsen J.L., Saleur H., Algebraic {B}ethe
 ansatz for the open~{XXZ} spin chain with non-diagonal boundary terms
 via~{$U_{\mathfrak{q}}\mathfrak{sl}_2$} symmetry,
 \href{https://doi.org/10.3842/SIGMA.2023.046}{\textit{SIGMA}} \textbf{19}
 (2023), 046, 47~pages,
 \href{http://arxiv.org/abs/2212.09696}{arXiv:2212.09696}.

\bibitem{CGS}
Chernyak D., Gainutdinov A.M., Saleur H., {$U_{\mathfrak
 q}\mathfrak{sl}_2$}-invariant non-compact boundary conditions for the {XXZ}
 spin chain, \href{https://doi.org/10.1007/jhep11(2022)016}{\textit{J.~High
 Energy Phys.}} \textbf{2022} (2022), no.~11, 016, 64~pages,
 \href{http://arxiv.org/abs/2207.12772}{arXiv:2207.12772}.

\bibitem{CG92}
Cremmer E., Gervais J.L., The quantum strip: {L}iouville theory for open
 strings, \href{https://doi.org/10.1007/BF02101093}{\textit{Comm. Math.
 Phys.}} \textbf{144} (1992), 279--301.

\bibitem{Da98}
Damiani I., La {$R$}-matrice pour les alg\`ebres quantiques de type affine non
 tordu, \href{https://doi.org/10.1016/S0012-9593(98)80104-3}{\textit{Ann. Sci.
 \'Ecole Norm. Sup.}} \textbf{31} (1998), 493--523.

\bibitem{VG94}
de~Vega H.J., Gonz\'alez-Ruiz A., Boundary {$K$}-matrices for the~{$XYZ$},
 {$XXZ$} and~{$XXX$} spin chains,
 \href{https://doi.org/10.1088/0305-4470/27/18/021}{\textit{J.~Phys.~A}}
 \textbf{27} (1994), 6129--6137,
 \href{http://arxiv.org/abs/hep-th/9306089}{arXiv:hep-th/9306089}.

\bibitem{DN02}
Delius G.W., Nepomechie R.I., Solutions of the boundary {Y}ang--{B}axter
 equation for arbitrary spin,
 \href{https://doi.org/10.1088/0305-4470/35/24/102}{\textit{J.~Phys.~A}}
 \textbf{35} (2002), 341--348,
 \href{http://arxiv.org/abs/hep-th/0204076}{arXiv:hep-th/0204076}.

\bibitem{DF93}
Ding J.T., Frenkel I.B., Isomorphism of two realizations of quantum affine
 algebra~{$U_q({\mathfrak g}{\mathfrak l}(n))$},
 \href{https://doi.org/10.1007/BF02098484}{\textit{Comm. Math. Phys.}}
 \textbf{156} (1993), 277--300.

\bibitem{DKM03}
Donin J., Kulish P.P., Mudrov A.I., On a universal solution to the reflection
 equation, \href{https://doi.org/10.1023/A:1024438101617}{\textit{Lett. Math.
 Phys.}} \textbf{63} (2003), 179--194,
 \href{http://arxiv.org/abs/hep-th/0210242}{arXiv:hep-th/0210242}.

\bibitem{Dr0}
Drinfeld V.G., Quantum groups, in Proceedings of the {I}nternational {C}ongress
 of {M}athematicians, {V}ol.~1, 2 ({B}erkeley, {C}alif., 1986), American
 Mathematical Society, Providence, RI, 1987, 798--820.

\bibitem{Dr89b}
Drinfeld V.G., Quasi-{H}opf algebras and {K}nizhnik--{Z}amolodchikov equations,
 in Problems of {M}odern {Q}uantum {F}ield {T}heory ({A}lushta, 1989), \textit{Res.
 Rep. Phys.}, \href{https://doi.org/10.1007/978-3-642-84000-5_1}{Springer},
 Berlin, 1989, 1--13.

\bibitem{Dr89}
Drinfeld V.G., Quasi-{H}opf algebras, \textit{Leningrad Math.~J.} \textbf{1}
 (1990), 1419--1457.

\bibitem{E08}
Enriquez B., Quasi-reflection algebras and cyclotomic associators,
 \href{https://doi.org/10.1007/s00029-007-0048-2}{\textit{Selecta
 Math.~(N.S.)}} \textbf{13} (2007), 391--463,
 \href{http://arxiv.org/abs/math.QA/0408035}{arXiv:math.QA/0408035}.

\bibitem{FRT87}
Faddeev L.D., Reshetikhin N.Yu., Takhtajan L.A., Quantization of {L}ie groups
 and {L}ie algebras, in Algebraic Analysis, Vol.~I,
 \href{https://doi.org/10.1016/B978-0-12-400465-8.50019-5}{Academic Press},
 Boston, MA, 1988, 129--139.

\bibitem{FNR07}
Frappat L., Nepomechie R.I., Ragoucy E., A complete {B}ethe ansatz solution for
 the open spin-{$s$}~{$XXZ$} chain with general integrable boundary terms,
 \href{https://doi.org/10.1088/1742-5468/2007/09/p09009}{\textit{J.~Stat.
 Mech. Theory Exp.}} \textbf{2007} (2007), P09009, 18~pages,
 \href{http://arxiv.org/abs/0707.0653}{arXiv:0707.0653}.

\bibitem{FR92}
Frenkel I.B., Reshetikhin N.Yu., Quantum affine algebras and holonomic
 difference equations, \href{https://doi.org/10.1007/BF02099206}{\textit{Comm.
 Math. Phys.}} \textbf{146} (1992), 1--60.

\bibitem{GZ94}
Ghoshal S., Zamolodchikov A., Boundary~{$S$} matrix and boundary state in
 two-dimensional integrable quantum field theory,
 \href{https://doi.org/10.1142/S0217751X94001552}{\textit{Internat.~J. Modern
 Phys.~A}} \textbf{9} (1994), 3841--3885,
 \href{http://arxiv.org/abs/hep-th/9306002}{arXiv:hep-th/9306002}.

\bibitem{He17}
Hernandez D., Avancees concernant les {$R$}-matrices et leurs applications, \textit{Ast\'erisque} \textbf{407} (2019), 297--331, \href{http://arxiv.org/abs/hep-th/9306002}{arXiv:hep-th/9306002}.

\bibitem{inami96}
Inami T., Odake S., Zhang Y.-Z., Reflection {$K$}-matrices of the {$19$}-vertex
 model and~{$XXZ$} spin-{$1$} chain with general boundary terms,
 \href{https://doi.org/10.1016/0550-3213(96)00133-2}{\textit{Nuclear Phys.~B}}
 \textbf{470} (1996), 419--432,
 \href{http://arxiv.org/abs/hep-th/9601049}{arXiv:hep-th/9601049}.

\bibitem{JLM}
Jing N., Liu M., Molev A., Isomorphism between the {$R$}-matrix and {D}rinfeld
 presentations of quantum affine algebra: {T}ype~{$C$},
 \href{https://doi.org/10.1063/1.5133854}{\textit{J.~Math. Phys.}} \textbf{61}
 (2020), 031701, 41~pages,
 \href{http://arxiv.org/abs/1903.00204}{arXiv:1903.00204}.

\bibitem{JLMBD}
Jing N., Liu M., Molev A., Isomorphism between the {$R$}-matrix and {D}rinfeld
 presentations of quantum affine algebra: types~{$B$} and~{$D$},
 \href{https://doi.org/10.3842/SIGMA.2020.043}{\textit{SIGMA}} \textbf{16}
 (2020), 043, 49~pages,
 \href{http://arxiv.org/abs/1911.03496}{arXiv:1911.03496}.

\bibitem{Ka79}
Karowski M., On the bound state problem in $1+1$ dimensional field theories,
 \href{https://doi.org/10.1016/0550-3213(79)90600-X}{\textit{Nuclear Phys.~B}}
 \textbf{153} (1979), 244--252.

\bibitem{kashi}
Kashiwara M., On level-zero representations of quantized affine algebras,
 \href{https://doi.org/10.1215/S0012-9074-02-11214-9}{\textit{Duke Math.~J.}}
 \textbf{112} (2002), 117--175,
 \href{http://arxiv.org/abs/math.QA/0010293}{arXiv:math.QA/0010293}.

\bibitem{Tolstoy1992}
Khoroshkin S.M., Tolstoy V.N., The uniqueness theorem for the universal
 {$R$}-matrix, \href{https://doi.org/10.1007/BF00402899}{\textit{Lett. Math.
 Phys.}} \textbf{24} (1992), 231--244.

\bibitem{KR87}
Kirillov A.N., Reshetikhin N.Yu., Exact solution of the integrable~{$XXZ$}
 {H}eisenberg model with arbitrary spin.~{I}. {T}he ground state and the
 excitation spectrum,
 \href{https://doi.org/10.1088/0305-4470/20/6/038}{\textit{J.~Phys.~A}}
 \textbf{20} (1987), 1565--1585.

\bibitem{Klimyk}
Klimyk A., Schm\"udgen K., Quantum groups and their representations, \textit{Texts
 Monogr. Phys.}, \href{https://doi.org/10.1007/978-3-642-60896-4}{Springer},
 Berlin, 1997.

\bibitem{Ko20}
Kolb S., Braided module categories via quantum symmetric pairs,
 \href{https://doi.org/10.1112/plms.12303}{\textit{Proc. Lond. Math. Soc.}}
 \textbf{121} (2020), 1--31,
 \href{http://arxiv.org/abs/1705.04238}{arXiv:1705.04238}.

\bibitem{KY20}
Kolb S., Yakimov M., Symmetric pairs for {N}ichols algebras of diagonal type
 via star products,
 \href{https://doi.org/10.1016/j.aim.2020.107042}{\textit{Adv. Math.}}
 \textbf{365} (2020), 107042, 69~pages,
 \href{http://arxiv.org/abs/1901.00490}{arXiv:1901.00490}.

\bibitem{KR83}
Kulish P.P., Reshetikhin N.Yu., Quantum linear problem for the sine-{G}ordon
 equation and higher representations,
 \href{https://doi.org/10.1007/BF01084171}{\textit{J.~Sov. Math.}} \textbf{23}
 (1983), 2435--2441.

\bibitem{KRS81}
Kulish P.P., Reshetikhin N.Yu., Sklyanin E.K., Yang--{B}axter equations and
 representation theory.~{I},
 \href{https://doi.org/10.1007/BF02285311}{\textit{Lett. Math. Phys.}}
 \textbf{5} (1981), 393--403.

\bibitem{KSS92}
Kulish P.P., Sasaki R., Schwiebert C., Constant solutions of reflection
 equations and quantum groups,
 \href{https://doi.org/10.1063/1.530382}{\textit{J.~Math. Phys.}} \textbf{34}
 (1993), 286--304,
 \href{http://arxiv.org/abs/hep-th/9205039}{arXiv:hep-th/9205039}.

\bibitem{LS}
Levendorskii S., Soibelman Y., Stukopin V., The quantum {W}eyl group and the
 universal quantum {$R$}-matrix for affine {L}ie algebra~{$A^{(1)}_1$},
 \href{https://doi.org/10.1007/BF00777372}{\textit{Lett. Math. Phys.}}
 \textbf{27} (1993), 253--264.

\bibitem{book-M}
Majid S., Foundations of quantum group theory,
 \href{https://doi.org/10.1017/CBO9780511613104}{Cambridge University Press},
 Cambridge, 1995.

\bibitem{MN91}
Mezincescu L., Nepomechie R.I., Fusion procedure for open chains,
 \href{https://doi.org/10.1088/0305-4470/25/9/024}{\textit{J.~Phys.~A}}
 \textbf{25} (1992), 2533--2543.

\bibitem{NP15}
Nepomechie R.I., Pimenta R.A., Fusion for {A}d{S}/{CFT} boundary {S}-matrices,
 \href{https://doi.org/10.1007/JHEP11(2015)161}{\textit{J.~High Energy Phys.}}
 \textbf{2015} (2015), no.~11, 161, 16~pages,
 \href{http://arxiv.org/abs/1509.02426}{arXiv:1509.02426}.

\bibitem{RSV16}
Reshetikhin N., Stokman J., Vlaar B., Boundary quantum
 {K}nizhnik--{Z}amolodchikov equations and fusion,
 \href{https://doi.org/10.1007/s00023-014-0395-4}{\textit{Ann. Henri
 Poincar\'e}} \textbf{17} (2016), 137--177,
 \href{http://arxiv.org/abs/1404.5492}{arXiv:1404.5492}.

\bibitem{Skly88}
Sklyanin E.K., Boundary conditions for integrable quantum systems,
 \href{https://doi.org/10.1088/0305-4470/21/10/015}{\textit{J.~Phys.~A}}
 \textbf{21} (1988), 2375--2389.

\bibitem{Ter03}
Terwilliger P., Two relations that generalize the {$q$}-{S}erre relations and
 the {D}olan--{G}rady relations, in Physics and {C}ombinatorics~1999
 ({N}agoya), \href{https://doi.org/10.1142/9789812810199_0013}{World
 Scientific Publishing}, River Edge, NJ, 2001, 377--398.

\bibitem{Ter21}
Terwilliger P., The alternating central extension of the {$q$}-{O}nsager
 algebra, \href{https://doi.org/10.1007/s00220-021-04171-2}{\textit{Comm.
 Math. Phys.}} \textbf{387} (2021), 1771--1819,
 \href{http://arxiv.org/abs/2103.03028}{arXiv:2103.03028}.

\bibitem{Ter21d}
Terwilliger P., A conjecture concerning the {$q$}-{O}nsager algebra,
 \href{https://doi.org/10.1016/j.nuclphysb.2021.115391}{\textit{Nuclear
 Phys.~B}} \textbf{966} (2021), 115391, 26~pages,
 \href{http://arxiv.org/abs/2101.09860}{arXiv:2101.09860}.

\bibitem{Ter21c}
Terwilliger P., The {$q$}-{O}nsager algebra and its alternating central
 extension,
 \href{https://doi.org/10.1016/j.nuclphysb.2022.115662}{\textit{Nuclear
 Phys.~B}} \textbf{975} (2022), 115662, 37~pages,
 \href{http://arxiv.org/abs/2106.14041}{arXiv:2106.14041}.

\bibitem{Ter21b}
Terwilliger P., The compact presentation for the alternating central extension
 of the {$q$}-{O}nsager algebra,
 \href{https://doi.org/10.1016/j.jpaa.2023.107408}{\textit{J.~Pure Appl.
 Algebra}} \textbf{227} (2023), 107408, 22~pages,
 \href{http://arxiv.org/abs/2103.11229}{arXiv:2103.11229}.

\bibitem{Tolstoy1991}
Tolstoy V.N., Khoroshkin S.M., The universal {$R$}-matrix for quantum untwisted
 affine {L}ie algebras,
 \href{https://doi.org/10.1007/BF01077085}{\textit{Funct. Anal. Appl.}}
 \textbf{26} (1992), 69--71.

\bibitem{Tom}
tom Dieck T., H\"aring-Oldenburg R., Quantum groups and cylinder braiding,
 \href{https://doi.org/10.1515/form.10.5.619}{\textit{Forum Math.}}
 \textbf{10} (1998), 619--639.

\bibitem{T19}
Tsuboi Z., Generic triangular solutions of the reflection
 equation:~{$U_q(\widehat{{\rm sl}_2})$} case,
 \href{https://doi.org/10.1088/1751-8121/ab8853}{\textit{J.~Phys.~A}}
 \textbf{53} (2020), 225202, 16~pages,
 \href{http://arxiv.org/abs/1912.12808}{arXiv:1912.12808}.

\bibitem{T20}
Tsuboi Z., Universal {B}axter {TQ}-relations for open boundary quantum
 integrable systems,
 \href{https://doi.org/10.1016/j.nuclphysb.2020.115286}{\textit{Nuclear
 Phys.~B}} \textbf{963} (2021), 115286, 27~pages,
 \href{http://arxiv.org/abs/2010.09675}{arXiv:2010.09675}.

\bibitem{VW20}
Vlaar B., Weston R., A~{$Q$}-operator for open spin chains~{I}. {B}axter's~{TQ}
 relation,
 \href{https://doi.org/10.1088/1751-8121/ab8854}{\textit{J.~Phys.~A}}
 \textbf{53} (2020), 245205, 47~pages,
 \href{http://arxiv.org/abs/2001.10760}{arXiv:2001.10760}.

\bibitem{Ya67}
Yang C.N., Some exact results for the many-body problem in one dimension with
 repulsive delta-function interaction,
 \href{https://doi.org/10.1103/PhysRevLett.19.1312}{\textit{Phys. Rev. Lett.}}
 \textbf{19} (1967), 1312--1315.

\bibitem{YNZ06}
Yang W.-L., Nepomechie R.I., Zhang Y.-Z., {$Q$}-operator and {$T$}-{$Q$} relation
 from the fusion hierarchy,
 \href{https://doi.org/10.1016/j.physletb.2005.12.022}{\textit{Phys. Lett.~B}}
 \textbf{633} (2006), 664--670,
 \href{http://arxiv.org/abs/hep-th/0511134}{arXiv:hep-th/0511134}.

\bibitem{FZ80}
Zamolodchikov A.B., Fateev V.A., Model factorized {$S$}-matrix and an integrable
 spin-$1$ Heisenberg chain, \textit{Sov.~J. Nucl. Phys.} \textbf{32} (1980),
 581--590.

\bibitem{ZZ78}
Zamolodchikov A.B., Zamolodchikov A.B., Factorized {$S$}-matrices in two
 dimensions as the exact solutions of certain relativistic quantum field
 theory models,
 \href{https://doi.org/10.1016/0003-4916(79)90391-9}{\textit{Ann. Physics}}
 \textbf{120} (1979), 253--291.

\end{thebibliography}
\end{document}